\documentclass[12pt,a4paper,titlepage,twoside]{book}
\usepackage[a4paper,twoside,bindingoffset=1cm]{geometry}
\usepackage[latin1]{inputenc}
\usepackage[english]{babel}
\usepackage{rotating}
\usepackage{dsfont}
\usepackage{fancyhdr}
\usepackage{indentfirst} 
\usepackage{graphicx}
\usepackage{epstopdf}
\usepackage{newlfont}
\usepackage{amssymb}
\usepackage{amsmath}
\usepackage{latexsym}
\usepackage{amsthm}
\usepackage{amsfonts}
\usepackage{mathrsfs}
\usepackage{ushort}  
\usepackage{array}
\usepackage{enumerate}
 \usepackage[numbers]{natbib}
\PassOptionsToPackage{hyphens}{url}\usepackage{hyperref}
\usepackage{hyphenat}
\renewcommand{\chaptermark}[1]{\markboth{\chaptername\ \thechapter.\ \MakeUppercase{#1}}{} }
\renewcommand{\sectionmark}[1]{\markright{\thesection.\ #1}{} }

\newtheorem{theorem}{\bf Theorem}[chapter]
\newtheorem{assumption}{\bf Assumption}[chapter]
\newtheorem{proposition}{\bf Proposition}[chapter]
\newtheorem{lemma}[theorem]{\bf Lemma}
\newtheorem{corollary}{\bf Corollary}[chapter]
\newtheorem{definition}[theorem]{\bf Definition}
\newtheorem{remark}[theorem]{\bf Remark}
\newtheorem{example}[theorem]{\bf Example}

\newcommand{\beq}{\begin{equation}}
\newcommand{\bea}[1]{\begin{array}{#1} }
\newcommand{\eeq}{ \end{equation}}
\newcommand{\ea}{ \end{array}}
\newcommand{\btab}[1]{\begin{tabular}[t]{#1}}
\newcommand{\etab}{\end{tabular}}

\newcommand{\ds}{\displaystyle}

\newcommand{\ud}{\mathrm{d}}
\newcommand{\incrx}[1]{\left(#1(t_{i+1}^n)-#1(t_i^n)\right)}

\newcommand{\fqv}[1] {finite quadratic variation along $#1$ }
\newcommand{\prop}[1] {Proposition \ref{prop:#1}}
\newcommand{\thm}[1] {Theorem \ref{thm:#1}}
\newcommand{\cor}[1] {Corollary \ref{cor:#1}}
\newcommand{\defin}[1] {Definition \ref{def:#1}}
\newcommand{\ass}[1] {Assumption \ref{ass:#1}}
\newcommand{\rmk}[1] {Remark \ref{rmk:#1}}

\newcommand{\lem}[1] {Lemma \ref{lem:#1}}
\newcommand{\Sec}[1] {Section \ref{sec:#1}}
\newcommand{\chap}[1] {Chapter \ref{chap:#1}}
\newcommand{\eq}[1] {\eqref{eq:#1}}
\newcommand{\pa}[1] {\frac{\partial}{\partial #1}}
\newcommand{\ppa}[1] {\frac{\partial^2}{\partial #1^2}}
\newcommand{\norm}[1]{\left\lVert #1 \right\rVert} 
\newcommand{\abs}[1]{\left\lvert #1 \right\rvert} 
\newcommand{\pqv}[1]{\!\left\langle #1 \right\rangle\!} 
\newcommand{\Et}[2]{\EE\left[ #1 \mid \F_{#2}\right]} 
\newcommand{\conc}[1]{\underset{#1}{\oplus}}

\newcommand{\HRule}{\rule{\linewidth}{0.5mm}}

\def \g {{\gamma}}
\def \eps {{\varepsilon}}
\def \x {{\xi}}
\def \t {{\tau}}
\def \n {{\nu}}
\def \m {{\mu}}
\def \y {{\eta}}
\def \th {{\theta}}
\def \z {{\zeta}}
\def \a {{\alpha}}
\def \d {{\delta}}
\def \k {{\kappa}}
\def \b {{\beta}}
\def \s {{\sigma}}
\def \w {{\omega}}
\def \e {{\epsilon}}
\def \r {{\varrho}}
\def \p {{\varphi}}
\def\l {\lambda}

\def \O {{\Omega}}
\def \De {{\Delta}}
\def \G {{\Gamma}}
\def \L {{\Lambda}}

\def \R  {{\mathbb {R}}} 
\def \N {{\mathcal {N}}}
\def \C {{\mathcal {C}}}

\def \B {{\mathcal {B}}}
\def \A {\mathcal{A}}
\def \M {\mathcal{M}}
\def \P {{\cal{P}}}
\def \F {\mathcal{F}}  
\def \H {\mathcal{H}}
\def \K {\mathcal{K}}
\def \S {\mathcal{S}}
\def \Si {\Sigma}
\def \I {\mathcal{I}}
\def \W {\mathcal{W}}  
\def \WW {\mathbb{W}}

\def \NN {{\mathbb {N}}}
\def \VV  {{\mathbb {V}}}

\def \PP  {{\mathbb {P}}}
\def \QQ  {{\mathbb {Q}}}
\def \EE  {{\mathbb {E}}}  
\def \FF  {{\mathbb {F}}}  
\def \CC {\mathbb{C}}
\def \BB {\mathbb{B}}
\def \Ft {\left(\mathcal{F}_t\right)_{t\in[0,T]}}  
\def \Cloc {\mathbb{C}^{1,2}_{loc}}
\def \Cb {\mathbb{C}^{1,2}_{b}}

\def \LL {\mathcal{L}^2}
\def \DT {D([0,T],\R^d)}
\def \Ft {\left(\mathcal{F}_t\right)_{t\in[0,T]}}  
\def \Cloc {\mathbb{C}^{1,2}_{loc}}
\def \Cb {\mathbb{C}^{1,2}_{b}}

\def \DT {D([0,T],\R^d)}

\def \ind {{\mathds{1}}}

\def \ss {{$\s$-algebra }}

\def \Limn {\lim_{n\rightarrow \infty}}
\def \limn {\xrightarrow[n\rightarrow \infty]{}}

\def \limnp {\xrightarrow[n\rightarrow \infty]{\ \PP\ }}
\def \limucp {\xrightarrow{\ ucp(\PP)\ }}

\def \limmed {\lim\mathrm{med}_n}

\def \zs {, \qquad t\in[0,T]}
\def \MC {{Monte Carlo }}
\def \ito {It\^o}
\def \follmer {F\"{o}llmer}
\def \cadlag {c\`adl\`ag}
\def \caglad {c\`agl\`ad}

\def \lf {\left(}
\def \rg {\right)}
\def \ps {$(\O,\F,(\F_t)_{0\leq t\leq T},\PP)$}
\def \OT {$[0,T]$}
\def \supp {\mathrm{supp}}
\def \over {\overline}
\def \under {\underline}

\def \naf {non-anticipative functional}
\def \vd {\nabla_{\w}}

\def \hd {\mathcal{D}}
\def \dinf {\textrm{d}_{\infty}}

\def \tr {\mathrm{tr}}

\def \rint {\!\!\!\!\!{\phantom{\int}}^{(+)}\!\!\!\!\int}
\def \lint {\!\!\!\!\!{\phantom{\int}}^{(-)}\!\!\!\!\int}

\def \e {{\varepsilon}}
\def \n {{\nu}}
\def \m {{\mu}}

\def \zs {, \qquad t\in[0,T]}
\def \caratt {{\mathds{1}}}

\def \k {{\kappa}}

\def \H {{\tilde{H}}}

\def \MC {{Monte Carlo }}

\def \K {{\nu}}

\def \w {{\omega}}
\def \N {{\mathbb {N}}}

\def \x {{\xi}}
\def \e {{\varepsilon}}
\def \eps {{\varepsilon}}
\def \r {{\varrho}}

\def \tilde {\widetilde}

\def\pa {\partial}
\def \ss {{$\s$-algebra }}

\def \I {\mathcal{I}}
\def \P {{\cal{P}}}
\def \B {\mathscr{B}}

\def \à {{\`a }}
\def \è {{\`e }}
\def \ò {{\`o }}
\def \ù {{\`u }}

\def \xbar {{\bar{x}}}
\def \rle {{\bar{r}}}

\begin{document}

\frontmatter

\newgeometry{margin=3cm,centering}

\begin{titlepage}
\thispagestyle{empty}
\begin{center}

\vspace*{-1cm}

\includegraphics[trim=0 0.2cm 0 0, height=1.8cm,keepaspectratio=true]{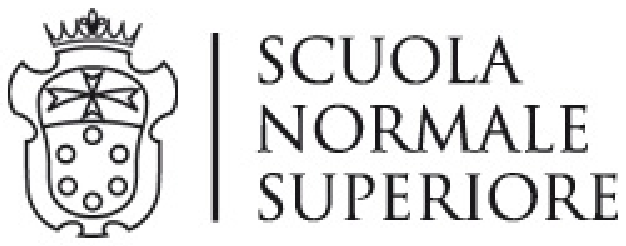}\hfill
\includegraphics[trim=0 -0.04cm 0 0, height=1.4cm,keepaspectratio=true]{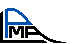}\hfill
\includegraphics[trim=0 0.8cm 0 -2cm, height=2.5cm,keepaspectratio=true]{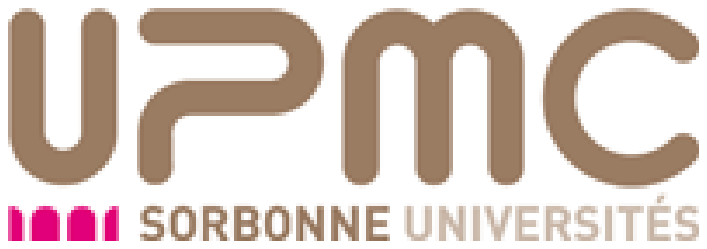}
\vspace{0.3cm}

\textsc{\LARGE Scuola Normale Superiore di Pisa}\\
{\it Classe di Scienze}\\[0.3cm]
\textsc{\LARGE Universit\'e Pierre et Marie Curie}\\
{\'Ecole Doctorale de Sciences Math\'ematiques de Paris Centre}\\
{\it Laboratoire de Probabilit\'es et Mod\`eles Al\'eatoires}

\vfill

\HRule

{\Large \bf Pathwise functional calculus\\and applications to continuous-time finance\\[0.5cm]
Calcul fonctionnel non-anticipatif\\et application en finance}

\HRule

\vfill

{\large  Candia Riga\\
Tesi di perfezionamento in Matematica per la Finanza\\
Th\`ese de doctorat de Math\'ematiques\\[0.3cm]
Dirig\'e par Rama Cont, co-dirig\'e par Sara Biagini\\
}
\vfill
Rapporteurs: Hans F\"ollmer et St\'ephane Cr\'epey
\vfill
Pr\'esent\'ee et soutenue publiquement le 26/06/2015, devant un jury compos\'e de:
\vfill
\begin{tabular}{lll}
Ambrosio Luigi&Scuola Normale Superiore di Pisa&Examinateur\\
Biagini Sara&Universit\`a di Pisa&Co-Directeur\\
Cont Rama&Universit\'e Pierre et Marie Curie&Directeur\\
Cr\'epey St\'ephane&Universit\'e d'Evry&Rapporteur\\
Marmi Stefano&Scuola Normale Superiore di Pisa&Examinateur\\
Tankov Peter	&Universit\'e Denis Diderot&Examinateur
\end{tabular}

\end{center}
\end{titlepage}

\restoregeometry

\chapter*{Abstract}
\markboth{Abstract}{}

This thesis develops a mathematical framework for the analysis of continuous\hyp time trading strategies which, in contrast to the classical setting of continuous-time finance, does not rely on stochastic integrals or other probabilistic notions.

Using the recently developed \lq non-anticipative functional calculus\rq, we first develop a pathwise definition of the gain process for a large class of continuous-time trading strategies which include the important class of delta-hedging strategies, as well as a pathwise definition of the self-financing condition.

Using these concepts, we propose a framework for analyzing the performance and  robustness of delta-hedging strategies for path-dependent derivatives across a given set of scenarios. Our setting allows for general path-dependent payoffs and does not require any probabilistic assumption on the dynamics of the underlying asset, thereby extending previous results on robustness of hedging strategies in the setting of diffusion models. We obtain a pathwise formula for the hedging error for a general path-dependent derivative and provide sufficient conditions ensuring the robustness of the delta hedge. We show in particular that robust hedges may be obtained in a large class of continuous exponential martingale models under a vertical convexity condition on the payoff functional. Under the same conditions, we show that discontinuities in the underlying asset always deteriorate the hedging performance. These results are applied to the case of Asian options and barrier options.

The last chapter, independent of the rest of the thesis,  proposes a novel method, jointly developed with Andrea Pascucci and Stefano Pagliarani, for analytical approximations in local volatility models with L\'evy jumps. The main result is an expansion of the characteristic function in a local L\'evy model, which
is worked out in the Fourier space by consid\'ering the adjoint formulation of the pricing problem.
Combined with standard Fourier methods, our result provides efficient and accurate pricing
formulae. In the case of Gaussian jumps, we also derive an explicit approximation of the
transition density of the underlying process by a heat kernel expansion; the approximation is
obtained in two ways: using PIDE techniques and working in the Fourier space. Numerical tests
confirm the effectiveness of the method.

\chapter*{Sommario}
\markboth{Sommario}{}

Questa tesi sviluppa un approccio \lq per traiettorie\rq\ alla modellizzazione dei mercati finanziari in tempo continuo, senza fare ricorso a delle ipotesi probabilistiche o a dei modelli stocastici. Lo strumento principale utilizzato in questa tesi \`e il calcolo funzionale non-anticipativo, una teoria analitica che sostituisce il calcolo stocastico solitamente utilizzato in finanza matematica.

Cominciamo nel Capitolo 1 introducendo la teoria di base del calcolo funzionale non-anticipativo e i suoi principali risultati che utilizzeremo nel corso della tesi.
Il Capitolo 2 mostra in dettaglio la versione probabilistica di tale calcolo, soprannominata \emph{Calcolo di \ito\ funzionale}, e mostra come essa permetta di estendere i risultati classici sulla valutazione e la replicazione dei derivati finanziari al caso di opzioni dipendenti dalla traiettoria dei prezzi. Inoltre illustriamo la relazione tra le equazioni alle derivate parziali con coefficienti dipendenti dal cammino e le equazioni differenziali stocastiche \lq backward\rq. Infine prendiamo in consid\'erazione altre nozioni deboli di soluzione a tali equazioni alle derivate parziali dipendenti dal cammino, utilizzate nella letteratura nel caso in cui non esistano soluzioni classiche.

In seguito, nel Capitolo 3, costruiamo un modello di mercato finanziario in tempo continuo, senza ipotesi probabilistiche e con un orizzonte temporale finito, dove i tempi di transazione sono rappresentati da una sequenza crescente di partizioni temporali, il cui passo converge a 0.
Identifichiamo le traiettorie \lq plausibili\rq\ con quelle che possiedono una variazione quadratica finita, nel senso di F\"ollmer, lungo tale sequenza di partizioni.
Tale condizione di plausibilit\`a sull'insieme dei cammini ammissibili rispetta il punto di vista delle condizioni \lq per traiettorie\rq\ di non-arbitraggio.

Completiamo il quadro introducendo una nozione \lq per traiettorie\rq\  di strategie auto-finanzianti su un insieme di traiettorie di prezzi. Queste strategie sono definite come limite di strategie semplici e auto-finanzianti, i cui tempi di transizione appartengono alla sequenza di partizioni temporali fissata.
Identifichiamo una classe speciale di strategie di trading che dimostriamo essere auto-finanzianti e il cui guadagno pu\`o essere calcolato traiettoria per traiettoria come limite di somme di Riemann.
Inoltre, presentiamo un risultato di replicazione per traiettorie e una formula analitica esplicita per stimare l'errore di replicazione.
Infine, definiamo una famiglia di operatori integrali indicizzati sui cammini come delle isometrie tra spazi normati completi.

Il Capitolo 4 utilizza questo quadro teorico per proporre un'analisi per traiettorie delle strategie di replicazione dinamica. Ci interessiamo in particolare alla robustezza della loro performance nel caso della replicazione di derivati dipendenti dalla traiettoria dei prezzi e monitorati in tempo continuo. Supponiamo che l'agente di mercato utilizzi un modello di martingala esponenziale di quadrato integrabile per calcolare il prezzo e il portafoglio di replicazione; analizziamo quindi la performance della strategia di delta-hedging quando viene applicata alla traiettoria realizzata dei prezzi del sottostante piuttosto che a una dinamica stocastica.

Innanzitutto, consid\'eriamo il caso in cui disponiamo di un funzionale di prezzo regolare e mostriamo che la replicazione tramite delta-hedging \`e robusta se la derivata verticale seconda del funzionale di prezzo ha lo stesso segno della differenza tra la volatilit\`a del modello e la volatilit\`a realizzata dei prezzi di mercato. Otteniamo cos\`i una formula esplicita per l'errore di replicazione data una traiettoria. Questa formula \`e l'analogo per traiettorie del risultato ottenuto da EL Karoui et al (1997) e la generalizza al caso dipendente dalla traiettoria, senza ricorrere a delle ipotesi probabilistiche o alla propiet\`a di Markov circa la dinamica reale dei prezzi di mercati. Presentiamo infine delle codizioni sufficienti affinch\'e il funzionale di valutazione abbia la regolarit\`a richiesta per tali risultati sullo spazio dei cammini continui.

Questi risultati permettono di analizzare la robustezza delle strategie di replicazione dinamica.
Forniamo una condizione sufficiente sul funzionale di payoff che assicura la positivit\`a della derivata verticale seconda del funzionale di prezzo, ovvero la convessit\`a di una certa funzione reale.
Analizziamo ugualmente il contributo di salti della traiettoria dei prezzi all'errore di replicazione ottenuto agendo sul mercato secondo la strategia di delta-hedging. Osserviamo che le discontinuit\`a deteriorano la performance della replicazione.
Nel caso speciale di un modello Black-Scholes generalizzato utilizzato dall'agente, se il derivato venduto ha un payoff monitorato a tempo discreto, allora il funzionale di prezzo \`e localmente regolare su tutto lo spazio dei cammini continui stoppati e le sue derivate, verticale e orizzontale, sono date in forma esplicita.
consid\'eriamo anche il caso di un modello con volatilit\`a dipendente dalla traiettoria dei prezzi, il modello Hobson-Rogers, e mostriamo come il problema di pricing sia anche in quel caso riconducibile all'equazione di pricing universale introdotta nel secondo capitolo.
Infine, mostriamo qualche esempio di applicazione della nostra analisi, precisamente la replicazione di opzioni asiatiche e barriera.

L'ultimo capitolo \`e uno studio indipendente dal resto della tesi, sviluppato insieme ad Andrea Pascucci e Stefano Pagliarani, in cui proponiamo un nuovo metodo di approssimazione analatica in modelli a volatilit\`a locale con salti di tipo L\'evy. Il risultato principale \`e un'espansione in serie della funzione caratteristica in un modello di L\'evy locale, ottenuta nello spazio di Fourier consid\'erando la formulazione aggiunta del problema di \lq pricing\rq. Congiuntamente ai metodi di Fourier standard, il nostro risultato fornisce formule di \lq pricing\rq\ efficienti e accurate. Nel caso di salti gaussiani, deriviamo anche un'approssimazione esplicita della densit\`a di transizione del processo sottostante tramite un'espansione con nucleo del calore; tale approssimazione \`e ottenuta in due modi: usando tecniche PIDE e lavorando nello spazio di Fourier. Test numerici confermano l'efficacit\`a del metodo.

\chapter*{R\'esum\'e}
\markboth{R\'esum\'e}{}

Cette th\`ese d\'eveloppe une approche trajectorielle pour  la mod\'elisation des march\'es financiers en temps continu, sans faire appel \`a des hypoth\`eses probabilistes ou \`a des mo\-d\`eles  stochastiques. L'outil principal dans cette th\`ese est le calcul fonctionnel non-anticipatif, un  cadre analytique qui remplace le calcul stochastique habituellement utilis\'e en finance
math\'ematique.

Nous commen\c cons dans le Chapitre 1 par introduire la th\'eorie de base du calcul fonctionnel non-anticipatif et ses principaux r\'esultats que nous utilisons tout au long de la th\`ese.
Le Chapitre 2 d\'etaille la contrepartie probabiliste de ce calcul, le Calcul d'\ito\ fonctionnel, et montre comment ce calcul permet d'\'etendre les r\'esultats classiques sur l'\'evaluation et la couverture des produits d\'eriv\'es au cas des options avec une d\'ependance trajectorielle. Par ailleurs, nous d\'ecrivons la relation entre les \'equations aux d\'eriv\'ees partielles avec coefficients d\'ependant du chemin et les \'equations diff\'erentielles stochastiques r\'etrogrades. Finalement, nous consid\'erons d'autres notions plus faibles de solution \`a ces \'equations aux d\'eriv\'ees partielles avec coefficients d\'ependant du chemin, lesquelles sont utilis\'ees dans la litt\'erature au cas o\`u des solutions classiques n'existent pas.

Ensuite nous mettons en place, dans le Chapitre 3, un mod\'ele de march\'e financier en temps continu, sans hypoth\`eses probabilistes et avec un horizon fini o\`u les temps de transaction sont repr\'esent\'es par  une suite embo\^it\'ee de partitions dont le pas converge vers $0$.
Nous proposons une condition de plausibilit\'e sur l'ensemble des chemins admissibles du point de vue des conditions trajectorielles de non-arbitrage.
Les trajectoires \lq plausibles\rq\ sont r\'ev\'el\'ees avoir une variation quadratique finie, au sens de F\"ollmer, le long de cette suite de partitions.

 Nous compl\'etons le cadre en introduisant une notion trajectorielle de strat\'egie  auto-finan\c cante sur un ensemble de trajectoires de prix.
\\Ces strat\'egies sont d\'efinies comme des limites de strat\'egies simples et auto-finan\c cantes, dont les temps de transactions appartiennent \`a la suite de partitions temporelles fix\'ee.
Nous identifions une classe sp\'eciale de strat\'egies de trading que nous prouvons \^etre auto-finan\c cantes et dont le gain peut \^etre calcul\'e trajectoire par trajectoire comme limite de sommes de Riemann.
Par ailleurs, nous pr\'esentons un r\'esultat de r\'eplication trajectorielle et une formule analytique explicite pour estimer l'erreur de couverture.
Finalement nous d\'efinissons une famille d'op\'erateurs int\'egrals trajectoriels (index\'es par les chemins) comme des isom\'etries entre des espaces norm\'es complets.

Le Chapitre 4 emploie ce cadre th\'eorique pour proposer une analyse trajectorielle des strat\'egies de couverture dynamique.
Nous nous int\'eressons en particulier \`a la robustesse de leur performance dans la couverture de produits d\'eriv\'es \emph{path-dependent} monitor\'es en temps continu. Nous supposons que l'agent utilise un mod\`ele de martingale exponentielle de carr\'e int\'egrable pour calculer les prix et les portefeuilles de couverture, et nous analysons la performance de la strat\'egie delta-neutre lorsqu'elle est appliqu\'ee \`a la trajectoire du prix sous-jacent r\'ealis\'e plut\^ot qu'\`a une dynamique stochastique.
D'abord nous consid\'erons le cas o\`u nous disposons d'une fonctionnelle de prix r\'eguli\`ere et nous montrons que la couverture delta-neutre est robuste si la d\'eriv\'ee verticale seconde de la fonctionnelle de prix est du m\^eme signe que la diff\'erence entre la volatilit\'e du mod\`ele et la volatilit\'e r\'ealis\'ee du march\'e. Nous obtenons aussi une formule explicite pour l'erreur de couverture sur une trajectoire donn\'ee. Cette formule est l'analogue trajectorielle du r\'esultat de El Karoui et al (1997) et le g\'en\'eralise  au cas \emph{path-dependent}, sans faire appel \`a des hypoth\'eses probabilistes ou \`a la propri\'et\'e de Markov. Enfin nous pr\'esentons des conditions suffisantes  pour que la fonctionnelle d'\'evaluation ait la r\'egularit\'e requise pour ces r\'esultats sur l'espace des chemins continus.

Ces r\'esultats permettent d'analyser la robustesse  des strat\'egies de couverture dynamiques.
Nous fournissons une condition suffisante sur la fonctionnelle de payoff qui assure la positivit\'e de la d\'eriv\'e verticale seconde de la fonctionnelle d'\'evaluation, i.e. la convexit\'e d'une certaine fonction r\'eelle.
Nous analysons \'egalement la contribution des sauts de la trajectoire des prix \`a l'erreur de couverture obtenue en \'echangeant sur le march\'e selon la strat\'egie delta-neutre. Nous remarquons que les discontinuit\'es d\'et\'eriorent la performance de la couverture. Dans le cas sp\'ecial d'un mod\`ele Black-Scholes g\'en\'eralis\'e utilis\'e par l'agent, si le produit d\'eriv\'e vendu a un payoff monitor\'e en temps discret, alors la fonctionnelle de prix est localement r\'eguli\`ere sur tout l'espace des chemins continus arr\^et\'es et ses d\'eriv\'ees verticale et horizontale sont donn\'ees dans une forme explicite. Nous consid\'erons aussi le cas d'un mod\`ele avec volatilit\'e d\'ependante de la trajectoire des prix, le mod\`ele Hobsons-Rogers, et nous montrons comment le probl\`eme de \lq pricing\rq\ peut encore \^etre r\'eduit \`a l'\'equation universelle introduite dans le Chapitre 2.
Finalement, nous montrons quelques applications de notre analyse, notamment la couverture des options Asiatiques et barri\`eres.

Le dernier chapitre, ind\'ependant du reste de la th\`ese,  est une \'etude en collaboration avec Andrea Pascucci and Stefano Pagliarani, o\`u nous proposons une nouvelle m\'ethode pour l'approximation analytique dans des mo\-d\`eles \`a volatilit\'e locale avec des sauts de type L\'evy.
Le r\'esultat principal est un d\'eveloppement asymptotique de la fonction caract\'eristique dans un mod\`ele de L\'evy local, qui est obtenu dans l'espace de Fourier en consid\'erant la formulation adjointe du probl\`eme de \lq pricing\rq.
Associ\'e aux m\'ethodes de Fourier standard, notre r\'esultat fournit des approximations pr\'ecises du prix.
Dans le cas de sauts gaussiens, nous d\'erivons aussi une approximation explicite de la densit\'e de transition du processus sous-jacent \`a l'aide d'une expansion avec noyau de la chaleur; cette approximation est obtenue de deux fa\c cons: en utilisant des techniques PIDE et en travaillant dans l'espace de Fourier.
Des test num\'eriques confirment l'efficacit\'e de la m\'ethode.

\chapter*{Acknowledgments}
\markboth{Acknowledgments}{}

First, I would like to thank my advisor Professor Rama Cont for giving me the valuable opportunity to be under his guidance and to join the team at the \emph{Laboratoire de Probabilit\'e et mod\`eles al\'eatoires} in Paris, for sharing with me his precious insights, and for encouraging my continuation in the academic research. He also tought me to have a more independent attitude to research.
I would also like to thank my co-supervisor Professor Sara Biagini for her patient and precious support at the beginning of my research project in Pisa.

I really thank Professors Stefano Marmi, Fabrizio Lillo and Mariano Giaquinta for awarding me the PhD fellowship at the \emph{Scuola Normale Superiore} in Pisa, and Stefano for his helpfulness as my tutor and for the financial support.
I am very thankful to Professors Franco Flandoli and Maurizio Pratelli who welcomed me at my arrival in Pisa, let me join their seminars at the department of Mathematics, and were always available for helpful discussions and advises.

I am very thankful to my master thesis advisor, Professor Andrea Pascucci, for guiding me into research, for his sincere advises and his constant availability. I thank as well Professor Pierluigi Contucci for his helpful encouragement and his collaboration, and Professors Paolo Foschi and Hykel Hosni for inviting me to present my ongoing research at the workshops organized respectively at \emph{Imperial College of London} and at \emph{Centro di Ricerca Matematica Ennio De Giorgi} in Pisa.

I would like to thank all my colleagues and friends that shared these three years of PhD with me in Pisa, Paris and Cambridge, especially Adam, Giacomo, Mario, Dario, Giovanni, Laura A., Alessandro, Olga, Fernando, Pierre-Louis, Eric A., Nina, Alice, Tibault, Nils, Hamed. I especially thank Professor Rama Cont's other PhD students, for their friendship and reciprocal support: Yi, Eric S., Laura S. and Anna. 

I thank all the administrative teams at \emph{Scuola Normale Superiore} in Pisa, at \emph{Laboratoire de Probabilit\'e et mod\`eles al\'eatoires} in Paris and at the \emph{Newton Institute for Mathematical Sciences} in Cambridge (UK), for being always very helpful.
I also thank the French Embassy in Rome for financing my staying in Paris and the \emph{Isaac Newton Institute} for financing my staying in Cambridge.

I am grateful to Professors Ashkan Nikeghbali and Jean-Charles Rochet for inviting me at the University of Zurich, for their interest and trust in my research, and for doing me the honour of welcoming me in their team.
I also thank Delia for her friendship and support since my arrival in Zurich.

I am very grateful to Professors Hans F\"ollmer and St\'ephane Cr\'epey for accepting to be referees for this thesis, for helpful comments, and for the very nice and supporting reports.
I would like to thank as well Professors Luigi Ambrosio and Peter Tankov for agreeing to be examiners for my defense.

Last but not least, I am infinitely thankful to my parents who have always supported me, my brother Kelvin for making me stronger, and my significant other Thomas for his love and for always inciting me to follow the right path for my studies even when this implied a long distance between us.

\markboth{\MakeUppercase{Contents}}{}
\tableofcontents

\chapter*{Notation}
\addcontentsline{toc}{chapter}{Notation}
\markboth{Notation}{}

\subsubsection{Acronyms and abbreviations}
\begin{description}
  \item[\cadlag] = right continuous with left limits
  \item[\caglad] = left continuous with right limits
  \item[SDE] = stochastic differential equation
  \item[BSDE] = backward stochastic differential equation
  \item[PDE] = partial differential equation
  \item[FPDE] = functional partial differential equation
  \item[PPDE] = path-dependent partial differential equation
  \item[EMM] = equivalent martingale measure
  \item[NA] = no-arbitrage condition
  \item[NA1] = ``no arbitrage of the first kind'' condition
  \item[NFL] = ``no free lunch'' condition
  \item[NFLVR] = ``no free lunch with vanishing risk'' condition
  \item[s.t.] = such that
  \item[a.s.] = almost surely
  \item[a.e.] = almost everywhere
  \item[e.g.] = exempli gratia $\equiv$ example given
  \item[i.e.] = id est $\equiv$ that is
\end{description}

\subsubsection{Basic mathematical notation}
\begin{description}
  \item[$\R^d_+$] = positive orthant in $\R^d$
  \item[\mbox{$\DT$ (resp. $D([0,T],\R^d_+)$)}] = space of \cadlag\ functions from $[0,T]$ to $\R^d$ (respectively $\R^d_+$), $d\in\NN$
  \item[\mbox{$C([0,T],\R^d_+)$ (resp. $C([0,T],\R^d_+)$)}] = space of continuous functions from $[0,T]$ to $\R^d$ (respectively $\R^d_+$), $d\in\NN$
  \item[$\S^d_+$] = set of symmetric positive-definite $d\times d$ matrices
  \item[$\FF=\Ft$] = natural filtration generated by the coordinate process 
  \item[$\FF^X=\Ft^X$] = natural filtration generated by a stochastic process $X$
  \item[$\EE^\PP$] = expectation under the probability measure $\PP$
  \item[$\xrightarrow{\ \PP\ }$] = limit in probability $\PP$
  \item[$\xrightarrow{\ ucp(\PP)\ }$] = limit in the topology defined by uniform convergence on compacts in probability $\PP$
  \item[$\cdot$] = scalar product in $\R^d$ (unless differently specified)
  \item[$\pqv{\cdot}$] = Frobenius inner product in $\R^{d\times d}$ (unless differently specified)
  \item[$\norm{\cdot}_\infty$] = sup norm in spaces of paths, e.g in $D([0,T],\R^d)$, $C([0,T],\R^d)$, $D([0,T],\R^d_+)$, $C([0,T],\R^d_+)$,\ldots
  \item[$\norm{\cdot}_p$] = $L^p$-norm, $1\leq p\leq\infty$
  \item[{$[\cdot]$} {($[\cdot,\cdot]$)}] = quadratic (co-)variation process
  \item[$\bullet$] = stochastic integral operator
  \item[$\tr$] = trace operator, i.e. $\tr(A)=\sum_{i=1}^dA_{i,i}$ where $A\in\R^{d\times d}$.
  \item[${}^t\!A$] = transpose of a matrix $A$
  \item[$x(t-)$] = left limit of $x$ at $t$, i.e. $\lim_{s\nearrow t}x(s)$
  \item[$x(t+)$] = right limit of $x$ at $t$, i.e. $\lim_{s\searrow t}x(s)$
  \item[$\De x(t)\equiv \De^-x(t)$] = left-side jump of $x$ at $t$, i.e. $x(t)-x(t-)$
  \item[$\De^+x(t)$] = right-side jump of $x$ at $t$, i.e. $x(t+)-x(t)$
  \item[$\partial_x$] = $\pa{x}$
  \item[$\partial_{xy}$] = $\frac{\partial^2}{\partial x \partial y}$
\end{description}

\subsubsection{Functional notation}
\begin{description}
\item[$x(t)$] = value of $x$ at time $t$, e.g. $x(t)\in\R^d$ if $x\in\DT$; 
\item[$x_t$] = $x(t\wedge\cdot)\in\DT$ the path of $x$ \lq stopped\rq\ at the time $t$; 
\item[$x_{t-}$] = $x\ind_{[0,t)}+x(t-)\ind_{[t,T]}\in\DT$; 
\item[$x_t^\d$] = $x_t+\d\ind_{[t,T]}\in\DT$ the \textit{vertical perturbation} -- of size and direction given by the vector  $\d\in\R^d$ -- of the path of $x$ stopped at $t$ over the future time interval $[t,T]$; 
\item[$\L_T$] = space of (\cadlag) stopped paths
\item[$\W_T$] = subspace of $\L_T$ of continuous stopped paths
\item[$\dinf$] = distance introduced on the space of stopped paths
  \item[$\hd F$] = horizontal derivative of a \naf\ $F$
  \item[$\vd F$] = vertical derivative of a \naf\ $F$
  \item[$\nabla_X$] = vertical derivative operator defined on the space of square-integrable $\F^X$-martingales
\end{description}

\mainmatter

\chapter*{Introduction}
\addcontentsline{toc}{chapter}{Introduction}
\markboth{\MakeUppercase{Introduction}}{}

\label{chap:intro}


The mathematical modeling of financial markets dates back to 1900, with the doctoral thesis~\citep{bachelier} of Louis Bachelier, who first introduce the Brownian motion as a model for the price fluctuation of a liquid traded financial asset. After a long break, in the mid-sixties, \citet{sam} revived Bachelier's intuition by proposing the use of geometric Brownian motion which, as well as stock prices, remains positive. This became soon a reference financial model, thanks to \citet{bs} and \citet{merton}, who derived closed formulas for the price of call options under this setting, later named the ``Black-Scholes model'', and introduced the novelty of linking the option pricing issue with hedging.
The seminal paper by \citet{harr-pliska} linked the theory of continuous-time trading to the theory of stochastic integrals, which has been used ever since as the standard setting in Mathematical Finance.

Since then, advanced stochastic tools have been used to describe the price dynamics of financial assets and its interplay with the pricing and hedging of financial derivatives contingent on the trajectory of the same assets.
The common framework has been to model the financial market as a filtered probability space \ps\ under which the prices of liquid traded assets are represented by stochastic processes $X=(X_t)_{t\geq0}$ and the payoffs of derivatives as functionals of the underlying price process. The probability measure $\PP$, also called \textit{real world, historical, physical} or \textit{objective} probability tries to capture the observed patterns and, in the equilibrium interpretation, represents the (subjective) expectation of the ``representative investor''. The objective probability must satisfy certain constraints of market efficiency, the strongest form of which requires $X$ to be a $\Ft$-martingale under $\PP$. However, usually, weaker forms of market efficiency are assumed by no-arbitrage considerations, which translate, by the several versions of the Fundamental Theorem of Asset Pricing (see \cite{schachermayer,schachermayerEMM} and references therein), to the existence of an equivalent \textit{martingale} (or \textit{risk-neutral}) \textit{measure} $\QQ$, that can be interpreted as the expectation of a ``risk-neutral investor'' as well as a consistent price system describing the market consensus.
The first result in this stream of literature (concerning continuous-time financial models) is found in \citet{ross78} in 1978, where the \emph{no-arbitrage} condition (NA) is formalized, then major advances came in 1979 by \citet{harr-kreps} and in 1981 by \citet{harr-pliska} and in particular by \citet{kreps81}, who introduced the \emph{no free lunch} condition (NFL), proven to be equivalent to the existence of a local martingale measure. More general versions of the Fundamental Theorem of Asset Pricing are due to \citet{ds94,ds98}, whose most general statement pertains to a general multi-dimensional semimartingale model and establishes the equivalence between the condition of \emph{no free lunch with vanishing risk} (NFLVR) and the existence of a sigma-martingale measure.
The model assumption that the price process behaves as a semimartingale comes from the theory of stochastic analysis, since it is known that there is a good integration theory for a stochastic process $X$ if and only if it is a semimartingale. At the same time, such assumption is also in agreement with the financial reasoning, as it is shown in~\cite{ds94} that a very weak form of no free lunch condition, assuring also the existence of an equivalent local martingale measure, is enough to imply that if $X$ is locally bounded then it must be a semimartingale under the objective measure $\PP$.
In \cite{ds-book} the authors present in a ``guided tour'' all important results pertaining to this theme.

The choice of an objective probability measure is not obvious and always encompasses a certain amount of model risk and model ambiguity. 
Recently, there has been a growing emphasis on the dangerous consequences of relying on a specific probabilistic model.
The concept of the so-called \textit{Knightian uncertainty}, introduced way back in 1921 by Frank Knight~\citep{knight}, while distinguishing between ``risk'' and ``uncertainty'', is still as relevant today and led to a new challenging research area in Mathematical Finance.
More fundamentally, the existence of a single objective probability does not even make sense, agreeing with the criticism raised by \citet{definetti31,definetti37}.

After the booming experienced in the seventies and eighties, in the late eighties the continuous-time modeling of financial markets evoked new interpretations that can more faithfully represent the economic reality.
In the growing flow of literature addressing the issue of model ambiguity, we may recognize two approaches:
\begin{itemize}
\item \textbf{model-independent}, where the single probability measure $\PP$ is replaced by a family $\P$ of plausible probability measures;
\item  \textbf{model-free}, that eliminates probabilistic a priori assumptions altogether, and relies instead on pathwise statements.
\end{itemize}

The first versions of the Fundamental Theorem of Asset Pricing under model ambiguity are presented in \cite{bouchard-nutz,bfm,abps} in discrete time, and \cite{sara-bkn} in continuous time, using a model-independent approach.

The model-free approach to effectively deal with the issue of model ambiguity also provides a solution to another problem affecting the classical probabilistic modeling of financial markets.
Indeed, in continuous-time financial models, the gain process of a self-financing trading strategy is represented as a stochastic integral.
However, despite the elegance of the probabilistic representation, some real concerns arise.
Beside the issue of the impossible consensus on a probability measure, the representation of the gain from trading lacks a pathwise meaning: while being a limit in probability of approximating Riemann sums, the stochastic integral does not have a well-defined value on a given \lq state of the world\rq.
This causes a gap in the use of probabilistic models, in the sense that it is not possible to compute the gain of a trading portfolio given the realized trajectory of the underlying asset price, which constitutes a drawback in terms of interpretation.

Beginning in the nineties, a new branch of the literature has addressed the issue of pathwise integration in the context of financial mathematics.


The approach of this thesis is probability-free. In the first part, we set up a framework for continuous-time trading where everything has a pathwise characterization. This purely analytical structure allows us to effectively deal with the issue of model ambiguity (or Knightian uncertainty) and the lack of a path-by-path computation of the gain of trading strategies.

A breakthrough in this direction was the seminal paper written by \citet{follmer} in 1981. He proved a pathwise version of the \ito\ formula, conceiving the construction of an integral of a $C^1$-class function of a \cadlag\ path with respect to that path itself, as a limit of non-anticipative Riemann sums. His purely analytical approach does not ask for any probabilistic structure, which may instead come into play only at a later point by considering stochastic processes that satisfy almost surely, i.e. for almost all paths, the analytical requirements. In this case, the so-called \emph{F\"ollmer integral} provides a path-by-path construction of the stochastic integral.
F\"ollmer's framework turns out to be of main interest in finance (see also \cite{schied-CPPI}, \cite[Sections 4,5]{follmer-schied}, and \cite[Chapter 2]{sondermann}) as it allows to avoid any probabilistic assumption on the dynamics of traded assets and consequently to avoid any model risk/ambiguity. Reasonably, only observed price trajectories are involved.

In 1994, \citet{bickwill} provided an interesting economic interpretation of F\"ollmer's pathwise calculus, leading to new perspectives in the mathematical modeling of financial markets.
Bick and Willinger reduced the computation of the initial cost of a replicating trading strategy to an exercise of analysis. Moreover, for a given price trajectory (state of the world), they showed one is able to compute the outcome of a given trading strategy, that is the gain from trade. 
Other contributions towards the pathwise characterization of stochastic integrals have been obtained via probabilistic techniques by Wong and Zakai (1965), \citet{bichteler}, \citet{karandikar} and \citet{nutz-int} (only existence), and via convergence of discrete-time economies by \citet{willtaq}.

We are interested only in the model-free approach: we set our framework in a similar way to \cite{bickwill}, and we enhance it  by the aid of the pathwise calculus for \naf s, 
developed by \citet{contf2010}. This theory extends the F\"ollmer's pathwise calculus to a large class of non-anticipative functionals.

Another problem related to the model uncertainty, addressed in the second part of this thesis is the \emph{robustness} of hedging strategies used by market agents to cover the risks involved in the sale of financial derivatives.
The issue of robustness came to light in the nineties, dealing mostly with the analysis of the performance, in a given complete model, of pricing and hedging simple payoffs under a mis-specification of the volatility process.
The problem under consideration is the following. Let us imagine a market participant who sells an (exotic) option with payoff $H$ and maturity $T$ on some underlying asset which is assumed to follow some model (say, Black-Scholes), at price given by
$$ V_t = E^{\mathbb{Q}}[H|{\cal F}_t]$$
and hedges the resulting profit and loss using the hedging strategy derived from the same model (say, Black-Scholes delta hedge for $H$).
However, the {\it true} dynamics of the underlying asset may, of course, be different from the assumed dynamics.
Therefore, the hedger is interested in a few questions: How good is the result of the hedging strategy? How 'robust' is it to model mis-specification? How does the hedging error relate to model parameters and option characteristics?
In 1998, \citet{elkaroui} provided an answer to the important questions above in the setting of diffusion models, for non-path-dependent options. They provided an explicit formula for the profit and loss, or \textit{tracking error} as they call it, of the hedging strategy. Specifically, they show that if the underlying asset follows a Markovian diffusion
$$\ud S_t= r(t)S(t)\ud t+ S(t)\sigma(t) \ud W(t) \qquad \text{under}\ \mathbb{P}$$
such that the discounted price $S/M$ is a square-integrable martingale, then a hedging strategy computed in a (mis-specified) model with local volatility $\sigma_0$, satisfying some technical conditions, leads to a tracking error equal to
$$\int_0^T \frac{\sigma_0^{2}(t,S(t))-\sigma^{2}(t)}{2}S(t)^2e^{\int_t^T r(s)\ud s}\overbrace{\partial_{xx}^2f(t,S(t))}^{\Gamma(t)}\ud t,$$
$\PP$-almost surely. This fundamental equation, called by \citet{davis} \lq the most important equation in option pricing theory\rq, shows that the exposure of a mis-specified delta hedge over a short time period is proportional to the Gamma of the option times the specification error measured in quadratic variation terms.
Other two papers studying the monotonicity and super-replication properties of non-path-dependent option prices under mis-specified models are \cite{bergman} and \cite{hobson}, respectively by PDE and coupling techniques.
The robustness of dynamic hedging strategies in the context of model ambiguity has been considered by several authors in the literature (\citet{bickwill,avlevyparas,lyons,cont2006}).
\citet{ss} studied the robustness of delta hedging strategies for discretely monitored path-dependent derivatives in a Markovian diffusion (\lq local volatility\rq) model from a pathwise perspective: they looked at the performance of the delta hedging strategy derived from some model when applied to the realized underlying price path, rather than to some supposedly true stochastic dynamics. 
In the present thesis, we investigate the robustness of delta hedging from this pathwise perspective, but we consider a general square-integrable exponential model used by the hedger for continuously - instead of discretely - monitored path-dependent derivatives.
In order to conduct this pathwise analysis, we resort to the pathwise functional calculus developed in \citet{contf2010} and the functional \ito\ calculus developed in \cite{contf2013,cont-notes}. In particular we use the results of \chap{path-trading} of this thesis, which provide an analytical framework for the analysis of self-financing trading strategies in a continuous-time financial market.

The last chapter of this thesis deals with a completely different problem, that is the search for accurate approximation formulas for the price of financial derivatives under a model with local volatility and L\'evy-type jumps.
Precisely, we consider a one-dimensional {\it local L\'evy model}: the risk-neutral dynamics of the underlying log-asset process $X$ is given by
$$\ud X(t)=\m(t,X(t-))\ud t+\s(t,X(t)) \ud W(t)+ \ud J(t),$$
where $W$ is a standard real Brownian motion on a filtered probability space
$(\O,\F,(\F_t)_{0\leq t\leq T},\mathbb{P})$ with the usual assumptions on the filtration and $J$
is a pure-jump L\'evy process, independent of $W$, with L\'evy triplet $(\m_{1},0,\n)$.
Our main result is a fourth order approximation formula of the characteristic function $\phi_{X^{t,x}(T)}$ of the log-asset price $X^{t,x}(T)$ starting from $x$ at time $t$, that is
 $$\phi_{X^{t,x}(T)}(\x)=E^\PP\left[e^{i\x X^{t,x}(T)}\right],\qquad \x\in\R,$$
In some particular cases, we also obtain an explicit approximation of the transition density of $X$.

Local L\'evy models of this form have attracted an increasing
interest in the theory of volatility modeling (see, for instance, \cite{AndersenAndreasen2000},
\cite{CGMY2004} and \cite{ContLantosPironneau2011}); however to date only in a few cases closed
pricing formulae are available. Our approximation formulas provide a way to compute efficiently and accurately option prices and sensitivities by using
standard and well-known Fourier methods (see, for instance, Heston \cite{Heston1993}, Carr and
Madan \cite{CarrMadan1999}, Raible \cite{Raible2000} and Lipton \cite{Lipton2002}).

We derive the approximation formulas by introducing an ``adjoint'' expansion method:
this is worked out in the Fourier space by considering the adjoint formulation of the pricing
problem.
Generally speaking, our approach makes use of Fourier analysis and PDE techniques.

The thesis is structured as follows:

\paragraph{Chapter 1}
The first chapter introduces the pathwise functional calculus, as developed by Cont and Fourni\'e \cite{contf2010,cont-notes}, and states some of its key results.
The most important theorem is a change-of-variable formula extending the pathwise \ito\ formula proven in \cite{follmer} to \naf s, and applies to a class of paths with finite quadratic variation. The chapther then includes a discussion on the different notions of \emph{quadratic variation} given by different authors in the literature.

\paragraph{Chapter 2}
The second chapter presents the probabilistic counterpart of the pathwise functional calculus, the so-called \lq functional \ito\ calculus\rq, following the ground-breaking work of Cont and Fourni\'e \cite{ContFournie09a,contf2013,cont-notes}.
Moreover, the weak functional calculus, which applies to a large class of square-integrable processes, is introduced. Then, in \Sec{kolmogorov} we show how to apply the functional \ito\ calculus to extend the relation between Markov processes and partial differential equations to the path-dependent setting. These tools have useful applications for the pricing and hedging of path-dependent derivatives. In this respect, we state the universal pricing and hedging formulas. Finally, in \Sec{PPDE}, we report the results linking forward-backward stochastic differential equations to path-dependent partial differential equations and we recall some of the recent papers investigating weak and viscosity solutions of such path-dependent PDEs.

\paragraph{Chapter 3}
\Sec{lit-path} presents a synopsis of the various approaches in the literature attempting a pathwise construction of stochastic integrals, and clarifies the connection with appropriate no-arbitrage conditions.
In \Sec{setting}, we set our analytical framework 
 and we start by defining \emph{simple trading strategies}, whose trading times are covered by the elements of a given sequence of partitions of the time horizon $[0,T]$ and for which the self-financing condition is straightforward.
We also remark on the difference between our setting and the ones presented in \Sec{arbitrage} about no-arbitrage and we provide some kind of justification, in terms of a condition on the set of admissible price paths, to the assumptions underlying our main results. 
In \Sec{self-fin}, we define equivalent self-financing conditions for (non-simple) trading strategies on a set of paths, whose gain from trading is the limit of gains of simple strategies and satisfies the pathwise counterpart equation of the classical self-financing condition. Similar conditions were assumed in \cite{bickwill} for convergence of general trading strategies.
In \Sec{gain}, we show the first of the main results of the chapter: in \prop{G} for the continuous case and in \prop{G-cadlag} for the \cadlag\ case, we obtain the path-by-path computability of the gain of path-dependent trading strategies in a certain class of $\R^d$-valued \caglad\ adapted processes, which are also self-financing on the set of paths with finite quadratic variation along $\Pi$.
For dynamic asset positions $\phi$ in the vector space of \emph{vertical 1-forms}, 
the gain of the corresponding self-financing trading strategy is well-defined as a \cadlag\ process $G(\cdot,\cdot;\phi)$ such that
 \begin{align*}
   G(t,\w;\phi)=&\int_0^t \phi(u,\w_u)\cdot\ud^{\Pi}\w \\
   =&\lim_{n\rightarrow\infty}\sum_{t^n_i\in\pi^n, t^n_i\leq t}\phi(t_i^n,\w^{n}_{t^n_i})\cdot(\w(t_{i+1}^n)-\w(t_i^n))
  \end{align*}
for all continuous paths of finite quadratic variation along $\Pi$, where $\w^n$ is a piecewise constant approximation of $\w$ defined in \eq{wn}. 
In \Sec{replication}, we present a pathwise replication result, \prop{hedge}, that can be seen as the model-free and path-dependent counterpart of the well known pricing PDE in mathematical finance, giving furthermore an explicit formula for the \emph{hedging error}. That is, if a \lq smooth\rq\ \naf\ $F$ solves
$$\left\{\bea{ll}
\hd F(t,\w_t)+\frac12\tr\lf A(t)\cdot\vd^2F(t,\w_t)\rg=0,\quad t\in[0,T), \w\in Q_A(\Pi) \\
F(T,\w)=H(\w),
\ea\right.$$
where $H$ is a continuous (in sup norm) payoff functional and $Q_A(\Pi)$ is the set of paths with absolutely continuous quadratic variation along $\Pi$ with density $A$, then the hedging error of the delta-hedging strategy for $H$ with initial investment $F(0,\cdot)$ and asset position $\vd F$ is
\beq\label{eq:er}
\frac12\int_{0}^T\tr\lf\vd^2F(t,\w)\cdot \lf A(t)-\tilde A(t)\rg\rg \ud t
\eeq
on all paths $\w\in Q_{\tilde A}(\Pi)$.
In particular, if the underlying price path $\w$ lies in $Q_A(\Pi)$, the delta-hedging strategy $(F(0,\cdot),\vd F)$ replicates the $T$-claim with payoff $H$ and its portfolio's value at any time $t\in[0,T]$ is given by $F(t,\w_t)$.
The explicit error formula \eq{er} is the purely analytical counterpart of the probabilistic formula  given in \cite{elkaroui}, where a mis-specification of volatility is considered in a stochastic framework.
Finally, in \Sec{isometry} we propose, in \prop{Iw}, the construction of a family of pathwise integral operators (indexed by the paths) as extended isometries between normed spaces defined as quotient spaces.

\paragraph{Chapter 4}
The last chapter begins with a review of the results, from the present literature, that focus on the problem of robustness which we are interested in, in particular the propagation of convexity and the hedging error formula for non-path-dependent derivatives, as well as a contribution to the pathwise analysis of path-dependent hedging for discretely-monitored derivatives.
In \Sec{path-robust}, we introduce the notion of robustness that we are investigating (see \defin{rob}): the delta-hedging strategy is robust on a certain set $U$ of price paths if it super-replicates the claim at maturity, when trading with the market prices, as far as the price trajectory belongs to $U$. We then state in \prop{robust} a first result which applies to the case where the derivative being sold admits a smooth pricing functional under the model used by the hedger: robustness holds if the second vertical derivative of the value functional, $\vd^2F$, is (almost everywhere) of same sign as the difference between the model volatility and the realized market volatility. Moreover, we give the explicit representation of the \emph{hedging error} at maturity, that is
$$\frac12\int_{0}^T \lf\s(t,\w)^2-\s^{\mathrm{mkt}}(t,\w)^2\rg\w^2(t)\vd^2F(t,\w) \ud t,$$
where $\s$ is the model volatility and $\s^{\mathrm{mkt}}$ is the realized market volatility, defined by $t\mapsto\s^{\mathrm{mkt}}(t,\w)=\frac1{\w(t)}\sqrt{\frac{\ud}{\ud t}[w](t)}$.
In \Sec{exist}, \prop{exist} provides a constructive existence result for a pricing functional which is twice left-continuously vertically differentiable on continuous paths, given a log-price payoff functional $h$ which is \emph{vertically smooth} on the space of continuous paths (see \defin{vsmooth}).
We then show in \Sec{convex}, namely in \prop{convex}, that a sufficient condition for the second vertical derivative of the pricing functional to be positive is the convexity of the real map
$$v^H(\cdot;t,\w):\R\rightarrow\R,\quad e\mapsto v^H(e;t,\w)=H\lf\w(1+e\ind_{[t,T]})\rg$$
in a neighborhood of 0. This condition may be readily checked for all path-dependent payoffs.
In \Sec{jumps}, we analyze the contribution of jumps of the price trajectory to the hedging error obtained trading on the market according to a delta-hedging strategy. We show in \prop{jumps} that the term carried by the jumps is of negative sign if the second vertical derivative of the value functional is positive.
In \Sec{HR}, we consider a specific pricing model with path-dependent volaility, the Hobson-Rogers model. 
Finally, in \Sec{ex}, we apply the results of the previous sections to common examples, specifically the hedging of discretely monitored path-dependent derivatives, Asian options and barrier options. In the first case, we show in \lem{BS} that in the Black-Scholes model the pricing functional is of class $\Cloc$ and its vertical and horizontal derivatives are given in closed form. Regarding Asian options, both the Black-Scholes and the Hobson-Rogers pricing functional have already been proved to be regular by means of classical results, and, assuming that the market price path lies in the set of paths with absolutely continuous \fqv{\mbox{the}} given sequence of partitions and the model volatility overestimates the realized market volatility, the delta hedge is \emph{robust}. Regarding barrier options, the robustness fails to be satisfied: Black-Scholes delta-hedging strategies for barrier options are not robust to volatility mis-specifications.

\paragraph{Chapter 5}

Chapter 5, independent from the rest of the thesis,  is based on joint work with Andrea Pascucci and Stefano Pagliarani.

In \Sec{sec1}, we present the general procedure that allows to approximate analy\-tically the
transition density (or the characteristic function), in terms of the solutions of a sequence of
nested Cauchy problems. Then we also prove explicit error bounds for the expansion that generalize some classical estimates. In \Sec{Merton} and \Sec{LV-J}, the previous Cauchy problems are solved explicitly by using different
approaches. Precisely, in \Sec{Merton} we focus on the special class of local L\'evy models with
Gaussian jumps and we provide a heat kernel expansion of the transition density of the underlying
process. The same results are derived in an alternative way in Subsection \ref{sec:secsimpl}, by
working in the Fourier space.

\Sec{LV-J} contains the main contribution of the chapter: we consider the general class of local L\'evy models and provide high order approximations of the characteristic function. Since all the computations are carried out in the Fourier space, we are forced to introduce {\it a dual formulation} of the approximating problems, which involves the adjoint (forward) Kolmogorov operator.
Even if at first sight the adjoint expansion method seems a bit odd, it turns out to be much more natural and simpler than the direct formulation. To the best of our knowledge, the interplay between perturbation methods and Fourier analysis has not been previously studied in finance.
Actually our approach seems to be advantageous for several reasons:
\begin{enumerate}[(i)]
  \item working in the Fourier space is natural and allows to get simple and clear results;
  \item we can treat the entire class of L\'evy processes and not only jump-diffusion processes or processes which can be approximated by heat kernel expansions --potentially, we can take as leading term of the expansion every process which admits an explicit characteristic function and not necessarily a Gaussian kernel; 
  \item our method can be easily adapted to the case of stochastic volatility or multi-asset models;
  \item higher order approximations are rather easy to derive and the approximation results are generally very accurate. Potentially, it is possible to derive approximation formulae for the characteristic function and plain vanilla options, at any prescribed order. For example, in Subsection \ref{HOA} we provide also the $3^{\text{rd}}$ and $4^{\text{th}}$ order expansions of the characteristic function, used in the numerical tests of \Sec{numeric}. A Mathematica notebook with the implemented formulae is freely available on \url{https://explicitsolutions.wordpress.com}.
\end{enumerate}
 Finally, in \Sec{numeric}, we present some numerical tests under the Merton and Variance-Gamma models and show the effectiveness of the analytical approximations compared with Monte Carlo simulation.

\chapter{Pathwise calculus for \naf s}
\label{chap:pfc}
\chaptermark{Pathwise functional calculus}

This chapter is devoted to the presentation of the pathwise calculus for non-anticipative functionals developed by \citet{contf2010} and having as main result a change of variable formula (also called \emph{chain rule}) for \naf s. This pathwise functional calculus extends the pathwise calculus introduced by F\"ollmer in his seminal paper \emph{Calcul d'\ito\ sans probabilit\'es} in 1981. Its probabilistic counterpart, called the \lq functional \ito\ calculus\rq\ and presented in \chap{fic}, can either stand by itself or rest entirely on the pathwise results, e.g. by introducing a probability measure under which the integrator process is a semimartingale. This shows clearly the pathwise nature of the theory, as well as \citeauthor{follmer} proved that the classical \ito\ formula has a pathwise meaning.
Other chain rules were derived in \cite{norvaisa} for extended Riemann-Stieltjes integrals and for a type of one-sided integral similar to F\"ollmer's one.

Before presenting the functional case we are concerned with, let us set the stage by introducing the pathwise calculus for ordinary functions.
First, let us give the definition of quadratic variation for a function that we are going to use throughout this thesis and review other notions of quadratic variation.

\section{Quadratic variation along a sequence of partitions}
\label{sec:qv}
\sectionmark{Quadratic variation of paths}

Let $\Pi=\{\pi_n\}_{n\geq1}$ be a sequence of partitions of $[0,T]$, that is for all $n\geq1$ $\pi_n=(t_i^n)_{i=0,\ldots,m(n)},\;0=t_0^n<\ldots<t_{m(n)}^n=T$. We say that $\Pi$ is \emph{dense} if $\cup_{n\geq1}\pi_n$ is dense in $[0,T]$, or equivalently the mesh $\abs{\pi^n}:=\max_{i=1,\ldots m(n)}|t^n_i-t^n_{i-1}|$ goes to 0 as $n$ goes to infinity, and we say that $\Pi$ is \emph{nested} if $\pi_{n+1}\subset\pi_{n}$ for all $n\in\NN$. 
\begin{definition}\label{def:qv1}
  Let $\Pi$ be a dense sequence of partitions of $[0,T]$, a \cadlag\ function $x:[0,T]\to\R$ is said to be of \emph{finite quadratic variation along $\Pi$} if there exists a non-negative \cadlag\ function $[x]_\Pi:[0,T]\to\R_+$ such that
\beq\label{eq:qv}
\forall t\in[0,T],\quad[x]_\Pi(t)=\Limn\sum_{\stackrel{i=0,\ldots,m(n)-1:}{t^n_i\leq t}}(x(t^n_{i+1})-x(t^n_{i}))^2<\infty
\eeq
and 
\begin{equation} \label{eq:qv-jumps}
[x]_\Pi(t)=[x]_\Pi^c(t)+\sum_{0<s\leq t}\De x^2(s) \zs,
\end{equation}
where $[x]_\Pi^c$ is a continuous non-decreasing function and $\De x(t):=x(t)-x(t-)$ as usual.
In this case, the non-decreasing function $[x]_\Pi$ is called the \emph{quadratic variation of $x$ along $\Pi$}.
\end{definition}
Note that the quadratic variation $[x]_\Pi$ depends strongly on the sequence of partitions $\Pi$. Indeed, as remarked in \cite[Example 2.18]{cont-notes}, for any real-valued continuous function we can construct a sequence of partition along which that function has null quadratic variation.

In the multi-dimensional case, the definition is modified as follows.
\begin{definition}\label{def:qvd}
An $\R^d$-valued \cadlag\ function $x$ is of \emph{finite quadratic variation along $\Pi$} if, for all $1\leq i,j\leq d$, $x^i,x^i+x^j$ have finite quadratic variation along $\Pi$. In this case, the function $[x]_\Pi$ has values in the set $\S^+(d)$ of positive symmetric $d\times d$ matrices:
$$\forall t\in[0,T],\quad[x]_\Pi(t)=\Limn\sum_{\stackrel{i=0,\ldots,m(n)-1:}{t^n_i\leq t}}\incrx{x}\cdot\,^t\!\!\incrx{x},$$
whose elements are given by
\begin{eqnarray*}
([x]_\Pi)_{i,j}(t) &=& \frac12\lf[x^i+x^j]_\Pi(t)-[x^i]_\Pi(t)-[x^j]_\Pi(t)\rg \\
 &=& [x^i,x^j]_\Pi^c(t)+\sum_{0<s\leq t}\De x^i(s)\De x^j(s) 
\end{eqnarray*}
for $i,j=1,\ldots d$.
\end{definition}

For any set $U$ of \cadlag\ paths with values in $\R$ (or $\R^d$), we denote by $Q(U,\Pi)$ the subset of $U$ of paths having \fqv{\Pi}.

Note that $Q(D([0,T],\R),\Pi)$ is not a vector space, because assuming $x^1,x^2\in\penalty0 Q(D([0,T],\R),\Pi)$ does not imply $x^1+x^2\in Q(D([0,T],\R),\Pi)$ in general. This is the reason of the additional requirement $x^i+x^j\in Q(D([0,T],\R),\Pi)$ in \defin{qvd}. As remarked in \cite[Remark 2.20]{cont-notes}, the subset of paths $x$ being $C^1$-functions of a same path $\w\in D([0,T],\R^d)$, i.e. 
$$\{x\in Q(D([0,T],\R),\Pi),\;\exists f\in C^1(\R^d,\R),\,x(t)=f(\w(t))\,\forall t\in[0,T]\},$$
is instead closed with respect to the quadratic variation composed with the sum of two elements.

Henceforth, when considering a function $x\in Q(U,\Pi)$, we will drop the subscript in the notation of its quadratic variation, thus denoting $[x]$ instead of $[x]_\Pi$.

\subsection{Relation with the other notions of quadratic variation}

An important distinguish is between \defin{qv1} and the notions of $2$-variation and local 2-variation considered in the theory of extended Riemann-Stieltjes integrals (see e.g. \citet[Chapters 1,2]{dudley-norvaisa} and \citet[Section 1]{norvaisa}). 
Let $f$ be any real-valued function on $[0,T]$ and $0<p<\infty$, the \emph{$p$-variation} of $f$ is defined as 
\beq\label{eq:p-var}
v_p(f):=\sup_{\k\in P[0,T]}s_p(f;\k)
\eeq
where $P[0,T]$ is the set of all partitions of $[0,T]$ and
$$s_p(f;\k)=\sum_{i=1}^n\abs{f(t_i)-f(t_{i-1})}^p,\quad\text{for }\k=\{t_i\}_{i=0}^n\in P[0,T].$$
The set of functions with finite $p$-variation is denoted by $\W_p$. We also denote by $\mathrm{vi}(f)$ the variation index of $f$, that is the unique number in $[0,\infty]$ such that
$$\bea{l}v_p(f)<\infty,\quad\mbox{for all }p>\mathrm{vi}(f),\\v_p(f)=\infty,\quad\mbox{for all }p<\mathrm{vi}(f)\ea.$$

For $1<p<\infty$, $f$ has the \emph{local $p$-variation} if the directed function $(s_p(f;\cdot),\mathfrak R)$, where $\mathfrak R:=\{\mathcal R(\k)=\left\{\pi\in P[0,T],\,\k\subset\pi\},\,\k\in P[0,T]\right\}$, converges.
An equivalent characterization of functions with local $p$-variation was introduced by \citet{love-young} and it is given by the Wiener class $\W^*_p$ of functions $f\in\W_p$ such that 
$$\limsup_{\k,\mathfrak R}s_p(f;\k)=\sum_{(0,T]}\abs{\De^-f}^p+\sum_{[0,T)}\abs{\De^+f}^p,$$
where the two sums converge unconditionally. 
We refer to \cite[Appendix A]{norvaisa} for convergence of directed functions and unconditionally convergent sums.
The Wiener class satisfies
$\cup_{1\leq q<p}\W_q\subset\W_p^*\subset \W_p$.

A theory on Stieltjes integrability for functions of bounded $p$-variation was developed by \citet{young36,young38} in the thirties and generalized among others by \cite{dudley-norvaisa99,norvaisa02} around the years 2000. 
According to Young's most well known theorem on Stieltjes integrability, if 
\beq\label{eq:ys}
f\in\W_p,\;g\in\W_q,\quad p^{-1}+q^{-1}>1,\,p,q>0,
\eeq
then the integral $\int_0^Tf\ud g$ exists: in the \emph{Riemann-Stieltjes} sense if $f,g$ have no common discontinuities, in the \emph{refinement Riemann-Stieltjes} sense if $f,g$ have no common discontinuities on the same side, and always in the \emph{Central Young} sense. \cite{dudley-norvaisa99} showed that under condition \eq{ys} also the \emph{refinement Young-Stieltjes} integral always exists.
However, in the applications, we often deal with paths of unbounded 2-variation, like sample paths of the Brownian motion. For example, given a Brownian motion $B$ on a complete probability space $(\O,\F,\PP)$, the pathwise integral $(RS)\!\int_0^Tf\ud B(\cdot,\w)$ is defined in the Riemann-Stieltjes sense, for $\PP$-almost all $\w\in\O$, for any function having bounded $p$-variation for some $p<2$, which does not apply to sample paths of $B$.
In particular, in Mathematical Finance, one necessarily deals with price paths having unbounded 2-variation. In the special case of a market with continuous price paths, as shown in \Sec{arbitrage}, \cite{vovk-proba} proved that non-constant price paths must have a variation index equal to 2 and infinite 2-variation in order to rule out \lq arbitrage opportunities of the first kind\rq. 
In the special case where the integrand $f$ is replaced by a smooth function of the integrator $g$, weaker conditions than \eq{ys} on the $p$-variation are sufficient (see \cite{norvaisa02} or the survey in \cite[Chapter 2.4]{norvaisa}) to obtain chain rules and integration-by-parts formulas for extended Riemann-Stieltjes integrals, like the refinement Young-Stieltjes integral, the symmetric Young-Stieltjes integral, the Central Young integral, the Left and Right Young integrals, and others.
However, these conditions are still quite restrictive.

 As a consequence, other notions of quadratic variation were formulated and integration theories for them followed.

\subsubsection{\follmer's quadratic variation and pathwise calculus}

In 1981, \citet{follmer} derived a pathwise version of the \ito\ formula, conceiving a construction path-by-path of the stochastic integral of a special class of functions. His purely analytic approach does not ask for any probabilistic structure, which may instead come into play only in a later moment by considering stochastic processes that satisfy almost surely, i.e. for almost all paths, a certain condition.
F\"ollmer considers functions on the half line $[0,\infty)$, but we present here his definitions and results adapted to the finite horizon time $[0,T]$.
His notion of quadratic variation is given in terms of weak convergence of measures and is renamed here in his name in order to make the distinguish between the different definitions.
\begin{definition}\label{def:qv-follmer}
Given a dense sequence $\Pi=\{\pi_n\}_{n\geq1}$ of partitions of $[0,T]$, for $n\geq1\; \pi_n=(t_i^n)_{i=0,\ldots,m(n)}$, $0=t_0^n<\ldots<t_{m(n)}^n<\infty$, a \cadlag\ function $x:[0,T]\to\R$ is said to have \emph{F\"ollmer's quadratic variation along} $\Pi$ if the Borel measures 
\beq\label{eq:xin}
\xi_n:=\sum\limits_{i=0}^{m(n)-1}\incrx{x}^2\d_{t_i^n},
\eeq
where $\d_{t_i^n}$ is the Dirac measure centered in $t_i^n$, converge weakly to a finite measure $\xi$ on $[0,T]$ with cumulative function $[x]$ and Lebesgue decomposition 
\beq\label{eq:dec-follmer}
[x](t)=[x]^c(t)+\sum_{0<s\leq t}\De x^2(s),\quad \forall t\in[0,T]
\eeq
where $[x]^c$ is the continuous part.
\end{definition}

\begin{proposition}[Follmer's pathwise \ito\ formula]
  Let $x:[0,T]\to\R$ be a \cadlag\ function having F\"ollmer's quadratic variation along $\Pi$.
  Then, for all $t\in[0,T]$, a function $f\in\C^2(\R)$ satisfies
\begin{align}
  \label{eq:follmer_ito}\nonumber
f(x(t))={}& f(x(0))+\int_0^tf'(x(s-))\ud x(s)+\frac12\int_{(0,t]}f''(x(s-))\ud[x](s) \\ \nonumber
&{}+\sum_{0<s\leq t}\lf f(x(s))-f(x(s-))-f'(x(s-))\De x(s)-\frac12f''(x(s-))\De x(s)^2 \rg \\\nonumber
={}& f(x(0))+\int_0^tf'(x(s-))\ud x(s)+\frac12\int_{(0,t]}f''(x(s))\ud[x]^c(s) \\ 
&{}+\sum_{0<s\leq t}\lf f(x(s))-f(x(s-))-f'(x(s-))\De x(s) \rg,
\end{align}
where the pathwise definition
\beq \label{eq:follmer_int}
\int_0^tf'(x(s-))\ud x(s):=\Limn \sum_{t_i^n\leq t}f'(x(t_i^n))\lf x(t_{i+1}^n\wedge T)-x(t_i^n\wedge T)\rg
\eeq
is well posed by absolute convergence.
\end{proposition}
The integral on the left-hand side of \eq{follmer_int} is referred to as the \emph{F\"ollmer integral} of $f\circ x$ with respect to $x$ along $\Pi$.

In the multi-dimensional case, where $x$ is $\R^d$-valued and $f\in\C^2(\R^d)$, the pathwise \ito\ formula gives
\begin{align} \nonumber
f(x(t))={}& f(x(0))+\int_0^t\nabla f(x(s-))\cdot \ud x(s)+\frac12\int_{(0,t]}\mathrm{tr}\lf \nabla^2f(x(s))\ud[x]^c(s) \rg\\\label{eq:follmer_Dito}
&{}+\sum_{0<s\leq t}\lf f(x(s))-f(x(s-))-\nabla f(x(s-))\cdot\De x(s) \rg
\end{align}
and
$$\int_0^t\nabla f(x(s-))\cdot \ud x(s):=\Limn \sum_{t_i^n\leq t}\nabla f(x(t_i^n))\cdot\incrx{x},$$ 
where $[x]=([x^i,x^j])_{i,j=1,\ldots,d}$ and, for all $t\geq0$,
\begin{align*}
[x^i,x^j](t)={}&\frac12\lf[x^i+x^j](t)-[x^i](t)-[x^j](t)\rg\\
{}={}&[x^i,x^j]^c(t)+\sum_{0<s\leq t}\De x^i(s)\De x^j(s).
\end{align*}
F\"ollmer also pointed out that the class of functions with finite quadratic variation is stable under $\C^1$ transformations and, given $x$ with finite quadratic variation along $\Pi$ and $f\in\C^1(\R^d)$, the composite function $y=f\circ x$ has finite quadratic variation 
$$[y](t)=\int_{(0,t]}\mathrm{tr}\lf \nabla^2f(x(s))^\mathrm{t}\ud[x]^c(s)\rg+\sum_{0<s\leq t}\De y^2(s).$$
Further, he has enlarged the scope of the above results by considering stochastic processes with almost sure finite quadratic variation along some proper sequence of partition.
For example, let $S$ be a semimartingale on a probability space \ps, it is well known that there exists a sequence of random partitions, $\Pi=(\pi_n)_{n\geq1}$, $|\pi_n|\limn0$ $\PP$-almost surely, such that $$\PP\lf\{\w\in\O,\; S(\cdot,\w) \text{ has F\"ollmer's quadratic variation along }\Pi\}\rg=1.$$
More generally, this holds for any so-called \textit{Dirichlet} (or \textit{finite energy}) \textit{process}, that is the sum of a semimartingale and a process with zero quadratic variation along the dyadic subdivisions. Thus, the pathwise \ito\ formula holds and the pathwise F\"ollmer integral is still defined for all paths outside a null set.

A last comment on the link between \ito\ and \follmer\ integrals is the following. For a semimartingale $X$ and a \cadlag\ adapted process $H$, we know that, for any $t\geq0$,
$$ \sum_{t_i^n\leq t} H(t_i^n)\cdot\incrx{x} \limnp \int_0^tH(s-)\cdot\ud X(s),$$
hence we have almost sure pathwise convergence by choosing properly an absorbing set of paths dependent on $H$, which is not of practical utility. However, in the case $H=f\circ X$ with $f\in\C^1$, we can select a priori the null set out of which the definition~\eqref{eq:follmer_int} holds and so, by almost sure uniqueness of the limit in probability, the F\"ollmer integral must coincide almost surely with the \ito\ integral.

\subsubsection{Norvai\u sa's quadratic variation and chain rules}

Norvai\u sa's notion of quadratic variation was proposed in \cite{norvaisa} in order to weaken the requirement of local 2-variation used to prove chain rules and integration-by-parts formulas for extended Riemann-Stieltjes integrals.
\begin{definition}\label{def:qv-norvaisa}
  Given a dense nested sequence $\l=\{\l_n\}_{n\geq1}$ of partitions of $[0,T]$, \emph{Norvai\u sa's quadratic $\l$-variation} of a regulated function $f:[0,T]\to\R$ is defined, if it exists, as a regulated function $H:[0,T]\to\R$ such that $H(0)=0$ and, for any $0\leq s\leq t\leq T$,
\beq\label{eq:Nqv}
H(t)-H(s)=\Limn s_2(f;\l_n\Cap[s,t]),
\eeq
\beq\label{eq:Njumps}
\De^-H(t)=(\De^-f(t))^2\quad \text{and}\quad \De^+H(t)=(\De^+f(t))^2,
\eeq
where $\l_n\Cap[s,t]:=(\l_n\cap[s,t])\cup\{s\}\cup\{t\}$, $\De^-x(t)=x(t)-x(t-)$, and $\De^+x(t)=x(t+)-x(t)$.
\end{definition}
In reality, Norvai\u sa's original definition is given in terms of an additive upper continuous function defined on the simplex of extended intervals of $[0,T]$, but he showed the equivalence to the definition given here and we chose to report the latter because it allows us to avoid introducing further notations.

Following F\"ollmer's approach in \cite{follmer}, \citet{norvaisa} also proved a chain rule for a function with finite $\l$-quadratic variation, involving a new type of integrals called Left (respectively Right) Cauchy $\l$-integrals.
We report here the formula obtained for the left integral, but a symmetric formula holds for the right integral.
 Given two regulated functions $f,g$ on $[0,T]$ and a dense nested sequence of partitions $\l=\{\l_n\}$, then \emph{the Left Cauchy $\l$-integral} $(LC)\!\int \phi\ud_\l g$ is defined on $[0,T]$ if there exists a regulated function $\Phi$ on $[0,T]$ such that $\Phi(0)=0$ and, for any $0\leq u<v\leq T$, 
$$\bea{c}\Phi(v)-\Phi(u)=\Limn S_{LC}(\phi,g;\l_n\Cap[u,v]),\\
\De^-\Phi(v)=\phi(v-)\De^-g(v),\quad\De^+\Phi(u)=\phi\De^+g(u),\ea$$
where $$S_{LC}(\phi,g;\k):=\sum_{i=0}^{m-1}\phi(t_i)(g(t_{i+1})-g(t_i))\quad\text{for any }\k=\{t_i\}_{i=0}^m.$$
In such a case, denote $(LC)\!\int_u^v\phi\ud_\l g:=\Phi(v)-\Phi(u).$
\begin{proposition}[Proposition 1.4 in \cite{norvaisa}]
  Let $g$ be a regulated function on $[0,T]$ and $\l=\{\l_n\}$ a dense nested sequence of partitions such that $\{t:\,\De^+g(t)\neq0\}\subset\cup_{n\in\NN}\l_n$. The following are equivalent:
  \begin{enumerate}[(i)]
  \item $g$ has Norvai\u sa's $\l$-quadratic variation;
  \item for any $C^1$ function $\phi$, $\phi\circ g$ is Left Cauchy $\l$-integrable on $[0,T]$ and, for any $0\leq u<v\leq T$,
\begin{align}\label{eq:chainrule-LC}
\Phi\circ g(v)-\Phi\circ g(u)={}&(LC)\!\int_u^v(\phi\circ g)\ud_\l g+\frac12\int_u^v(\phi'\circ g)\ud[g]^c_\l\\
&{}+\sum_{t\in[u,v)}\lf\De^-(\Phi\circ g)(t)-(\phi\circ g)(t-)\De^-g(t)\rg \nonumber\\
&{}+\sum_{t\in(u,v]}\lf\De^+(\Phi\circ g)(t)-(\phi\circ g)(t)\De^+g(t)\rg. \nonumber
\end{align}
  \end{enumerate}
\end{proposition}
Note that the change of variable formula \eq{chainrule-LC} gives the F\"ollmer's formula \eq{follmer_ito} when $g$ is right-continuous, and the Left Cauchy $\l$-integral coincides with the F\"ollmer integral along $\l$ defined in \eq{follmer_int}.

\subsubsection{Vovk's quadratic variation}
\citet{vovk-cadlag} defines a notion of quadratic variation along a sequence of partitions not necessarily dense in $[0,T]$ and uses it to investigate the properties of \lq typical price paths\rq, that are price paths which rule out arbitrage opportunities in his pathwise framework, following a game-theoretic probability approach.
\begin{definition}\label{def:qv-vovk}
  Given a nested sequence $\Pi=\{\pi_n\}_{n\geq1}$ of partitions of $[0,T]$, $\pi_n=(t_i^n)_{i=0,\ldots,m(n)}$ for all $n\in\NN$, a \cadlag\ function $x:[0,T]\to\R$ is said to have \emph{Vovk's quadratic variation along} $\Pi$ if the sequence $\{A^{n,\Pi}\}_{n\in\NN}$ of functions defined by 
$$A^{n,\Pi}(t):=\sum_{i=0}^{m(n)-1}(x(t^n_{i+1}\wedge t)-x(t^n_i\wedge t))^2,\quad t\in[0,T],$$
converges uniformly in time. In this case, the limit is denoted by $A^\Pi$ and called the Vovk's quadratic variation of $x$ along $\Pi$.
\end{definition}
An interesting result in \cite{vovk-cadlag} is that typical paths have the Vovk's quadratic variation along a specific nested sequence $\{\t_n\}_{n\geq1}$ of partitions composed by stopping times and such that, on each realized path $\w$, $\{\t_n(\w)\}_{n\geq1}$ \emph{exhausts} $\w$, i.e. $\{t:\,\De\w(t)\neq0\}\subset\cup_{n\in\NN}\t_n(\w)$ and, for each open interval $(u,v)$ in which $\w$ is not constant, $(u,v)\cap(\cup_{n\in\NN}\t_n(\w))\neq\emptyset$.

The most evident difference between definitions \ref{def:qv1}, \ref{def:qv-follmer}, \ref{def:qv-norvaisa}, \ref{def:qv-vovk} is that the first two of them require the sequence of partitions to be dense, the third one requires the sequence of partitions to be dense and nested, and the last one requires a nested sequence of partitions.
Moreover, Norvai\v sa's definition is given for a regulated, rather than \cadlag, function.

Vovk proved that for a nested sequence $\Pi=\{\pi_n\}_{n\geq1}$ of partitions of $[0,T]$ that exhausts $\w\in D([0,T],\R)$, the following are equivalent:
\begin{enumerate}[(a)]
\item $\w$ has Norvai\v sa's quadratic $\Pi$-variation;
\item $\w$ has Vovk's quadratic variation along $\Pi$;
\item $\w$ has \emph{weak quadratic variation of $\w$ along $\Pi$}, i.e. there exists a \cadlag\ function $V:[0,T]\to\R$ such that
$$V(t)=\Limn\sum_{i=0}^{m(n)-1}(x(t^n_{i+1}\wedge t)-x(t^n_i\wedge t))^2$$
for all points $t\in[0,T]$ of continuity of $V$ and it satisfies \eq{qv-jumps} where $[x]_\Pi$ is replace by $V$.
\end{enumerate}
Moreover, if any of the above condition is satisfied, then $H=A^\Pi=V$.

If, furthermore, $\Pi$ is also dense, than $\w$ has F\"ollmer's quadratic variation along $\Pi$ if and only if it has any of the quadratic variations in (a)-(c), in which case  $H=A^\Pi=V=[\w]$.

In this thesis, we will always consider the quadratic variation of a \cadlag\ path $w$ along a dense nested sequence $\Pi$ of partitions that exhausts $\w$, in which case our \defin{qv1} is equivalent to all the other ones mentioned above. It is sufficient to note that condition (b) implies that $\w$ has finite quadratic variation according to \defin{qv1} and $[\w]=A$, because the properties in  \defin{qv1} imply the ones in \defin{qv-follmer}, which, by Proposition 4 in \cite{vovk-cadlag}, imply condition (b). Therefore, we denote $\bar k(n,t):=\max\{i=0,\ldots,m(n)-1:\,t^n_i\leq t\}$ and note that
\begin{multline*}
A^{n,\Pi}(t)-\sum_{\stackrel{i=0,\ldots,m(n)-1:}{t^n_i\leq t}}(x(t^n_{i+1})-x(t^n_{i}))^2=\\
=(\w(t)-\w(t^n_{\bar k(n,t)}))^2-(\w(t^n_{\bar k(n,t)+1})-\w(t^n_{\bar k(n,t)}))^2\limn 0
\end{multline*}
by right-continuity of $\w$ if $t\in\cup_{n\in\NN}\pi_n$, and by the assumption that $\Pi$ exhausts $\w$ if $t\notin\cup_{n\in\NN}\pi_n$.

\section{Non-anticipative functionals}\label{sec:pre}

First, we resume the functional notation we are adopting in this thesis, according to the lecture notes \cite{cont-notes}, which unify the different notations from the present papers on the subject into a unique clear language.

As usual, we denote by $\DT$ the space of \cadlag\ functions on $[0,T]$ with values in $\R^d$.
Concerning maps $x\in\DT$, for any $t\in[0,T]$ we denote:
\begin{itemize}
\item $x(t)\in\R^d$ its value at $t$; 
\item $x_t=x(t\wedge\cdot)\in\DT$ its path \lq stopped\rq\ at time $t$; 
\item $x_{t-}=x\ind_{[0,t)}+x(t-)\ind_{[t,T]}\in\DT$; 
\item for $\d\in\R^d$, $x_t^\d=x_t+\d\ind_{[t,T]}\in\DT$ the \textit{vertical perturbation} of size $\d$ of the path of $x$ stopped at $t$ over the future time interval $[t,T]$; 
\end{itemize}

A \textit{non-anticipative functional} on $\DT$ is defined as a family of functionals on $\DT$ adapted to the natural filtration $\FF=(\F_t)_{t\in[0,T]}$ of the canonical process on $\DT$, i.e. $F=\{F(t,\cdot),\,t\in[0,T]\},$ such that
$$\forall t\in[0,T],\quad F(t,\cdot):\DT\mapsto\R\text{ is }\F_t\text{-measurable}.$$
It can be viewed as a map on the space of 'stopped' paths $\L_T:=\{(t,x_t):\:(t,x)\in[0,T]\times\DT\}$, that is in turn the quotient of $[0,T]\times\DT$ by the equivalence relation $\sim$ such that $$\forall(t,x),(t',x')\in[0,T]\times\DT,\quad(t,x)\sim(t',x') \iff t=t',x_t=x'_{t}.$$
Thus, we will usually write a \naf\ as a map $F:\L_T\to\R^d$.

The space $\L_T$ is equipped with a distance $\dinf$, defined by
$$\dinf((t,x),(t',x'))=\sup_{u\in[0,T]}|x(u\wedge t)-x'(u\wedge t')|+|t-t'|=||x_t-x'_{t'}||_\infty+|t-t'|,$$
for all $(t,x),(t',x')\in\L_T$.
Note that $(\L_T,\dinf)$ is a complete metric space and the subset of continuous stopped paths,
$$\W_T:=\{(t,x)\in\L_T:\,x\in\C([0,T],\R^d)\},$$
is a closed subspace of $(\L_T,\dinf)$. 

We recall here all the notions of functional regularity that will be used henceforth.
\begin{definition}\label{def:regF}
A \naf\ $F$ is:
\begin{itemize}
\item \emph{continuous at fixed times} if, for all $t\in[0,T]$, 
$$F(t,\cdot):\lf\lf\{t\}\times\DT\rg/\sim,||\cdot||_\infty\rg\mapsto\R$$
is continuous, that is 
$$\bea{c}\forall x\in\DT, \forall\e>0,\,\exists\eta>0:\quad \forall x'\in\DT,\\
||x_t-x'_t||_\infty<\eta\quad\Rightarrow\quad|F(t,x)-F(t,x')|<\e;\ea$$
\item \emph{jointly-continuous}, i.e. $F\in\CC^{0,0}(\L_T)$, if 
$F:\lf\L_T,\dinf\rg\to\R$ is continuous; 
\item \emph{left-continuous}, i.e. $F\in\CC_l^{0,0}(\L_T)$, if 
$$\bea{c}\forall (t,x)\in\L_T, \forall\e>0,\,\exists\eta>0:\quad \forall h\in[0,t],\,\forall (t-h,x')\in\L_T,\\
\quad \dinf((t,x),(t-h,x'))<\eta\quad\Rightarrow\quad|F(t, x)-F(t-h,x')|<\e;\ea$$
a symmetric definition characterizes the set $\CC_r^{0,0}(\L_T)$ of \emph{right-continuous} functionals;
\item \emph{boundedness-preserving}, i.e. $F\in\BB(\L_T)$, if, 
$$\bea{c}\forall K\subset\R^d\text{ compact, }\forall t_0\in[0,T],\,\exists C_{K,t_0}>0;\quad \forall t\in[0,t_0],\,\forall (t,x)\in\L_T,\\
x([0,t])\subset K \Rightarrow |F(t,x)|<C_{K,t_0}.\ea$$
\end{itemize}
\end{definition}

Now, we recall the notions of differentiability for \naf s.
\begin{definition}\label{def:derF}
  A \naf\ $F$ is said:
\begin{itemize}
\item \emph{horizontally differentiable at} $(t,x)\in\L_T$ if the limit
$$\lim_{h\rightarrow0^+}\frac{F(t+h,x_{t})-F(t,x_t)}{h}$$
exists and is finite, in which case it is denoted by $\hd F(t,x)$; if this holds for all $(t,x)\in\L_T$ and $t<T$, then the \naf\ $\hd F=(\hd F(t,\cdot))_{t\in[0,T)}$ is called the \emph{horizontal derivative} of $F$;
\item \emph{vertically differentiable at} $(t,x)\in\L_T$ if the map 
$$\R^d\rightarrow\R,\; e\mapsto F(t, x_t^e)$$ is differentiable at 0 and in this case its gradient at 0 is denoted by $\vd F(t, x)$;
if this holds for all $(t,x)\in\L_T$, then the $\R^d$-valued \naf\ $\vd F=(\vd F(t,\cdot))_{t\in[0,T]}$ is called the \emph{vertical derivative} of $F$.  
\end{itemize}
\end{definition}

Then, the class of smooth functionals is defined as follows:
\begin{itemize}
\item $\CC^{1,k}(\L_T)$ the set of \naf s $F$ which are 
\begin{itemize}
\item horizontally differentiable with $\hd F$ continuous at fixed times,
\item $k$ times vertically differentiable with $\vd^j F\in\CC^{0,0}_l(\L_T)$ for $j=0,\ldots,k$;
\end{itemize}
\item $\CC^{1,k}_b(\L_T)$ the set of \naf s $F\in\CC^{1,k}(\L_T)$ such that $\hd F,\vd F,\ldots,\vd^kF\in\BB(\L_T)$. 
\end{itemize}
However, many examples of functionals in applications fail to be globally smooth, especially those involving exit times. Fortunately, the global smoothness characterizing the class $\Cb(\L_T)$ is in fact sufficient but not necessary to get the functional \ito\ formula. Thus, we will often require only the following weaker property of local smoothness, introduced in \cite{fournie}.
  A \naf\ $F$ is said to be \emph{locally regular}, i.e. $F\in\Cloc(\L_T)$, if $F\in\CC^{0,0}(\L_T)$ and there exist a sequence of stopping times $\{\t_k\}_{k\geq0}$ on $(\DT,\F_T,\FF)$, such that $\t_0=0$ and $\t_k\to_{k\to\infty}\infty$, and a family of \naf s $\{F^k\in\Cb(\L_T)\}_{k\geq0},$
such that
$$F(t,x_t)=\sum_{k\geq0}F^k(t,x_t)\ind_{[\t_k( x),\t_{k+1}(x))}(t) \zs.$$

\section{Change of variable formulae for functionals}

In 2010, \citet{contf2010} extended the \follmer's change of variable formula to \naf s on $\DT$, hence allowing to define an analogue of the \follmer\ integral for functionals. The pathwise formulas are also viable for a wide class of stochastic process in an ``almost-sure'' sense. The setting of \citet{contf2010} is more general than what we need, so we report here its main results in a simplified version.


\begin{remark}[Proposition 1 in \cite{contf2010}]\label{rmk:regularity}
Useful pathwise regularities follow from the continuity of \naf s:
\begin{enumerate}
\item If $F\in\CC_l^{0,0}(\L_T)$, then for all $x\in\DT$ the path $t\mapsto F(t,x_{t-})$ is left-continuous;
\item If $F\in\CC_r^{0,0}(\L_T)$, then for all $x\in\DT$ the path $t\mapsto F(t,x_t)$ is right-continuous;
\item If $F\in\CC^{0,0}(\L_T)$, then for all $x\in\DT$ the path $t\mapsto F(t,x_t)$ is \cadlag\ and continuous at each point where $x$ is continuous.
\item If $F\in\BB(\L_T)$, then $\forall x\in\DT$ the path $t\mapsto F(t,x_t)$ is bounded.
\end{enumerate}
\end{remark}

Below is one of the main results of  \cite{contf2010}: the \emph{change of variable formula for \naf s of \cadlag\ paths}. We only report the formula for \cadlag\ paths because the change of variable formula for functionals of continuous paths (\cite[Theorem 3]{contf2010}) can then be obtained with straightforward modifications. 

\begin{theorem}[Theorem 4 in \cite{contf2010}]\label{thm:fif-d}
  Let $x\in Q(\DT,\Pi)$ such that
\begin{equation}  \label{eq:ass_w}
\sup\limits_{t\in[0,T]\setminus\pi^n}|\De x(t)|\limn0.
\end{equation}
 and denote
\beq\label{eq:wn}
x^n:=\sum_{i=0}^{m(n)-1}x(t^n_{i+1}-)\ind_{[t^n_i,t^n_{i+1})}+x(T)\ind_{\{T\}}
\eeq
Then, for any $F\in\Cloc(\L_T)$, the limit
\begin{equation}  \label{eq:int-d}
  \Limn\sum_{i=0}^{m(n)-1}\vd F(t_i^n,x^{n,\De x(t^n_i)}_{t^n_i-})(x(t_{i+1}^n)-x(t_i^n))
\end{equation}
exists, denoted by $\int_0^T\vd F(t,x_{t-})\cdot\ud^{\Pi}x$, and 
\begin{align}  \label{eq:fif-d}
  F(T,x)={} & F(0,x)+\int_0^T\vd F(t,x_{t-})\cdot\ud^{\Pi}x+ \\
 &{} +\int_0^T\hd F(t,x_{t-})\ud t+\int_0^T\frac12\tr\lf\vd^2F(t,x_{t-})\ud[x]_\Pi^c(t)\rg+\nonumber  \\
 &{} +\sum_{u\in(0,T]}\lf F(u,x)-F(u,x_{u-})-\vd F(u,x_{u-})\cdot\De x(u)\rg.\nonumber
\end{align}
\end{theorem}
Note that the assumption \eq{ass_w} can always be removed, simply by including all jump times of the \cadlag\ path $\w$ in the fixed sequence of partitions $\Pi$. Hence, in the sequel we will omit such an assumption.

The proof, in the simpler case of continuous paths, turns around the idea of rewriting the variation of $F(\cdot,x)$ on \OT\ as the limit for $n$ going to infinity of the sum of the variations of $F(\cdot,x^n)$ on the consecutive time intervals in the partition $\pi^n$. In particular, these variations can be decomposed along two directions, horizontal and vertical. That is:
$$F(T,x_T)-F(0,x_{0})=\Limn\sum_{i=0}^{m(n)-1}\lf F(t_{i+1}^n,x^{n}_{t^n_{i+1}-})-F(t_{i}^n,x^{n}_{t^n_{i}-})\rg,$$
where
\begin{align}
F(t_{i+1}^n,x^{n}_{t^n_{i+1}-})-F(t_{i}^n,x^{n}_{t^n_{i}-})={}&F(t_{i+1}^n,x^{n}_{t^n_i})-F(t_{i}^n,x^{n}_{t^n_{i}})\label{eq:incr1}\\
&{}+F(t_{i}^n,x^{n}_{t^n_i})-F(t_{i}^n,x^{n}_{t^n_{i}-})\label{eq:incr2}.
\end{align}
Then, it is possible to rewrite the two increments on the right-hand side in terms of increments of two functions on $\R^d$.
Indeed: defined the left-continuous and right-differentiable function $\psi(u):=F(t_i^n+u,x^{n}_{t^n_i})$, \eq{incr1} is equal to 
$$\psi(h^n_i)-\psi(0)=\int_{t^n_i}^{t^n_{i+1}}\hd F(t,x^{n}_{t^n_{i}})\ud t,$$ while, defined the function $\phi(u):=F(t_i^n,x^{n,u}_{t^n_i-})$ of class $\C^2(B(0,\y_n),\R)$, where 
$$\y_n:=\sup\{\abs{x(u)-x(t^n_{i+1})}+\abs{t^n_{i+1}-t^n_i},\,0\leq i\leq m(n)-1,\,u\in[t^n_i,t^n_{i+1})\},$$
\eq{incr2} is equal to 
$$\phi(\d x^n_i)-\phi(0)=\vd F(t_i^n,x^{n}_{t^n_i-})\cdot\d x^n_i+\frac12\tr\lf\vd^2F(t_i^n,x^{n}_{t^n_i-})\,^t\!(\d x^n_i)\d x^n_i\rg+r^n_i,$$
where $\d x^n_i:=x(t^n_{i+1})-x(t^n_i)$ and
$$r^n_i\leq K\abs{\d x^n_i}^2\sup_{u\in B(0,\y_n)}\abs{\vd^2F(t_i^n,x^{n,u}_{t^n_i-})-\vd^2F(t_i^n,x^{n}_{t^n_i-})}.$$
The sum over $i=0,\ldots,m(n)-1$ of \eq{incr1}, by the dominated convergence theorem, converges to $\int_{0}^T\hd F(t,x_t)\ud t$. On the other hand, by Lemma 12 in \cite{contf2010} and weak convergence of the Radon measures in \eq{xin}, we have
$$\sum_{i=0}^{m(n)-1}\frac12\tr\lf\vd^2F(t_i^n,x^{n}_{t^n_i-})\,^t\!(\d x^n_i)\d x^n_i\rg \limn \int_{0}^T\frac12\tr\lf\vd^2F(t,x_t) \ud[x](t)\rg$$
and the sum of the remainders goes to 0.
Therefore, the limit of the sum of the first order terms exists and the change of variable formula (see \eq{fif-c} below) holds.

The route to prove the change of variable formula for \cadlag\ paths is much more intricate than in the continuous case, but the idea is the following.
We can rewrite the variation of $F$ over \OT\ as before, but now we separate the indexes between two complementary sets $I_1(n),I_2(n)$. Namely: let $\eps>0$ and let $C_2(\eps)$ be the set of jump times such that $\sum_{s\in C_2(\eps)}\abs{\De x(s)}^2<\eps^2$ and $C_1(\eps)$ be its complementary finite set of jump times, denote $I_1(n):=\{i\in\{1,\ldots m(n)\}:\,(t^n_i,t^n_{i+1}]\cap C_1(\eps)\neq0\}$ and $I_2(n):=\{i\in\pi^n: i\notin I_1(n)\}$, then
\begin{align*}
F(T,x_T)-F(0,x_{0})={}&\Limn\sum_{i\in I_1(n)}\lf F(t_{i+1}^n,x^{n,\De x(t^n_{i+1})}_{t^n_{i+1}-})-F(t_i^n,x^{n,\De x(t^n_{i})}_{t^n_i-})\rg+\\
&{}+\Limn\sum_{i\in I_2(n)}\lf F(t_{i+1}^n,x^{n,\De x(t^n_{i+1})}_{t^n_{i+1}-})-F(t_i^n,x^{n,\De x(t^n_{i})}_{t^n_i-})\rg.
\end{align*}
The first sum converges, for $n$ going to infinity, to $\sum_{u\in C_1(\eps)}\lf F(u,x_u)-F(u,x_{u-})\rg$, while the increments in the second sum are further decomposed into a horizontal and two vertical variations. After many steps:
\begin{align}
&F(T,x_T)-F(0,x_{0})=\nonumber\\
={}&\int_{(0,T]}\hd F(t,x_t)\ud t+\int_{(0,T]}\frac12\tr\lf\vd^2F(t,x_t) \ud[x](t)\rg+\nonumber\\
&{}+\Limn\sum_{i=0}^{m(n)-1}\vd F_{t_i^n}(x^{n,\De x(t^n_i)}_{t^n_i-},v^n_{t^n_i-})\cdot(x(t_{i+1}^n)-x(t_i^n))+\nonumber\\
&{}+\sum_{u\in C_1(\eps)}\lf F(u,x_u)-F(u,x_{u-})-\vd F(u,x_{u-})\cdot\De x(u)\rg+\a(\eps),\label{eq:sum}
\end{align}
where $\a(\eps)\leq K(\eps^2+T\eps)$.
Finally, the sum in \eq{sum} over $C_1(\eps)$ converges, for $\eps$ going to 0, to the same sum over $(0,T]$ and the formula \eq{fif-d} holds.

It is important to remark that to obtain the change of variable formula on continuous paths it suffices to require the smoothness of the restriction of the \naf\ $F$ to the subspace of continuous stopped paths (see \cite[Theorems 2.27,2.28]{cont-notes}). To this regard, it is defined the class $\Cb(\W_T)$ of \naf s $F$ such that there exists an extension $\tilde F$ of class $\Cb(\L_T)$ that coincides with $F$ if restricted to $\W_T$. Then, the following theorem holds:
\begin{theorem}[Theorems 2.29 in \cite{cont-notes}]\label{thm:fif-c}
 For any $F\in\Cloc(\W_T)$ and $x\in Q(C([0,T],\R^d),\Pi)$, the limit
\begin{equation}  \label{eq:int-c}
  \Limn\sum_{i=0}^{m(n)-1}\vd F(t_i^n,x^{n}_{t^n_i})(x(t_{i+1}^n)-x(t_i^n))
\end{equation}
exists, denoted by $\int_0^T\vd F(t,x_{t})\cdot\ud^{\Pi}x$, and
\begin{align}  \label{eq:fif-c}
  F(T,x)={} & F(0,x)+\int_0^T\vd F(t,x_{t})\cdot\ud^{\Pi}x+ \\
 &{} +\int_0^T\hd F(t,x_{t})\ud t+\int_0^T\frac12\tr\lf\vd^2F(t,x_{t})\ud[x](t)\rg.\nonumber
\end{align}
\end{theorem}

As remarked in \cite{contf2010}, the change of variable formula \eq{fif-d} also holds in the case of right-continuous functionals instead of left-continuous, by redefining the pathwise integral \eq{int-d} as
$$\Limn\sum_{i=0}^{m(n)-1}\vd F_{t_{i+1}^n}(x^{n}_{t^n_i},v^n_{t^n_i})\cdot(x(t_{i+1}^n)-x(t_i^n))$$
and the stepwise approximation $x^n$ in \eq{wn} as
$$x^n:=\sum_{i=0}^{m(n)-1}x(t^n_{i})\ind_{[t^n_i,t^n_{i+1})}+x(T)\ind_{\{T\}}.$$

\chapter{Functional \ito\ Calculus}
\label{chap:fic}

The  \lq\ito\ calculus\rq\ is a  powerful tool at the core of stochastic analysis and lies at the foundation of modern Mathematical Finance. It is a calculus which applies to functions of the current state of a stochastic process, and extends the standard differential calculus to functions of processes with non-smooth paths of infinite variation. However, in many applications, uncertainty affects the current situation even through the whole (past) history of the process and it is necessary to consider functionals, rather than functions, of a stochastic process, i.e. quantities of the form
$$F(X_t),\quad\text{where }X_t=\{X(u), u\in[0,t]\}.$$
These ones appear in many financial applications, such as the pricing and hedging of path-dependent options, and in (non-Markovian) stochastic control problems.
One framework allowing to deal with functionals of stochastic processes is  the Fr\'echet calculus, but many path-dependent  quantities intervening in stochastic analysis  are  not Fr\'echet-differentiable. This instigated the development of a new theoretical framework to deal with functionals of a stochastic process: the Malliavin calculus \cite{malliavin,nualart09}, which is a weak (variational) differential  calculus for functionals on the Wiener space. The theory of Malliavin calculus has found many applications in financial mathematics, specifically to problems dealing with path-dependent instruments. 
However, the Malliavin derivative involves perturbations affecting the whole path (both past and future) of the process. This notion of perturbation is not readily interpretable in applications such as optimal control, or hedging, where the quantities are required to be causal or non-anticipative processes.

In an  insightful paper, Bruno Dupire~\cite{dupire}, inspired by methods used by practitioners for the sensitivity analysis of path-dependent derivatives, introduced a new notion of functional derivative, and used it to extend the \ito\ formula to the path-dependent case.
Inspired by Dupire's work, Cont and Fourni\'e \cite{ContFournie09a,contf2010,contf2013} developed a rigorous mathematical framework for a  path-dependent extension of the \ito\ calculus, the  Functional \ito\ Calculus~\cite{contf2013}, as well as a purely pathwise functional calculus~\cite{contf2010} (see \chap{pfc}), proving the pathwise nature of some of the results obtained in the probabilistic framework.

The idea is to control the variations of a functional along a path by controlling its sensitivity to horizontal and vertical perturbations of the path, by defining functional derivatives corresponding to infinitesimal versions of these perturbations.
These tools led to 
\begin{itemize}\item a new class of {\bf ``path-dependent PDEs''} on the space of \cadlag\ paths $D([0,T],R^d)$, extending the Kolmogorov equations to a non-Markovian setting,
\item a {\bf universal hedging formula} and a {\bf universal pricing equation} for path-dependent options.\end{itemize}

In this chapter we develop the key concepts and main results of the Functional Ito calculus, following \citet{contf2013,cont-notes}.

\section{Functional \ito\ formulae}
\sectionmark{The Functional \ito\ formula}
\label{sec:fif}

The change of variable formula \eq{fif-d} implies as a corollary the extension of the classical \ito\ formula to the case of \naf s, called the \emph{functional \ito\ formula}. This holds for very general stochastic processes as Dirichlet process, in particular for semimartingales. We report here the results obtained with respect to \cadlag\ and continuous semimartingales, in which case the pathwise integral \eq{int-d} coincides almost surely with the stochastic integral. The following theorems correspond to Proposition 6 in \cite{contf2010} and Theorem 4.1 in \cite{contf2013}, respectively. 

\begin{theorem}[Functional \ito\ formula: \cadlag\ case]\label{thm:fif-sm}
Let $X$ be a $\R^d$-valued semimartingale on $(\O,\F,\PP,\FF)$ and $F\in\Cloc(\L_T)$, then, for all $t\in[0,T)$, 
\begin{align*}
  F(t,X_t)={} & F(0,X_{0})+\int_{(0,t]}\vd F(u,X_{u-})\cdot\ud X(u)+ \\
 &{} +\int_{(0,t]}\hd F(u,X_{u-})\ud u+\int_{(0,t]}\frac12\tr\lf\vd^2F(u,X_{u-}) \ud[X]^c(u)\rg \\
 &{} +\sum_{u\in(0,t]}\lf F(u,X_{u})-F(u,X_{u-})-\vd F(u,X_{u-})\cdot\De X(u)\rg,
\end{align*}
$\PP$-almost surely. In particular, $(F(t,X_t), t\in[0,T])$ is a semimartingale.
\end{theorem}

\begin{theorem}[Functional \ito\ formula: continuous case]\label{thm:fif-csm}
Let $X$ be a $\R^d$-valued continuous semimartingale on $(\O,\F,\PP,\FF)$ and $F\in\Cloc(\W_T)$, then, for all $t\in[0,T)$, 
\begin{align}\label{eq:fif-csm}
  F(t,X_t)={} & F(0,X_{0})+\int_0^t\vd F(u,X_{u})\cdot\ud X(u)+ \\\nonumber
 &{} +\int_0^t\hd F(u,X_{u})\ud u+\int_0^t\frac12\tr\lf\vd^2F(u,X_{u})\ud[X](u)\rg
\end{align}
$\PP$-almost surely. In particular, $(F(t,X_t), t\in[0,T])$ is a semimartingale.
\end{theorem}

Although the functional \ito\ formulae are a consequence of the stronger pathwise change of variable formulae, \citet{contf2013,cont-notes} also provided a direct probabilistic proof for the functional \ito\ formula for continuous semimartingales, based on the classical \ito\ formula. 
The proof follows the lines of the proof to \thm{fif-d} in the case of continuous paths, first considering the case of $X$ having values in a compact set $K$, $\PP$-almost surely, then going to the general case.
The $i$-th increment of $F(t,X_t)$ along the $n^{\mathrm{th}}$ partition $\pi_n$ is decomposed as:
\begin{align*}
F(t_{i+1}^n,X^{n}_{t^n_{i+1}-})-F(t_{i}^n,X^{n}_{t^n_{i}-})={}&F(t_{i+1}^n,X^{n}_{t^n_i})-F(t_{i}^n,X^{n}_{t^n_{i}})\\
&{}+F(t_{i}^n,X^{n}_{t^n_i})-F(t_{i}^n,X^{n}_{t^n_{i}-}).
\end{align*}
 The horizontal increment is treated analogously to the pathwise proof, while for the vertical increment,  the classical \ito\ formula is applied to the partial map, which is a $\C^2$-function of the continuous $(\F_{t^n_i+s})_{s\geq0}$-semimartingale $(X(t^n_i+s)-X(t^n_i),\,s\geq0)$. The sum of the increments of the functionals along $\pi_n$ gives:
\begin{align*}
  F(t,X^n_t)-F(0,X^n_0)={}&\int_0^t\hd F(u,X^n_{i(u)})\ud u\\
&{}+\frac12\int_0^t\tr\lf\vd^2F(t^n_{\bar k(u,n)},X_{t^n_{\bar k(u,n)}-}^{n,X(u)-X(t^n_{\bar k(u,n)})})\ud[X](u)\rg\\
&{}+\int_0^t\vd F(t^n_{\bar k(u,n)},X_{t^n_{\bar k(u,n)}-}^{n,X(u)-X(t^n_{\bar k(u,n)})})\cdot\ud X(u).
\end{align*}
Formula \eq{fif-csm} then follows by applying the dominated convergence theorem to the Stieltjes integrals on the first two lines and the dominated convergence theorem for stochastic integrals to the stochastic integral on the third line.
As for the general case, it suffices to take a sequence of increasing compact sets $(K_n)_{n\geq0}$, $\cup_{n\geq0}K_n=\R^d$, define the stopping times $\bar\t_k:=\inf\{s<t,\,X_s\notin K_k\}\wedge t$, and apply the previous result to the stopped process $(X_{t\wedge\bar\t_k})$. Finally, taking the limit for $k$ going to infinity completes the proof.

As an immediate corollary, if $X$ is a local martingale, for any $F\in\Cb$, $F(X_t,A_t)$ has finite variation if and only if $\vd F_t=0$ $\ud[X](t)\times\ud\PP$-almost everywhere.

\section{Weak functional calculus and martingale representation}
\sectionmark{Weak functional calculus}
\label{sec:weak}

\citet{contf2013} extended the pathwise theory to a weak functional calculus that can be applied to all square-integrable martingales adapted to the filtration $\FF^X$ generated by a given $\R^d$-valued square-integrable \ito\ process $X$.
\citet{cont-notes} carries the extension further, that is to all square-integrable semimartingales.
Below are the main results on the functional \ito\ calculus obtained in \cite{contf2013,cont-notes}.

Let $X$ be the coordinate process on the canonical space $\DT$ of $\R^d$-valued \cadlag\ processes and $\PP$ be a probability measure under which $X$ is a square-integrable semimartingale such that
\beq
\ud [X](t)=\int_0^tA(u)\ud u
\eeq
for some $\S^d_+$-valued \cadlag\ process $A$ satisfying 
\beq\label{eq:Anon-deg}
\mathrm{det}(A(t))\neq0\text{ for almost every }t\in[0,T],\ \PP\text{-almost surely}.
\eeq
Denote by $\FF=\Ft$ the filtration $(\F^X_{t+})_{t\in[0,T]}$ after $\PP$-augmentation.
Then define
\beq\label{eq:CX}
\Cloc(X):=\{Y:\;\exists F\in\Cloc,\;Y(t)=F(t,X_t)\; \ud t\times\ud\PP\text{-a.e.}\}.
\eeq
Thanks to the assumption \eq{Anon-deg}, for any adapted process $Y\in\Cb(X)$, the \emph{vertical derivative of $Y$ with respect to $X$}, $\nabla_XY(t)$, is well defined as $\nabla_XY(t)=\vd F(t,X_t)$ where $F$ satisfies \eq{CX}, and it is unique up to an evanescent set independently of the choice of $F\in\Cb$ in the representation \eq{CX}.

\thm{fif-sm} leads to the following representation for smooth local martingales.
\begin{proposition}[Prop. 4.3 in \cite{cont-notes}]
  Let $Y\in\Cb(X)$ be a local martingale, then
$$Y(T)=Y(0)+\int_0^T\nabla_XY(t)\cdot\ud X(t).$$
\end{proposition}
On the other hand, under specific assumptions on $X$, this leads to an explicit martingale representation formula. 
\begin{proposition}[Prop. 4.3 in \cite{cont-notes}]\label{prop:mg-repr}
If $X$ is a square-integrable $\PP$-Brownian martingale, for any square integrable $\FF$-martingale $Y\in\Cloc(X)$, then $\nabla_XY$ is the unique process in the Hilbert space
$$\LL(X):=\left\{\phi\,\text{progressively-measurable},\;\EE^\PP\left[\int_0^T|\phi(t)|^2\ud[X](t)\right]<\infty\right\},$$ 
endowed with the norm $\ds\norm{\phi}_{\LL(X)}:=\EE^\PP\left[\int_0^T|\phi(t)|^2\ud[X](t)\right]^{\frac12}$,
such that
$$Y(T)=Y(0)+\int_0^T\nabla_XY(t)\cdot\ud X(t)\quad \PP\text{-a.s.}$$  
\end{proposition}

This is used in \cite{contf2013} to extend the domain of the vertical derivative operator $\nabla_X$ to the space of square-integrable $\FF$-martingales $\M^2(X)$, by a density argument.
 

On the space of smooth square-integrable martingales, $\Cb(X)\cap\M^2(X)$, which is dense in $\M^2(X)$, an integration-by-parts formula holds: for any $Y,Z\in \Cb(X)\cap\M^2(X)$,
$$\EE[Y(T)Z(T)]=\EE\left[\int_0^TY(T)Z(T)\ud[X](t)\right].$$
By this and by density of $\{\nabla_XY,\,Y\in\Cloc(X)\}$ in $\LL(X)$, the extension of the vertical derivative operator follows.
\begin{theorem}[Theorem 5.9 in \cite{contf2013}]\label{th:weakvd}
  The operator $\nabla_X:\Cb(X)\cap\M^2(X)\rightarrow\LL(X)$ admits a closure in $\M^2(X)$. Its closure is a bijective isometry 
\beq\label{eq:vdM2}
\nabla_X:\M^2(X)\rightarrow\LL(X), \quad \int_0^\cdot\phi(t)\ud X(t)\mapsto\phi,
\eeq
characterized by the property that, for any $Y\in\M^2$, $\nabla_X Y$ is the unique element of $\LL(X)$ such that $$\forall Z\in \Cb(X)\cap\M^2(X),\quad \EE[Y(T)Z(T)]=\EE\left[\int_0^T\nabla_XY(t)\nabla_XZ(t)\ud[X](t)\right].$$
In particular $\nabla_X$ is the adjoint of the \ito\ stochastic integral 
$$I_X:\LL(X)\rightarrow\M^2(X), \quad \phi\mapsto\int_0^\cdot\phi(t)\cdot\ud X(t),$$
in the following sense: for all $\phi\in\LL(X)$ and for all $Y\in\M^2(X)$,
$$\EE\left[Y(T)\int_0^T\phi(t)\cdot\ud X(t)\right]=\EE\left[\int_0^T\nabla_XY(T)\phi(t)\ud[X](t)\right].$$
\end{theorem}

Thus, for any square-integrable $\FF$-martingale $Y$, the following martingale representation formula holds:
\beq\label{eq:mgrepr}
Y(T)=Y(0)+\int_0^T\nabla_XY(t)\cdot\ud X(t), \quad \PP\text{-a.s.}
\eeq



Then, denote by $A^2(\FF)$ the space of $\FF$-predictable absolutely continuous processes $H=H(0)+\int_0^\cdot h(u)\ud u$ with finite variation, such that 
$$\norm{H}^2_{\A^2}:=\EE^{\PP}\left[ \abs{H(0)}^2+ \int_0^T\abs{h(u)}^2\ud u\right]<\infty$$
and by $\S^{1,2}(X)$ the space of square-integrable $FF$-adapted special semimartingales, $\S^{1,2}(X)=\M^2(X)\oplus\A^2(\FF)$, equipped with the norm $\norm{\cdot}_{1,2}$ defined by
$$\norm{S}_{1,2}^2:=\EE^\PP\left[[M](T)\right]+\norm{H}^2_{\A^2}, \quad S\in\S^{1,2}(X),$$
where $S=M+H$ is the unique decomposition of $S$ such that $M\in\M^2(X)$, $M(0)=0$ and $H\in\A^2(\FF)$, $H(0)=S(0)$.

The vertical derivative operator admits a unique continuous extension to $\S^{1,2}(X)$ such that its restriction to $\M^2(X)$ coincides with the bijective isometry in \eq{vdM2} and it is null if restricted to $\A^2(\FF)$.

By iterating this construction it is possible to define a series of \lq Sobolev\rq\ spaces $\S^{k,2}(X)$ on which the vertical derivative of order $k$, $\nabla^k_X$ is defined as a continuous operator. We restrict our attention to the space of order 2:
$$\S^{2,2}(X):=\{Y\in\S^{1,2}(X):\;\nabla_X Y\in\S^{1,2}(X)\},$$
equipped with the norm $\norm{\cdot}_{2,2}^2$ defined by
$$\norm{Y}_{2,2}^2=\norm{H}^2_{\A^2}+\norm{\nabla_XY}_{\LL(X)}+\norm{\nabla^2_{X}Y}_{\LL(X)},\quad Y\in\S^{2,2}(X).$$

Note that the second vertical derivative of a process $Y\in\S^{2,2}(X)$ has values in $\R^d\times\R^d$ but it needs not be a symmetric matrix, differently from the (pathwise) second vertical derivative of a smooth functional $F\in\C^{1,2}_b(\L_T)$.

The power of this construction is that it is very general, e.g. it applies to functionals with no regularity, and it makes possible to derive a \lq weak functional \ito\ formula\rq\ involving vertical derivatives of square-integrable processes and a weak horizontal derivative defined as follow.
For any $S\in S^{2,2}(X)$, the weak horizontal derivative of $S$ is the unique $\FF$-adapted process $\hd S$ such that: for all $t\in[0,T]$
\beq\label{eq:weak-hd}
\quad \int_0^t\hd S(u)\ud u=S(t)-S(0)-\int_0^t\nabla_XS\ud X-\frac12\int_0^t\tr(\nabla^2_XS(u)\ud[X](u))
\eeq
and $\EE^\PP\left[\int_0^T\abs{\hd S(t)}^2\ud t\right]<\infty$.
\begin{proposition}[Proposition 4.18 in \cite{cont-notes}]
For any $S\in S^{2,2}(X)$, the following \lq weak functional \ito\ formula\rq\ holds $\ud t\times\ud\PP$-almost everywhere:
\beq\label{eq:weak-ito}
S(t)=S(0)+\int_0^t\nabla_XS\ud X+\frac12\int_0^t\tr(\nabla^2_XS\ud[X])+\int_0^t\hd S(u)\ud u.
\eeq
\end{proposition}

\section{Functional Kolmogorov equations}
\label{sec:kolmogorov}

Another important result in \cite{cont-notes} is the characterization of smooth harmonic functionals as solutions of functional Kolmogorov equations.
Specifically, a \naf\ $F:\L_T\to\R$ is called \emph{$\PP$-harmonic} if $F(\cdot,X_\cdot)$ is a $\PP$-local martingale, where $X$ is the unique weak solution to the path-dependent stochastic differential equation
$$\ud X(t)=b(t,X_t)\ud t+\s(t,X_t)\ud W(t),\quad X(0)=X_0,$$
where $b,\s$ are \naf s with enough regularity and $W$ is a $d$-dimensional Brownian motion on $(\DT,\F_T,\PP)$.

\begin{proposition}[Theorem 5.6 in \cite{cont-notes}]\label{prop:harmonic}
  If $F\in\Cb(\W_T)$, $\hd F\in\CC^{0,0}_l(\W_T)$, then $F$ is a $\PP$-harmonic functional if and only if it satisfies
\beq\label{eq:FPDE-cont}
\hd F(t,\w_t)+b(t,\w_t)\vd F(t,\w_t)+\frac12\tr\lf\vd^2F(t,\w_t)\s(t,\w_t){}^t\!\s(t,\w_t)\rg=0
\eeq
for all $t\in[0,T]$ and all $\w\in\supp(X)$, where
\begin{align}  \label{eq:supp}
  \supp(X):=\big\{&\w\in C([0,T],\R^d):\;\PP(X_T\in V)>0\\
&\forall \text{ neighborhood $V$ of }\w\text{ in }\lf C([0,T],\R^d),\norm{\cdot}_\infty\rg\big\},\nonumber
\end{align} 
is the topological support of $(X,\PP)$ in $(C([0,T],\R^d),\norm{\cdot}_\infty)$.
\end{proposition}

Analogously to classical finite-dimensional parabolic PDEs, we can introduce the notions of sub-solution and super-solution of the functional (or path-dependent) PDE \eq{FPDE-cont}, for which \cite{cont-notes} proved a comparison principle allowing to state uniqueness of solutions.

\begin{definition}
  $F\in\CC^{1,2}(\L_T)$ is called a \emph{sub-solution} (respectively \emph{super-solution}) of \eq{FPDE-cont} on a domain $U\subset\L_T$ if, for all $(t,\w)\in U$,
  \begin{equation} \label{eq:sub}
    \hd F(t,\w_t)+b(t,\w_t)\vd F(t,\w_t)+\frac12\tr\lf\vd^2F(t,\w_t)\s(t,\w_t){}^t\!\s(t,\w_t)\rg\geq0
  \end{equation}
(resp. $\hd F(t,\w_t)+b(t,\w_t)\vd F(t,\w_t)+\frac12\tr\lf\vd^2F(t,\w_t)\s(t,\w_t){}^t\!\s(t,\w_t)\rg\leq0$).
\end{definition}

\begin{theorem}[Comparison principle (Theorem 5.11 in \cite{cont-notes})]
  Let $\under F\in\CC^{1,2}(\L_T)$ and  $\over F\in\CC^{1,2}(\L_T)$ be respectively a sub-solution and a super-solution of \eq{FPDE-cont}, such that
$$\bea{c}\forall\w\in C([0,T,\R^d),\quad \under F(T,\w)\leq\over F(T.\w),\\
\EE^\PP\left[ \sup_{t\in[0,T]}|\under F(t,X_t)-\over F(t,X_t)|\right]<\infty.\ea$$
Then, $$\forall t\in[0,T),\,\forall\w\in\supp(X),\quad\under F(t,X_t)\leq\over F(t,X_t).$$
\end{theorem}
This leads to a uniqueness result on the topological support of $X$ for $\PP$-uniformly integrable solutions of the functional Kolmogorov equation.

\begin{theorem}[Uniqueness of solutions (Theorem 5.12 in \cite{cont-notes})]
  Let $H:(C([0,T],\R^d),\norm{\cdot}_\infty)\to\R$ be continuous and let $F^1,F^2\in\Cb(\L_T)$ be solutions of \eq{FPDE-cont} verifying 
$$\bea{c}\forall\w\in C([0,T,\R^d),\quad F^1(T,\w)=F^2(T.\w)=H(\w_T),\\
\EE^\PP\left[ \sup_{t\in[0,T]}|F^1(t,X_t)-F^2(t,X_t)|\right]<\infty.\ea$$
Then:
$$\forall (t,\w)\in[0,T]\times\supp(X),\quad F^1(t,\w)=F^2(t,\w).$$
\end{theorem}

The uniqueness result, together with the representation of $\PP$-harmonic functionals as solutions of a functional Kolmogorov equation, leads to a Feynman-Kac formula for \naf s.

\begin{theorem}[Feynman-Kac, path-dependent (Theorem 5.13 in \cite{cont-notes})]
  Let $H:(C([0,T],\R^d),\norm{\cdot}_\infty)\to\R$ be continuous and let $F\in\Cb(\L_T)$ be a solution of \eq{FPDE-cont} verifying $F(T,\w)=H(\w_T)$ for all $\w\in C([0,T,\R^d)$ and $\EE^\PP\left[\sup_{t\in[0,T]}|F(t,X_t)|\right]<\infty$.
Then:
$$\quad F(t,\w)=\EE^\PP[H(X_T)|\F_t]\quad\ud t\times\ud\PP\text{-a.s.}$$
\end{theorem}

\subsection{Universal pricing and hedging equations}
\label{sec:hedgeprice}

Straightforward applications to the pricing and hedging of path-dependent derivatives then follow from the representation of $\PP$-harmonic functionals.

Now we consider the point of view of a market agent and we suppose that the asset price process $S$ is modeled as the coordinate process on the path space $\DT$, and it is a square-integrable martingale under a pricing measure $\PP$,
$$\ud S(t)=\s(t,S_t)\ud W(t).$$
Let $H:\DT\to\R$ be the payoff functional of a path-dependent derivative that the agent wants to sell. The price of such derivative at time $t$ is computed as 
$$Y(t)=\EE^\PP\left[H(S_T)\mid\F_t\right].$$

The following proposition is a direct corollary of \prop{mg-repr}.
\begin{proposition}[Universal hedging formula]
If $\EE^\PP\left[\abs{H(S_T)}^2\right]<\infty$ and if the price process has a smooth functional representation of $S$, that is $Y\in\Cloc(S)$, then:
\begin{align}
\PP\text{-a.s.}\quad H&=\EE^\PP\left[H(S_T)\mid\F_t\right]+\int_t^T\nabla_S Y(u)\cdot\ud S\label{eq:implicithedge}\\
&=\EE^\PP\left[H(S_T)\mid\F_t\right]+\int_t^T\vd F(u,S_u)\cdot\ud S,\label{eq:univ-hedge}
\end{align}
where $Y(t)=F(t,S_t)$ $\ud t\times\ud\PP$-almost everywhere and $\vd F(\cdot,S_\cdot)$ is the unique (up to indistinguishable processes) asset position process of the hedging strategy for $H$.
\end{proposition}
We refer to the equation \eq{univ-hedge} as the \lq universal hedging formula\rq, because it gives an explicit representation of the hedging strategy for a path-dependent option $H$. The only dependence on the model lies in the computation of the price $Y$.
\begin{remark}
If the price process does not have a smooth functional representation of $S$, but the payoff functional still satisfies $\EE^\PP\left[\abs{H(S_T)}^2\right]<\infty$, then the equation \eq{implicithedge} still holds.
\end{remark}
In this case, the hedging strategy is not given explicitly, being the vertical derivative of a square-integrable martingale, but it can be uniformly approximated by regular functionals that are the vertical derivatives of smooth \naf s.
Namely: there exists a sequence of smooth functionals 
$$\{F^n\in\Cb(\L_T),\,F^n(\cdot,S_\cdot)\in\M^2(S),\,\norm{F^n(\cdot,S_\cdot)}_2<\infty\}_{n\geq1},$$ 
where $$\norm{Y}_2:=\EE^\PP\left[\abs{Y(T)}^2\right]^{\frac12}<\infty,\quad Y\in\M^2(S),$$
such that $$\norm{F^n(\cdot,S_\cdot)-Y}_2\limn0\quad \text{and}\quad \norm{\nabla_S Y-\nabla_S F^n(\cdot,S_\cdot)}_{\LL(S)}\limn0.$$

For example, \citet{contlu} compute an explicit approximation for the integrand in the representation \eq{implicithedge}, which cannot be itself computed through pathwise perturbations. They allow the underlying process $X$ to be the strong solution of a path-dependent stochastic differential equation with non-anticipative Lipschitz-continuous and non-degenerate coefficients, then they consider the Euler-Maruyama scheme of such SDE. They proved the strong convergence of the Euler-Maruyama approximation to the original process.
By assuming that the payoff functional $H:(\DT,\norm{\cdot}_\infty)\to\R$ is continuous with polynomial growth, they are able to define a sequence $\{F_n\}_{n\geq1}$ of smooth functionals $F_n\in\CC^{1,\infty}(\L_T)$ that approximate the pricing functional and provide thus a smooth functional approximation sequence $\{\vd F_n(\cdot,S_\cdot)\}_{n\geq1}$ for the hedging process $\nabla_SY$.

Another application is derived from \prop{harmonic} for the pricing of path-dependent derivatives.
\begin{proposition}[Universal pricing equation]\label{prop:universalprice} If there exists a smooth functional representation of the price process $Y$ for $H$, i.e.
$$  \exists F\in\Cb(\W_T):\quad F(t,S_t)=\EE^{\PP}[H(S_T)|\F_t]\quad \ud t\times \ud\PP\text{-a.s.},$$
such that $\hd F\in\CC^{0,0}_l(\W_T)$,
then the following path-dependent partial differential equation holds on the topological support of $S$ in $\lf C([0,T],\R^d),\norm{\cdot}_\infty\rg$ for all $t\in[0,T]$:
\beq\label{eq:pricing}
\hd F(t,\w_t)+\frac12\tr\lf\vd^2F(t,\w_t)\s(t,\w_t)\,{}^t\!\s(t,\w_t)\rg=0.
\eeq
\end{proposition}

\begin{remark}
  If there exists a smooth functional representation of the price process $Y$ for $H$, but the horizontal derivative is not left-continuous, then the pricing equation \eq{pricing} cannot hold on the whole topological support of $S$ in $\lf C([0,T],\R^d),\norm{\cdot}_\infty\rg$, but it still holds for $\PP$-almost every $\w\in C([0,T],\R^d)$.
\end{remark}

\section{Path-dependent PDEs and BSDEs}
\label{sec:PPDE}

In the Markovian setting, there is a well-known relation between backward stochastic differential equations (BSDEs) and semi-linear parabolic PDEs, via the so-called nonlinear Feynman-Kac formula introduced by \citet{pardoux-peng92} (see also \citet{pardoux-peng90} for the introduction to BSDEs and \citet{elkaroui-peng-quenez} for a comprehensive guide on BSDEs and their application in finance).
This relation can be extended to a non-Markovian setting using the functional \ito\ calculus.

Consider the following forward-backward stochastic differential equation (FBSDE) with path-dependent coefficients:
\begin{eqnarray}
  X(t)&=&x+\int_0^tb(u,X_{u})\ud u+\int_0^t\s(u,X_{u})\cdot\ud W(u)\label{eq:FBSDE1}\\
  Y(t)&=&H(X_T)+\int_t^Tf(u,X_{u},Y(u),Z(u))\ud u-\int_t^TZ(u)\cdot\ud X(u)\label{eq:FBSDE2},
\end{eqnarray}
where $W$ is a $d$-dimensional Brownian motion on $(D([0,T],\R^d),\PP)$, $\FF=\Ft$ is the $\PP$-augmented natural filtration of the coordinate process $X$, the terminal value is a square-integrable $\F_T$-adapted random variable, i.e. $H\in L^2(\O,\F_T,\PP)$, and the coefficients
$$b:\W_T\to\R^d,\ \s:\W_T\to\R^{d\times d},\ f:\W_T\times\R\times\R^d\to\R$$
are assumed to satisfy the standard assumptions that guarantee that the process $M$, $M(t)=\int_0^t\s(u,X_{u})\cdot\ud W(u)$ is a square-integrable martingale, and the forward equation \eq{FBSDE1} has a unique strong solution $X$ satisfying $\EE^\PP\left[\sup_{t\in[0,T]}|X(t)|^2\right]<\infty$. Moreover, assuming also $det\lf\s(t,X_{t-},X(t))\rg\neq0$ $\ud t\times\ud\PP$-almost surely, they guarantee that the FBSDE \eq{FBSDE1}-\eq{FBSDE2} has a unique solution $(Y,Z)\in\S^{1,2}(M)\times\L^2(M)$ such that $\EE^\PP\left[\sup_{t\in[0,T]}|Y(t)|^2\right]<\infty$ and $Z=\nabla_MY$.

The following is the extension of the non-linear Feynman-Kac formula of \cite{pardoux-peng92} to the non-Markovian setting.

\begin{theorem}[Theorem 5.14 in \cite{cont-notes}]\label{thm:FK-cont}
Let $F\in\Cloc(\W_T)$ be a solution of the path-dependent PDE
$$\begin{cases}
  \hd F(t,\w)+f(t,\w_t,F(t,\w)\vd F(t,\w))+\frac12\tr(\s(t,\w)\,^t\!\s(t,\w)\vd^2F(t,\w))=0\\
  F(T,\w)=H(\w_T)
\end{cases}$$
for $(t,\w)\in[0,T]\times\supp(X)$. Then, the pair $(Y,Z)=(F(\cdot,X_\cdot),\vd F(\cdot,X_\cdot))$ solves the FBSDE \eq{FBSDE1}-\eq{FBSDE2}.
\end{theorem}
Together with the standard comparison theorem for BSDEs, \thm{FK-cont} provides a comparison principle for functional Kolmogorov equations and uniqueness of the solution.

To prove existence of a solution to \eq{FPDE-cont}, additional regularity of the coefficients is needed. A result in this direction is provided by \citet{peng}, using BSDEs where the forward process is a Brownian motion.
\citet{peng} considers the following backward stochastic differential equation:
\beq\label{eq:peng}
Y^{(t,\g)}(s)=H(W^{(t,\g)}_T)+\int_s^Tf(W^{(t,\g)}_u,Y^{(t,\g)}(u),Z^{(t,\g)}(u))\ud u-\int_s^TZ^{(t,\g)}(u)\ud W(u),
\eeq
where $W$ is the coordinate process on the Wiener space $(C([0,T],\R^d),\PP)$ and, for all $(t,\g)\in\L_T$, $W^{(t,\g)}=\g\ind_{[0,t)}+(\g(t)+W-W(t))\ind_{[t,T]}$.
Note that the notation has been rearranged to be consistent with the presentation in this thesis.

The BSDE \eq{peng} has a unique solution $(Y^{(t,\g)},Z^{(t,\g)})\in S^2([t,T])\times M^2([t,T])$, where $M^2([t,T])$ and $S^2([t,T])$ denote respectively the space of $\R^m$-valued processes $X$ such that $X\in L^2([t,T]\times\O,\ud t\times\ud\PP)$ and $\R^{m\times d}$-valued processes $X$ such that $\EE^\PP[\sup_{u\in[t,T]}|X(u)|^2]<\infty$, both adapted to the completion of the filtration generated by $\{W(u)-W(t),\,u\in[t,T]\}$, under the following assumptions on the coefficients:
\begin{enumerate}
\item $H:\L_T\to\R^m$ satisfies
  \begin{enumerate}
  \item $\psi^{(t,\g)}:\R^d\to\R^m,\,e\mapsto H(\g+e\ind_{[t,T]})$ is twice differentiable in 0 for all $(t,\g)\in[0,T]\times D([0,T],\R^d)$,
  \item $\abs{H(\g_T)-H(\g'_T)}\leq C(1+\norm{\g_T}_\infty^k+\norm{\g'_T}_\infty^k)\norm{\g_T-\g'_T}_\infty$ for all $\g,\g'\in D([0,T],\R^d)$,
  \item $\partial^j_e\psi^{(t,\g)}(0)-\partial^j_e\psi^{(t',\g')}(0)\leq C(1+\norm{\g_T}_\infty^k+\norm{\g'_T}_\infty^k)(\abs{t-t'}+\norm{\g_T-\g'_T}_\infty)$ for all $\g,\g'\in D([0,T],\R^d)$, $t,t'\in[0,T]$, $j=1,2$;
  \end{enumerate}
\item $f:\L_T\times\R^m\times\R^{m\times d}\to\R^m$ is continuous; for any $(t,\g)\in\L_T$ and $s\in[0,t]$ $(x,y,z)\mapsto f(t,\g_t+x\ind_{[s,T]},y,z)$ is of class $C^3(\R^d\times\R^m\times\R^{m\times d},\R^m)$ with first-order partial derivatives and second-order partial derivatives with respect to $(y,z)$ uniformly bounded, and all partial derivatives up to order three growing at most as a polynomial at infinity; for any $(t,y,z)$, $\g\mapsto f(t,\g_t,y,z)$ satisfies assumptions 1(a),1(b),1(c) replacing $H$ with $f(t,\cdot_t,y,z)$, $\g\mapsto \partial_yf(t,\g_t,y,z),\partial_zf(t,\g_t,y,z)$ satisfy assumptions 1(a),1(b) and 1(c) with only $j=1$, and $$\g\mapsto \partial_{yy}f(t,\g_t,y,z),\partial_{zz}f(t,\g_t,y,z),\partial_{yz}f(t,\g_t,y,z)$$ satisfy the assumptions 1(a),1(b).
\end{enumerate}
The functional Kolmogorov equation associated is the following: for all $\g\in D([0,T],\R^d)$ and $t\in[0,T]$,
\beq\label{eq:FPDE-peng}
\begin{cases}\hd F(t,\g_t)+\frac12\tr(\vd^2F(t,\g_t))+f(t,\g_t,F(t,\g_t),\vd F(t,\g_t))=0,\\F(T,\g_T)=H(\g_T)
\end{cases}
\eeq
First, by the functional \ito\ formula, they obtain the analogue of \thm{FK-cont}, then they prove the converse result: the \naf\ $F$ defined by $F(t,\g)=Y^{(t,\g_t)}(t)$ is the unique $\CC^{1,2}(\L_T)$-solution of the functional Kolmogorov equation \eq{FPDE-peng}. This significant result is achieved based on the theory of BSDEs.

Another approach to study the connection between PDEs and SDEs in the path-dependent case is provided by \citet{flandoli-zanco}, who reformulate the problem into an infinite-dimensional setting on Banach spaces, where solutions of the SDE are intended in the mild sense and the Kolmogorov equations are defined appropriately. However, in the infinite-dimensional framework, the regularity requirements are very strong, involving at least Fr\'echet differentiability.

\subsection{Weak and viscosity solutions of path-dependent PDEs}
\label{sec:viscosity}

The results seen above in \Sec{kolmogorov} require a regularity that is often difficult to prove and classical solutions of the above path-dependent PDEs may fail to exist. To find a way around this issue, more general notions of solution have been proposed, analogously to the Markovian case where weak solutions of PDEs are considered or viscosity solutions are used to link solutions of BSDEs to the associated PDE.

\citet{cont-notes} proposed the following notion of \emph{weak solution}, using the weak functional \ito\ calculus presented in \Sec{weak} and generalizing \prop{harmonic}.

Consider the stochastic differential equation \eq{FBSDE1} with path-dependent coefficients such that $X$ is the unique strong solution and $M$ is a square-integrable martingale.

Denote by $\WW^{1,2}(\PP)$ the Sobolev space of $\ud t\times\ud\PP$-equivalence classes of \naf s $F:\L_T\to\R$ such that the process $S=F(\cdot,X_\cdot)$ belongs to $\S^{1,2}(M)$, equipped with the norm $\norm{\cdot}_{\WW^{1,2}}$ defined by
\begin{align*}
\norm{F}^2_{\WW^{1,2}}:={}&\norm{F(\cdot,X_\cdot)}_{1,2}^2\\
={}&\EE^\PP\left[|F(0,X_0)|^2+\int_0^T\tr(\nabla_M F(t,X_t)^t\nabla_M F(t,X_t)\ud[M](t))\right.\\
&\left.\quad{}+\int_0^T|v(t)|^2\ud t\right],
\end{align*}
where $F(t,X_t)=V(t)+\int_0^t\nabla_MS\ud M$ and $V(t)=S(0)+\int_0^tv(u)\ud u$, $V\in A^2(\FF)$.
Equivalently, $\WW^{1,2}(\PP)$ can be defined as the completion of $(\Cb(\L_T),\norm{\cdot}_{\WW^{1,2}})$.

Note that, in general, it is not possible to define for $F\in\WW^{1,2}(\PP)$ the $\FF$-adapted process $\hd F(\cdot,X_\cdot)$, because it requires $F\in\S^{2,2}(M)$. On the other hand, the finite-variation part of $S$ belongs to the Sobolev space $H^1([0,T])$, so the process $U$ defined by
$$U(t):=F(T,X_T)- F(t,X_t)-\int_t^T\nabla_MF(u,X_u)\ud M(u),\quad t\in[0,T],$$
has paths in $H^1([0,T])$, almost surely.
By integration by parts, for all $\Phi\in A^2(\FF)$, $\Phi(t)=\int_0^t\phi(u)\ud u$ for $t\in[0,T]$,
$$\bea{l}\int_0^T\Phi(t)\frac\ud{\ud t}\lf F(t,X_t)-\int_0^t\nabla_MF(u,X_u)\ud M(u)\rg \ud t\\
=\int_0^T\Phi(t)\lf-\frac\ud{\ud t}U(t)\rg\ud t\\
=\int_0^T\phi(t)\lf F(T,X_T)-F(t,X_t)-\int_t^T\nabla_MF(u,X_u)\ud M(u)\rg \ud t.
\ea$$

Thus, the following notion of weak solution is well defined.
\begin{definition}
  A \naf\ $F\in\WW^{1,2}(\PP)$ is called a \emph{weak solution} of the path-dependent PDE \eq{FPDE-cont} on $\supp(X)$ with terminal condition $H(X_T)\in L^2(\O,\PP)$ if, for all $\phi\in L^2([0,T]\times\O,\ud t\times\ud\PP)$, it satisfies
\beq\label{eq:weak}
\begin{cases}
\EE^\PP\left[\int_0^T\phi(t)\lf H(X_T)-F(t,X_t)-\int_t^T\nabla_MF(u,X_u)\ud M(u)\rg \ud t\right]=0,\\
F(T,X_T)=H(X_T).
\end{cases}
\eeq
\end{definition}

Using the tools provided by the functional \ito\ calculus presented in this chapter, different notions of viscosity solutions have been recently proposed, depending on the path-dependent partial differential equation considered. 
\citet{ektz} proposed a notion of viscosity solution for semi-linear parabolic path-dependent PDEs that allows to extend the non-linear Feynman-Kac formula to non-Markovian case. \citet{ektz1} generalizes the definition of viscosity solutions introduced in \cite{ektz} to deal with fully non-linear path-dependent parabolic PDEs. Then, in \cite{ektz2} they prove a comparison result for such viscosity solutions that implies a well-posedness result. \citet{cosso} extended the results of \cite{ektz2} to the case of a possibly degenerate diffusion coefficient for the forward process driving the BSDE.

We remark that, although these approaches are useful to study solutions of path-dependent PDEs from a theoretical point of view and in applications, the problem studied in this thesis cannot be faced by means of viscosity or weak solutions. This is due to the fact that the change of variable formula for \naf s and the pathwise definition of the F\"ollmer integral are the key tools that allow us to achieve the robustness results, and they require smoothness ($\CC^{1,2}$ regularity) of the portfolio value functionals.

\chapter{A pathwise approach to continuous-time trading}
\label{chap:path-trading}
\chaptermark{Pathwise continuous-time trading}

The \ito\ theory of stochastic integration defines the integral of a general non-anticipative integrand as either an $L^2$ limit or a limit in probability of non-anticipative Riemann sums. The resulting integral is therefore defined almost-surely but does not have a well-defined value along a given sample path.
If one interprets such an integral as the gain of a strategy, this poses a problem for interpretation: the gain cannot necessarily be defined for a given scenario, which does not make sense financially.
It is therefore important to dispose of a construction which allows to give  a meaning to such integrals in a pathwise sense.

In this Chapter, after reviewing in \Sec{lit-path} various approaches proposed in the literature for the pathwise construction of integrals with respect to stochastic processes, we present an analytical setting where the pathwise computation of the gain from trading is possible for a class of continuous-time trading strategies which includes in particular \lq delta-hedging\rq\ strategies. This construction also allows for a pathwise definition of the self-financing property.

\section{Pathwise integration and model-free arbitrage}
\label{sec:lit-path}

\subsection{Pathwise construction of stochastic integrals}
\label{sec:pathint}

A first attempt to a pathwise construction of the stochastic integral deals with Brownian integrals and dates back to the sixties, due to \citet{wong-zakai}. They stated that, for a restricted class of integrands, the sequence of Riemann-Stieltjes integrals obtained by replacing the Brownian motion with a sequence of approximating smooth integrators converges in mean square (hence pathwise along a properly chosen subsequence) to a Stratonovich integral. 
This approach is based on approximating the integrator process.

In 1981, \citet{bichteler} obtained almost-sure convergent subsequences by using stopping times. Namely, given a \caglad\ process $\phi$ and a sequence of non-negative real numbers $(c_n)_{n\geq0}$ such that $\sum\limits_{n\geq0}c_n<\infty$, by defining for each $n\geq0$ a sequence of stopping times $T^n_0=0$, $T_{k+1}^n=\inf\{t>T^n_k:\,|\phi(t)-\phi(T^n_k)|>c_n\}$, $k\geq 0$, for a certain class of integrands $M$ (more general than square-integrable martingales) the following holds: for almost all states $\w\in\O$, $\lf\int\phi\ud M\rg(\w)$ is the uniform limit on every bounded interval of the pathwise integrals $\lf\int\phi^n\ud M\rg(\w)$ of the approximating elementary processes $\phi^n(t)=\sum\limits_{k\geq0}\phi(T^n_k)\ind_{(T^n_k,T^n_{k+1}]}(t)$, $t\geq0$.
Though Bichteler's method is constructive, it involves stopping times. Moreover, note that the $\PP$-null set outside of which convergence does not hold depends on $\phi$.

\subsubsection{Pathwise stochastic integration by means of ``skeleton approximations''}
In 1989, Willinger and Taqqu~\cite{willtaq} proposed a constructive method to compute stochastic integrals path-by-path by making both time and the probability space discrete.
The discrete and finite case contains the main idea of their approach and shows the connection between the completeness property, i.e. the martingale representation property, and stochastic integration.
It is given a probability space $(\O,\F,\PP)$ endowed with a filtration $\FF=(\F_t)_{t=0,1,\ldots,T}$ generated by minimal partitions of $\O$, $\F_t=\s(\P_t)$ for all $t=0,1,\ldots,T$, and an $\R^{d+1}$-valued $(\FF,\PP)$-martingale $Z=(Z(t))_{t=0,1,\ldots,T}$ with components $Z^0\equiv1$ and $Z^1(0)=\ldots=Z^d(0)=0$. They denote by $\Phi$ the space of all $\R^{d+1}$-valued $\FF$-predictable stochastic processes $\phi=(\phi(t))_{t=0,1,\ldots,T}$, where $\phi(t)$ is $\F_{t-1}$-measurable $\forall t=1,\ldots,T$, and such that 
\begin{equation}
  \label{eq:wt_1}
\phi(t)\cdot Z(t)=\phi(t+1)\cdot Z(t)\quad \PP\text{-a.s., }t=0,1,\ldots,T,  
\end{equation}
where by definition $\phi_0\equiv\phi_1$. Property~\eqref{eq:wt_1} has an interpretation in the context of discrete financial markets as the \textit{self-financing} condition for a strategy $\phi$ trading the assets $Z$, in the sense that at each trading date the investor rebalances his portfolio without neither withdrawing nor paying any cash. Moreover, it implies 
$$(\phi\bullet Z)(t):=\phi(1)\cdot Z(0)+\sum_{s=1}^t\phi(s)\cdot(Z(s)-Z(s-1))=\phi(t)\cdot Z(t)\quad \PP\text{-a.s., }t=0,1,\ldots,T,  $$
where $\phi\bullet Z$ is the \textit{discrete stochastic integral} of the predictable process $\phi$ with respect to $Z$. The last equation is still meaningful in financial terms, having on the left-hand side the initial investment plus the accumulated gain and on the right-hand side the current value of the portfolio.
The $\R^{d+1}$-valued $(\FF,\PP)$-martingale $Z$ is defined to be \textit{complete} if for every real random variable $Y\in L^1(\O,\F,\PP)$ there exists $\phi\in\Phi$ such that for $\PP$-almost all $\w\in\O$, $Y(\w)=(\phi\bullet Z)(T,\w)$, i.e.
\begin{equation}
  \label{eq:wt_2}
  \{\phi\bullet Z,\;\phi\in\Phi\}=L^1(\O,\F,\PP).
\end{equation}
The $(Z,\Phi)$-representation problem~\eqref{eq:wt_2} is reduced to a duality structure between the completeness of $Z$ and the uniqueness of an equivalent martingale measure for $Z$, which are furthermore proved (\citet{willtaq87}) to be equivalent to a technical condition on the flow of information and the dynamics of $Z$, that is: $\forall t=1,\ldots,T,\;A\in\P_{t-1},$
\begin{equation}
  \label{eq:wt_3}
  \dim\lf\mathrm{span}\lf\{Z(t,\w)-Z(t-1,\w),\,\w\in A\}\rg\rg=\sharp\{A'\subset\P_t:\,A'\subset A\}-1.
\end{equation}
The discrete-time construction extends then to stochastic integrals of continuous martingales, by using a ``skeleton approach''. The probability space $(\O,\F,\PP)$ is now assumed to be complete and endowed with a filtration $\FF^Z=\FF=\Ft$, where $Z=(Z(t))_{t\in[0T]}$ denotes an $\R^{d+1}$-valued continuous $\PP$-martingale with the components $Z^0\equiv1$ and $Z^1(0)=\ldots=Z^d(0)=0$ $\PP$-a.s. and $\FF$ satisfies the usual condition and is \textit{continuous} in the sense that, for all measurable set $B\in\F$, the $(\FF,\PP)$-martingale $\lf\PP(B|\F_t)\rg_{t\in[0,T]}$ has a continuous modification. 
The key notion of the skeleton approach is the following.
\begin{definition}\label{def:wt}
  A triplet $(I^\z,\FF^\z,\z)$ is a \emph{continuous-time skeleton} of $(\FF,Z)$ if:
  \begin{enumerate}
  \item[(i)] $I^\z$ is a finite partition $0=t^\z_0<\ldots<t^\z_{N}=:T^\z\leq T$;
  \item[(ii)] for all $t\in[0,T]$, $\F^\z_t=\sum\limits_{k=0}^{N-1}\F^\z_{t^\z_k}\ind_{[t^\z_k,t^\z_{k+1})}(t)$, such that for all $k=0,\ldots,N$ there exists a minimal partition of $\O$ which generates the sub-\ss $\F^\z_{t^\z_k}\subset\F_{t^\z_k}$;
  \item[(iii)]for all $t\in[0,T]$, $\z(t)=\sum\limits_{k=0}^{N-1}\z_{t^\z_k}\ind_{[t^\z_k,t^\z_{k+1})}(t)$ where $\z_{t^\z_k}$ is $\F^\z_{t^\z_k}$-measurable for all $k=0,\ldots,N$.
  \end{enumerate}
Given an $\R^{d+1}$-valued stochastic process $\nu=(\nu(t))_{t\in[0,T]}$ and $I^\nu,\FF^\nu$ satisfying (i),(ii), $(I^\nu,\FF^\nu,\nu)$ is called a $\FF^\nu$-predictable (continuous-time) skeleton if, for all $t\in[0,T]$, $\nu(t)=\sum\limits_{k=1}^{N}\nu_{t^\nu_k}\ind_{(t^\nu_{k-1},t^\nu_k]}(t)$ where $\nu_{t^\nu_k}$ is $\F^\nu_{t^\nu_{k-1}}$-measurable for all $k=1,\ldots,N$.
\\A sequence of continuous-time skeletons $(I^n,\FF^n,\z^n)_{n\geq0}$ is then called a  continuous-time skeleton approximation of $(\FF,Z)$ if the sequence of time partitions $(I^n)_{n\geq0}=\{0=t^n_0<\ldots<t^n_{N^n}=:T^n\leq T\}_{n\geq0}$ has mesh going to 0 as $n\rightarrow\infty$, the \textit{skeleton filtrations} $\FF^n$ converge to $\FF$ in the sense that, for each $t\in[0,T]$, 
$$\F_t^0\subset\cdots\subset\F_t^{n-1}\subset\F_t^n\subset\s\lf\underset{k\geq0}\cup\F^k_t\rg=\F_t$$ and the \textit{skeleton processes} $\z^n$ converge to $Z$ uniformly in time, as $n\rightarrow\infty$, $\PP$-a.s.
\end{definition}
Given $\bar Y\in L^1(\O,\F,\PP)$ and considered the $(\FF,\PP)$-martingale $Y=(Y(t))_{t\in[0,T]}$, $Y(t)=\EE^\PP[\bar Y|\F_t]\,\PP\text{-a.s.}$, the pathwise construction of stochastic integrals with respect to $Z$ runs as follows.
\begin{enumerate}
  \item Choose a complete continuous-time skeleton approximation $(I^n,\FF^n,\z^n)_{n\geq0}$ of $(\FF,Z)$ such that, defined $Y^n=\lf Y^n_t=\EE^\PP[\bar Y|\F^n_t]\,\PP\text{-a.s.}\rg_{t\in[0,T^n]}$ for all $n\geq0$, the sequence $(I^n,\FF^n,Y^n)_{n\geq0}$ defines a continuous-time skeleton approximation of $(\FF,Y)$.
  \item Thanks to the completeness characterization in discrete time, for each $n\geq0$, there exists an $\FF^n$-predictable skeleton $(I^n,\FF^n,\phi^n)$ such that
$$\phi^n(t^n_k)\cdot \z^n(t^n_k)=\phi^n(t^n_{k+1})\cdot \z^n(t^n_k)\quad \PP\text{-a.s., }k=0,1,\ldots,N^n,$$ and $$Y^n=(\phi^n\bullet\z^n)(T^n)=\phi^n(T^n)\cdot\z^n(T^n)\quad \PP\text{-a.s.} $$
  \item Define the pathwise integral
    \begin{equation}
      \label{eq:wt_4}
      \int_0^t\phi(s,\w)\cdot Z(s,\w):=\Limn(\phi^n\bullet\z^n)(t,\w),\quad t\in[0,T]
    \end{equation}
for $\PP$-almost all $\w\in\O$, namely on the set of scenarios $\w$ where the discrete stochastic integrals converge uniformly.
\end{enumerate}
\citet{willtaq} applied their methodology to obtain a convergence theory in the context of models for continuous security market with exogenously given equilibrium prices.
Thanks to the preservation of the martingale property and completeness and to the pathwise nature of their approximating scheme, they were able to characterize important features of continuous security models by convergence of ``real life'' economies, where trading occurs at discrete times. In particular, for a continuous security market model represented by a probability space $(\O,\F,\PP)$ and an $(\FF,\PP)$-martingale $Z$ on $[0,T]$, the notions of ``no-arbitrage'' and ``self-financing'' are understood through the existence of converging discrete market approximations $(T^n,\FF^n,\z^n)$ which are all free of arbitrage opportunities (as $\z^n$ is an $(\FF^n,\PP)$-martingale) and complete.
Moreover, the characterization~\eqref{eq:wt_3} of completeness in finite market models relates the structure of the skeleton filtrations $\FF^n$ to the number of non-redundant securities needed to complete the approximations $\z^n$.

However, this construction lacks an appropriate convergence result of the sequence $(\phi^n)_{n\geq0}$ to the predictable integrand $\phi$; moreover it deals exclusively with a given martingale in the role of the integrator process, which restricts the spectrum of suitable financial models.

\subsubsection{Continuous-time trading without probability}
In 1994, \citet{bickwill} looked at the current financial modeling issues from a new perspective: they provided an economic interpretation of \follmer's pathwise \ito\ calculus in the field of continuous-time trading models.
\follmer's framework turns out to be of interest in finance, as it allows to avoid any probabilistic assumption on the dynamics of traded assets and consequently any resulting model risk. Reasonably, only observed price trajectories are needed.
Bick and Willinger reduced the computation of the initial cost of a replicating trading strategy to an exercise of analysis.
For a given stock price trajectory (state of the world), they showed one is able to compute the outcome of a given trading strategy, that is the gain from trading. 

The set of possible stock price trajectories is taken to be the space of positive \cadlag\ functions, $D([0,T],\R_+)$, and trading strategies are defined only based on the past price information.

They define a \textit{simple trading strategy} to be a couple $(V_0,\phi)$ where $V_0:\R_+\rightarrow\R$ is a measurable function representing the initial investment depending only on the initial stock price and $\phi:(0,T]\times D([0,T],\R_+)\rightarrow\R$ is such that, for any trajectory $S\in  D([0,T],\R_+)$, $\phi(\cdot,S)$ is a \caglad\ stepwise function on a time grid $0\equiv\t_0(S)<\t_1(S)<\ldots<\t_m(S)\equiv T$, and satisfies the following \lq adaptation\rq\ property: for all $t\in(0,T],$ given $S_1,S_2\in  D([0,T],\R_+)$, if $S_{1\mid_{(0,t]}}=S_{2\mid_{(0,t]}}$, then $\phi(t+,S_1)=\phi(t+,S_2)$, where $\phi(t+,\cdot):=\lim\limits_{u\searrow t}\phi(u,\cdot)$. The value $\phi(t,S)$ represents the amount of shares of the stock held at time $t$. They restrict the attention to self-financing portfolios of the stock and bond (always referring to their discounted prices), so that the number of bonds in the portfolio is described by the map $\psi:(0,T]\rightarrow\R$, 
$$\psi(t)=V_0(S(0))-\phi(0+,S)S(0)-\sum_{j=1}^{m}S(\t_j\wedge t)(\phi(t_{j+1}\wedge t,S)-\phi(t_j\wedge t,S)).$$
The cumulative gain is denoted by 
$$G(t,S)=\sum_{j=1}^m\phi(t_j\wedge t,S)(S(t_j\wedge t)-S(t_{j-1}\wedge t)).$$
The self-financing assumption supplies us with the following well-known equation linking the gain to the value of the portfolio,
\beq\label{eq:bw_sf}
V(t,S):=\psi(t)+\phi(t,S)S(t)=V_0(S(0))+G(t,S),
\eeq
and makes $V$ be a \cadlag\ function in time.

Then, they define a \textit{general trading strategy} to be a triple $(V_0,\phi,\Pi)$ where $\phi(\cdot,S)$ is more generally a \caglad\ function, satisfying the same \lq adaptation\rq\ property and $\Pi=(\pi_n(S))_{n\ge1}$ is a sequence of partitions $\pi_n\equiv\pi_n(S)=\{0=\t^n_0<\ldots<\t^n_{m^n}=T\}$ whose mesh tends to $0$ and such that $\pi_n\cap[0,t]$ depends only on the price trajectory up to time $t$.
To any such triple is associated a sequence of simple trading strategies $\{(V_0,\phi^n)\}$, where $\phi^n(t,S)=\sum\limits_{j=1}^{m^n}\ind_{(\t^n_j,\t^n_{j+1}]}\phi(\t^n_j+,S)$, and for each $n\geq1$ the correspondent numbers of bonds, cumulative gains and portfolio values are denoted respectively by $\psi^n$, $G^n$ and $V^n$. 
They define a notion of \textit{convergence for $S$} of a general trading strategy $(V_0,\phi,\Pi)$ involving several conditions, that we simplify in the following:
\begin{enumerate}
\item $\ds \exists \Limn\psi^n(t,S)=:\psi(t,S)<\infty$ for all $t\in(0,T]$;
\item $\psi(\cdot,S)$ is a \caglad\ function;
\item $\psi(t+,S)-\psi(t,S)=-S(t)\lf\phi(t+,S)-\phi(t,S)\rg$ for all $t\in(0,T)$.
\end{enumerate}
The limiting gain and portfolio value of the approximating sequence, if exist, are denoted by $\ds G^n(t,S)=\Limn G(t,S)$ and $\ds V(t,S)=\Limn V^n(t,S)$.
Note that condition 1. can be equivalently reformulated in terms of $G$ or $V$ and, in case it holds, equation \eq{bw_sf} is still satisfied by the limiting quantities. Assuming 1., Condition 2. is equivalent to the equation
\beq\label{eq:bw_sf2}
V(t,S)-V(t-,S)=\phi(t,S)(S(t)-S(t-)) \quad \forall t\in(0,T],
\eeq
while condition 3. equates to the right-continuity of $V(\cdot,S)$.
In this setting, the objects of main interest can be expressed in terms of properly defined \lq one-sided\rq\  integrals, namely
\beq\label{eq:bw_psi}
\psi(t,S)=V_0(S(0))-\phi(0+,S)S(0)-\!\!\rint_0^t S(u) \ud \phi(u+,S)+S(t)(\phi(t+,S)-\phi(t,S)),
\eeq
where the \emph{right integral} of $S$ with respect to $(\phi(\cdot+,S),\Pi)$ is defined as
\beq\label{eq:bw_right}
\rint_0^t S(u) \ud \phi(u+,S):=\Limn \sum_{j=1}^{m^n}S(\t^n_j\wedge t) \lf\phi((t^n_j\wedge t)+,S)-\phi((t_{j-1}\wedge t)+,S) \rg,
\eeq
and $G(t,S)=\lint_0^t\phi(u+,S)\ud S(u)$, where the \emph{left integral} of $\phi(\cdot+,S)$ with respect to $(S,\Pi)$ is defined as
\beq\label{eq:bw_left}
\lint_0^t\phi(u+,S)\ud S(u):=\Limn \sum_{j=1}^{m^n}\phi(\t^n_{j-1}+,S) \lf S(t^n_j\wedge t)-S(t_{j-1}\wedge t) \rg.
\eeq
The existence and finiteness of either integral is equivalent to condition 1., hence equation~\eq{bw_sf} turns into the following integration-by-parts formula:
$$\lint_0^t\phi(u+,S)\ud S(u)=\phi(t+,S)S(t)-\phi(0+,S)S(0)-\rint_0^t S(u) \ud \phi(u+,S).$$
It is important to note that the one-sided integrals can exist even if the correspondent Riemann-Stieltjes integrals do not, in which case the right-integral may differ in value from the left-integral with respect to the same functions. When the Riemann-Stieltjes integrals exist, they necessarily coincide respectively with \eq{bw_right} and \eq{bw_left}. Moreover, these latter are associated to a specific sequence of partitions $\Pi$ along which convergence for $S$ holds true.
Once established the set-up, \citeauthor{bickwill} provide a few examples showing how to compute the portfolio value in different situations where convergence holds for $S$ in a certain sub-class of $  D([0,T],\R_+)$, along an arbitrary sequence of partitions.


Finally, they use the pathwise calculus introduced in \citep{follmer} to compute the portfolio value of general trading strategies depending only on time and on the current observed price in a smooth way. 

Their two main claims, slightly reformulated, are the following.
\begin{proposition}[Proposition 2 in \cite{bickwill}]\label{prop:bw1}
  Let $f:[0,T]\times\R_+\rightarrow\R$ be such that $f\in\C^2([0,T)\times\R_+)\cap\C(\{T\}\times\R_+)$ and $\Pi$ be a given sequence of partitions whose mesh tends to $0$. If the price path $S\in  D([0,T],\R_+)$ has finite quadratic variation along $\Pi$ and if $f,\partial_{x}f,\partial_{t}f,\partial_{tx}f,\partial_{xx}f,\partial_{tt}f$ have finite left limits in $T$, then the trading strategy $(0,\phi,\Pi)$, where $\phi(t,S)=f_x(t-,S(t-))$, converges for $S$ and its portfolio value at any time $t\in[0,T]$ is given by
  \begin{align}
    \lint_0^t\phi(u+,S)\ud S(u)={}& f(t,S(t))-f(0,S(0))-\int_0^t\partial_{t}f(u,S(u))\ud u  \label{eq:bw1}\\
&{} -\frac12\int_{[0,t]}\partial_{xx}f(u-,S(u-))\ud[S](u) \nonumber \\
\nonumber & {}-\sum_{u\leq t}\Big[f(u,S(u))-f(u-,S(u-)) \\
\nonumber & \left. {}-\partial_{x}f(u-,S(u-))\De S(u)-\frac12\partial_{xx}f(u-,S(u-))\De S^2(u)\right] .   
  \end{align}
\end{proposition}
This statement is a straightforward application of the \follmer's equation~\eq{follmer_Dito} by the choice $x(t)=(t,S(t))$, which makes the definition of the \follmer's integral~\eq{follmer_int} equivalent to the sum of a Riemann integral and a left-integral, i.e.
$$\int_0^t\nabla f(x(u-))\cdot \ud x(u)=\int_0^t\partial_tf(u,S(u))\ud u+\lint_0^t\partial_xf(u,S(u))\ud S(u).$$
Moreover, the convergence is ensured by remarking that the pathwise formula~\eq{bw1} implies that the portfolio value $V(t,S)=\lint_0^t\phi(u+,S)\ud S(u)$ is a \cadlag\ function and has jumps 
$$\De V(t)= \partial_xf(t-,S(t-))\De S(t)=\phi(t,S)\De S(t)\text{ for all }t\in(0,T],$$
hence conditions 2. and 3. are respectively satisfied.

The second statement is a direct implication of the previous one and provides a non-probabilistic version of the pricing problem for one-dimensional diffusion models.
\begin{proposition}[Proposition 3 in \cite{bickwill}]\label{prop:bw2}
  Let $f:[0,T]\times\R_+\rightarrow\R$ be such that $f\in\C^2([0,T)\times\R_+)\cap\C(\{T\}\times\R_+)$ and $f,\partial_{x}f,\partial_{t}f,\partial_{tx}f,\partial_{xx}f,\partial_{tt}f$ have finite left limits in $T$, and let $\Pi$ be a given sequence of partitions whose mesh tends to $0$. Assume that $f$ satisfies the partial differential equation
\beq\label{eq:bw_pde}
\partial_tf(t,x)+\frac12\b^2(t,x)\partial_{xx}f(t,x)=0,\quad t\in[0,T],x\in\R_+,
\eeq 
where $\b:[0,T]\times\R_+\rightarrow\R$ is a continuous function.
If the price path $S\in  D([0,T],\R_+)$ has finite quadratic variation along $\Pi$ of the form $[S](t)=\int_0^t\b^2(u,S(u))\ud u$ for all $t\in[0,T]$, then the trading strategy $(f(0,S(0)),\phi,\Pi)$, where $\phi(t,S)=\partial_xf(t-,S(t-))$, converges for $S$ and its portfolio value at time $t\in[0,T]$ is $f(t,S(t))$.
\end{proposition}

Following Bick and Willinger's approach, all that has to be specified is the set of all possible scenarios and the trading instructions for each possible scenario. 
The investor's probabilistic beliefs can then be considered as a way to express the set of all possible scenarios together with their odds, however there may be no need to consider them. Indeed, by taking any financial market model in which the price process satisfies almost surely the assumptions  of either above proposition, the portfolio value of the correspondent trading strategy, computed pathwise, will provide almost surely the model-based value of such portfolio.
In this way, on one hand the negligible set outside of which the pathwise results do not hold depends on the specific sequence of time partitions, but on the other hand we get a path-by-path interpretation of the hedging issue, which was missing in the stochastic approach.

\subsubsection{Karandikar's pathwise construction of stochastic integrals}
In 1994, \citet{karandikar} proposed another pathwise approach to stochastic integration for continuous time stochastic processes. She proved a pathwise integration formula, first for Brownian integrals, then for the general case of semimartingales and a large class of integrands.
It is fixed a complete probability space $(\O,\F,\PP)$, equipped with a filtration $(\F_t)_{t\geq0}$ satisfying the usual conditions.
\begin{proposition}[Pathwise Brownian integral, \cite{karandikar}]
  Let $W$ be a $(\F_t)$-Brownian motion and $Z$ be a \cadlag\ $(\F_t)$-adapted process. For all $n\geq1$, let $\{\t^n_1\}_{i\geq0}$ be the random time partition defined by
$$\t^n_0:=0,\quad\t^n_{i+1}:=\inf\{t\geq\t^n_i:\:|Z(t)-Z(\t^n_i)|\geq2^{-n}\},\quad i\geq0,$$
and $(Y^n(t))_{t\geq0}$ be a stochastic process defined by, for all $t\in[0,\infty)$, 
$$Y^n(t):=\sum_{i=0}^{\infty}Z(\t^n_i\wedge t)(W(\t^n_{i+1}\wedge t)-W(\t^n_i\wedge t)).$$
Then, for all $T\in[0,\infty)$, almost surely, $\ds \sup_{t\in[0,T]}\left|Y^n(t)-\int_0^tZ\ud W\right|\limn0$.
\end{proposition}
The proof hinges on the Doob's inequality for $p=2$, which says that a \cadlag\ martingale $M$ such that, for all $t\in[0,T]$, $\EE[|M(t)|^2]<\infty$, satisfies $$\norm{\sup_{t\in[0,T]}|M(t)|}_{L^2(\PP)} \leq 4\norm{M(T)}_{L^2(\PP)}.$$
Indeed, by taking $M(t)=\int_0^t(Z^n-Z)\ud W$, where $\ds Z^n:=\sum_{i=1}^{\infty}Z(\t^n_{i-1})\ind_{[\t^n_{i-1},\t^n_i)}$, the Doob's inequality holds and gives
$$\EE\left[\sup_{t\in[0,T]}\left|Y^n(t)-\int_0^tZ\ud W\right|^2\right]\leq 4T2^{-2n},$$ by the definitions of $\{\t^n_i\}$ and $Y^n$.
\\Finally, by denoting $\ds U_n:=\sup_{t\in[0,T]}\left|Y^n(t)-\int_0^tZ\ud W\right|$, the H\"older's inequality implies that 
$$\EE\left[\sum_{n\geq1}U_n\right]\leq2\sqrt{T}\sum_{n\geq1}2^{-n}<\infty,$$ 
hence, almost surely, $\sum\limits_{n\geq1}U_n<\infty$, whence the claim.

The generalization to semimartingale integrators is the following.
\begin{proposition}[Pathwise stochastic integral, \cite{karandikar}]
\label{Prop:kar}
  Let $X$ be a semimartingale and $Z$ be a \cadlag\ $(\F_t)$-adapted process. For all $n\geq1$, let $\{\t^n_1\}_{i\geq0}$ be the time partition defined as in the previous theorem and $Y^n$ be the process defined by, for all $t\in[0,\infty)$, 
$$Y^n(t):=Z(0)X(0)+\sum_{i=1}^{\infty}Z(\t^n_{i-1}\wedge t)(X(\t^n_{i}\wedge t)-X(\t^n_{i-1}\wedge t)).$$
Then, for all $T\in[0,\infty)$, almost surely, 
$$\ds \sup_{t\in[0,T]}\left|Y^n(t)-\int_0^tZ(u-)\ud X(u)\right|\limn0.$$
\end{proposition}
The proof is carried out analogously to the Brownian case, using some basic properties of semimartingales and predictable processes. Precisely, $X$ can be decomposed as $X=M+A$, where $M$ is a locally square-integrable martingale and $A$ has finite variation on bounded intervals, and let $\{\s_k\}_{k>0}$ be a sequence of stopping times increasing to $\infty$ such that $C_k=\EE\left[\pqv{M}(\s_k)\right]<\infty$. By rewriting $Y^n(t)=\int_0^tZ^n\ud X$, where
$$Z^n:=Z(0)\ind_{0}+\sum_{i=1}^\infty Z(\t^n_i)\ind_{(\t^n_i,\t^n_{i+1}]},$$
the Doob's inequality 
gives
$$\EE\left[\sup_{t\in[0,\s_k]}\left|\int_0^t(Z^n(u)-Z(u-))\ud M\right|^2\right]\leq 4C_k2^{-2n},$$ by the definitions of $\{\t^n_i\}$. Then, proceeding as before and using $\s_k\nearrow\infty$, for all $T\in[0,\infty)$, almost surely 
$$\ds \sup_{t\in[0,T]}\left|\int_0^t(Z^n(u)-Z(u-))\ud M(u)\right|\limn0.$$
As regards the Stieltjes integrals with respect to $A$, the uniform convergence of $Z^n$ to the left-continuous version of $Z$ implies directly that, almost surely,
$$\ds \sup_{t\in[0,T]}\left|\int_0^t(Z^n(u)-Z(u-))\ud A(u)\right|\limn0.$$


The main tool in Karandikar's pathwise characterization of stochastic integrals is the martingale Doob's inequality.
A recent work by \citet{traj-doob} establishes a deterministic version of the Doob's martingale inequality, which provides an alternative proof of the latter, both in discrete and continuous time. 
Using the trajectorial counterparts, they also improve the classical Doob's estimates for non-negative \cadlag\ submartingales by using the initial value of the process, obtaining sharp inequalities.

These continuous-time inequalities are proven by means of ad hoc constructed pathwise integrals.
First, let us recall the following notion of pathwise integral (see \cite[Chapter 2.5]{norvaisa}):
\begin{definition}
  Given two \cadlag\ functions $f,g:[0,T]\rightarrow[0,\infty)$, the \textrm{Left Cauchy-Stieltjes integral} of $g$ with respect to $f$ is defined as the limit, denoted $(LCS)\!\!\int_0^Tg\ud f$, of the directed function $(S_{LC}(f;\cdot),\mathfrak R)$, where the \emph{Left Cauchy sum} is defined by
  \begin{equation}
    \label{eq:left-C-S}
    S_{LC}(g,f;\k):=\sum_{t_i\in\k}g(t_i)(f(t_{i+1})-f(t_i)),\quad\k\in P[0,T].
  \end{equation}
\end{definition}
\citet{traj-doob} are interested in the particular case where the integrand is of the form $g=h(\bar f)$ and $h:[0,\infty)\rightarrow[0,\infty)$ is a continuous monotone function. 
In this case, the limit of the sums in \eq{left-C-S} in the sense of refinements of partitions exists if and only if its predictable version $\ds (LCS)\!\!\int_0^Tg(t-)\ud f(t):=\Limn\sum_{t^n_i\in\pi^n}g(t^n_i-)(f(t^n_{i+1})-f(t^n_i))$ exists for every dense sequence of partitions $(\pi^n)_{n\geq0}$, in which case the two limits coincide.
By monotonicity of $g$ and rearranging the finite sums, it follows that $\int_0^Tg(t)\ud f(t)$ is well defined if and only if $\int_0^Tf(t)\ud g(t)$ is; if so, they lead to the following integration-by-parts formula:
\begin{align}
  \nonumber
  (LCS)\!\!\int_0^Tg(t)\ud f(t)={}&g(T)f(T)-g(0)f(0)-(LCS)\!\!\int_0^Tf(t)\ud g(t)\\
&{}-\sum_{0\leq t\leq T}\De g(t)\De f(t).\label{eq:ibp}
\end{align}
By the assumptions on $h$, the two integrals exist and the equation \eq{ibp} holds.
Moreover, given a martingale $S$ on $(\O,\F,(\F_t)_{t\geq},\PP)$ and taking $f$ to be the path of $S$, the Left Cauchy-Stieltjes integral coincides almost surely with the \ito\ integral, i.e. 
$$(h(\bar S)\bullet S)(T,\w)=(LCS)\!\!\int_0^Th(\bar S(t-,\w))\ud S(t,\w),\text{ for $\PP$-almost all }\w\in\O.$$
Indeed, \citet{karandikar} showed the almost sure uniform convergence of the sums in \eq{left-C-S} to the \ito\ integral along a specific sequence of random partitions; therefore, by the existence of the pathwise integral and uniqueness of the limit, the two coincide.

The trajectorial Doob inequality obtained in continuous time and using the pathwise integral defined above is the following.
\begin{proposition}
  Let $f:[0,T]\rightarrow[0,\infty)$ be a \cadlag\ function, $1<p<\infty$ and $h(x):=-\frac{p^2}{p-1}x^{p-1}$, then
$$\bar{f}^p(T)\leq(LCS)\!\!\int_0^Th(\bar{f}(t))\ud f(t)-\frac p{p-1}f(0)^p+\lf\frac p{p-1}\rg^pf(T)^p.$$
\end{proposition}

\subsubsection{Pathwise integration under a family of measures}
In 2012, motivated by problems involving stochastic integrals under families of measures, \citet{nutz-int} proposed a different pathwise ``construction'' of the \ito\ integral of an arbitrary predictable process under a general set of probability measures $\P$ which is not dominated by a finite measure and under which the integrator process is a semimartingale.
However, his result concerns only existence and does not provide a constructive procedure to compute such integral.

Let us briefly recall his technique. It is fixed a measurable space $(\O,\F)$ endowed with a right-continuous filtration $\FF^*=(\F^*_t)_{t\in[0,1]}$ which is $\P$-universally augmented. $X$ denotes a \cadlag\ $(\FF^*,\PP)$-semimartingale for all $\PP\in\P$ and $H$ is an $\FF^*$-predictable process.
The approach is to average $H$ in time in order to obtain approximating processes of finite variation which allow to define (pathwise) a sequence of Lebesgue-Stieltjes integrals converging in medial limit to the \ito\ integrals. To this aim, a domination assumption is needed, but it is imposed at the level of characteristics, thus preserving the non-dominated nature of $\P$ encountered in all applications.
So, it is assumed that there exists a predictable \cadlag\ increasing process $A$ such that
$$B^\PP+\pqv{X^c}^\PP+(x^2\wedge1)\ast\nu^\PP\ll A \quad \PP\text{-a.s., for all }\PP\in\P,$$
where $(B^\PP,\pqv{X^c}^\PP,\nu^\PP)$ is the canonical triplet (i.e. the triplet associated with the truncation function $h(x)=x\ind_{\{|x|<1\}}$) of predictable characteristics of $X$.
The main result is the following.
\begin{theorem}
  Under the assumption above, there exists an $\FF^*$-adapted \cadlag\ process $\lf\int_0^t H\ud X\rg_{t\in[0,1]}$ such that $\int_0^\cdot H\ud X=(H\bullet X)^\PP$ $\PP$-almost surely, for all $\PP\in\P$, where the construction of $\lf\int H\ud X\rg(\w)$ involves only $H(\w)$ and $X(\w)$.
\end{theorem}
The proof stands on two lemmas.
Without loss of generality and to simplify notation, it is set $X(0)=0$ and defined $H(t)=A(t)=0$ for all $t<0$; it is also assumed that $X$ has bounded jumps, $|\De X|\leq1$, $H$ is uniformly bounded, $|H|\leq c$, and $A(t)-A(s)\geq t-s$ for all $0\leq s\leq t\leq1$.
\begin{lemma}\label{lem:nutz1}
  For all $n\geq1$, the processes $H^n,Y^n$, defined by
$$\bea{l}
H^n(0):=0,\quad H^n(t):=\frac1{A_t-A_{t-\frac1n}}\int_{t-\frac1n}^tH(s)\ud A(s),\quad 0<t\leq1,\\
Y^n:=H^nX-\int_0^\cdot X(s-)\ud H^n(s),
\ea$$
are well defined (pathwise) in the Lebesgue-Stieltjes sense and
$$Y^n=(H^n\bullet X)^\PP\ \PP\text{-a.s.},\quad Y^n\limucp(H\bullet X)^\PP\text{ for all }\PP\in\P.$$
\end{lemma}

\begin{lemma}\label{lem:nutz2}
  Let $(Y^n)_{n\geq1}$ be a sequence of $\FF^*$-adapted \cadlag\ processes and assume that for each $\PP\in\P$ there exists a \cadlag\ process $Y^\PP$ such that $Y^n(t)\limnp Y^\PP(t)$ for all $t\in[0,1]$. Then, there exists an $\FF^*$-adapted \cadlag\ process $Y$ such that $Y=Y^\PP$ $\PP$-almost surely for all $\PP\in\P$.
\end{lemma}

The first claim in Lemma \ref{lem:nutz1} is a consequence of the assumptions on $H,A$, while the convergence in $ucp(\PP)$ is implied by the $L^2(\PP)$ convergence
$$\EE^\PP\left[\sup_{t\in[0,1]}\left|\int_0^tH^n(s)\ud X(s)-\int_0^tH(s)\ud X(s)\right|^2\right]\limn0,$$ 
which in turn is proven thanks to the convergence of $H^n(\w)$ to $H(\w)$ in $L^1([0,1],\ud A(\w))$ for all $\w\in\O$.

Instead, Lemma \ref{lem:nutz2} relies on the notion of \emph{Mokobodzki's medial limit}, a kind of \lq projective limit\rq\ of convergence in measure.
More precisely, the medial limit $\limmed$ is a map on the set of real sequences, such that, if $(Z_n)_{n\geq1}$ is a sequence of random variables on a measurable space, the medial limit defines a universally 
measurable random variable Z, $Z(\w):=\limmed Z_n(\w)$, such that, if for some probability measure $\PP$, $Z_n\limnp Z^\PP$, then $Z^\PP=Z$ $\PP$-almost surely.

However, as anticipated above, Nutz's method does not give a pathwise computation of stochastic integrals, though it supplies us with a process which coincides $\PP$-almost surely with the $\PP$-stochastic integral for each $\PP$ in the set of measures $\P$ and is a limit in $ucp(\PP)$ of approximating Stieltjes integrals.

\subsection{Model-free arbitrage strategies}
\label{sec:arbitrage}

Once we have at our disposal a pathwise notion of gain process, a natural next step is to examine the corresponding notion of arbitrage strategy.

The literature investigating arbitrage notions in financial markets admitting uncertainty is recent and there are different approaches to the subject. The mainstream approach is that of model-uncertainty, where arbitrage notions are reformulated for families of probability measures in a way analogous to the classical case of a stochastic model. However, most of the contributions in this direction deal with discrete-time frameworks. In continuous time, recent results are found in \cite{sara-bkn}.

An important series of papers exploring arbitrage-like notions by a model-free approach is due to Vladimir Vovk (see e.g. \citet{vovk-vol,vovk-proba,vovk-rough,vovk-cadlag}). 
 He introduced an outer measure (see \cite[Definition 1.7.1]{tao} for the definition of \emph{outer measure}) on the space of possible price paths, called \emph{upper price} (\defin{upperP}), as the minimum super-replication price of a very special class of European contingent claims. The important intuition behind this notion of upper price is that the sets of price paths with zero upper price, called \emph{null sets}, allow for the infinite gain of a positive portfolio capital with unitary initial endowment. The need to guarantee this type of market efficiency in a financial market leads to discard the null sets.
\citeauthor{vovk-proba} says that a property holds for \emph{typical paths} if the set of paths where it does not hold is null, i.e. has zero upper price. 
Let us give some details.
\begin{definition}[Vovk's upper price]\label{def:upperP}
  The \emph{upper price} of a set $E\subset\O$ is defined as 
\beq\label{eq:upperP}
\bar\PP(E):=\inf_{S\in\S}\{S(0)|\,\forall\w\in\O,\; S(T,\w)\geq\ind_E(\w)\},
\eeq
where $\S$ is the set of all \emph{positive capital processes} $S$, that is: $S=\sum_{n=1}^\infty\K^{c_n,G_n}$, where $\K^{c_n,G_n}$ are the portfolio values of bounded simple predictable strategies trading at a non-decreasing infinite sequence of stopping times $\{\t^n_i\}_{i\geq1}$, such that for all $\w\in\O$ $\t^n_i(\w)=\infty$ for all but finitely many $i\in\NN$, with initial capitals $c_n$ and with the constraints $\K^{c_n,G_n}\geq0$ on $[0,T]\times\O$ for all $n\in\NN$ and $\sum_{n=1}^\infty c_n<\infty$.
\end{definition}
It is immediate to see that $\bar\PP(E)=0$ if and only if there exists a positive capital process $S$ with initial capital $S(0)=1$ and infinite capital at time $T$ on all paths in $E$, i.e. $S(T,\w)=\infty$ for all $\w\in E$.

Depending on what space $\O$ is considered, Vovk obtained specific results. In particular, he investigated properties of typical paths that concern their measure of variability. The most general framework considered is $\O=D([0,T],\R_+)$. He proved in \cite{vovk-rough} that typical paths $\w$ have a \emph{$p$-variation index} less or equal to 2, which means that the $p$-variation is finite for all $p>2$, but we have no information for $p=2$ (a stronger result is stated in \cite[Proposition 1]{vovk-rough}). If we relax the positivity and we restrict to \cadlag\ path with all components having \lq moderate jumps\rq\ in the sense of \eq{mod-jumps}, then \citet{vovk-cadlag} obtained appealing results regarding the quadratic variation of typical paths along special sequences of random partitions.
Indeed, by adding a control on the size of the jumps, in the sense of considering the sample space $\O=\O_\psi$, defined as
\beq\label{eq:mod-jumps}
\O_\phi:=\left\{\w\in D([0,T],\R)\bigg|\,\forall t\in(0,T],\;\abs{\De\w(t)}\leq\psi\lf \sup_{s\in[0,t)}\abs{\w(s)}\rg\right\}
\eeq
for a given non-decreasing function $\psi:[0,\infty)$, \citet{vovk-cadlag} obtained finer results. In particular, he proved the existence for typical paths of the quadratic variation in \defin{qv-vovk} along a special sequence of nested vertical partitions. It is however important to remark (\cite[Proposition 1]{vovk-cadlag}) that the same result applies to all sequences of nested partitions of dyadic type, and that any two sequences of dyadic type give the same value of quadratic variation for typical paths. A sequence of nested partitions is called of \emph{dyadic type} if it is composed of stopping times such that there exist a polynomial $p$ and a constant $C$ and
\begin{enumerate}
\item for all $\w\in\O_\psi$, $n\in\NN_0$, $0\leq s<t\leq T$, if $\abs{\w(t)-\w(s)}>C2^{-n}$, then there is an element of the $n^{th}$ partition which belongs to $(s,t]$,
\item for typical $\w$, from some $n$ on, the number of finite elements of the $n^{th}$ partition is at most $p(n)2^{2n}$.
\end{enumerate}

The sharper results are obtained when the sample space is $\O=C([0,T],\R)$ (or equivalently $\O=C([0,T],[0,\infty))$). In this case, in \cite{vovk-proba} it is proved that typical paths are constant or have a $p$-variation which is finite for all $p>2$ and infinite for $p\leq2$ (stronger results are stated in \cite[Corollaries 4.6,4.7]{vovk-proba}.
Note that the situation changes remarkably from the space of \cadlag\ paths to the space of continuous paths. Indeed, no (positive) \cadlag\ path which is bounded away from zero and has finite total variation can belong to a null set in $D([0,T],\R^d_+)$, while all continuous paths with finite total variation belong to a null set in $C([0,T],\R^d_+)$.

A similar notion of outer measure is introduced by \citet{perk-promel} (see also \citet{perkowski-thesis}), which is more intuitive in terms of hedging strategies. He considers portfolio values that are limits of simple predictable portfolios with the same positive initial capital and whose correspondent simple trading strategies never risk more than the initial capital.
\begin{definition}[Definition 3.2.1 in \cite{perkowski-thesis}]\label{def:outerP}
  The \emph{outer content} of a set $E\subset\O:=C([0,T],\R^d)$ is defined as 
\beq\label{eq:outerP}
\tilde\PP(E):=\inf_{(H^n)_{n\geq1}\in\H_{\l,s}}\{\l|\,\forall\w\in\O,\; \liminf_{n\to\infty}(\l+(H^n\bullet\w)(T))\geq\ind_E(\w)\},
\eeq
where $\H_{\l,s}$ is the set of all \emph{$\l$-admissible simple strategies}, that is of bounded simple predictable strategies $H^n$ trading at a non-decreasing infinite sequence of stopping times $\{\t^n_i\}_{i\geq1}$, $\t^n_i(\w)=\infty$ for all but finitely many $i\in\NN$ for all $\w\in\O$, such that $(H^n\bullet\w)(t)\geq-\l$ for all $(t,\w)\in[0,T]\times\O$.
\end{definition}
Analogously to Vovk's upper price, the $\tilde\PP$-null sets are identified with the sets where the inferior limit of some sequence of 1-admissible simple strategies brings infinite capital at time $T$. This characterization is shown to be a model-free interpretation of the condition of \emph{no arbitrage of the first kind} (NA1) from mathematical finance, also referred to as \emph{no unbounded profit with bounded risk} (see e.g. \cite{kk2007,kardaras}). Indeed, in a financial model where the price process is a semimartingale on some probability space $(\O,\F,\PP)$, the (NA1) property holds if the set $\{1+(H\bullet S)(T),\,H\in\H_1\}$ is bounded in $\PP$-probability, i.e. if 
$$\lim_{c\to\infty}\sup_{H\in\H_{1,s}}\PP(1+(H\bullet S)(T)\geq c)=0.$$
On the other hand, \cite[Proposition 3.28]{perkowski-thesis} proved that an event $A\in\F$ which is $\tilde\PP$-null has zero probability for any probability measure on $(\O,\F)$ such that the coordinate process satisfies (NA1).

However, the characterization of null sets in \cite{perk-promel,perkowski-thesis} is possibly weaker than Vovk's one. In fact, the outer measure $\tilde\PP$ is dominated by the outer measure $\bar\PP$.


A distinct approach to a model-free characterization of arbitrage is proposed by \citet{riedel}, although he only allows for static hedging. He considers a Polish space $(\O,\ud)$ with the Borel sigma-field and he assumes that there are $D$ uncertain assets in the market with known non-negative prices $f_d\geq0$ at time 0 and uncertain values $S_d$ at time $T$, which are continuous on  $(\O,\ud)$, $d=1,\ldots,D$. A portfolio is a vector $\pi$ in $\R^{D+1}$ and it is called an \emph{arbitrage} if $\pi\cdot f\leq0$, $\pi\cdot S\geq0$ and $\pi\cdot S(\w)>0$ for some $\w\in\O$, where $f_0=S_0=1$. Thus the classical ``almost surely'' is replaced by ``for all scenarios'' and ``with positive probability'' is replaced by ``for some scenarios''.
The main theorem in \cite{riedel} is a model-free version of the FTAP and states that the market is \emph{arbitrage-free} if and only if there exists a \emph{full support martingale measure}, that is a probability measure whose topological support in the polish space of reference is the full space and under which the expectation of the final prices $S$ is equal to the initial prices $f$.
This is proven thanks to the continuity assumption of $S(\w)$ in $\w$ on one side and a separation argument on the other side. Even without a prior probability assumption, it shows that, if there are no (static) arbitrages in the market, it is possible to introduce a pricing probability measure, which assigns positive probability to all open sets.

\section{The setting}
\label{sec:setting}

We consider a continuous-time frictionless market open for trade during the time interval $[0,T]$, where $d$ risky (non-dividend-paying) assets are traded besides a riskless security, named \lq bond'. The latter is assumed to be the numeraire security and we refer directly to the forward asset prices and portfolio values, which makes this framework of simplified notation without loss of generality. 
Our setting does not make use of any (subjective) probabilistic assumption on the market dynamics and we construct trading strategies based on the realized paths of the asset prices. 

Precisely, we consider the metric space $(\O,||\cdot||_\infty)$, $\O:=D([0,T],\R^d_+)$, provided with the Borel sigma-field $\F$ and the canonical filtration $\FF=\Ft$, that is the natural filtration generated by the coordinate process $S$, $S(t,\w):=\w(t)$ for all $\w\in\O$, $t\in[0,T]$. Thinking of our financial market, $\O$ represents the space of all possible trajectories of the asset prices up to time $T$. When considering only continuous price trajectories, we will restrict to the subspace $\O^0:=C([0,T],\R^d_+)$.


In such analytical framework, we think of a continuous-time path\hyp dependent trading strategy as determined by the value of the initial investment and the quantities of asset and bond holdings, given by functions of time and of the price trajectory.
\begin{definition}
  A \emph{trading strategy} in $(\O,\F)$ is any triple $(V_0,\phi,\psi)$, where $V_0:\O\to\R$ is $\F_0$-measurable and $\phi=(\phi(t,\cdot))_{t\in(0,T]},\psi=(\psi(t,\cdot))_{t\in(0,T]}$ are $\FF$-adapted \caglad\ processes on $(\O,\F)$, respectively with values in $\R^d$ and in $\R$. The portfolio value $V$ of such trading strategy at any time $t\in[0,T]$ and for any price path $\w\in\O$ is given by $$V(t,\w;\phi,\psi)=\phi(t,\w)\cdot\w(t)+\psi(t,\w).$$
\end{definition}
Economically speaking, $\phi(t,\w),\psi(t,\w)$ represent the vectors of the number of assets and bonds, respectively, held in the trading portfolio at time $t$ in the scenario $w\in\O$. The left-continuity of the trading processes comes from the fact that any revision to the portfolio will be executed the instant just after the time the decision is made. On the other hand, their right-continuous modifications $\phi(t+,\w),\psi(t+,\w)$, defined by 
$$\phi(t+,\w):=\lim\limits_{s\searrow t}\phi(s,\w),\ \psi(t+,\w):=\lim\limits_{s\searrow t}\psi(s,\w),\quad\forall\w\in\O,\,t\in[0,T)$$
represent respectively the number of assets and bonds in the portfolio just after any revision of the trading portfolio decided at time $t$.
The choice of strategies adapted to the canonical filtration conveys the realistic assumption that any trading decision makes use only of the price information available at the time it takes place.

We aim to identify \emph{self-financing trading strategies} in this pathwise framework, that is portfolios where changes in the asset position are necessarily financed by buying or selling bonds without adding or withdrawing any cash. In particular, we look for those of them which trade continuously in time but still allow for an explicit computation of the gain from trading.
In the classical literature about continuous-time financial market models, unlike for discrete-time models, we don't have a general pathwise characterization of self-financing dynamic trading strategies, mainly because of the probabilistic characterization of the gain in terms of a stochastic integral with respect to the asset price process. In the same way, the number of bonds which continuously rebalances the portfolio has no pathwise representation.

Here, we start from considering strategies where the portfolio is rebalanced only a finite number of times, for which the self-financing condition is well established and whose gain is given by a pathwise integral, equal to a Riemann sum.

Henceforth, we will take as given a dense nested sequence of time partitions, $\Pi=(\pi^n)_{n\geq1}$, i.e. $\pi^n=\{0=t^n_0<t^n_1<\ldots,t^n_{m(n)}=T\}$, $\pi^n\subset\pi^{n+1}$, $\abs{\pi^n}\limn\infty$.

We denote by $\Si(\Pi)$ the set of simple predictable processes whose jump times are covered by one of the partitions in $\Pi$\footnote{We could assume in more generality that the jump times are only covered by $\cup_{n\geq1}\pi^n$, but at the expense of more complicated formulas}:
\begin{align*}
\Si(\pi^n):={}&\bigg\{\phi:\;\forall i=0,\ldots,m(n)-1,\;\exists \l_i\,\F_{t^n_i}\mbox{-measurable }\R^d\mbox{-valued}\\
&\quad\mbox{random variable on }(\O,\F),\;\phi(t,\w)=\sum_{i=0}^{m(n)-1}\l_i(\w)\ind_{(t^n_i,t^n_{i+1}]}\bigg\},\\
\Si(\Pi):={}&\underset{n\geq1}\cup\Si(\pi^n).
\end{align*}

\section{Self-financing strategies}
\label{sec:self-fin}

\begin{definition}
  $(V_0,\phi,\psi)$ is called a \emph{simple self-financing trading strategy} if it is a trading strategy such that $\phi\in\Si(\pi^n)$ for some $n\in\NN$ and
\begin{align}
\nonumber\psi(t,\w;\phi)={}&V_0-\phi(0+,\w)\cdot\w(0)-\sum_{i=1}^{m(n)-1}\w(t^n_i\wedge t)\cdot(\phi(t^n_{i+1}\wedge t,\w)-\phi(t^n_i\wedge t,\w)) \\
  ={}& V_0-\phi(0+,\w)\cdot\w(0)-\sum_{i=1}^{k(t,n)}\w(t^n_i)\cdot(\l_{i}(\w)-\l_{i-1}(\w)),\label{eq:psi-sf}
\end{align}
where $\phi(t,\w)=\sum_{i=0}^{m(n)-1}\l_i(\w)\ind_{(t^n_i,t^n_{i+1}]}$ and $k(t,n):=\max\{i\in\{1,\ldots,m\}\;:\;t^n_i<t\}$.
The \emph{gain} of such a strategy is defined at any time $t\in[0,T]$ by
\begin{align*}
G(t,\w;\phi):={}&\sum_{i=1}^{m(n)}\phi(t^n_{i}\wedge t,\w)\cdot(\w(t^n_{i}\wedge t)-\w(t^n_{i-1}\wedge t)) \\
  ={}& \sum_{i=1}^{k(t,n)}\l_{i-1}(\w)\cdot(\w(t^n_{i})-\w(t^n_{i-1}))+\l_{k(t,n)}(\w)\cdot(\w(t)-\w(t^n_{k(t,n)})).
\end{align*}
\end{definition}
In the following, when there is no ambiguity, we drop the dependence of $k$ on $t,n$ and write $k\equiv k(t,n)$.

Note that the definition \eq{psi-sf} is equivalent to requiring that the trading strategy $(V_0,\phi,\psi)$ satisfies
$$V(t,\w;\phi,\psi)\equiv V(t,\w;\phi)=V_0+G(t,\w;\phi).$$

Since a simple self-financing trading strategy is uniquely determined by its initial investment and the asset position at all times, we will drop the dependence on $\psi$ of the quantities involved. For instance, when we are referring to a simple self-financing strategy $(V_0,\phi)$, we implicitly refer to the triplet $(V_0,\phi,\psi)$ with $\psi\equiv\psi(\cdot,\cdot;\phi)$ defined in \eq{psi-sf}.

\begin{remark}
The portfolio value $V(\cdot,\cdot;\phi)$ of a simple self-financing strategy $(V_0,\phi,\psi)$ is a \cadlag\ $\FF$-adapted process on $(\O,\F)$, satisfying
$$\Delta V(t,\w;\phi)=\phi(t,\w)\cdot\Delta\w(t),\quad \forall t\in[0,T],\w\in\O.$$
\end{remark}
The right-continuity of $V$ comes from the definition \eq{psi-sf}, which implies, for all $t\in[0,T]$ and $\w\in\O$,
$$\psi(t,\w)+\phi(t,\w)\cdot\w(t)=\psi(t+,\w)+\phi(t+,\w)\cdot\w(t).$$

Below, we are going to establish the self-financing conditions for (non\hyp simple) trading strategies.

\begin{definition}\label{def:path-sf} 
  Given an $\F_0$-measurable random variable $V_0:\O\to\R$ and an $\FF$-adapted $\R^d$-valued \caglad\ process $\phi=(\phi(t,\cdot))_{t\in(0,T]}$ on $(\O,\F)$, we say that $(V_0,\phi)$ is a \emph{self-financing trading strategy on} $U\subset\O$ if there exists a sequence of self-financing simple trading strategies $\{(V_0,\phi^n,\psi^n), n\in\NN\}$, such that 
$$\forall\w\in U,\,\forall t\in[0,T],\quad\phi^n(t,\w)\limn\phi(t,\w),$$
and any of the following conditions is satisfied:
 \begin{enumerate}[(i)]
 \item there exists an $\FF$-adapted real-valued \cadlag\ process $G(\cdot,\cdot;\phi)$ on $(\O,\F)$ such that, for all $t\in[0,T],\w\in U$,
$$G(t,\w;\phi^n)\limn G(t,\w;\phi)\quad\text{and}\quad\Delta G(t,\w;\phi)=\phi(t,\w)\cdot\Delta\w(t);$$
 \item there exists an $\FF$-adapted real-valued \cadlag\ process $\psi(\cdot,\cdot;\phi)$ on $(\O,\F)$ such that, for all $t\in[0,T],\w\in U$,
$$\psi^n(t,\w)\limn\psi(t,\w;\phi)$$
and $$\psi(t+,\w;\phi)-\psi(t,\w;\phi)=-\w(t)\cdot\lf\phi(t+,\w)-\phi(t,\w)\rg;$$
 \item there exists an $\FF$-adapted real-valued \cadlag\ process $V(\cdot,\cdot;\phi)$ on $(\O,\F)$ such that, for all $t\in[0,T],\w\in U$,
$$V(t,\w;\phi^n)\limn V(t,\w;\phi)\quad\text{and}\quad\Delta V(t,\w;\phi)=\phi(t,\w)\cdot\Delta\w(t).$$
 \end{enumerate}
\end{definition}

\begin{remark}\label{rmk:path-sf}
It is easy to see that the three conditions (i)-(iii) of Definition \ref{def:path-sf} are equivalent. If any of them is fulfilled, the limiting processes $G,\psi,V$ are respectively the gain, bond holdings and portfolio value of the self-financing strategy $(V_0,\phi)$ on $U$ and they satisfy, for all $t\in[0,T],\w\in U$,
\begin{equation}
   \label{eq:sf}
V(t,\w;\phi)=V_0+G(t,\w;\phi)
 \end{equation}
and
\begin{equation}
  \label{eq:psi}
  \psi(t,\w;\phi)=V_0-\phi(0+,\w)-\Limn\sum_{i=1}^{m(n)}\w(t^n_i\wedge t)\cdot(\phi^n(t^n_{i+1}\wedge t,\w)-\phi^n(t^n_{i}\wedge t,\w)).
\end{equation}
\end{remark}
Equation \eq{sf} is the pathwise counterpart of the classical definition of self-financing in probabilistic financial market models. However, in our purely analytical framework, we couldn't take it directly as the self-financing condition because some prior assumptions are needed to define path-by-path the quantities involved.

\section{A plausibility requirement}
\label{sec:reasonable}

The results reviewed in Subsection \ref{sec:arbitrage} cannot directly be applied to our framework, because the partitions considered there consist of stopping times, i.e. depend on the path, while we want to work with a fixed sequence of partitions $\Pi$ rather than with a random one.
Nonetheless, we can deduce that if we consider a singleton $\{\w\}$, where $\w\in\O_\psi$ with $\O_\psi$ defined in \eq{mod-jumps}, and our sequence of partition is of dyadic type for $\w$, then the property of finite quadratic variation for $\w$ is necessary to prevent the existence of a positive capital process, according to \defin{upperP}, trading at times in $\Pi$, that starts from a finite initial capital but ends up with infinite capital at time $T$.
However, the conditions imposed on the sequence of partitions are difficult to check.

Instead, we turn around the point of view: we want to keep our sequence of partitions $\Pi$ fixed and to identify the right subset of paths in $\O$ that is \emph{plausible} working with. 
To do so, we propose the following notion of \emph{plausibility} that, together with a technical condition on the paths, suggests that it is indeed a good choice to work on set of price paths with finite quadratic variation along $\Pi$, as we do in all the following sections.
\begin{definition}
  A set of paths $U\subset\O$ is called \emph{plausible} if there does not exist a sequence $(V_0^n,\phi^n)$ of simple self-financing strategies such that:
  \begin{enumerate}[(i)]
  \item the correspondent sequence of portfolio values, $\{V(t,\w;\phi^n)\}_{n\geq1}$, is non-decreasing for all paths $\w\in U$ at any time $t\in[0,T]$,
  \item the correspondent sequence of initial investments $\{V^n_0(\w_0)\}_{n\geq1}$ converges for all paths $\w\in U$,
  \item the correspondent sequence of gains along some path $\w\in U$ at the final time $T$ grows to infinity with $n$, i.e. $G(T,\w;\phi^n)\limn\infty$.
  \end{enumerate}
\end{definition}

\begin{proposition}
Let $U\subset\O$ be a set of price paths satisfying, for all $(t,\w)\in[0,T]\times U$ and all $n\in\NN$, 
\beq\label{eq:cn-conv}
\sum_{n=1}^\infty\lf\sum_{i=0}^{m(n-1)-1}\!\!\!\!\!\!\sum_{\stackrel{j,k:\,j\neq k,}{t^{n-1}_i\leq t_j^n,t^n_k<t^{n-1}_{i+1}}}\!\!\!\!\!\!(\w(t^n_{j+1}\wedge t)-\w(t^n_j\wedge t))\cdot(\w(t^n_{k+1}\wedge t)-\w(t^n_k\wedge t))\rg^-
\eeq
is finite, where $(x)^-:=\max\{0,-x\}$ denotes the negative part of $x\in\R$. Then, if $U$ is plausible, all paths $\w\in U$ have \fqv{\Pi}.
\end{proposition}

\proof
First, let us write explicitly what the condition \eq{cn-conv} means in terms of the relation between the $\w$ and the sequence of nested partitions $\Pi$.
Let $d=1$ for sake of notation. Denote by $A^n$ the $n^{th}$-approximation of the quadratic variation along $\Pi$, i.e.
$$A^n(t,\w):=\sum_{i=0}^{m(n)-1}(\w(t^n_{i+1}\wedge t)-\w(t^n_i\wedge t))^2\quad\forall(t,\w)\in[0,T]\times\O.$$
Then:
\begin{align*}
  &A^n(t,\w)-A^{n-1}(t,\w)=\\
={}&\sum_{i=0}^{m(n)-1}(\w(t^n_{i+1}\wedge t)-\w(t^n_i\wedge t))^2-\sum_{i=0}^{m(n-1)-1}(\w(t^{n-1}_{i+1}\wedge t)-\w(t^{n-1}_i\wedge t))^2\\
={}&\sum_{i=0}^{m(n-1)-1}\!\lf\sum_{t^{n-1}_i\leq t_j^n< t^{n-1}_{i+1}}\!\!(\w(t^n_{j+1}\wedge t)-\w(t^n_j\wedge t))^2-(\w(t^{n-1}_{i+1}\wedge t)-\w(t^{n-1}_i\wedge t))^2\rg\\
={}&{}-2\sum_{i=0}^{m(n-1)-1}\!\sum_{\stackrel{j,k:\,j\neq k,}{t^{n-1}_i\leq t_j^n,t^n_k<t^{n-1}_{i+1}}}\!(\w(t^n_{j+1}\wedge t)-\w(t^n_j\wedge t))(\w(t^n_{k+1}\wedge t)-\w(t^n_k\wedge t)).
\end{align*}
Thus the series in \eq{cn-conv} is exactly the series $\sum_{n=1}^\infty (A^n(t,\w)-A^{n-1}(t,\w))^-$.
Now, for $n\in\NN$, let us define a simple predictable process $\phi^n\in\Si(\pi^n)$ by
  \begin{align}
    \label{eq:Vn}
    \phi^n(t,\w):={}&{}-2\sum_{i=0}^{m(n)-1}\w(t^n_i)\ind_{(t^n_i,t^n_{i+1}]}(t)
  \end{align}
Then, we can rewrite the $n^{\mathrm{th}}$ approximation of the quadratic variation of $\w$ at time $t\in[0,T]$ as 
\begin{align}
  A^n(t,\w)={}&\w(t)^2-\w(0)^2-2\sum_{i=0}^{m(n)-1}\w(t^n_i)(\w(t^n_{i+1}\wedge t)-\w(t^n_i\wedge t))\nonumber\\
={}&\w(t)^2-\w(0)^2+G(t,\w;\phi^n)\nonumber\\
={}&V(t,\w;\phi^n)-c_n,  \label{eq:An}
\end{align}
where $c_n=\w(0)^2-\w(t)^2+V^n_0(\w_0)$.
We want to define the initial capitals $V^n_0$ in such a way that the sequence of simple self-financing strategies $(V_0^n,\phi^n)$ has non decreasing portfolio values at any time and the sequence of initial capitals converges. 
By writing
\beq\label{eq:kn}
A^n(t,\w)-A^{n-1}(t,\w)+k_n=V(t,\w;\phi^n)-V(t,\w;\phi^{n-1}),
\eeq
where $k_n=c_n-c_{n-1}=V^{n}_0(\w_0)-V^{n-1}_0(\w_0)$, we see that the monotonicity of $\{V(t,\w;\phi^n)\}_{n\in\NN}$ is obtained by opportunely choosing a finite $k_n\geq0$ (i.e. by choosing $V^n_0$), which is made possible by the boundedness of $\abs{A^n(t,\w)-A^{n-1}(t,\w)}$, implied by condition \eq{cn-conv}.
However, it is not sufficient to have $k_n<\infty$ for all $n\in\NN$, but we need the convergence of the series $\sum_{n=1}^\infty k_n$.
This is provided again by condition \eq{cn-conv}, because the minimum value of $k_n$ satisfying the positivity of \eq{kn} for all $t\in[0,T]$ is indeed $\max_{t\in[0,T]}(A^n(t,\w)-A^{n-1}(t,\w))^-$.
On the other hand, since both the sequence $\{V(t,\w;\phi^n)\}_{n\geq1}$ for any $t\in[0,T]$ and the sequence $\{V^n_0\}_{n\geq1}$ are regular, i.e. they have limit for $n$ going to infinity, by \eq{An} the sequence $\{A^n(t,\w)\}_{n\geq1}$ is also regular. Finally, since the sequence of initial capitals converges, the equation \eq{An} implies that the sequence of approximations of the quadratic variation of $\w$ converges if and only if $\{G(T,\w;\phi^n)\}_{n\geq1}$ converges.
But $U$ is a plausible set by assumption, thus convergence must hold.
\endproof

\section{Pathwise construction of the gain process}
\label{sec:gain}

In the following two propositions we show that we can identify a special class of (pathwise) self-financing trading strategies, respectively on the set of continuous price paths with \fqv{\Pi} and on the set of \cadlag\ price paths with \fqv{\Pi}, whose gain is computable path-by-path as a limit of Riemann sums.

\begin{proposition}[Continuous price paths]\label{prop:G}
 Let $\phi=(\phi(t,\cdot))_{t\in(0,T]}$ be an $\FF$-adapted $\R^d$-valued \caglad\ process on $(\O,\F)$ such that there exists $F\in\Cloc(\L_T)\cap\CC^{0,0}(\W_T)$ satisfying
  \begin{equation}
    \phi(t,\w)=\vd F(t,\w_{t})\quad\forall\w\in Q(\O,\Pi),t\in[0,T].
  \end{equation}
Then, there exists a \cadlag\ process $G(\cdot,\cdot;\phi)$ such that, for all $\w\in Q(\O^0,\Pi)$ and $t\in[0,T]$,
 \begin{align}
   G(t,\w;\phi)={}&\int_0^t \phi(u,\w_u)\cdot\ud^{\Pi}\w \label{eq:fi}\\
   \label{eq:pathint}
   ={}&\lim_{n\rightarrow\infty}\sum_{t^n_i\leq t}\vd F(t_i^n,\w^{n}_{t^n_i-})\cdot(\w(t_{i+1}^n\wedge T)-\w(t_i^n\wedge T)),
  \end{align}
where $\w^n$ is defined as in \eq{wn}.
Moreover, $\phi$ is the asset position of a pathwise self-financing trading strategy on $Q(\O^0,\Pi)$ with gain process $G(\cdot,\cdot;\phi)$.
\end{proposition}
\begin{proof}
  First of all, under the assumptions, the change of variable formula for functionals of continuous paths holds (\citep[Theorem 3]{contf2010}), which ensures the existence of the limit in \eq{pathint} and provide us with the definition of the F\"ollmer integral in \eq{fi}.
Then, we observe that each $n^{th}$ sum in the right-hand side of \eqref{eq:pathint} is exactly the accumulated gain of a pathwise self-financing strategy which trades only a finite number of times. Precisely, let us define, for all $\w\in\O$ and all $t\in[0,T)$, 
\begin{align*}
    \phi^n(t,\w):={}&\phi(0+,\w)\ind_{\{0\}}(t)+\sum_{i=0}^{m(n)-1}\phi\left(t^n_{i},\w^{n}_{t^n_i}\right)\ind_{(t^n_{i},t^n_{i+1}]}(t),\\
\intertext{and}
\psi^n(t,\w):={}& V_0-\phi(0+,\w)-\sum_{i=1}^{m(n)-1}\w(t^n_{i}\wedge t)\cdot(\phi^n(t^n_{i+1}\wedge t,\w)-\phi^n(t^n_{i}\wedge t,\w)),
\end{align*}
then $(V_0,\phi^n,\psi^n)$ is a simple self-financing strategy, with cumulative gain $G(\cdot,\cdot;\phi^n)$ given by
\begin{align*} 
G(t,\w;\phi^n)={}&\sum_{i=1}^{k}\vd F\lf t^n_{i-1},\w^{n}_{t^n_{i-1}-}\rg\cdot(\w(t^n_{i})-\w(t^n_{i-1}))\\
&\,+\vd F\lf t^n_{k},\w^{n}_{t^n_{k}-}\rg\cdot(\w(t)-\w(t^n_{k})).
\end{align*}
and portfolio value $V(\cdot,\cdot;\phi^n)$ given by
$$V(t,\w;\phi^n)=\psi^n(t,\w)+\phi^n(t,\w)\cdot\w(t)=V_0+G(t,\w;\phi^n).$$
Then, we have to verify that the simple asset position $\phi^n$ converges pointwise to $\phi$, i.e.
$$\forall\w\in\O,\,\forall t\in[0,T],\quad|\phi^n(t,\w)-\phi(t,\w)|\limn0.$$
This is true, because by assumption $\vd F\in\CC_l^{0,0}(\L_T)$ and this implies that the path $t\mapsto\vd F(t,\w_{t-})=\vd F(t,\w_{t})$ is left-continuous (see \rmk{regularity}). Indeed, for each $t\in[0,T],\w\in\O$ and $\e>0$, there exist $\bar n\in\NN$ and $\eta>0$, such that, for all $n\geq\bar n$, 
$$\dinf\left((t^n_k,\w^{n}_{t_k^n-}),(t,\w)\right)=\max\left\{||\w^n_{t^n_k-},\w_{t^n_k-}||_\infty,\sup_{u\in[t^n_k,t)}|\w(t^n_k)-\w(u)|\right\}+|t-t_k^n|<\eta,$$
where $k\equiv k(t,n):=\max\{i\in\{1,\ldots,m\}\;:\;t^n_i<t\}$, and
\begin{align*}
|\phi^n(t,\w)-\phi(t,\w)|={}&\abs{\phi(t_k^n,\w_{t_k^n}^{n})-\phi(t,\w)}\\
={}&\abs{\vd F(t_k^n,\w_{t_k^n-}^{n})-\vd F(t,\w)}\\
\leq{}&\e.
\end{align*}
We have thus built a sequence of self-financing simple trading strategies approximating $\phi$ and, if the realized price path $\w$ is continuous with finite quadratic variation along $\Pi$, then the gain of the simple strategies converges to a real-valued \cadlag\ function $G(\cdot,\w;\phi)$.
Namely, for all $t\in[0,T]$ and $\w\in Q(\O^0,\Pi)$,
$$G(t,\w;\phi^n)\limn G(t,\w;\phi),\quad G(t,\w;\phi)=\int_0^t\vd F(u,\w)\cdot\ud^{\Pi}\w.$$
Moreover, by the assumptions on $F$ and by \rmk{regularity}, the map $t\mapsto F(t,\w_{t})$ is continuous for all $\w\in C([0,T],\R^d)$. Therefore, by the change of variable formula for functionals of continuous paths, $G(\cdot,\w;\phi)$ is continuous for all $\w\in Q(\O^0,\Pi)$.

Thus, the process $G(\cdot,\cdot;\phi)$ satisfies the condition (i) in Definition \ref{def:path-sf} and so it is the gain process of the self-financing trading strategy with initial value $V_0$ and asset position $\phi$, on the set of continuous paths with \fqv{\Pi.}
\end{proof}

\begin{corollary}\label{cor:path-sf}
  Let $\phi$ be as in \prop{G}, then $\psi(\cdot,\cdot;\phi)$, defined for all $t\in[0,T]$ and $\w\in Q(\O^0,\Pi)$ by
  \begin{align*}
  \psi(t,\w;\phi)={}&V_0-\phi(0+,\w)\\
&{}-\Limn\sum_{i=1}^{k(t,n)}\w(t^n_i)\cdot\lf\vd F\lf t^n_{i},\w^{n}_{t^n_{i}-}\rg-\vd F\lf t^n_{i-1},\w^{n}_{t^n_{i-1}-}\rg\rg,
  \end{align*}
is the bond holding process of the self-financing trading strategy $(V_0,\phi)$ on $Q(\O^0,\Pi)$.
\end{corollary}

\begin{proposition}[C\`adl\`ag price paths]\label{prop:G-cadlag}
 Let  $\phi=(\phi(t,\cdot))_{t\in(0,T]}$ be an $\FF$-adapted $\R^d$-valued \caglad\ process on $(\O,\F)$ such that there exists $F\in\Cloc(\L_T)\cap\CC^{0,0}_r(\L_T)$ with $\vd F\in\CC^{0,0}(\L_T)$, satisfying
  \begin{equation*}
    \phi(t,\w)=\vd F(t,\w_{t-})\quad\forall\w\in Q(\O,\Pi),\,t\in[0,T].
  \end{equation*}
Then, there exists a \cadlag\ process $G(\cdot,\cdot;\phi)$ such that, for all $\w\in Q(\O,\Pi)$ and $t\in[0,T]$,
 \begin{align}
   \nonumber
   G(t,\w;\phi)={}&\int_0^t \phi(u,\w_u)\cdot\ud^{\Pi}\w \\
   \label{eq:pathint-cadlag}
   ={}&\lim_{n\rightarrow\infty}\sum_{t^n_i\leq t}\vd F\lf t_i^n,\w^{n,\De\w(t_i^n)}_{t^n_i-}\rg\cdot(\w(t_{i+1}^n\wedge T)-\w(t_i^n\wedge T)),
  \end{align}
where $\w^n$ is defined as in \eq{wn}.
Moreover, $\phi$ is the asset position of a pathwise self-financing trading strategy on $Q(\O,\Pi)$ with gain process $G(\cdot,\cdot;\phi)$.
\end{proposition}
\begin{proof}
The proof follows the lines of that of \prop{G}, using the change of variable formula for functionals of \cadlag\ paths instead of continuous paths, which entails the definition of the pathwise integral \eq{pathint-cadlag}.
For all $\w\in\O$ and $t\in[0,T]$, we define
\begin{align*}
    \phi^n(t,\w):={}&\phi(0,\w)\ind_{\{0\}}(t)+\sum_{i=0}^{m(n)-1}\phi\left(t^n_{i}+,\w^{n,\De\w(t^n_{i})}_{t^n_i-}\right)\ind_{(t^n_{i},t^n_{i+1}]}(t)\\
\intertext{and}
\psi^n(t,\w):={}& V_0-\phi(0+,\w)-\sum_{i=1}^{m(n)-1}\w(t^n_{i}\wedge t)\cdot(\phi^n(t^n_{i+1}\wedge t,\w)-\phi^n(t^n_{i}\wedge t,\w)).
\end{align*}
then $(V_0,\phi^n,\psi^n)$ is a simple self-financing strategy, with cumulative gain $G(\cdot,\cdot;\phi^n)$ given by
\begin{align*}
G^n(t,\w)={}&\sum_{i=1}^{k}\vd F\lf t^n_{i-1},\w^{n,\De\w(t^n_{i-1})}_{t^n_{i-1}-}\rg\cdot(\w(t^n_{i})-\w(t^n_{i-1}))\\
&\,+\vd F\lf t^n_{k},\w^{n,\De\w(t^n_{k})}_{t^n_{k}-}\rg\cdot(\w(t)-\w(t^n_{k})), 
\end{align*}
Finally, we verify that
$$\forall\w\in\O,\,\forall t\in[0,T],\quad|\phi^n(t,\w)-\phi(t,\w)|\limn0.$$
This is true, by the left-continuity of $\vd F$: for each $t\in[0,T],\w\in\O$ and $n\in\N$, we have that
$\forall \e>0$, $\exists \eta=\eta(\e)>0$, $\exists\bar n=\bar n(t,\eta)\in\NN$ such that, $\forall n\geq\bar n$,
$$\dinf\left(\w^{n,\De\w(t^n_k)}_{t_k^n-},\w_{t-}\right)=\max\left\{||\w^n_{t^n_k-},\w_{t^n_k-}||_\infty,\sup_{u\in[t^n_k,t)}|\w(t^n_k)-\w(u)|\right\}+|t-t_k^n|<\eta,$$
hence
\begin{align*}
|\phi^n(t,\w)-\phi(t,\w)|={}&\abs{\lim_{s\searrow t^n_k}\phi(s,\w_{t_k^n-}^{n,\De\w(t^n_k)})-\phi(t,\w)}\\
={}&\lim_{s\searrow t^n_k}\abs{\vd F(s,\w_{t_k^n-}^{n,\De\w(t^n_k)})-\vd F(t,\w_{t-})}\\
\leq&\e.
\end{align*}
Therefore:
$$G(t,\w;\phi^n)=\limn G(t,\w;\phi),\quad G(t,\w;\phi)=\int_{(0,t]}\vd F(u,\w_{u-})\cdot\ud^{\Pi}\w,$$
where $G(t,\w;\phi)$ is an $\FF$-adapted real-valued process on $(\O,\F)$.
Moreover, by the change of variable formula \eq{fif-d} and \rmk{regularity}, it is \cadlag\ with left-side jumps
\begin{align*}
  \De G(t,\w;\phi)={}&\lim_{s\nearrow t}(G(t,\w;\phi)-G(s,\w;\phi))\\
  ={}& F(t,\w_{t})- F(t,\w_{t-})-\lf F(t,\w_{t})- F(t,\w_{t-})-\vd F(t,\w_{t-})\cdot\De\w(t)\rg\\
  ={}&\vd F(t,\w_{t-})\cdot\De\w(t).
\end{align*}
\end{proof}

\begin{corollary}
  Let $\phi$ be as in \prop{G-cadlag}, then $\psi(\cdot,\cdot;\phi)$, defined for all $t\in[0,T]$ and $\w\in Q(\O,\Pi)$ by
\begin{align*}
\psi(t,\w;\phi)={}&V_0-\phi(0+,\w)\\
&{}-\Limn\sum_{i=1}^{k(t,n)}\w(t^n_i)\cdot\lf\vd F_{t^n_{i}}\lf\w^{n,\De\w(t^n_{i})}_{t^n_{i}-}\rg-\vd F_{t^n_{i-1}}\lf\w^{n,\De\w(t^n_{i-1})}_{t^n_{i-1}-}\rg\rg
\end{align*}
is the bond position process of the trading strategy $(V_0,\phi,\psi)$ which is self-financing on $Q(\O,\Pi)$.
\end{corollary}

\section{Pathwise replication of contingent claims}
\label{sec:replication}

A non-probabilistic replication result restricted to the non-path-dependent case was obtained by \citet{bickwill}, as shown in Propositions \ref{prop:bw1},\ref{prop:bw2} in \Sec{pathint} of this thesis.
Here, we state the generalization to the replication problem for path-dependent contingent claims.

First, let us introduce the notation.

\begin{definition}\label{def:hedging_error}
  The \emph{hedging error} of a self-financing trading strategy $(V_0,\phi)$ on $U\subset D([0,T],\R^d_+)$ for a path-dependent derivative with payoff $H$ in a scenario $\w\in U$ is the value 
$$V(T,\w;\phi)-H(\w)=V_0(\w)+G(T,\w;\phi)-H(\w).$$
$(V_0,\phi)$ is said to \emph{replicate} $H$ on $U$ if its hedging error for $H$ is null on $U$, while it is called a \emph{super-strategy} for $H$ on $U$ if its hedging error for $H$ is non-negative on $U$, i.e.
$$V_0(\w)+ G(T,\w;\phi)\geq H(\w_T)\quad\forall\w\in U.$$
\end{definition}

For any \cadlag\ function with values in $\S^+(d)$, say $A\in D([0,T],\S^+(d))$, we denote by
$$Q_A(\Pi):=\left\{\w\in Q(\O,\Pi):\;[\w](t)=\int_0^tA(s)\ud s\quad\forall t\in[0,T]\right\}$$
the set of functions with finite quadratic variation along $\Pi$ and whose quadratic variation is absolutely continuous with density $A$. Note that the elements of $Q_A(\Pi)$ are continuous, by \eq{qv-jumps}.
\begin{proposition}\label{prop:hedge}
Consider a path-dependent contingent claim with exercise date $T$ and a continuous payoff functional $H:(\O,\norm{\cdot}_\infty)\mapsto\R$. Assume that there exists a \naf\ $F\in\Cloc(\W_T)\cap\CC^{0,0}(\W_T)$ that satisfies
\beq\label{eq:fpde1}
\left\{\bea{ll}
\hd F(t,\w)+\frac12\tr\lf A(t)\cdot\vd^2F(t,\w)\rg=0,& t\in[0,T),\w\in Q_A(\Pi)\\
F(T,\w)=H(\w).&
\ea\right.\eeq
Let $\tilde A\in D([0,T],\S^+(d))$. Then, the hedging error of the trading strategy $(F(0,\cdot),\vd F)$, self-financing on $Q(\O^0,\Pi)$, for $H$ in any price scenario $\w\in Q_{\tilde A}(\Pi)$, is 
\begin{equation}\label{eq:err}
  \frac12\int_{0}^T\tr\lf (A(t)-\tilde A(t))\vd^2F(t,\w)\rg \ud t.
\end{equation}
In particular, the trading strategy $(F(0,\cdot),\vd F)$ replicates the contingent claim $H$ on $Q_A(\Pi)$ and its portfolio value at any time $t\in[0,T]$ is given by $F(t,w_t)$.
\end{proposition}

\begin{proof}[Proof]
  By \prop{G}, the gain at time $t\in[0,T]$ of the trading strategy $(F(0,\cdot),\vd F)$ in a price scenario $\w\in Q(\O^0,\Pi)$ is given by
  $$G(t,\w;\vd F)=\int_0^t\vd F(u,\w_u)\cdot\ud^{\Pi}\w(u).$$ Moreover, this strategy is self-financing on $Q(\O^0,\Pi)$, hence, by Remark \ref{rmk:path-sf}, its portfolio value at any time $t\in[0,T]$ in any scenario $\w\in Q(\O^0,\Pi)$ is given by
$$V(t,\w)=F(0,\w_{0})+\int_0^t\vd F(u,\w_u)\cdot\ud^{\Pi}\w.$$ 
In particular, since $F$ is smooth, we can apply the change of variable formula for functionals of continuous paths. This, by using the functional partial differential equation \eqref{eq:fpde1}, for all $\w\in Q_{\tilde A}(\Pi)$, gives 
\begin{align*}
V(T,\w)={}&F(0,\w_{0})+\int_{0}^T\vd F(t,\w)\cdot\ud^{\Pi}\w \\
 ={}&F(T,\w_T)-\int_{0}^T\hd F(t,\w)\ud t-\frac12\int_{0}^T\tr\lf \tilde A(t)\vd^2F(t,\w)\rg \ud t\\
={}&H-\frac12\int_{0}^T\tr\lf(\tilde A(t)-A(t))\vd^2F(t,\w)\rg \ud t.
\end{align*}
\end{proof}

\section{Pathwise isometries and extension of the pathwise integral}
\label{sec:isometry}
\sectionmark{Pathwise isometries and extension of the pathwise integral}

We denote $\mathring Q(\O,\Pi)$ the set of price paths $\w$ of non-trivial finite quadratic variation, that is $\w\in Q(\O,\Pi)$ such that $[\w](T)>0$.
Then, given $\w\in \mathring Q(\O,\Pi)$, we consider the measure space $([0,T],\B([0,T]),\ud[\w])$, where $\B([0,T])$ is the family of Borel sets of\OT\ and $\ud[\w]$ denotes the finite measure on \OT\ associated with $[\w]$. Here, we define the space of measurable $\R^d$-valued functions on $[0,T]$ with finite second moment with respect to the measure $\ud[\w]$, that is
\begin{align*}
\mathfrak L^2([0,T],[\w]):=\bigg\{&f:([0,T],\B([0,T]))\to\R^d\mbox{ measurable}:\\
&\int_0^T\pqv{f(t)\,^t\!f(t),\ud[\w](t)} <\infty\bigg\},
\end{align*}
where $\pqv{\cdot}$ denotes the Frobenius inner product, i.e. $\pqv{A,B}=\tr(^t\!AB)=\sum_{i,j}A_{i,j}B_{i,j}$.
Then, consider the set
\begin{align*}
\mathfrak L^2(\FF,[\w]):=\big\{&\phi\;\R^d\mbox{-valued, progressively measurable process on }(\O,\F,\FF),\\
&\phi(\cdot,\w)\in\mathfrak L^2([0,T],[\w])\big\}
\end{align*}
and 
we equip it with the following semi-norm:
$$\norm{\phi}^2_{[\w],2}:=\int_0^T\pqv{\phi(t,\w)\,^t\!\phi(t,\w),\ud[\w](t)},\quad \phi\in\mathfrak L^2(\FF,[\w])$$
We also define the quotient of the space of real-valued paths with finite quadratic variation by its subspace of paths with zero quadratic variation:
$$\bar Q(D([0,T],\R),\Pi):=Q(D([0,T],\R),\Pi)/ker([\cdot](T)),$$
where 
$ker([\cdot](T))=\{v\in Q(D([0,T],\R),\Pi):\;[v](T)=0\}$. 

\begin{proposition}\label{prop:Iw}
  For any price path $\w\in\mathring Q(\O,\Pi)$, let us define the pathwise integral operator 
\begin{eqnarray} \nonumber
  I^\w:\left(\bar\Si(\Pi),\norm{\cdot}_{[\w],2}\right)\!\!\!\!&\to&\lf \bar Q(D([0,T],\R),\Pi),\sqrt{[\cdot](T)}\rg\\
\phi&\mapsto&\int\phi\cdot\ud^\Pi\w, \label{eq:Iw1}
\end{eqnarray}
where $\bar\Si(\Pi):=\Si(\Pi)/ker(\norm{\cdot}_{[\w],2})$ and
\begin{align*}
  ker(\norm{\cdot}_{[\w],2})=\bigg\{&z=(z^1,\ldots,z^d)\in\mathfrak L^2(\FF,[\w]):\;\forall i,j=1,\ldots,d,\;\\
&[\w]_{i,j}\lf\{t\in[0,T]:\,z^i(t,\w)\neq0,\,z^j(t,\w)\neq0\}\rg=0\bigg\}.
\end{align*}
 $I^\w$ is an isometry between two normed spaces: 
\begin{equation}
\forall\phi\in\bar\Si(\Pi),\quad\left[\int\phi\cdot\ud^\Pi\w\right](T)=\int_0^T\pqv{\phi(t,\w)^t\!\phi(t,\w),\ud[\w](t)}.\label{eq:iso}
\end{equation}
Moreover, $I^w$ admits a closure on $L^2(\FF,[\w]):=\mathfrak L^2(\FF,[\w])/ker(\norm{\cdot}_{[\w],2})$, that is the isometry 
\begin{equation}
  \label{eq:Iw2}\bea{rcl}
  \tilde I^\w:\lf L^2(\FF,[\w]),\norm{\cdot}_{[\w],2}\rg&\to&\lf\bar Q(\O,\Pi),\sqrt{[\cdot](T)}\rg,\\\phi&\mapsto&\int\phi\cdot\ud^\Pi\w.
\ea
\end{equation}
\end{proposition}
\proof
The space $\left(\mathfrak L^2(\FF,[\w]),\norm{\cdot}_{[\w],2}\right)$ is a semi-normed space and its quotient with respect to the kernel of $\norm{\cdot}_{[\w],2}$ is a normed space, which is also a Banach space by the Riesz-Fischer theorem.
Moreover, for any $\phi\in\Si(\Pi)$, it holds
\begin{align*}
&  \int_0^T\pqv{\phi(t,\w)^t\!\phi(t,\w),\ud[\w](t)}=\\
={}& \sum_{i=1}^{m(n)}tr\lf\phi(t^n_i,\w)^t\!\phi(t^n_i,\w)([\w](t^n_i)-[\w](t^n_{i-1}))\rg\\
={}&\sum_{i=1}^{m(n)}tr\lf\phi(t^n_i,\w)^t\!\phi(t^n_i,\w)\lim_{m\to\infty}\sum_{t^n_{i-1}<t^m_j\leq t^n_i}(\w(t^m_j)-\w(t^m_{j-1}))^t\!(\w(t^m_j)-\w(t^m_{j-1}))\rg\\
={}&\lim_{m\to\infty}\sum_{t^m_j\in\pi^m}tr\lf\phi(t^m_j,\w)^t\!\phi(t^m_j,\w)(\w(t^m_j)-\w(t^m_{j-1}))^t\!(\w(t^m_j)-\w(t^m_{j-1}))\rg\\
={}&\lim_{m\to\infty}\sum_{t^m_j\in\pi^m}\lf\int_{t^m_{j-1}}^{t^m_j}\phi(\cdot,\w)\cdot\ud^\Pi\w\rg^2\\
={}&\left[\int\phi(\cdot,\w)\cdot\ud^\Pi\w\right](T).
\end{align*}
Finally, since $\lf \bar Q(D([0,T],\R),\Pi),\sqrt{[\cdot](T)}\rg$ is a Banach space and $\bar\Si(\Pi)$ is dense in $\lf L^2(\FF,[\w]),\norm{\cdot}_{[\w],2}\rg$, we can uniquely extend the isometry $I^\w$ in \eq{Iw1} to the isometry $\tilde I^\w$ in \eq{Iw2}.
\endproof

\begin{remark}
  For any $\w\in\mathring Q(\O,\Pi)$ and any $\phi\in L^2(\FF,[\w])$,
the pathwise integral of $\phi$ with respect to $\w$ along $\Pi$ is given by a limit of Riemann sums:
\beq
\int\phi\cdot\ud^\Pi\w =\Limn \sum_{t^n_i\in\pi^m}\phi^n(t^n_i,\w)\cdot(\w(t^n_i)-\w(t^n_{i-1})),
\eeq
independently of the sequence $(\phi^n)_{n\geq1}\in\bar\Si(\Pi)$ such that $$\norm{\phi^n(\cdot,\w)-\phi(\cdot,\w)}_{[\w],2}\limn0.$$
\end{remark}
Indeed, the definition of the isometry in \eq{Iw2} entails that, given $\phi(\cdot,\w)\in L^2(\FF,[\w])$, for any sequence $(\phi^n(\cdot,\w))_{n\geq1}\in\bar\Si(\Pi)$ such that $$\norm{\phi^n(\cdot,\w)-\phi(\cdot,\w)}_{[\w],2}\limn0,$$
then 
\beq\label{eq:limqv}
\left[\sum_{t^n_i\in\pi^m}\phi^n(t^n_i,\w)\cdot(\w(t^n_i)-\w(t^n_{i-1})) - \int\phi\cdot\ud^\Pi\w\right](T)\limn0.
\eeq
Since $\sqrt{[\cdot](T)}$ defines a norm on $ \bar Q(D([0,T],\R),\Pi)$, \eq{limqv} implies that the pathwise integral of $\phi$ with respect to $\w$ along $\Pi$ is a pointwise limit of Riemann sums:
$$\int\phi\cdot\ud^\Pi\w=\Limn\sum_{t^n_i\in\pi^m}\phi^n(t^n_i,\w)\cdot(\w(t^n_i)-\w(t^n_{i-1})),$$
independently of the chosen approximating sequence $(\phi^n)_{n\geq1}$.

\chapter{Pathwise Analysis of dynamic hedging strategies}
\label{chap:robust}
\chaptermark{Pathwise Analysis of dynamic hedging}

The issue of model uncertainty and its impact on the pricing and hedging of derivative securities has been the focus of a lot of research in the quantitative finance literature (see e.g. \citet{avlevyparas,bickwill,cont2006,lyons}).
Starting with Avellaneda et al.'s Uncertain Volatility Model \cite{avlevyparas}, the literature has focused on the analysis of the performance of pricing and hedging simple payoffs under model uncertainty.
The dominant approach in this stream of literature was to replace the assumption of a given, known, probability measure by a family of probability measures which reflects model uncertainty, and look for bounds on prices and performance measures for trading strategies using a worst-case analysis across the family of possible models.

A typical problem to consider is the hedging of a contingent claim. Consider a market participant who issues a contingent claim with payoff $H$ and maturity $T$ on some underlying asset. To price and hedge this claim, the issuer uses a pricing model (say, Black-Scholes), computes the price as
$$ V_t = E^{\mathbb{Q}}[H|{\cal F}_t]$$
and hedges the resulting profit and loss using the hedging strategy derived from the same model (say, Black-Scholes delta hedge for $H$).
However, the {\it true} dynamics of the underlying asset may, of course, be different from the assumed dynamics.
Therefore, the hedger is interested in a few questions: How good is the result of the model-based hedging strategy in a realistic scenario? How 'robust' is it to model mis-specification? How does the the hedging error relate to model parameters and option characteristics?
In 1998, \citet{elkaroui} provided an answer to these questions in the case of non-path-dependent options in the context of Markovian diffusion models. They provided an explicit formula for the profit and loss of the hedging strategy. \citet{elkaroui} showed that, when the underlying asset follows a Markovian diffusion
$$\ud S_t= \mu(t)S(t)\ud t+ S(t)\sigma_0(t,S(t)) \ud W(t) \qquad \text{under}\ \mathbb{P}^0,$$
 a hedging strategy
computed in a (mis-specified) local volatility model with volatility $\sigma$:
$$\ud S_t= r(t)S(t)\ud t+ S(t)\sigma(t,S(t)) \ud W(t) \qquad \text{under}\ \mathbb{Q}^\sigma$$
 leads, under some technical conditions on $\sigma,\sigma_0$ to a P\&L equal to
\beq\label{eq:elk}
\int_0^T \frac{\sigma^{2}(t,S(t))-\sigma_0^2(t,S(t))}{2}S(t)^2e^{\int_t^T r(s)\ud s}\overbrace{\partial^2_{xx}f(t,S(t))}^{\Gamma(t)}\ud t.
\eeq
$\mathbb{P}^0-$almost surely.
This fundamental result, called by Mark Davis \lq the most important equation in option pricing theory\rq\ \cite{davis}, shows that the exposure of a mis-specified delta hedge over  a short time period is proportional to the Gamma of the option times the specification error measured in quadratic variation terms.

In this chapter, we contribute to this line of analysis by developing  a general framework for analyzing  the performance and robustness of  delta hedging strategies for path-dependent derivatives across a given set of scenarios. Our approach is based on the pathwise financial framework introduced in \chap{path-trading}, which takes advantage of the non-anticipative functional calculus developed in \cite{contf2010}, which extends F\"ollmer's pathwise approach to \ito\ calculus \cite{follmer} to a functional setting.
Our  setting allows for general path-dependent payoffs and does not require any probabilistic assumption on the dynamics of the underlying asset, thereby extending previous results on robustness of hedging strategies in the setting of diffusion models to a much more general setting which is closer to the scenario analysis approach used by risk managers.
We obtain a pathwise formula for the hedging error for a general path-dependent derivative  and provide sufficient conditions ensuring the robustness of the delta hedge. Under the same conditions, we show that discontinuities in the underlying asset always deteriorate the hedging performance. We show in particular that robust hedges may be obtained in a large class of continuous exponential martingale models under a vertical convexity condition on the payoff functional. We apply these results to the case of hedging strategies for Asian options and  barrier options, both in the Black Scholes model with time-dependent volatility and in a model with path-dependent characteristics, the Hobson-Rogers model \cite{hobson-rogers}.

\section{Robustness of hedging  under model uncertainty: a survey}

\subsection{Hedging under uncertain volatility}

Two fundamental references in the literature on model uncertainty are \citet{avlevyparas} and \citet{lyons}.
\citet{avlevyparas} proposed a novel approach to  pricing and hedging under \lq volatility risk\rq: the \emph{Uncertain Volatility Model}. Instead of looking for the most accurate model (in terms of forward volatility of asset prices), they work under the assumption that the volatility is bounded between two extreme values. In particular, they assume that future stock prices are \ito\ processes
\beq\label{eq:unvol}
\ud S(t)=S(t)\lf\s(t)\ud W(t)+\mu(t)\ud t\rg,
\eeq
where $\mu,\s$ are adapted process such that $\s_{\min}\leq\s\leq\s_{\max}$ and $W$ is a standard Brownian motion. 
The problem under consideration was the pricing and hedging of a derivative security paying a stream of cash-flows at $N$ future dates: $f_1(S(t_1)),\ldots,f_N(S(t_N))$, where $f_j$ are known functions. By denoting $\P$ the class of probability measures on the set of paths under which the coordinate process $S$ has a dynamics \eq{unvol} for some $\s$ between the bounds, then in absence of arbitrage opportunities it is possible to construct an optimal (in the sense that the initial cost is minimal) self-financing portfolio that hedges a short position in the derivative and gives a non-negative value after paying out all the cash flows. This optimal portfolio consists of an initial capital $p^+(t,S(t))$ and a risky position $\partial_{S}p^+(t,S(t))$, where $p^+(t,S(t))=\sup_{\PP\in\P}\EE^\PP\left[\sum_{j=1}^Ne^{-r(t_j-t)}f_j(S(t_j))\right]$ is obtained by solving the Black-Scholes-Barenblatt equation
\begin{align*}
\partial_{t}p^+(t,S(t))+\frac12S(t)^2\s^*\lf\partial_{SS}p^+(t,S(t))\rg^2\partial_{SS}p^+(t,S(t))\\
=-\sum_{k=1}^{N-1}f_j(S(t))\d_{t_k}(t),\quad t<t_N,
\end{align*}
with final condition $p^+(t,S(t))=f_N(S(t_N))$ where the function $\s^*$ is defined as $\s^*(s)=\s_{\min}\ind_{(-\infty,0)}(s)+\s_{\max}\ind_{[0,\infty)}(s)$.

On the other hand, \citet{lyons} analyzes the same problem of \citet{avlevyparas} but uses a pathwise approach, in view of F\"ollmer's formula \eq{follmer_ito}. The security process $S$ is multi-dimensional and the only assumption is that it has finite quadratic variation at any time $t\geq0$ along the sequence of dyadic partitions and that the quadratic variation function $A=\{A_{i,j}\}_{i,j\in I}$ is such that, for all $u\geq0$, $A(u)$ belongs to the set
$$\bea{rl}O(\l,\L,K(u,S(u)):=&\!\!\!\left\{\g=\{\g_{i,j}\}_{i,j\in I}\text{ positive symmetric matrix, }\right.\\
&\left.\forall v\in\R^I_+,\; \l\,^t\!vK(u,S(u))v<^t\!v\g v<\L\,^t\!vK(u,S(u))v\right\},\ea$$
where $\l\leq1,\L\geq1$ are given constants and $K$ is a reference model for the squared volatility of the security, e.g. $K_{i,j}(t,s)=\s_{i,j}(t,s)s_is_j$. The main result in \cite{lyons} claims that there exists a hedging strategy with an initial investment $f(0,S(0))$ that replicates a derivative paying $F(\t,S(\t))$ at the first occasion $\t$ that the security $(t,S(t))$ leaves a fixed smooth domain $U\subset\R\times\R^I_+$. Moreover, such a strategy returns at any time $t<T$ an excess stream of money equal to
$$\int_0^t\frac12\lf\sum_{i,j\in I}(\tilde A_{i,j}(u,S(u))-A_{i,j}(u,S(u))\partial_{s_i s_j}f\rg(u,S(u))\ud u$$
and at time $T$ it holds exactly $F(T,S(T))$.
This is an application of the pathwise \ito\ formula proven by F\"ollmer and of the PDE theory, which guarantees that under appropriate conditions on $K$ the Pucci-maximal equation
$$\bea{l}\sup_{a\in O(\l,\L,K(u,S(u)))}\lf\frac12\sum_{i,j\in I}a_{i,j}\partial_{s_i s_j}f\rg(u,s)+\partial_{u}f(u,s)=0,\quad(u,s)\in U,\\
f(u,s)=F(u,s),\quad (u,s)\in\partial_pU\ea$$
has a smooth solution $f$ which is also the solution of the linear equation
$$\lf\frac12\sum_{i,j\in I}\tilde A_{i,j}\partial_{s_i s_j}f\rg(u,s)+\partial_{u}f(u,s)=0,\quad\tilde A_{i,j}\in O(\l,\L,K(u,s)).$$

In 1996, \citet{bergman} established the properties of European option prices as functions of the model parameters in case the underlying asset follows a one-dimensional diffusion or belongs to a certain restricted class of multi-dimensional diffusions, or stochastic volatility models, by using PDE methods. Their results have implications in the robustness analysis of pricing and hedging derivatives.
They assume absence of arbitrage opportunities and that the following stochastic differential equations are well-defined in terms of path-by-path uniqueness  of solutions and that parameters allow for the application of the Feynman-Kac theorem.
In the one-dimensional case, they assume that the risk-neutral dynamics of the underlying asset process $S$ is
\begin{equation}
  \label{eq:1dim}
  \ud S(t)=S(t)r(t)\ud t+S(t)\s(t,S(t))\ud W(t),
\end{equation}
where $W$ is a standard Brownian motion.
This holds the \textit{no-crossing} property, i.e.
\begin{equation}
  \label{eq:nocross}
s_2\geq s_1\;\Rightarrow\;S^{t,s_2}(u)\geq S^{t,s_1}(u),\;\text{almost surely}, \forall u\geq t,
\end{equation}
where $S^{s,t}$ solves \eq{1dim} with $S^{s,t}(t)=s$.
Indeed, fixed a realization $W(\cdot,\w)$ of the Brownian motion in \eq{1dim} and the correspondent paths $S^{t,s_2}(\cdot,\w)$ and $S^{t,s_1}(\cdot,\w)$, if there exists a time $\bar s\geq t$ such that $S^{t,s_2}(\bar s,\w)=S^{t,s_1}(\bar s,\w)$, then the two paths will coincide from $\bar s$ onwards, by the Markov property.
This property allows a claim price to inherit monotonicity from the payoff.
In the two-dimensional case, they assume that the risk-neutral dynamics is given by
\begin{equation}
  \label{eq:2dim}
\left\{\bea{ll}
  \ud S(t)={}&S(t)r(t)\ud t+S(t)\s(t,S(t),Y(t))\ud W^1(t),  \\
  \ud Y(t)={}&(\b(t,S(t),Y(t))-\l(t,S(t),Y(t)))\th(t,S(t),Y(t))\ud t \\
           &+\th(t,S(t),Y(t))\ud W^2(t),
\ea\right.
\end{equation}
where $W^1,W^2$ are standard Brownian motions with quadratic co-variation $[W^1,W^2](t)=\rho(t,S(t),Y(t))\ud t$.
Despite the fact that, unfortunately, multi-dimensional diffusions do not exhibit in general a similar behavior, there are conditions under which the process $S$ solving \eq{2dim} holds the no-crossing property \eq{nocross} as well.
A first important result concerns the inheritance of monotonicity from option prices and establishes bounds on the risky position of a delta-hedging portfolio.
\begin{theorem}[Theorem 1 in \cite{bergman}]
Let the payoff function $g$ be one-sided differentiable and at each point $x$ we also allow either $g'(x-)=\pm\infty$ or $g'(x+)=\pm\infty$. Suppose that $S$ follows either the one-dimensional diffusion~\eq{1dim}, or the two-dimensional diffusion~\eq{2dim} with the additional property that the drift and diffusion parameters do not depend on $s$. Then
$$\inf_x (\min\{g'(x-),g'(x+)\})\leq \partial_{s}v\leq\sup_x(\min\{g'(x-),g'(x+)\}),$$
uniformly in $s,t$, where $v$ is the value of the European claim with payoff $g$.
\end{theorem}
This follows directly by the no-crossing property and an application of the generalized intermediate value theorem of real analysis.
A second important result proves the inheritance of convexity of the claim price from the payoff function, which was already known for proportional one-dimensional diffusions (Black-Scholes setting).
\begin{theorem}[Theorem 2 in \cite{bergman}]\label{th:bergman2}
  Suppose that $S$ follows either the one-dimensional diffusion~\eq{1dim}, or the two-dimensional diffusion~\eq{2dim} with the additional property that the drift and diffusion parameters do not depend on $s$ and there exists a function $G:[0,\infty)^2\rightarrow\R$ such that
$$G(t,y)=\s(t,s,y)\th(t,s,y)\rho(t,s,y).$$
Then, if the payoff function is convex (concave), the calms value is a convex (concave) function of the current underlying price.
\end{theorem}
The proof proceeds by applying the Feynman-Kac theorem to write the claim value as the solution of a Cauchy problem with final datum given by the payoff function $g$; then, by taking the $s$-partial derivative of the PDE, we get a new Cauchy problem for $\partial_{s}v$ with final datum $g'$. It suffices to apply again the Feynman-Kac theorem, taking into account the hypothesis on coefficients, to write $\partial_{s}v$ as an expectation of $g'$ composed to a new stochastic process which holds the no-crossing property. Finally, the no-crossing property gives the monotonicity of $\partial_{s}v$ and equivalently the convexity (concavity) of $v$ in the underlying asset price.
A consequence of the previous results in terms of robustness analysis of hedging strategies is the extension of the comparative statics known in a Black-Scholes setting to a one-dimensional diffusion. In particular, an ordering in the volatility functions is preserved in the claim value functions:
\begin{theorem}[Theorem 6 in \cite{bergman}]
  Let $\s_1(t,s)\geq\s_2(t,s)$ for all $s,t$ and strict inequality holds in some region, then $v_1(t,s)\geq v_2(t,s)$ for all $s,t$.
\end{theorem}
This result turns out to be of special interest if one has knowledge of deterministic bounds on the volatility and the claim to hedge is a plain vanilla option, e.g. a call option; in such a case  it implies to have both the call option and its Delta bounded respectively by the correspondent Black-Scholes call prices and appropriate Black-Scholes Deltas.
\begin{theorem}[Theorem 8 in \cite{bergman}]
If for all $s,t$, $\ushort\s(t)\leq\s(t,s)\leq\bar\s(t)$, then, for all $s,t$,
$$\bea{cc}c^{\mathrm{BS}(\ushort\s)}(t,s)\leq c(t,s)\leq c^{\mathrm{BS}(\bar\s)}(t,s),\\\partial_{s}c^{\mathrm{BS}(\bar\s)}(t,s'')\leq \partial_{s}c(t,s)\leq \partial_{s}c^{\mathrm{BS}(\bar\s)}(t,s'),\ea$$
where $s',s''$ solve respectively
$$\bea{l}c^{\mathrm{BS}(\ushort\s)}(t,s)=c^{\mathrm{BS}(\bar\s)}(t,s'')+\partial_{s}c^{\mathrm{BS}(\bar\s)}(t,s'')(s-s''),\\c^{\mathrm{BS}(\ushort\s)}(t,s)=c^{\mathrm{BS}(\bar\s)}(t,s')-\partial_{s}c^{\mathrm{BS}(\bar\s)}(t,s')(s'-s).\ea$$
\end{theorem}
The bounds on the delta are an immediate consequence of bounds on the call price and of inherited convexity. When the values of $s$ and $c(t,s)$ are observed, these bounds can even be tightened.
Finally, they remark that relaxing either the continuity or the Markov property in the one-dimensional case, or the restrictions on the two-dimensional diffusion, the no-crossing property does not need to hold, hence call option prices may exhibit unexpected behaviors.

In 1998, \citet{elkaroui} derived results analogous to \citet{bergman} for both European and American options under a one-dimensional diffusion setting, by an independent approach based on stochastic flows rather than PDEs.
While completeness is not assumed, the market is equipped with the strongest form of no-arbitrage condition, namely discounted stock prices are martingales under the objective probability measure $\PP$. The stock price is assumed to follow \begin{equation}\label{eq:elk-dS}
  \ud S(t)=r(t)S(t)\ud t+\s(t)S(t)\ud W(t),
\end{equation}
where $W$ is a standard $(\Ft,\PP)$-Brownian motion, the interest rate $r$ is a deterministic function in $L^1([0,T],\ud t)$ and the volatility process $\s$ is non-negative, $\Ft$-adapted, almost surely in $L^1([0,T],\ud t)$ and such that the discounted stock price
$$\frac{S(t)}{M(t)}=S(0)\exp\left(\int_0^t\s(u)\ud W(u)-\frac12\int_0^t\s^2(u)\ud u\right),\quad 0\leq t\leq T,$$
is a square-integrable martingale.
A trading strategy, or \emph{portfolio process}, is defined as a bounded adapted process, while a \emph{payoff function} is defined as a convex function on $\R_+$ having bounded one-sided derivatives.
Let $h$ be the payoff function of a European contingent claim, $\phi$ a portfolio process and $P$ an adapted process such that $P(T)=h(S(T))$ (called a \emph{price process}), the \emph{tracking error} associated with $(P,\phi)$ is defined as $e:=V-P$, where $V$ is the value process of the self-financing portfolio with trading strategy $\phi$ and initial investment $V(0)=P(0)$. Then, $(P,\phi)$ is called a
\begin{itemize}
\item \emph{replicating strategy} if $\frac e M\equiv0$, in which case the hedger exactly replicates the option at maturity, i.e. $V(T)=h(S(T))$, and $P(0)=\EE^\PP\left[\frac{h(S(T))}{M(T)}\right]$ is an arbitrage price for the claim;
\item\emph{super-strategy} if $\frac e M$ is non-decreasing, in which case the hedger super-replicates a short position in the claim at maturity, i.e. $V(T)\geq h(S(T))$, and $P(0)\geq\EE^\PP\left[\frac{h(S(T))}{M(T)}\right]$;
\item \emph{sub-strategy} if $\frac e M$ is non-increasing, hence the hedger super-replicates a long position in the claim and the above inequalities are reversed.
\end{itemize}
The main purpose in \cite{elkaroui} is to analyze the performance of a hedging portfolio derived from a model with mis-specified volatility.
First, assuming completeness, they provide two counterexamples of the familiar properties of option prices, when volatility is allowed to be stochastic in a path-dependent manner. On the one hand, a volatility process depending on the initial stock price and the driving Brownian motion may cause the value of a European call to fail the monotonicity property, even if the volatility is non-decreasing in the initial stock price, as it happens for
\begin{equation}\label{eq:counterex1}
  \s(t)=\ind_{\{W(t)<S(0)\}}\ind_{\{t\leq T_a\}},\quad a>0,\quad T_a:=\inf\{t\geq0,W(t)=a\}.
\end{equation}
 On the other hand, even when the underlying dynamics allows the claim value to preserve both monotonicity and convexity, it may happen that an ordering on volatilities is not passed on to the respective call values, e.g.
\begin{equation}\label{eq:counterex2}
  \s(t)\leq\hat\s(t):=\ind_{\{t\leq T_a\}}\quad\text{but}\quad v(x)>\hat v(x)=0\;\forall x\in(0,a).
\end{equation}
Given a mis-specified model
\begin{equation}\label{eq:misS}
  \ud S_\g(t)=S_\g(t) r(t)\ud t+S_\g(t)\g(t,S_\g(t))\ud W(t),
\end{equation}
where the only source of randomness in the volatility is the dependence on the current stock price, the following theorem states the important property of propagation of convexity, also obtained by \citet{bergman}, for one-dimensional diffusions, but the proof follows a completely independent approach.
\begin{theorem}[Theorem 5.2 in \cite{elkaroui}]\label{th:elkaroui1}
  Suppose that $\g:[0,T]\times\R_+\rightarrow\R$ is continuous and bounded from above and $s\mapsto\partial_{s}(s\g(t,s))$ is Lipschitz-continuous and bounded in $\R_+$, uniformly in $t\in[0,T]$. Then, if $h$ is a payoff function, the mis-specified claim value $$v_\g(x)=\EE^\PP\left[h(S_\g(T))|S_\g(0)=x\right]$$ is a convex function of $x>0$.
\end{theorem}
Indeed, by denoting $S_\g^x$ the solution of \eq{misS} with initial value $S_\g^x(0)=x$ and by applying the \ito\ formula to the process $\partial_{x}S_\g^x$, 
the discounted process $\z^x=\lf\frac{\partial_ xS_\g^x(t)}{M(t)}\rg_{t\in[0,T]}$ turns out to be the exponential martingale of $(N(t))_{t\in[0,T]}$, $N(t)=\int_0^t\partial_{s}(S_\g^x(u)\g(u,S_\g^x(u)))\ud W(u)$, i.e.
$\z^x(t)=\linebreak[0]\exp\big\{N(t)-\frac12\pqv{N}(t)\big\}.$
Then, Girsanov's theorem says that the process $W^x$, defined by $W^x(t)=W(t)-\int_0^t\partial_{s}(S_\g^x(u)\g(u,S_\g^x(u)))\ud u$, is a $\PP^x$-Brownian motion, where $\frac{\ud\PP^x}{\ud\PP}=\z(T)$.
The idea now is to prove that $v$ has increasing one-sided derivatives. In order to do that, the first step is to bound the incremental ratios $\frac{v_\g(y)-v_\g(x)}{y-x}$, for $y>x$, in such a way to be able to apply on both sides a version of Fatou's lemma. This gives
\begin{align*}\EE^{\PP^x}\left[h'(S_\g^x(T)+)\right]\leq{}&\liminf_{y\searrow x}\frac{v_\g(y)-v_\g(x)}{y-x}\\
\leq{}&\limsup_{y\searrow x}\frac{v_\g(y)-v_\g(x)}{y-x}\leq\EE^{\PP^x}\left[h'(S_\g^x(T)+)\right],
\end{align*}
and an analogous estimate holds for $y<x$, $y\nearrow x$, thus $$v_\g'(x\pm)=\EE^{\PP^x}\left[h'(S_\g^x(T)\pm)\right].$$
Let us notice that, to achieve the above bounds, it is used the same no-crossing property \eq{nocross} which is fundamental in \cite{bergman}.
Lastly, to remove the dependence on $x$ of the expectation operators, they define a new process $\tilde S^x$, whose law under $\PP$ is the same as the law of $S_\g^x$ under $\PP^x$ and which still holds the no-crossing property, hence rewrite $v_\g'(x\pm)=\EE^\PP\left[h'(\tilde S^x(T)\pm)\right]$.
From the last argument it also follows that the one-sided derivatives of $v$ have the same bounds as $h$.
Under additional requirements, \citet{elkaroui} proved a robustness principle similar to Theorem~\ref{th:bergman2} but also providing the explicit formula of the tracking error, which is fundamental to monitor hedging risks.
\begin{theorem}\label{th:elkaroui2}
  Under the assumptions of Theorem~\ref{th:elkaroui1}, let $r,\g$ be H\"older-continuous in their arguments. Then, if
  \begin{equation}\label{eq:vol-dom}
    \s(t)\leq\g(t,S(t))\text{ for Lebesgue-almost all }t\in[0,T], \;\PP-a.s.,
  \end{equation}
then $(P_\g,\De_\g)$ is a super-strategy, where $P_\g(t):=v_\g(t,S(t))$ and $\De_\g(t):=\partial_{s}v_\g(t,S(t))$ for all $t\in[0,T]$. If the volatilities satisfy the reversed inequality in \eq{vol-dom}, then $(P_\g,\De_\g)$ is a sub-strategy. Moreover, the tracking error associated with $(V_\De,P_\g)$ is
\begin{equation}\label{eq:disc-e}
  e_\g(t)=M(t)\frac12\int_0^t\lf\g^2(u,S(u))-\s^2(u)\rg S^2(u)\partial^2_{xx}v_\g(u,S(u))\frac{\ud u}{M(u)}.
\end{equation}

\end{theorem}
Indeed, under the assumptions, the value function $v_\g$ defined by
$$v_\g(t,x):=\EE\left[e^{-\int_t^Tr(u)\ud u}h(S_\g^{t,x}(T))\right],\quad t\in[0,T],\;x>0,$$
where $S_\g^{t,x,}$ is the solution of \eq{misS} with initial condition $S_\g^{t,x}(t)=x$, belongs to $\C^{1,2}([0,T)\times\R_+)\cap\C([0,T]\times\R_+)$ and satisfies the partial differential equation $L_\g v_\g=0$ on $[0,T)\times\R_+$, with the operator defined by
\begin{equation}\label{eq:Lgamma}
  L_\g f(t,x):=\partial_{t}f(t,x)+r(t)x\partial_{x}f(t,x)+\frac12\g^2(t,x)x^2\partial^2_{xx}f(t,x)-r(t)f(t,x).
\end{equation}
 Then, the value $V_{\De}$ of the self-financing portfolio $\De_\g$ will evolve according to
$$\ud V_\De(t)=r(t)V_\De(t)\ud t+\De_\g(t)(\ud S(t)-r(t)S(t)\ud t),$$
whereas 
the price process is governed by
\begin{eqnarray*}
\ud P_\g(t)&={}&r(t)P_\g(t)\ud t+\De_\g(t)(\ud S(t)-r(t)S(t)\ud t)\\
  &&+\frac12\lf\s^2(t)-\g^2(t,S(t))\rg S^2(t)\partial^2_{xx}v_\g(t,S(t))\ud t.
\end{eqnarray*}
Finally, the convexity of $v_\g$ and the domination of the mis-specified volatility over the \lq true\rq\ one end the proof.
Important remarks about weakening the assumption \eq{vol-dom} are reported in the appendix of \cite{elkaroui}.
By the way, under the regularity requirements, equation \eq{disc-e} for the discounted tracking error is still true, independently of the domination of volatilities. If $\s,\g$ are both non-negative, square-integrable and deterministic functions of time, satisfying
\begin{equation}
  \label{eq:weak-dom}
  \lf\int_t^T\s^2(u)\ud u \rg^{\frac12}\leq\lf\int_t^T\g^2(u)\ud u \rg^{\frac12},\quad\text{for all }0\leq t\leq T,
\end{equation}
then the mis-specified value of the claim succeeds to dominate the true price, but the mis-specified delta-hedging portfolio is not guaranteed to replicate the option at maturity, in the sense that the expected tracking error under the market probability measure can be negative.

In 1998, \citet{hobson} also addressed the monotonicity and super-replication properties of options prices under mis-specified models. The theorems presented in \cite{hobson} are similar to the results found in \cite{bergman} and \cite{elkaroui}, but the author uses a further different approach, based on coupling techniques.

The setting is that of a continuous-time frictionless market with finite horizon $T$, where the interest rate is set to $r=0$ and the stock price process $S$ is a weak solution to the stochastic differential equation
\begin{equation}\label{eq:hob-dS}
  \ud S(t)=S(t)\s(t)\ud B(t),\quad S(0)=s_0,
\end{equation}
for some standard Brownian motion $B$ on a stochastic basis $(\O,\F,\PP)$ and an adapted volatility process $\s$.
For the moment, completeness of the model is assumed, so that options prices are given by $\PP$-expectations of the respective claims at maturity.
The first main theorem goes under the name of ``option price monotonicity''.
\begin{theorem}\label{th:hob-mono}
  Let $h$ be a convex function and consider two candidate models for \eq{hob-dS}, namely $\s(\cdot)=\tilde\s(\cdot,S(\cdot))$ or $\s(\cdot)=\hat\s(\cdot,S(\cdot))$, such that $\hat\s(t,s)\geq\tilde\s(t,s)$ for all $t\in[0,T]$, $s\in\R$. Then, the European option with payoff $h(S(T))$ has a higher value under the model with volatility $\hat\s$ than under the one with volatility $\tilde\s$.
\end{theorem}
The proof is based on the joint application on the Brownian representation of local martingales and a coupling argument. Precisely, fixed a Brownian motion $W$ issued of $s_0$, define, for each model, a strictly increasing process $\t$ as the solution, for almost all $\w\in\O$, of the ordinary differential equation
$$\frac{\ud \t(t;\w))}{\ud t}=\frac1{W^2(t;\w)\s^2(\t(t;\w),W(t;\w))},\quad t\in[0,T].$$
Then, define $A(\cdot;\w)$ as the inverse of $\t(\cdot;\w)$ and consider the process $P=W(A)$ (again one for each model). This is a local martingale whose quadratic variation has time-derivative given by
$$\partial_t A(t)=W^2(A(t))\s^2(\t(A(t)),W(A(t)))=P^2(t)\s^2(t,P(t))$$
Thus, $P$ is a weak solution to the SDE $\ud P(t)=P(t)\s(t,P(t))\ud B$ for some Brownian motion $B$. By this representation, $\hat A\geq\tilde A$ on [0,T], almost surely. Indeed, at time 0, $\hat P(0)=\tilde P(0)=s_0$ and $\hat A(0)=\tilde A(0)=0$; afterward, if $\hat P(t)=\tilde P(t)$ then $\ud\hat A(t)\geq\ud\tilde A(t)$ and if $\hat A(t)=\tilde A(t)$ then $\hat P(t)=\tilde P(t)$.
Finally, by Jensen's inequality and properties of the Brownian motion,
\begin{eqnarray*}
  \EE[h(\hat P(T))]&={}&\EE\left[\Et{h(\hat P(T))}{\tilde A(T)}\right]\\
&\geq&\EE\left[h\lf\Et{\hat P(T)}{\tilde A(T)}\rg\right]\\
&={}&\EE\left[h\lf\Et{W(\tilde A(T))+(W(\hat A(T))-W(\tilde A(T)))}{\tilde A(T)}\rg\right]\\
&={}&\EE[h(\tilde P(T))].
\end{eqnarray*}
Notice that Hobson's method allows to generalize the statement of the theorem in two directions:
\begin{itemize}
\item it does not require the completeness assumption, which is used only in the last step of proof, when pricing the European claim by taking the expectation under the risk-neutral probability $\PP$, and can be omitted provided an agreed pricing measure;
\item it has not to restrict to diffusion models, as the same construction applies also to the case of path-dependent volatility $\s(t)=\s(t,S_t)$, provided that $\t$ and its inverse can be defined and by assuming that, for all $t\in[0,T],s\in\R$,
\begin{equation}\label{eq:path-dom}
\hat\s(t,\hat s_t)\geq\tilde\s(t,\tilde s_t)\quad\forall \hat s_t,\tilde s_t\in
\{\{f(u\wedge t)\}_{u\in[0,T]},\; f(0)=s_0,\; f(t)=s\}
\end{equation}
The contradiction that seems to arise with the counterexample \eq{counterex2} in \cite{elkaroui} is not consistent here. In fact, in \cite{elkaroui} the price process is defined to be the strong solution of the SDE \eq{elk-dS}, so that the coupling argument could not be applied, while in \cite{hobson} it is instead a weak solution. In effect, what matters to the aim of derivative pricing and hedging is the law of the price process, rather than its relation with a specific Brownian motion.
\end{itemize}

The second property of option prices addressed by \citeauthor{hobson} is the preservation of convexity from the payoff to the value function. This is then used to derive the so-called \lq super-replication property\rq.
\begin{theorem}\label{th:hob-conv}
  Suppose the asset price follows the complete diffusion model \eq{hob-dS}
where the volatility function has sufficient regularity to ensure that the solution is unique-in-law {\upshape(e.g. $s\mapsto s\s(t,s)$ Lipschitz)} and a true martingale {\upshape(e.g. $\s$ bounded)}. If $h$ is a convex payoff function, then the claim value at each time prior to maturity is convex in the current underlying price.
\end{theorem}
The coupling argument used here is the following.
Take $0<z<y<x$ and define $X,Y,Z$ as the solutions to \eq{hob-dS} with respect to independent Brownian motions and starting point respectively $x,y,z$ at time 0. Denote the crossing times with $H_X:=\inf\{t\geq0,X(t)=Y(t)\}$ and $H_Y:=\inf\{t\geq0,Y(t)=Z(t)\}$, and $\t:=H_X\wedge H_Y\wedge T$. Conditionally on $\{\t=H_X\}$ (respectively on $\t=H_Y$), $X(T)\stackrel d= Y(T)$ (respectively $Y(T)\stackrel d= Z(T)$), while on $\{\t=T\}$ we have $Z(T)<Y(T)<X(T)$.
Thus,  by using the identities in law and the convexity of $h$,
\begin{align*}
\EE[(X(T)-Z(T))h(Y(T))]\leq{}&\EE[(Y(T)-Z(T))h(X(T))]\\
&{}+\EE[(X(T)-Y(T))h(Z(T))].
\end{align*}
Then, the independence of the driving Brownian motions gives
$$(x-z)\EE[h(Y(T))]\leq(x-y)\EE[h(Z(T))]+(y-z)\EE[h(X(T))],$$
that is the convexity of the option price, by arbitrariness of starting points.

It should be noticed that this proof cannot be extended to non-diffusion models, where the identities in law could not be used.

The same property is also proved in \cite{bergman} and \cite{elkaroui}, however both require more restrictive conditions, such as the differentiability of the diffusion coefficient $s^2\s^2(t,s)$ and a bounded (possibly one-sided) derivative for $h$.
In case $h$ has a derivative bounded by a constant $C$ on $[0,\infty)$, then bounds on the option price and its spatial derivative at any time $t\in[0,T]$ are a direct consequence:
$$h(0)-CS(t)\leq v(t,S(t))\leq h(0)+CS(t),\quad \left|\partial_{s}v(t,S(t))\right|\leq C.$$

In \cite{elkaroui} the property of inherited convexity is used to prove robustness of a delta-hedging portfolio, accordingly to their definition.
\citeauthor{hobson} reproduces the same steps to prove the \lq super-replication property\rq, stated as follows.
\begin{theorem}\label{th:hob-super}
  Under the model assumption of Theorem~\ref{th:hob-conv}, assume also that option prices from the model are of class $\C^{1,2}([0,T]\times\R)$ (e.g. $\s>0$ and H\"older continuous). If the model volatility $\s$ dominates the true volatility $\hat\s$, i.e. $\s(t,s)\geq\hat\s(t,s)$ for all $t\in[0,T]$, $s\in\R$, and if the payoff function is convex, then pricing and hedging according to the model will super-replicate the option payout.
\end{theorem}

In order to prove that the model price dominates the true price, the portfolio value process, in particular the stochastic integral $\int_0^\cdot\partial_{s}v(u,S(u))\ud S(u)$, has to be a martingale. In case of a payoff function with bounded derivative, this is achieved by assuming that $\EE\left[\lf\int_0^T S^2(u)\s^2(u,S(u))\ud u\rg^{\frac12}\right]<\infty$, which makes $S$ itself a true martingale, even if not necessarily square-integrable.

\subsection{Robust hedging of discretely monitored options}
\label{sec:ss}

More recently, \citet{ss} revisited the notion of robustness by considering the performance of a model-based hedging strategy when applied to the realized observed path of the underlying asset price, rather than to some supposedly \lq true\rq\ model, inspired by the \follmer's pathwise \ito\ calculus.
\citet{ss} studied the performance of  delta hedging strategies for a path-dependent discretely monitored derivative, obtained under a local volatility model.

The stock price process $S$ is assumed to follow a local volatility model where the volatility process is a deterministic function of time and the current stock price,
\begin{equation}\label{eq:ss-dS}
  \ud S(t)=S(t)\s(t,S(t))\ud W(t),
\end{equation}
where the local volatility function is assumed to satisfy the following regularity conditions.
\begin{assumption}\label{ass:ss}
  \begin{itemize}\item[]
  \item $\s\in\C^1([0,T]\times\R_+,\R_+)$, bounded above and below away from 0;
  \item $s\mapsto s\s(t,s)$ Lipschitz continuous, uniformly in $t\in[0,T]$.
  \end{itemize}
\end{assumption}
The derivatives considered here have a path-dependent claim of the form $H(S)=h(S(t_1),\ldots,S(t_n))$, where $0=t_0<t_1<\ldots t_n\leq T$ and $h:[0,\infty)^n\rightarrow[0,\infty)$ is continuous and satisfies $h(x)\leq C(1+|x|^p)$ for all $x\in[0,\infty)^n$ and certain $C,p\geq0$, in which case $h$ is referred to as a \emph{payoff function}.

Using the Markov property, the price at time $t\in[t_k,t_{k+1})$ is given by
{\setlength\arraycolsep{2pt}  
\begin{eqnarray}
\nonumber  v(t,s_1,\ldots,s_k,s)&=&\EE[H(S)\mid S(t_1)=s_1,\ldots,S(t_k)=s_k,S(t)=s]\\
\label{eq:ss-v} &=&\EE[h(s_1,\ldots,s_k,S(t_{k+1}),\ldots,S(t_n))\mid S(t)=s]
\end{eqnarray}}
We denote $\ds v(t,x):=\sum_{k=1}^n\ind_{[t_k,t_{k+1})}(t) v(t,s_1,\ldots,s_k,s)$, where $x\in\C([0,T],\R_+)$ is a deterministic function matching the observed stock price path, i.e. $x(t_1)=s_1,\ldots,x(t_k)=s_k,x(t)=s$. It is also assumed that all observed price paths are continuous and have finite quadratic variation along a fixed sequence of time partitions $\{\pi^n\}_{n\geq1}$, $\pi^n=(t_i^n)_{i=0,\ldots,m(n)}$, $0=t_0^n<\ldots<t_{m(n)}^n=T$ for all $n\geq1$, with mesh going to 0.
The following result shows the regularity of the value function and makes use of the \follmer's pathwise calculus presented in \chap{pfc}.
\begin{proposition}\label{prop:hob-pde}
  Let $h$ be a payoff function. Under Assumption~\ref{ass:ss}, the map $(t,s)\mapsto v(t,x)$ belongs to
$\ds \C^{1,2}\Big({\setlength{\extrarowheight}{-0.7cm}\bea{c}\scriptstyle n-1\\\cup\\\scriptstyle k=0\ea} (t_k,t_{k+1})\times[0,\infty)\Big)\cap\C([0,T]\times[0,\infty))$ and satisfies the partial differential equation
  \begin{equation}\label{eq:ss-pde}
    \partial_{t}v(t,x)+\frac12\s^2(t,s)s^2\partial_{ss}v(t,x)=0,\quad t\in{\setlength{\extrarowheight}{-0.7cm}\bea{c}\scriptstyle n-1\\\cup\\\scriptstyle k=0\ea}(t_k,t_{k+1}),\,s\in[0,\infty).
  \end{equation}
Furthermore, the \follmer\ integral $\int_0^T\partial_{s}v(t,x)\ud x(t)$ is well defined and the pathwise \ito\ formula holds:
$$v(T,x)=v(0,x)+\int_0^T\partial_{s}v(t,x)\ud x(t)+\frac12\int_0^T\partial_{ss}v(t,x)\ud\pqv{x}(t)+\int_0^T\partial_{t}v(t,x)\ud t.$$
\end{proposition}
The regularity and the PDE characterization of the value function are proven by backward induction and using the following standard result for a European non-path-dependent option with payoff $h:[0,\infty)\rightarrow\R_+ $, that is: let $v(t,s):=\EE[h(S(T))\mid S(t)=s]$, then $v\in\C^{1,2}([0,T]\times(0,\infty))\cap\C([0,T]\times[0,\infty))$, satisfies a polynomial growth condition in $s$ uniformly in $t\in[0,T]$ and solves the Cauchy problem \eq{ss-pde} on $[0,T]\times(0,\infty)$. So, at step 1, let $t\in[t_{n-1},t_n)$, the problem reduces to the standard case. Then, at each step $k>1$, let $t\in[t_{n-k},t_{n-k+1})$,  define the auxiliary function
$$h_k(s)=\EE[h(s_1,\ldots,s_{n-k},s,S(t_{n-k+2}),\ldots,S(t_n))\mid S(t_{n-k+1})=s],$$ which is a payoff function such that $v(t,s_1,\ldots,s_{n-k},s)=\EE[h_k(S(t_{n-k+1}))\mid S(t)=s]$ and again the standard result applies.

Using the same notation above for $x$ and $H$, \citeauthor{ss} defined the delta-hedging strategy for $H$ obtained from the model~\eq{ss-dS} to be \textit{robust} if, when the model volatility \textit{overestimates} the market volatility, i.e. $\int_r^t\s^2(u,x(u))x^2(u)\ud u\geq \pqv{x}(t)-\pqv{x}(r)$ for all $0\leq r<t\leq T$, or equivalently $\s(t,x(t))\geq\sqrt{\z(t)}$, where $\pqv{x}(t)=\int_0^t\z(u)x^2(u)\ud u$ and $\z\geq0$, for Lebesgue-almost every $t\in[0,T]$, then
  \begin{equation}
    \label{eq:super}
    v(0,x)+\int_{0}^T\partial_{s}v(u,x) \ud^\Pi x(u) \geq H(x).
  \end{equation}

They pointed out that, under the assumptions of Proposition~\ref{prop:hob-pde}, the positivity of the option Gamma leads to a robust delta-hedging strategy.
An application of this first basic result is the generalized Black-Scholes model, where the value function of any convex payoff function is again convex and hence the corresponding delta hedge is robust. This follows directly from the fact that a geometric Brownian motion with time-dependent volatility is affine in its starting point and convexity is invariant under affine transformations.

However, in a general local volatility model, convexity of a payoff function does not guarantee the robustness property. Indeed, the main theorem in \cite{ss} spots sufficient conditions on the payoff function resulting in convexity for the value function and consequent robustness for the delta hedge.
\begin{theorem}
  If the payoff function $h$ is \emph{directionally convex}, i.e. for all $i=1,\ldots,n$ the map $x_i\mapsto h(x_1,\ldots,x_i,\ldots,x_n)$ is convex and has increasing right-derivative  with respect to any other component $j=1,\ldots,n$, then, for all $k=1,\ldots,n$ and for any $t\in[t_k,t_{k+1})$, the value function $(s_1,\ldots,s_k,s)\mapsto v(t,s_1,\ldots,s_k,s)$ is also directionally convex and hence convex in the last variable, and the delta-hedging strategy is robust.
\end{theorem}
The crucial step in the proof of the above theorem is the inherited directional convexity of a map of the form
$$u(s_1,\ldots,s_n)=\EE[h(s_1,\ldots,s_{n-1},S(T))\mid S(t)=s_n],$$
which is proven by means of the notion of Wright convexity.
Furthermore, given a directionally convex function of $k+1$ arguments $u(s_1,\ldots,s_{k+1})$, the contraction $\tilde u(s_1,\ldots,s_{k})=u(s_1,\ldots,s_k,s_k)$ is also directionally convex. By this remark, the proof ends by induction on $k=0,\ldots,n$, noticing that for $t\in[t_{n-k},t_{n-k+1})$ the value function can be written as
$$v(t,s_1,\ldots,s_{n-k},s)=\EE[v(t_{n-k+1}s_1,\ldots,s_{n-k},S(t_{n-k+1}),S(t_{n-k+1}))\mid S(t)=s].$$

A counter-example consisting of a local volatility model where the delta hedge fails to be robust in case of any convex payoff which is not identically linear and is positively homogeneous, implies that every payoff function that is both positively homogeneous and directionally convex must be linear.

The results obtained in \cite{ss} in the context of robustness of hedging strategies are specific to one-dimensional local volatility models. In more general models, the issue of propagation of convexity is quite intricate: in multivariate local volatility models, the convexity of prices of European options depends on the volatility matrix and value functions of European call options may fail to be convex.

\section{Robustness and the hedging error formula}
\label{sec:path-robust}
In this thesis, we consider the following problem: a market participant sells a path-dependent derivative with maturity $T$ and payoff functional $H$ and uses a model of preference to compute the price of such derivative and the corresponding hedging strategy.

This situation is typical of financial institutions issuing derivatives and subject to risk management constraints.
The behavior of the underlying asset during the lifetime of the derivative may or may not correspond to a typical trajectory of the model used by the issuer for constructing the hedging strategy. More importantly, the hedger only experiences a single path for the underlying so it is not even clear what it means to assess whether the model correctly describes the risk along this path. The relevant question for the hedger is to assess, ex-post, the performance of the hedging strategy in the realized scenario and to quantify, ex-ante, the magnitude of possible losses across different plausible risk scenarios. This calls for a scenario analysis --or pathwise analysis-- of the performance of such hedging strategies. In fact such scenario analysis, or stress testing, of hedging strategies are routinely performed in financial institutions using simulation methods, but a theoretical framework for such a pathwise analysis was missing.

In the general case where either the payoff or the volatility are path-dependent, the value at time t of the claim will be a non-anticipative functional of the path of the underlying asset.

In this chapter, we keep to the one-dimensional case and we work on the canonical space of continuous paths $(\O,\F,\FF)$, where $\O:=C([0,T],\R_+)$, $\F$ is the Borel sigma-field and $\FF=\Ft$ is the natural filtration of the coordinate process $S$, given by $S(u,\w)=\w(u)$ for all $\w\in\O$, $t\in[0,T]$.
The coordinate process $S$ represents the asset price process and we assume that the hedger's model consists in a square-integrable martingale measure for $S$:
\begin{assumption}\label{ass:S}
  The market participant prices and hedges derivative instruments assuming that the underlying asset price $S$ evolves according to $\ud S(t)=\s(t) S(t) \ud W(t)$, i.e.
\begin{equation}
  \label{eq:S}
  S(t)=S(0)e^{\int_0^t\s(u)\ud W(u)-\frac12\int_0^t\s(u)^2\ud u},\,t\in[0,T],
\end{equation}
where $W$ is a standard Brownian motion on $(\O,\F,\FF,\PP)$ and the volatility $\s$ is a non-negative $\FF$-adapted process such that 
$S$ is a square-integrable $\PP$-martingale.
\end{assumption}
This assumption includes the majority of models commonly used for pricing and hedging derivatives. The assumption of square-integrability is not essential and may be removed by localization arguments but we will retain it to simplify some arguments. Note that this is an assumption on the pricing model used by the hedger, not an assumption on the evolution of the underlying asset itself. We will not make any assumption on the process generating the dynamics of the underlying asset.

\begin{assumption}\label{ass:H}
  Let $H:D([0,T],\R)\mapsto\R$ be the payoff of a path-dependent derivative with maturity $T$, such that $\EE^{\PP}[|H(S_T)|^2]<\infty$.
\end{assumption}

Under Assumptions \ref{ass:S} and \ref{ass:H}, the replicating portfolio for $H$ is given by the delta-hedging strategy $(Y(0),\nabla_SY)$ and its value process coincides with $Y$.

We denote by
\beq  \label{eq:suppS}
  \supp(S,\PP):=\big\{\w\in\O:\;\PP(S_T\in V)>0\;\forall \text{neighborhood $V$ of }\w\text{ in }\lf\O,\norm{\cdot}_\infty\rg\big\},
\eeq
the \emph{topological support of $(S,\PP)$ in $(\O,\norm{\cdot}_\infty)$}, that is the smallest closed set in $(\O,\norm{\cdot}_\infty)$ such that it contains $S_T$ with $\PP$-measure equal to one.
Since $S$ may not have full support in $(\O,\norm{\cdot}_\infty)$, we will need to specifically work on the support of $S$ in order to pass from equations that hold $\PP$-almost surely for functionals of the price process $S$ to pathwise equations for functionals defined on the space of stopped paths.

Throughout this chapter, we consider a fixed sequence of partitions $\Pi=(\pi^n)_{n\geq1}$, $\pi^n=\{0=t^n_0<t^n_1<\ldots,t^n_{m(n)}=T\}$, with mesh going to 0 as $n$ goes to $\infty$.
For paths of absolutely continuous finite quadratic variation along $\Pi$, we define the \emph{local realized volatility} as
$$\s^{\mathrm{mkt}}:[0,T]\times\mathcal A\to\R,\quad (t,\w)\mapsto\s^{\mathrm{mkt}}(t,\w)=\frac1{\w(t)}\sqrt{\frac{\ud}{\ud t}[\w](t)},$$
where
$$\mathcal A:=\{\w\in Q(\O,\Pi),\;t\mapsto[\w](t)\text{ is absolutely continuous}\}.$$

Our main results apply to paths with finite quadratic variation along the given sequence $\Pi$ of partitions, as it is a necessary assumption in the theory of functional pathwise calculus. However, as remarked in Subsection \ref{sec:reasonable}, this assumption is also reasonable in terms of avoiding undesirable strategies that carry infinite gain with bounded initial capital on some paths.

If $Y\in\Cb(S)$, with $Y(t)=F(t,S_t)$ $\ud t\times\ud\PP$-almost surely, the universal hedging equation \eq{univ-hedge} holds and the asset position of the hedger's portfolio at almost any time $t\in[0,T]$ and for $\PP$-almost all scenarios $\w$, is given by $\nabla_SY(t,\w)=\vd F(t,\w)$.
Note that, even if the \naf\ $F:\W_T\mapsto\R$ does depend on the choice of the functional representation $F$ of $Y$ such that $Y(t)=F(t,\w)$ for Lebesgue-almost all $t\in[0,T]$ and $\PP$-almost all $\w$, the process $\nabla_SY(\cdot)=\vd F(\cdot,S_\cdot)$ does not, up to indistinguishable processes.
Moreover, if it also satisfies $F\in\CC^{0,0}(\W_T)$, according to \prop{G} the trading strategy $(F(0,\cdot),\vd F)$ is self-financing on $Q(\O,\Pi)$ and allows a path-by-path computation of the gain from trading as a \follmer\ integral.
 We will therefore restrict to this class of pathwise trading strategies, which are of main interest:
\beq\label{eq:nabla}
\VV:=\{\vd F,\quad F\in\Cloc(\W_T)\cap\CC^{0,0}(\W_T)\}.
\eeq
Note that $\VV$ has a natural structure of vector space; we call its elements \emph{vertical 1-forms}.

In line with \rmk{path-sf}, the portfolio value of a self-financing trading strategy $(V_0,\phi)$ with asset position a vertical 1-form $\phi=\vd F$ and initial investment $V_0=F(0,\cdot)$ will be given by, at any time $t\in[0,T]$ and in any scenario $\w\in Q(\O,\Pi)$,
$$V(t,\w)=F(0,\w)+\int_{0}^t\vd F(u,\w)\ud^\Pi\w(u).$$
The portfolio value functional $V(T,\cdot)$ at the maturity date can be different from the payoff $H$ with strictly positive $\PP$-measure. What is important about this mis-replication is the sign of the difference between the portfolio value at maturity and the payoff in a given scenario. 

By the arguments above and recalling \defin{hedging_error}, we remark that
 the hedging error of a trading strategy $(V_0,\phi)$, with $\phi\in\VV$, for a derivative with payoff $H$ and in a scenario $\w\in Q(\O,\Pi)$, is given by
$$V(T,\w)-H(\w_T)=V_0(\w)+\int_{0}^T\phi(u,\w)\ud^\Pi\w(u)-H(\w_T).$$
Moreover, $(V_0,\phi)$ is a super-strategy for $H$ on $U\subset Q(\O,\Pi)$ if 
$$V_0(\w)+\int_{0}^T\phi(u,\w) \ud^\Pi\w(u) \geq H(\w_T)\quad\forall\w\in U.$$

\begin{definition}\label{def:rob}
  Given $F\in\Cloc(\W_T)\cap\CC^{0,0}$ such that $Y(t)=F(t,S_t)$ $\ud t\times\ud\PP$-almost surely, the delta-hedging strategy $(Y(0),\nabla_S Y)$ for $H$ is said to be \emph{robust} on $U\subset Q(\O,\Pi)$ if $(F(0,\cdot),\vd F)$ is a \emph{super-strategy} for $H$ on $U$.
\end{definition}

\begin{proposition}[Pathwise hedging error formula]\label{prop:robust}
  If there exists a \naf\ $F:\L_T\to\R$ such that
  \begin{align}
&F\in\Cb(\W_T)\cap\CC^{0,0}(\W_T),\quad \hd F\in\CC^{0,0}_l(\W_T),&\label{eq:regF}\\
&F(t,S_t)=\EE^{\PP}[H(S_T)|\F_t]\quad \ud t\times \ud\PP\text{-a.s.}    \label{eq:valueF}
  \end{align}
then, the hedging error of the delta hedge $(F(0,\cdot),\vd F)$ along any path $\w\in Q(\O,\Pi)\cap\supp(S,\PP)$
is explicitly given by
 \begin{align*}
 &\!\!\!\!\!\!V_0(\w)+\int_{0}^T\vd F(u,\w) \ud^\Pi\w(u)-H(\w_T)\\
 ={}&\frac12\int_{0}^T\s(t,\w)^2\w^2(t)\vd^2F(t,\w)\ud t-\frac12\int_{0}^T\vd^2F(t,\w)\ud[\w](t).
\end{align*}
In particular, if $\w\in\A\cap\supp(S,\PP)$, then
  \begin{align}
 &\!\!\!\!\!\!V_0(\w)+\int_{0}^T\vd F(u,\w) \ud^\Pi\w(u)-H(\w_T)\nonumber\\
 ={}& \frac12\int_{0}^T \lf\s(t,\w)^2-\s^{\mathrm{mkt}}(t,\w)^2\rg\w^2(t)\vd^2F(t,\w) \ud t. \label{eq:tr_err}
  \end{align}
Furthermore, if for all $\w\in U\subset(\A\cap\supp(S,\PP))$ and Lebesgue-almost every $t\in[0,T)$,
\beq\label{eq:supervol}
\vd^2F(t,\w)\geq0\text{ (resp. $\leq$),\quad and \quad}\s(t,\w)\geq\s^{\mathrm{mkt}}(t,\w)\text{ (resp.$\leq$)},
\eeq
then the  delta hedge for $H$ is robust on $U$.
\end{proposition}
\proof
Assumptions \eq{regF}-\eq{valueF} imply $Y\in\Cb(S)$, with $Y(t)=F(t,S_t)$ $\ud t\times\ud\PP$-almost surely, thus $F(\cdot,S_\cdot)$ satisfies the functional \ito\ formula for functionals of continuous semimartingales \eq{fif-csm}.
Moreover, by \prop{universalprice}, the universal pricing equation holds: for all $\w\in\supp(S,\PP)$,
\beq\label{eq:fpde}
\hd F(t,\w)+\frac12\vd^2F(t,\w)\s^2(t,\w)\w^2(t)=0\quad \forall t\in[0,T)
\eeq
By \prop{G} and using the pathwise change of variable formula for functionals of continuous paths (\thm{fif-c}), the value of the hedger's portfolio at maturity is given by, for all $\w\in Q(\O,\Pi)$,
\begin{align}
V(T,\w)={}&F(0,\w_{0})+\int_{0}^T\vd F(t,\w) \ud^\Pi\w(t)\nonumber\\
={}& H-\int_{0}^T\hd F(t,\w)\ud t-\frac12\int_{0}^T\vd^2F(t,\w) \ud[\w](t).\label{eq:V}
\end{align}
Then, using the equations \eq{V} and \eq{fpde}, we get an explicit expression for the hedging error along any path $\w$ in $\A\cap\supp(S,\PP)$ as
\begin{align*}
V(T,\w)-H ={}& \int_{0}^T\lf\frac12{\s}^2(u,\w)\w^2(u)\vd^2F(u,\w)-\frac12\vd^2F(u,\w) \ud[\w](t)\rg \ud u \\
 &{} -\int_{0}^T\hd F(u,\w)\ud t - \int_{0}^T\frac12{\s}(u,\w)^2\w^2(u) \vd^2F(u,\w)\ud u \\
={}& \frac12\int_{0}^T \lf{\s}(u,\w)^2-\s^{\mathrm{mkt}}(t,\w)^2\rg\w^2(u)\vd^2F(u,\w) \ud u.
\end{align*}
Moreover, the inequalities \eq{supervol} imply that, for all $\w\in U$,
\begin{align*}
V(T,\w)\geq{}& H-\int_{0}^T\hd F(t,\w)\ud t-\frac12\int_{0}^T\s(t,\w)^2\w^2(t)\vd^2F(t,\w)\ud t\\
={}&H.
\end{align*}
This proves the robustness of the delta hedge on $U$.
\endproof
\begin{remark}
  \prop{robust} simply requires the price trajectory to have an absolutely continuous quadratic variation in a pathwise sense, but does not assume any specific probabilistic model. Nevertheless, it applies to any model whose sample paths fulfill these properties almost-surely: this applies in particular to diffusion models and other models based on continuous semimartingales analyzed in \cite{avlevyparas,bergman,elkaroui,hobson}. However, note that we do not even require the price process to be a semimartingale. For example, our results also hold when the price paths are generated by a (functional of a) fractional Brownian motion with index $H\geq\frac12$.
\end{remark}

\section{The impact of jumps}
\label{sec:jumps}

The presence of jumps in the price trajectory affects the hedging error of the delta-hedging strategy in an unfavorable way.

\begin{proposition}[Impact of jumps on delta hedging]\label{prop:jumps}
  If there exists a \naf\ $F:\L_T\to\R$ such that
  \begin{align*}
&F\in\Cb(\L_T)\cap\CC^{0,0}(\L_T),\quad \vd F\in\CC^{0,0}(\L_T),\quad \hd F\in\CC^{0,0}_l(\W_T)\\
&F(t,S_t)=\EE^{\PP}[H(S_T)|\F_t]\quad \ud t\times \ud\PP\text{-a.s.}    
  \end{align*}
then, for any $\w\in Q(D([0,T],\mathbb{R}_+),\Pi)$ such that $[\w]^c$ is absolutely continuous, the hedging error of the delta hedge $(F(0,\cdot),\vd F)$ for $H$ is explicitly given by
  \begin{align} \label{eq:tr_err-jumps}
     &\frac12\int_{0}^T \lf\s(t,\w)^2-\s^{\mathrm{mkt}}(t,\w)^2\rg\w^2(t)\vd^2F(t,\w) \ud t\\
&-\sum_{t\in(0,T]}\lf F(t,\w_t)-F(t,\w_{t-})-\vd F(t,\w_{t-})\cdot\De\w(t)\rg.
  \end{align}
\end{proposition}
\proof
We follow the same steps as in the proof of \prop{robust}, with the appropriate modifications.
The universal pricing equation holds on the support of $S$, that is, for all $\w\in\supp(S,\PP)$,
$$\hd F(t,\w)+\frac12\vd^2F(t,\w)\s^2(t,\w)\w^2(t)=0\text{ for Lebesgue-a.e. }t\in[0,T).$$
By \prop{G-cadlag} and using the pathwise change of variable formula for functionals of \cadlag\ paths (\thm{fif-d}), the value of the hedger's portfolio at maturity in the scenario $\w$ is given by
\begin{align}
V(T,\w)={}&F(0,\w_{0})+\int_{0}^T\vd F(t,\w) \ud^\Pi\w(t)\nonumber\\
={}& H-\int_{0}^T\hd F(t,\w)\ud t-\frac12\int_{0}^T\vd^2F(t,\w) \ud[\w]^c(t)\label{eq:V1}\\
&-\sum_{t\in(0,T]}\lf F(t,\w_t)-F(t,\w_{t-})-\vd F(t,\w_{t-})\cdot\De\w(t)\rg.\label{eq:V2}
\end{align}
Then, using the equations \eq{V1}, \eq{V2} and \eq{fpde}, we get an explicit expression for the hedging error in the scenario $\w$:
\begin{align*}
V(T,\w)-H ={}& \frac12\int_{0}^T \lf{\s}(u,\w)^2-\s^{\mathrm{mkt}}(u,\w)^2\rg\w^2(u)\vd^2F(u,\w) \ud u\\
&-\sum_{t\in(0,T]}\lf F(t,\w_t)-F(t,\w_{t-})-\vd F(t,\w_{t-})\De\w(t)\rg.
\end{align*}
\endproof

\begin{remark}
Using a Taylor expansion of $e\mapsto F(t,\w_{t-}+e\ind_{[t,T]})$, we can rewrite the hedging error as
\begin{align*}
V(T,\w)-H ={}& \frac12\int_{0}^T \lf{\s}(u,\w)^2-\s^{\mathrm{mkt}}(u,\w)^2\rg\w^2(u)\vd^2F(u,\w) \ud u\\
&{}-\frac12\sum_{t\in(0,T]} \vd^2F(t,\w_{t-}+\x\ind_{[t,T]})\De\w(t)^2,
\end{align*}
for an appropriate $\xi\in B(0,\abs{\De\w(t)})$. This shows that the exposure to jump risk is quantified by the Gamma of the option computed in a \lq jump scenario\rq, i.e. along a vertical perturbation of the original path.
\end{remark}

\section{Regularity of pricing functionals}
\label{sec:exist}

\prop{robust} requires some regularity on the pricing functional $F$, which is in general defined as a conditional expectation, therefore it is not obvious to verify such regularities for $F$ on the space of stopped paths.
In \prop{exist}, we give sufficient conditions on the payoff functional which lead to a \emph{vertically smooth} pricing functional.

\begin{definition}\label{def:vsmooth}
  A functional $h:D([0,T],\R)\mapsto\R$ is said to be \emph{vertically smooth on $U\subset D([0,T],\R)$} if $\forall(t,w)\in[0,T]\times U$ the map
\begin{eqnarray*}g^h(\cdot;t,\w):\R&\to&\R,
\\ e&\mapsto &h\lf\w+e\ind_{[t,T]}\rg\end{eqnarray*}
is twice continuously differentiable on a neighborhood $V$ and such that there exist $K,c,\b>0$ such that, for all $\w,\w'\in U$, $t,t'\in[0,T]$,
$$\abs{\partial_ e g^h(e;t,\w)}+\abs{\partial_{ee}g^h(e;t,\w)}\leq K,\quad e\in V,$$
and
\beq\label{eq:gh-lip}
\bea{c}
\abs{\partial_e g^h(0;t,\w)-\partial_e g^h(0;t',\w')}+\abs{\partial_e^2 g^h(0;t,\w)-\partial_e^2 g^h(0;t',\w')}\\
\leq c(\norm{\w-\w'}_\infty+\abs{t-t'}^\b).
\ea
\eeq
\end{definition}

We define, for all $t\in[0,T]$, the concatenation operator $\conc{t}$ as
$$\bea{rl}
\conc{t}:&D([0,T],\R)\times D([0,T],\R)\rightarrow D([0,T],\R),\\
&(\w,\w')\mapsto\w\underset{t}{\oplus}\w'=\w\ind_{[0,t)}+\w'\ind_{[t,T]}.
\ea$$
This will appear in the proof of Propositions \ref{prop:exist} and \ref{prop:convex}.

The following result shows how to construct a (vertically) smooth version of the conditional expectation that gives the price of a path-dependent contingent claim.
\begin{proposition}\label{prop:exist}
 Let $H:(D([0,T],\R),\norm{\cdot}_\infty)\mapsto\R$ a locally-Lipschitz payoff functional such that $\EE^{\PP}[|H(S_T)|]<\infty$ and define $h:(D([0,T],\R)\to\R$ by $h(\w_T)=H(\exp\w_T)$, where $\exp\w_T(t):=e^{\w(t)}$ for all $t\in[0,T]$. If $h$ is vertically smooth on $\C([0,T],\R_+)$ in the sense of \defin{vsmooth}, then
\begin{equation}
  \label{eq:F02}
\exists F\in\CC^{0,2}_b(\W_T)\cap\CC^{0,0}(\W_T),\quad F(t,S_t)=\EE^{\PP}[H(S_T)|\F_t]\quad \ud t\times \ud\PP\text{-a.s.}
\end{equation}
\end{proposition}

\proof
The first step is to construct analytically a regular \naf\ representation $F:\L_T\mapsto\R$ of the claim price, then the properties of regularity and vertical smoothness of $F$ will follow from the conditions of the payoff $H$.

By Theorem 1.3.4 in \cite{str-var} on the existence of regular conditional distributions, for any $t\in[0,T]$ there exists a regular conditional distribution $\{\PP^{(t,\w)},\,\w\in\O\}$ of $\PP$ given the (countably generated) sub-$\s$-algebra $\F_t\subset\F$, i.e. a family of probability measures $\PP^{(t,\w)}$ on $(\O,\F)$ such that
\begin{enumerate}
\item $\forall B\in\F$, the map $\O\ni\w\mapsto\PP^{(t,\w)}(B)\in[0,1]$ is $\F_t$-measurable;
\item $\forall A\in\F_t,\forall B\in\F$, $\PP(A\cap B)=\int_A\PP^{(t,\w)}(B)\PP(\ud\w)$;
\item $\forall A\in\F_t, \forall\w\in\O$, $\PP^{(t,\w)}(A)=\ind_A(\w)$.
\end{enumerate}
Moreover, for any random variable $Z\in L^1(\O,\F,\PP)$, it holds
$$\EE^{\PP^{(t,\w)}}[|Z|]<\infty\text{ and }\EE^{\PP}\left[Z|\F_t\right](\w)=\EE^{\PP^{(t,\w)}}[Z]\text{ for }\PP\text{-almost all }\w\in\O.$$
By taking $Z=H(S_T)$, since $\PP^{(t,\w)}$ is concentrated on the subspace $\O^{(t,\w)}:=\{\w'\in\O:\w'_t=\w_t\}$, we can rewrite $\EE^{\PP^{(t,\w)}}[H(S_T)]=\EE^{\PP^{(t,\w)}}[H(\w\underset{t}{\oplus} S_T)]$.

For any $t\in[0,T],x>0$, we denote $\PP^{(t,x)}$ the law of the stochastic process $x\ind_{[0,t)}+S^{(t,x)}\ind_{[t,T]}$ on $(\O,\F,\PP)$, where $\{S^{(t,x)}(u)\}_{u\in[t,T]}$ is defined by
\begin{equation}
  \label{eq:Seps}
  S^{(t,x)}(u)= x+\int_t^u\s(r)S^{(t,x)}(r)\ud W(r),\quad u\in[t,T].
\end{equation}
Note that $S$ has the same law under $\PP^{(t,x+\eps)}$ that $S\lf1+\frac\eps x\rg$ has under $\PP^{(t,x)}$.
Indeed:
\begin{align*}
S^{(t,x+\eps)}={}&\lf x+\eps+\int_t^\cdot\s(u)S^{(t,x+\eps)}(u)\ud W(u)\rg\ind_{[t,T]}  \\
={}&(x+\eps)e^{\int_t^\cdot\s(s)\ud W(s)-\frac12\int_t^\cdot\s^2(u)\ud u}\ind_{[t,T]}\\
={}&S^{(t,x)}\lf1+\frac\eps x\rg,
\end{align*}
hence we have the following identities in law
\begin{align*}
\mathrm{Law}(S,\PP^{(t,x+\eps)})={}&\mathrm{Law}\lf(x+\eps)\ind_{[0,t)}+S^{(t,x+\eps)}\ind_{[t,T]},\PP\rg\\
={}&\mathrm{Law}\lf\lf x\ind_{[0,t)}+S^{(t,x)}\ind_{[t,T]}\rg\lf1+\frac\eps x\rg,\PP\rg\\
={}&\mathrm{Law}\lf S\lf1+\frac\eps x\rg,\PP^{(t,x)}\rg.
\end{align*}
Then, consider the \naf\ $F:\L_T\rightarrow\R$ defined by, for all $(t,\w)\in\L_T$,
\begin{align}
F(t,\w)={}&\EE^{\PP^{(t,\w(t))}}\left[H\lf\w\conc{t}S_T\rg\right]  \label{eq:Fw}\\
={}&\EE^{\PP}\left[H\lf \w\conc{t}\w(t)e^{\int_t^\cdot\s(s)\ud W(s)-\frac12\int_t^\cdot\s^2(u)\ud u}\ind_{[t,T]} \rg\right]. \nonumber
\end{align}
If computed respectively on a continuous stopped path $(t,\w)\in\W_T$ and on its vertical perturbation in $t$ of size $\eps$, it gives
$$F(t,\w)=\EE^{\PP^{(t,\w)}}\left[H\lf \w\conc{t}S_T \rg\right]=\EE^{\PP}\left[H(S_T)|\F_t\right](\w)\quad \PP\text{-a.s.},$$
$$F(t,\w^\eps_t)=\EE^{\PP^{(t,\w(t)+\eps)}}\left[H\lf \w\conc{t}S_T \rg\right]=\EE^{\PP^{(t,\w)}}\left[H\lf\w\conc{t}\lf S_T\lf1+\frac\eps{\w(t)}\rg\rg\rg\right].$$

Since $H$ is locally Lipschitz continuous, given $(t,\w)\in[0,T]\times C([0,T],\R_+)$, there exist $\y=\y(\w)>0$ and $K_\w\geq 0$ such that 
$$\|\w -\w'\|_\infty \leq\y(\w) \quad \Rightarrow\quad |H(\omega)-H(\omega')| \leq K_\omega \|\w -\w'\|_\infty.$$

Now, we prove the joint-continuity, by showing the computation for the right side - the other being analogous because of symmetric properties; this also proves continuity at fixed times.
So, given $(t,\w)\in\W_T$, for $t'\in[t,T]$, $(t',\w')\in\W_T$ such that $\dinf((t,\w),(t',\w'))\leq\y$, then:
\begin{align*}
&\!\!\!\!\!\!\!\! \abs{F(t,\w)-F(t',\w')}=\\
={}&\abs{\EE^{\PP^{(t,\w)}}\left[H\lf\w\conc{t}S_T\rg\right]-\EE^{\PP^{(t',\w')}}\left[H\lf\w'\conc{t'}S_T\rg\right]} \\
={}& \EE^{\PP}\left[\left\lvert H\lf\w\ind_{[0,t)}+\w(t)e^{\int_{t}^{\cdot}\s(u)\ud W(u)-\frac12\int_{t}^{\cdot}\s^2(u)\ud u}\ind_{[t,T]}\rg \right.\right.\\
 &\left.\left.\quad\quad\quad\quad -H\lf\w'\ind_{[0,t')}+\w'(t')e^{\int_{t'}^{\cdot}\s(u)\ud W(u)-\frac12\int_{t'}^{\cdot}\s^2(u)\ud u}\ind_{[t',T]}\rg \right\rvert\right]\\
\leq{}&\,K_\w\,\EE^{\PP}\left[\norm{(\w-\w')\ind_{[0,t)}}_\infty \right. 
+\norm{\big(\w(t)e^{\int_t^{\cdot}\s(u)\ud W(u)-\frac12\int_t^{\cdot}\s^2(u)\ud u}-\w'\big)\ind_{[t,t')}}_\infty \\
&\,\left.+\norm{\big(\w(t)e^{\int_t^{\cdot}\s(u)\ud W(u)-\frac12\int_t^{\cdot}\s^2(u)\ud u}-\w'(t')e^{\int_{t'}^{\cdot}\s(u)\ud W(u)-\frac12\int_{t'}^{\cdot}\s^2(u)\ud u}\big)\ind_{[t',T]}}_\infty\right] \\
\leq{}&\,K_\w \lf\y+|\w(t)|\EE^{\PP}\left[\norm{\big(e^{\int_t^{\cdot}\s(u)\ud W(u)-\frac12\int_t^{\cdot}\s^2(u)\ud u}-1\big)\ind_{[t,t')}}_\infty\right]+\y \right.\\
&{}+|\w(t)|\EE^{\PP}\left[\norm{e^{\int_{t'}^{\cdot}\s(u)\ud W(u)-\frac12\int_{t'}^{\cdot}\s^2(u)\ud u}\ind_{[t',T)}}_\infty\abs{e^{\int_t^{t'}\s(u)\ud W(u)-\frac12\int_t^{t'}\s^2(u)\ud u}-1}\right] \\
&\left. +\y\EE^{\PP}\left[\norm{e^{\int_{t'}^{\cdot}\s(u)\ud W(u)-\frac12\int_{t'}^{\cdot}\s^2(u)\ud u}\ind_{[t',T)}}_\infty\right] \rg
\end{align*}

\begin{align}
\leq{}&K_\w\left[ 2\y+|\w(t)|\lf\EE^{\PP}\bigg[\sup_{s\in[t,t')}\abs{e^{\int_t^s\s(u)\ud W(u)-\frac12\int_t^s\s^2(u)\ud u}-1}\bigg] \right.\right. \nonumber\\
 &\left.{} +\EE^{\PP}\bigg[\sup_{s\in[t',T)}\abs{e^{\int_{t'}^s\s(u)\ud W(u)-\frac12\int_{t'}^s\s^2(u)\ud u}}\bigg]\EE^{\PP}\left[\abs{e^{\int_t^{t'}\s(u)\ud W(u)-\frac12\int_t^{t'}\s^2(u)\ud u}-1}\right] \rg \nonumber\\
 &\left.{} +\y\EE^{\PP}\bigg[\sup_{s\in[t',T)}\abs{e^{\int_{t'}^s\s(u)\ud W(u)-\frac12\int_{t'}^s\s^2(u)\ud u}}\bigg]\right]\label{eq:jc-1}
\end{align}
The first and third expectations in \eq{jc-1} go to 0 as $t'$ tends to $t$, indeed:

\begin{align*}
0\leq{}&\EE^{\PP}\left[\abs{e^{\int_t^{t'}\s(u)\ud W(u)-\frac12\int_t^{t'}\s^2(u)\ud u}-1}\right]\\
\leq{}&\EE^{\PP}\bigg[\sup_{s\in[t,t')}\abs{e^{\int_t^{\cdot}\s(u)\ud W(u)-\frac12\int_t^{\cdot}\s^2(u)\ud u}-1}\bigg] \\
\leq{}&\EE^{\PP}\bigg[\sup_{s\in[t,t')}\abs{e^{\int_t^{\cdot}\s(u)\ud W(u)-\frac12\int_t^{\cdot}\s^2(u)\ud u}-1}^2\bigg]^{\frac12},\text{ by H\"older's inequality} \\
\leq{}&2\EE^{\PP}\left[\abs{e^{\int_t^{t'}\s(u)\ud W(u)-\frac12\int_t^{t'}\s^2(u)\ud u}-1}^2\right]^{\frac12}, \text{ by Doob's martingale inequality} \\
={}&2\lf\EE^{\PP}\left[(M(t')-1)^2\right]\rg^{\frac12}\\
={}&2\sqrt{\EE^{\PP}\Big[[M](t')\Big]},
\end{align*}
where $M$ denotes the exponential martingale
$$M(s)=e^{\int_t^s\s(u)\ud W(u)-\frac12\int_t^s\s(u)^2\ud u},\quad s\in[t,T].$$
So, the expectation goes to 0 as $t'$ tends to $t$, by \ass{S}.
On the other hand, the second and fourth expectations in \eq{jc-1} are bounded above, again by H\"older's and Doob's martingale inequalities:
\begin{align*}
\EE^{\PP}\bigg[\sup_{s\in[t',T)}\abs{e^{\int_{t'}^s\s(u)\ud W(u)-\frac12\int_{t'}^s\s^2(u)\ud u}}\bigg]\leq{}&\EE^{\PP}\bigg[\sup_{s\in[t',T)}e^{2\int_{t'}^s\s(u)\ud W(u)-\int_{t'}^s\s^2(u)\ud u}\bigg]^{\frac12}\\
\leq{}&2\EE^{\PP}\left[\lf\frac{M(T)}{M(t')}\rg^2\right]^{\frac12}\\
={}&2\EE^{\PP}\left[\frac{[M](T)}{M(t')}-1\right]^{\frac12},
\end{align*}
which is finite by \ass{S}.

The vertical incremental ratio of F is given by
\begin{eqnarray*}
\frac{F(t,\w^\eps_t)-F(t,\w)}\eps&=&\frac1\eps \EE^{\PP^{(t,\w)}}\left[H\lf \w\conc{t}S_T \lf1+\frac{\eps}{\w(t)}\ind_{[t,T]}\rg \rg - H\lf\w\conc{t}S_T\rg\right]\\
&=&\frac1\eps \EE^{\PP^{(t,\w)}}\left[h\lf\log\lf\frac{\w\conc{t}S_T\lf1+\frac{\eps}{\w(t)}\ind_{[t,T]}\rg}{\w(0)}\rg \rg\right.\\
 &&\left.\phantom{\lf\frac{\lf\frac{\eps}{\w(t)}\rg}{\w(0)}\rg}- h\lf\log\lf\frac{\w\conc{t}S_T}{\w(0)}\rg \rg\right]\\
&=&\frac1\eps \EE^{\PP^{(t,\w)}}\left[h\lf\log\lf\frac{\w\conc{t}S_T}{\w(0)}\rg+\log\lf1+\frac\eps{\w(t)}\rg\ind_{[t,T]} \rg \right.\\
&&\left.\qquad\qquad{}- h\lf\log\lf\frac{\w\conc{t}S_T}{\w(0)}\rg \rg\right].
\end{eqnarray*}
Then, the vertical smoothness of h allows to use a dominated convergence argument to go to the limit for $\eps$ going to 0 inside the expectation. So we get:
\begin{eqnarray*}
  \vd F(t,\w)&=&\frac1{\w(t)}\EE^{\PP^{(t,\w)}}\left[\partial_{e}g^h\lf0;t,\log\lf\frac{\w\conc{t}S_T}{\w(0)}\rg\rg\right],\\
  \vd^2F(t,\w)&=&\frac1{\w(t)^2}\lf\EE^{\PP^{(t,\w)}}\left[\ppa{e}g^h\lf0;t,\log\lf\frac{\w\conc{t}S_T}{\w(0)}\rg\rg\right]\right.\\
&&\left.\qquad-\EE^{\PP^{(t,\w)}}\left[\partial_{e}g^h\lf0;t,\log\lf\frac{\w\conc{t}S_T}{\w(0)}\rg\rg\right]\rg
\end{eqnarray*}
The joint continuity of the first and second-order vertical derivative of $F$ are proved similarly, by means of the H\"older condition \eq{gh-lip}. 
Indeed, if $\dinf((t,\w),(t,\w'))<\eta$, then:

\begin{align}
 &\!\!\!\!\!\!\!\! \abs{\vd F(t,\w)-\vd F(t',\w')}=\nonumber\\
={}&\abs{\frac1{\w(t)}\EE^{\PP^{(t,\w)}}\left[\partial_{e} g^h\lf0;t,\log\lf\frac{\w\conc{t}S_T}{\w(0)}\rg\rg\right]\right. \nonumber \\
&\left.-\frac1{\w'(t')}\EE^{\PP^{(t',\w')}}\left[\partial_{e} g^h\lf0;t',\log\lf\frac{\w'\conc{t'}S_T}{\w'(0)}\rg\rg\right]}  \nonumber\\
={}&\frac{1}{\w(t)\w'(t')}\EE^{\PP}\left[\left|\w'(t')\partial_{e}g^h\lf0;t,\log\lf\frac{\w\ind_{[0,t)}+\w(t)e^{\int_{t}^{\cdot}\s(u)\ud W(u)-\frac12\int_{t}^{\cdot}\s^2(u)\ud u}\ind_{[t,T]}}{\w(0)}\rg\rg\right.\right. \nonumber\\
&\left.\left.-\w(t)\partial_{e}g^h\lf0;t',\log\lf\frac{\w'\ind_{[0,t')}+\w(t')e^{\int_{t'}^{\cdot}\s(u)\ud W(u)-\frac12\int_{t'}^{\cdot}\s^2(u)\ud u}\ind_{[t',T]}}{\w'(0)}\rg\rg\right|\right] \nonumber\\
\leq{}&\frac{1}{\w(t)(\w(t)-\y)}\Bigg\{\EE^{\PP}\left[\y\abs{\partial_{e}g^h\lf0;t,\log\lf\frac{\w\ind_{[0,t)}+\w(t)e^{\int_{t}^{\cdot}\s(u)\ud W(u)-\frac12\int_{t}^{\cdot}\s^2(u)\ud u}\ind_{[t,T]}}{\w(0)}\rg\rg}\right] \nonumber\\
&{}+K|\w(t)|\lf|t'-t|^\b+\norm{\lf\log\frac\w{\w(0)}-\log\frac{\w'}{\w'(0)}\rg\ind_{[0,t)}}_\infty\right.\nonumber\\
&{}+\EE^{\PP}\Bigg[\norm{\lf\log\lf\frac{\w(t)}{\w(0)}e^{\int_t^{\cdot}\s(u)\ud W(u)-\frac12\int_t^{\cdot}\s^2(u)\ud u}\rg-\log\frac{\w'}{\w'(0)}\rg\ind_{[t,t')}}_\infty \nonumber\\
&\left.{}+\left\lVert \lf\log\lf\frac{\w(t)}{\w(0)}e^{\int_t^{\cdot}\s(u)\ud W(u)-\frac12\int_t^{\cdot}\s^2(u)\ud u}\rg-\log\lf\frac{\w'(t')}{\w'(0)}e^{\int_{t'}^{\cdot}\s(u)\ud W(u)-\frac12\int_{t'}^{\cdot}\s^2(u)\ud u}\rg\rg\ind_{[t',T]}\right\rVert_\infty \bigg]\rg\Bigg\} \nonumber\\
\leq{}&\frac{1}{\w(t)(\w(t)-\y)}\Bigg\{\y C_1 +K|\w(t)|\lf|t'-t|^\b+2\y'\right. \label{eq:jc-3}\\
&{}+\EE^{\PP}\left[\norm{\lf\int_t^{\cdot}\s(u)\ud W(u)-\frac12\int_t^{\cdot}\s^2(u)\ud u\rg\ind_{[t,t')}}_\infty\right] \nonumber\\
&{}+\EE^{\PP}\left[\abs{\int_t^{t'}\s(u)\ud W(u)-\frac12\int_t^{t'}\s^2(u)\ud u}\right]\Bigg\} \nonumber\\
\leq{}&K'\lf\y+|t'-t|^\b+2\y'+3\EE^{\PP}\left[\abs{\int_{t}^{t'}\s(u)\ud W(u)}^2\right]^{\frac12}+\bar\s^2(t'-t)\rg \label{eq:jc-4}
\end{align}
The two constants $C_1$ and $\y'$ in \eq{jc-3} come respectively from the uniform bound on $\partial_{e} g^h$ and from the bound of $\norm{\log\frac\w{\w(0)}-\log\frac{\w'}{\w'(0)}}_\infty$, while to obtain \eq{jc-4} we used the H\"older's and Doob's martingale inequalities.
\endproof

\section{Vertical convexity as a condition for robustness}
\label{sec:convex}

The path-dependent analogue of the convexity property that plays a role in the analysis of hedging strategies turns out to be the following.

\begin{definition}\label{def:verticalconvex}
  A \naf\ $G:\L_T\to\R$ is called \emph{vertically convex on $U\subset\L_T$} if, for all $(t,\w)\in U$, there exists a neighborhood $V\subset\R$ of 0 such that the map
$$\bea{rcl}V&\to&\R\\
e&\mapsto&G\lf t,\w+e\ind_{[t,T]}\rg
\ea$$
is convex.
\end{definition}
It is readily observed that if $F\in\CC^{0,2}$ is vertically convex on $U$, then $\vd^2F(t,\w)\geq0$ for all $(t,\w)\in U$.

We now provide a sufficient condition on the payoff functional which ensures that the vertically smooth value functional in \eq{F02} is vertically convex.
\begin{proposition}[Vertical convexity of pricing functionals]\label{prop:convex}
  Assume that, for all $(t,\w)\in\mathbb T\times\supp(S,\PP)$, there exists an interval $\mathcal I\subset\R$, $0\in\mathcal I$, such that the map
  \begin{equation}    \label{eq:gh}
  \bea{rcl}  v^H(\cdot;t,\w):\mathcal I&\to&\R,\\
  e&\mapsto&v^H(e;t,\w)=H\lf\w(1+e\ind_{[t,T]})\rg
  \ea
  \end{equation}
is convex.
If the value functional $F$ defined in \eq{Fw} is of class $\CC^{0,2}(\W_T)$, then it is vertically convex on $\mathbb T\times\supp(S,\PP)$. In particular:
\beq\label{eq:vd2F}
\forall(t,\w)\in\mathbb T\times\supp(S,\PP),\quad \vd^2F(t,\w)\geq0.
\eeq
\end{proposition}
\proof
We only need to show that convexity of the map in \eq{gh} is inherited by the map $e\mapsto F(t,\w_t^e)$, which is also twice differentiable in 0 by assumption, hence \eq{vd2F} follows.
A simple way of proving convexity of a continuous function is through the property of Wright-convexity, introduced by \citet{wright} in 1954.
Precisely, we want to prove that for every $(t,\w)\in\mathbb T\times\supp(S,\PP)$, for all $\eps,e>0$ such that $\frac{e}{\w(t)},\frac{e+\eps}{\w(t)}\in\mathcal I$, the map
$$\I'\to\R,\quad e\mapsto F(t,\w^{e+\eps}_t)-F(t,\w^e_t)$$ is increasing:
\begin{align*}
F(t,\w^{e+\eps}_t)-F(t,\w^e_t)={}&\EE^{\PP^{(t,\w)}}\left[H\lf \lf\w\conc{t}S_T\rg\lf1+\frac{e+\eps}{\w(t)}\ind_{[t,T]}\rg \rg \right.\\
&\left.\quad\quad\quad- H\lf \lf\w\conc{t}S_T\rg\lf1+\frac{e}{\w(t)}\ind_{[t,T]}\rg \rg\right]\\
={}&\EE^{\PP^{(t,\w)}}\left[v^H\lf\frac{e+\eps}{\w(t)};t,\w\conc{t}S_T\rg-v^H\lf\frac{e}{\w(t)};t,\w\conc{t}S_T\rg\right].
\end{align*}
Since $v^H(\cdot;t,\w)$ is continuous and convex, hence Wright-convex, on $\I$, the random variable inside the expectation is pathwise increasing in $e$. Hence also $\mathcal I'\ni e\mapsto F(t,\w_t^e)$ is Wright-convex, where $\I':=\w(t)\I\subset\R$, $0\in\mathcal I'$. Therefore, $F$ is vertically convex. Moreover, since $F\in\CC^{0,2}(\W_T)$, \defin{verticalconvex} implies that
$$\forall(t,\w)\in\mathbb T\times\supp(S\PP),\quad \vd^2F(t,\w)\geq0.$$
\endproof

\begin{remark}
  If there exists an interval $\mathcal I\subset\R$, $B\lf0,\frac{\abs{\De\w(t)}}{\w(t)}\rg\subset\mathcal I$, such that the map $v^H(\cdot;t,\w)$ defined in \eq{gh} is convex, then
\beq\label{eq:vd2F-jumps}
\vd^2F(t,\w_{t-}+\x\ind_{[t,T]})\geq0\quad\forall\xi\in B(0,\abs{\De\w(t)}).
\eeq
\end{remark}

\section{A model with path-dependent volatility: Hobson-Rogers}
\label{sec:HR}
\sectionmark{A model with path-dependent volatility: Hobson-Rogers}

In the model proposed by \citet{hobson-rogers}, under the market probability $\tilde\PP$, the discounted log-price process $Z$, $Z(t)=\log S(t)$ for all $t\in[0,T]$, is assumed to solve the stochastic differential equation
$$\frac{\ud Z(t)}{Z(t)}=\s(t,Z_t)\ud \tilde W(t)+\mu(t,Z_t)\ud t,$$
where $\tilde W$ is a $\tilde\PP$-Brownian motion and $\s,\mu$ are non anticipative functionals of the process itself, which can be rewritten as Lipschitz-continuous functions of the current time, price and offset functionals of order up to $n$:
$$\bea{c}
\s(t,\w)=\s^n(t,\w(t),o^{(1)}(t,\w),\ldots,o^{(n)}(t,\w)),\\
\mu(t,\w)=\mu^n(t,\w(t),o^{(1)}(t,\w),\ldots,o^{(n)}(t,\w)),\\
o^{(m)}(t,\w)=\int_0^\infty \l e^{-\l u}(\w(t)-\w(t-u))^m\ud u,\quad m=1,\ldots,n.
\ea$$
Note that, in the original formulation in \cite{hobson-rogers}, the authors take into account the interest rate and denote by $Z(t)=\log(S(t)e^{-rt})$ the discounted log-price. We use the same notation for the forward log-prices instead.

Even if the coefficients of the SDE are path-dependent functionals, \cite{hobson-rogers} proved that the $n+1$-dimensional process $(Z,O^{(1)},\ldots,O^{(n)})$ composed of the price process and the offset processes up to order $n$, $O^{(m)}(t):=o^{(m)}(t,Z_t)$, is a Markov process.
In the special case $n=1$ and $\s^n(t,x,o)=\s^n(o)$, $\mu^n(t,x,o)=\mu^n(o)$, denoted $O:=O^{(1)}$, they proved the existence of an equivalent martingale measure $\PP$ defined by
$${\frac{\ud\PP}{\ud\tilde\PP}}\rvert_{\F_t}=\exp\left\{-\int_0^t\th(O(u))\ud W(u)-\frac12\int_0^t\th(O(u))^2\ud u\right\},$$
where $\th(o)=\frac12\s^n(o)+\frac{\mu^n(o)}{\s^n(o)}$.
Then, the offset process solves
\begin{eqnarray*}
  \ud O(t)&=&\s^n(O(t))\ud\tilde W(t)+(\mu^n(O(t))-\l O(t))\ud t\\
&=&\s^n(O(t))\ud W(t)-\frac12(\s^n(O(t))^2+\l O(t))\ud t,
\end{eqnarray*}
where $W$ is the $\PP$-Brownian motion defined by $W(t)=\tilde W(t)+\int_0^t\th(O(u))\ud u$.
So, the (forward) price process solves
\beq\label{eq:HR}
\ud S(t)=S(t)\s^n(O(t))\ud W(t),
\eeq
where $W$ is a standard Brownian motion on $(\O,\F,\FF,\PP)$ and $\s^n:\R\to\R$ is a Lipschitz-continuous function, satisfying some integrability conditions such that the correspondent pricing PDEs admit a classical solution.

The price of a European contingent claim with payoff $H(S(T))$, satisfying appropriate integrability and growth conditions, is given by ,for all $(t,\w)\in\W_T$,
$$F(t,\w)=f(t,\w(t),o(t,\w)),\quad o(t,\w)=\int_0^\infty\l e^{-\l u}(\w(t)-\w(t-u))\ud u,$$
where $f$ is the solution $f\in C^{1,2,2}([0.T)\times\R_+\times\R)\cap\C([0.T]\times\R_+\times\R)$ of the partial differential equation on $[0,T)\times\R_+\times\R$
$$\frac{\s^n(o)^2}2(x^2\partial^2_{xx}f+2x\partial_{xo}f+\partial_{oo}f)-\lf\frac12\s^n(o)^2+\l o\rg\partial_o f+\partial_{t}f=0,$$
where $f\equiv f(t,x,o)$, with final datum $f(T,x,o)=H(x)$.
Using a change of variable, the pricing problem simplifies to solving the following degenerate PDE on $[0,T]\times\R\times\R$:
\beq\label{eq:pde-HR}
\frac12\s^n(x_1-x_2)^2(\partial_{x_1x_1}u-\partial_{x_1}u)+\l(x_1-x_2)\partial_{x_2}u-\partial_t u=0,
\eeq
where $u\equiv u(T-t,x_1,x_2)=f(t,e^{x_1},x_1-x_2)$, with initial condition $u(0,x_1,x_2)=H(e^{x_1})$.
Note that the pricing PDE~\eq{pde-HR} reduces to the universal pricing equation~\eq{fpde}, where, for all $(t,\w)\in\W_T$,
$$F(t,\w)=u(T-t,\log\w(t),\log\w(t)-o(t,\w)),$$
and
\begin{eqnarray*}
\hd F(t,\w)&=&-\partial_t u(T-t,\log\w(t),\log\w(t)-o(t,\w))\\
&&{}+\l\partial_{x_2}u(T-t,\log\w(t),\log\w(t)-o(t,\w)), \\
\vd F(t,\w)&=&\partial_{x_1}u(T-t,\log\w(t),\log\w(t)-o(t,\w)),\\
\quad\vd^2 F(t,\w)&=&\frac1{\w(t)^2}(\partial_{x_1x_1}u(T-t,\log\w(t),\log\w(t)-o(t,\w))\\
&&\quad\quad{}-\partial_{x_1}u(T-t,\log\w(t),\log\w(t)-o(t,\w)).
\end{eqnarray*}


\section{Examples}
\label{sec:ex}

We now show how the above results apply to specific examples of hedging
strategies for path-dependent derivatives.

\subsection{Discretely-monitored path-dependent derivatives}
\label{sec:discr}

The simplest class of path-dependent derivatives are those which are discretely-monitored. The robustness of delta-hedging strategies for discretely-monitored path-dependent derivatives was studied in \cite{ss} as shown in \Sec{ss}. 
In the case of a Black-Scholes pricing model with time-dependent volatility, we show such results may be derived, without probabilistic assumptions on the true price dynamics, as a special case of the results presented above, and we obtain explicit expressions for the first and second order sensitivities of the pricing functional (see also Cont and Yi [9]).

The following lemma describes the regularity of pricing functionals for
discretely-monitored options in a Black-Scholes model with time-dependent
volatility $\s:[0,T]\rightarrow\R_+$ such that $\int_0^T\s^2(t)\ud t<\infty$.
The regularity assumption on the payoff functional is weaker then the ones required for \prop{exist}, thanks to the finite dimension of the problem.

\begin{lemma}[Discretely-monitored path-dependent derivatives]\label{lem:BS}
 Let $H:D([0,T],\R_+)$ and assume that there exist a partition $0=t_0<t_1<\ldots<t_n\leq T$ and a function $h\in C^2_b(\R^n;\R_+)$ such that
$$\forall \w\in D([0,T],\R_+),\quad H(\w_T)=h(\w(t_1),\w(t_2),\ldots,\w(t_n)).$$
Then, the \naf\ $F$ defined in \eq{Fw} is locally regular, that is $F\in\Cloc(\W_T)$, with horizontal and vertical derivatives given in a closed form.
\end{lemma}
\proof
For any $\w\in\O$ and $t\in[0,T]$, let us denote $\bar k\equiv\bar k(n,t):=\max\{i\in\{1,\ldots,n\}\;:\;t_i\leq t\}$, then for $s$ small enough $t+s\in[t_{\bar k},t_{\bar k+1})$ and we have
\begin{align*}
&\!\!\!\!F(t+s,\w_t)-F(t,\w_t)\\
 ={}&\EE^{\QQ}\left[H\lf \w(t_1),\ldots,\w(t_{\bar k}),\w(t)e^{\int_{t+s}^{t_{\bar k+1}}\s(u)\ud W(u)-\frac12\int_{t+s}^{t_{\bar k+1}}\s^2(u)\ud u},\ldots,\right.\right.\\
&\qquad\qquad\left.\w(t)e^{\int_{t+s}^{t_{n}}\s(u)\ud W(u)-\frac12\int_{t+s}^{t_n}\s^2(u)\ud u}\rg+{}\\
&\quad\quad {}-H\lf\w(t_1),\ldots,\w(t_{\bar k}),\w(t)e^{\int_{t}^{t_{\bar k+1}}\s(u)\ud W(u)-\frac12\int_{t}^{t_{\bar k+1}}\s^2(u)\ud u},\ldots,\right.\\
&\qquad\qquad\left.\w(t)e^{\int_{t}^{t_{n}}\s(u)\ud W(u)-\frac12\int_{t}^{t_n}\s^2(u)\ud u}\rg\bigg]\\
={}&\idotsint H\lf\w(t_1),\ldots,\w(t_{\bar k}),\w(t)e^{y_1},\ldots,\w(t)e^{y_{n-\bar k}}\rg\prod_{i=1}^{n-\bar k}\frac{e^{-\frac{\lf y_i+\frac12\int_{t+s}^{t_{\bar k+i}}\s^2(u)\ud u\rg^2}{2\int_{t+s}^{t_{\bar k+i}}\s^2(u)\ud u}}}{\sqrt{2\pi\int_{t+s}^{t_{\bar k+i}}\s^2(u)\ud u}}\ud y_i \\
&{}- \idotsint H\lf\w(t_1),\ldots,\w(t_{\bar k}),\w(t)e^{y_1},\ldots,\w(t)e^{y_{n-\bar k}}\rg\prod_{i=1}^{n-\bar k}\frac{e^{-\frac{\lf y_i+\frac12\int_{t}^{t_{\bar k+i}}\s^2(u)\ud u\rg^2}{2\int_{t}^{t_{\bar k+i}}\s^2(u)\ud u}}}{\sqrt{2\pi\int_{t}^{t_{\bar k+i}}\s^2(u)\ud u}}\ud y_i .
\end{align*}
By denoting $$v_i(s):=\frac{e^{-\frac{\lf y_i+\frac12\int_{t+s}^{t_{\bar k+i}}\s^2(u)\ud u\rg^2}{2\int_{t+s}^{t_{\bar k+i}}\s^2(u)\ud u}}}{\sqrt{2\pi\int_{t+s}^{t_{\bar k+i}}\s^2(u)\ud u}},\quad i=1,\ldots,n-\bar k,$$
dividing by $s$ and taking the limit for $s$ going to 0, we obtain
\begin{align}
\hd F(t,\w)={}&\lim_{s\rightarrow0}\frac{F(t+s,\w_t)-F(t,\w_t)}s \nonumber\\
={}&\sum_{j=1}^{n-\bar k}\idotsint H\lf\w(t_1),\ldots,\w(t_{\bar k}),\w(t)e^{y_1},\ldots,\w(t)e^{y_{n-\bar k}}\rg\!\!\!\!\prod_{\bea{c}\scriptstyle{i=1,\ldots,n-\bar k}\\\scriptstyle{i\neq j}\ea}\!\!\!\!\!\!v_j'(0)v_i(0)\ud y_i\ud y_j,
\end{align}
where, for $i=1,\ldots,n-\bar k$,
$$\bea{l}v_i'(0)=\frac{v_i(0)\s^2(t)}{2\lf \int_{t}^{t_{\bar k+i}}\s^2(u)\ud u\rg^2}\left( \left( y_i+\frac12\int_{t}^{t_{\bar k+i}}\s^2\ud u\rg^{\phantom{1}}\!\int_{t}^{t_{\bar k+i}}\s^2(u)\ud u\right.\\
\left.\quad\qquad\qquad\qquad\qquad\qquad-\lf y_i+\frac12\int_{t}^{t_{\bar k+i}}\s^2(u)\ud u\rg^2+\int_{t}^{t_{\bar k+i}}\s^2(u)\ud u \rg.\ea $$
Moreover, the first and second vertical derivatives are explicitly computed as:
\begin{align}
    \vd F(t,\w)={}&\sum_{j=1}^{n-k}\idotsint \partial_{k+j} H\lf\w(t_1),\ldots,\w(t_{k}),\w(t)e^{y_1},\ldots,\w(t)e^{y_{n-\bar k}}\rg e^{y_j}\prod_{i=1}^{n-k}v_i(0)\ud y_i,\\
    \vd^2F(t,\w)={}&\sum_{i,j=1}^{n-k}\idotsint\partial_{k+i,k+j} H\lf\w(t_1),\ldots,\w(t_{k}),\w(t)e^{y_1},\ldots,\w(t)e^{y_{n-\bar k}}\rg e^{y_i+y_j}\prod_{l=1}^{n-k}v_l(0)\ud y_l,
\end{align}
where $k\equiv k(n,t):=\max\{i\in\{1,\ldots,n\}\;:\;t_i<t\}$.
\endproof


\subsection{Robust hedging for Asian options}
\label{sec:asian}

Asian options, which are options on the average price computed across a certain fixing period, are  commonly traded in currency and commodities markets.
The payoff of Asian options depends on an average of prices during the lifetime of the option, which can be of two types: an arithmetic average
$$M^A(T)=\int_0^TS(u)\mu(\ud u),$$
or a geometric average
$$M^G(T)=\int_0^T\log S(u)\mu(\ud u).$$
We consider Asian call options with date of maturity $T$, whose payoff is given by a continuous functional on $(D([0,T],\R),\norm{\cdot}_\infty)$:
$$\bea{ll}
H^A(S_T)=(M^A(T)-K)^+=:\Psi^A(S(T),M^A(T))&\text{arithmetic Asian call},\\
H^G(S_T)=(e^{M^G(T)}-K)^+=:\Psi^G(S(T),M^G(T))&\text{geometric Asian call}.
\ea$$
Various weighting schemes may be considered:
\begin{itemize}
\item if $\mu(\ud u)=\d_{\{T\}}(\ud u)$, we reduce to an European option, with strike price $K$;
\item if $\mu(\ud u)=\frac1T\ind_{[0,T]}(u)\ud u$, we have a \textit{fixed strike} Asian option, with strike price $K$;
\item in the arithmetic case, if $\mu(\ud u)=\d_{\{T\}}(\ud u)-\frac1T\ind_{[0,T]}(u)\ud u$ and $K=0$, we have a \textit{floating strike} Asian option; the geometric floating strike Asian call has instead payoff $(S(T)-e^{M^G(T)})^+$ with $\mu(\ud u)=\frac1T\ind_{[0,T]}(u)\ud u$.
\end{itemize}
Here, we consider the hedging strategies for fixed strike Asian options, first in a Black-Scholes pricing model, where the volatility is a deterministic function of time, then in a model with path-dependent volatility, the Hobson-Rogers model introduced in \Sec{HR}.
First, we show that these models admit a smooth pricing functional.
Then, we show that the assumptions of \prop{convex} are satisfied, which
leads to robustness of the hedging strategy.

\subsubsection{Black-Scholes delta-hedging for Asian options}
\label{sec:BS-asian}

In the Black-Scholes model, the value functional of such options can be computed in terms of a standard function of three variables (see e.g. \cite[Section 7.6]{pascucci}). In the arithmetic case: for all $(t,\w)\in\W_T$,
\beq\label{eq:Ff-BS-arit}
F(t,\w)=f(t,\w(t),a(t,\w)),\quad a(t,\w)=\int_0^t\w(s)\ud s,
\eeq
where $f\in C^{1,2,2}([0.T)\times\R_+\times\R_+)\cap\C([0.T]\times\R_+\times\R_+)$ is the solution of the following Cauchy problem with final datum:
\beq \label{eq:asianpde-BS-arit}
\begin{cases}
\frac{\s^2(t)x^2}2\partial^2_{xx}f(t,x,a)+x\partial_{a}f(t,x,a)+\partial_{t}f(t,x,a)=0,&t\in[0,T),\,a,x\in\R_+\\
f(T,x,a)=\Psi^A\lf x,\frac{a}{T}\rg.&
\end{cases}
\eeq
Different parametrizations were suggested in order to facilitate the computation of the solution, which is however not in a closed form. For example, \cite{dupire} shows a different characterization which improves the numerical discretization of the problem, while \cite{rogershi} reduces the pricing issue to the solution of a parabolic PDE in two variable, thus decreasing the dimension of the problem, as done in \cite{ingersoll} for the case of a floating-strike Asian option.

In the geometric case: for all $(t,\w)\in\W_T$,
\beq\label{eq:Ff-BS-geom}
F(t,\w)=f(t,\w(t),g(t,\w)),\quad g(t,\w)=\int_0^t\log\w(s)\ud s,
\eeq
where $f\in C^{1,2,2}([0.T)\times\R_+\times\R)\cap\C([0.T]\times\R_+\times\R)$ is the solution of the following Cauchy problem with final datum: for $t\in[0,T)$, $x\in\R_+$, $g\in\R$,
\beq \label{eq:asianpde-BS-geom}
\begin{cases}
\frac{\s^2(t)x^2}2\partial_{xx}^2f(t,x,g)+\log x\partial_{g}f(t,x,g)+\partial_{t}f(t,x,g)=0,\\
f(T,x,g)=\Psi^G\lf x,\frac{g}{T}\rg.
\end{cases}
\eeq
As in the arithmetic case, the dimension of the problem \eq{asianpde-BS-geom} can be reduced to two by a change of variable. Moreover, in this case, it is possible to obtain a Kolmogorov equation associated to a degenerate parabolic operator that has a Gaussian fundamental solution.

We remark that the pricing PDEs~\eq{asianpde-BS-arit},\eq{asianpde-BS-geom} are both equivalent to the functional partial differential equation~\eq{fpde} for $F$ defined respectively by \eq{Ff-BS-arit} and \eq{Ff-BS-geom}.
Indeed, computing the horizontal and vertical derivatives of $F$ yields
$$\bea{l}
\hd F(t,\w)=\partial_t f(t,\w(t),a(t,\w))+\w(t)\partial_a f(t,\w(t),a(t,\w)), \\
\vd F(t,\w)=\partial_x f(t,\w(t),a(t,\w)),\quad\vd^2 F(t,\w)=\partial_{xx}^2 f(t,\w(t),a(t,\w))
\ea $$
for the arithmetic case, and
$$\bea{l}
\hd F(t,\w)=\partial_t f(t,\w(t),g(t,\w))+\log\w(t)\partial_g f(t,\w(t),g(t,\w)), \\
\vd F(t,\w)=\partial_x f(t,\w(t),g(t,\w)),\quad\vd^2 F(t,\w)=\partial_{xx}^2 f(t,\w(t),g(t,\w))
\ea $$
for the geometric case.

Thus, the standard pricing problems for the arithmetic and geometric Asian call options turn out to be particular cases of \prop{hedge}, with $A=\s^2\w^2$. In particular, the delta-hedging strategy is given by
\begin{align*}
\phi(t,\w)=\vd F(t,\w)={}&\partial_{x} f(t,\w(t),a(t,\w))\quad\text{(arithmetic), or}\\
={}&\partial_{x} f(t,\w(t),g(t,\w))\quad\text{(geometric)}.
\end{align*}

The following claim is an application of \prop{convex}.
\begin{corollary}\label{cor:BS-robust}
  If the Black-Scholes volatility term structure over-estimates the realized market volatility, i.e.
$$\s(t)\geq\s^{\mathrm{mkt}}(t,\w)\quad \forall\w\in \A\cap\supp(S,\PP)$$
then the Black-Scholes delta hedges for the Asian options with payoff functionals
$$\bea{ll}
H^A(S_T)=(\frac1T\int_0^TS(t)\ud t-K)^+&\text{arithmetic Asian call},\\
H^G(S_T)=(e^{\frac1T\int_0^T\log S(t)\ud t}-K)^+&\text{geometric Asian call},
\ea$$
are robust on $\A\cap\supp(S,\PP)$. Moreover, the hedging error at maturity is given by
$$\frac12\int_0^T \lf{\s}(t)^2-\s^{\mathrm{mkt}}(t,\w)^2\rg\w^2(t) \ppa{x}f \ud t,$$
where $f$ stays for, respectively, $f(t,\w(t),a(t,\w))$ solving the Cauchy problem \eq{asianpde-BS-arit}, and  $f(t,\w(t),g(t,\w))$ solving the Cauchy problem \eq{asianpde-BS-geom}.
\end{corollary}
Let us emphasize again that the hedger's profit-and-loss depends explicitly on the Gamma of the option and on the distance of the Black-Scholes volatility from the realized volatility during the lifetime of the contract.

\proof
The integrability of $H^A,H^G$ in $(\O,\PP)$ follows from the Feynman-Kac representation of the solution of the Cauchy problems with final datum \eq{asianpde-BS-arit}, \eq{asianpde-BS-geom}.

By the functional representation in~\eq{Ff-BS-arit}, respectively \eq{Ff-BS-geom}, the pricing functional $F$ is smooth, i.e. it satisfies \eq{regF}.
If the assumptions of \prop{convex} are satisfied, we can thus apply \prop{robust} to prove the robustness property.
We have to check the convexity of the map $v^H(\cdot;t,\w)$ in \eq{gh} for all $(t,\w)\in[0,T]\times Q(\O,\Pi)$.
Concerning the arithmetic Asian call option, we have:
\begin{align*}
  v^{H^A}(e;t,\w)={}&H^A\lf\w(1+e\ind_{[t,T]})\rg\\
={}&\lf\frac1T\lf\int_0^t\w(u)\ud u+\int_t^T\w(u)(1+e)\rg-K\rg^+\\
={}&\lf m(T)+\frac e T(a(T)-a(t))-K \rg^+\\
={}&\frac{a(T)-a(t)}T\lf e-K'\rg^+,
\end{align*}
where $m(T)=\frac1T a(T)$ and $K'=\frac{KT-a(T)}{a(T)-a(t)}$, which is clearly convex in $e$.

As for the geometric Asian call option, we have:
\begin{align*}
  v^{H^G}(e;t,\w)={}&H^G\lf\w(1+e\ind_{[t,T]})\rg\\
={}&\lf e^{\frac1T\int_0^t\log\w(u)\ud u}e^{\frac1T\int_t^T\log(\w(u)(1+e))\ud u}-K\rg^+\end{align*}
which is a convex function in $e$ around 0, since $\w$ is bounded away from 0 on $[0,T]$. Indeed: $e\mapsto\int_t^T\log(\w(u)(1+e))\ud u$ is convex since it is the integral in $u$ of a function of $(u,e)$ which is convex in $e$ by preservation of convexity under affine transformation; then $e\mapsto e^{\frac1T\int_t^T\log(\w(u)(1+e))\ud u}$ is convex because it is the composition of a convex increasing function and a convex function.
\endproof

\begin{remark}
The robustness of the Black-Scholes-delta hedging for the arithmetic Asian option is in fact a direct consequence of \prop{robust}.
Indeed, in the Black-Scholes framework, the Gamma of an Asian call option is non-negative, as it has been shown for different closed-form analytic approximations found in the literature. An example can be seen in \cite{mil-posner}, where the density of the arithmetic mean is approximated by a reciprocal gamma distribution which is the limit distribution of an infinite sum of correlated log-normal random variables. This already implies the condition \eq{vd2F}.
\end{remark}

\subsubsection{Hobson-Rogers  delta-hedging for Asian options}
\label{sec:RH}

We have already shown in \Sec{HR} that the Hobson-Rogers model admits a smooth pricing functional for suitable non-path-dependent payoffs. \citet{pascucci-difra} proved that also the problem of pricing and hedging a geometric Asian option can be similarly reduced to a degenerate PDE belonging to the class of Kolmogorov equations, for which a classical solution exists. In this case, the pricing functional can be written as a function of four variables
\beq\label{eq:Fu-geom}
F(t,\w)=u(T-t,\log\w(t),\log\w(t)-o(t,\w),g(t,\w)),
\eeq
where $u$ is the classical solution of the following Cauchy problem on $[0,T]\times\R\times\R\times\R$:
\beq\label{eq:asianpde-HR-geom}\begin{cases}
\frac12\s^n(x_1-x_2)^2(\partial^2_{x_1x_1}u-\partial_{x_1}u)+\l(x_1-x_2)\partial_{x_2}u+x_1\partial_{x_3}u-\partial_t u=0,\\
u(0,x_1,x_2,x_3)=\Psi^G(e^{x_1},\frac{x_3}T).
\end{cases}\eeq


The following claim is the analogous of \cor{BS-robust} for the Hobson-Rogers model; the proof is omitted because it follows exactly the same arguments as the proof of \cor{BS-robust}.
\begin{corollary}
  If the Hobson-Roger volatility in \eq{HR} over-estimates the realized market volatility, i.e.
$$\s(t,\w)=\s^n(o(t,\w))\geq\s^{\mathrm{mkt}}(t,\w)\quad \forall\w\in \A\cap\supp(S,\PP)$$
then the Hobson-Rogers delta hedge for the geometric Asian option with payoff functional
$$H^G(S_T)=(e^{\frac1T\int_0^T\log S(t)\ud t}-K)^+$$
is robust on $\A\cap\supp(S,\PP)$. Moreover, the hedging error at maturity is given by
$$\frac12\int_0^T \lf{\s^n}(o(t,\w))^2-\s^{\mathrm{mkt}}(t,\w)^2\rg\w^2(t) \ppa{x}u(T-t,\log\w(t),\log\w(t)-o(t,\w),g(t,\w)) \ud t,$$
where $u$ is the solution of the Cauchy problem \eq{asianpde-HR-geom}.
\end{corollary}

Other models that generalize Hobson-Rogers and allow to derive finite-dimensional Markovian representation for the price process and its arithmetic mean are given by \citet{pascucci-foschi,salvatore-tankov}. They thus guarantee the existence of a smooth pricing functional for arithmetic Asian options, then robustness of the delta hedge can be proved the same way as we showed in the Black-Scholes and Hobson-Rogers cases.

\subsection{Dynamic  hedging of barrier options}

Barrier options are examples of path-dependent derivatives for which delta-hedging strategies are not robust.

Consider the case of an up-and-out barrier call option with strike price $K$ and barrier $U$, whose payoff functional is
\beq \label{eq:barrier}
 H(S_T)=(S(T)-K)^+\ind_{\{\overline S(T)<U\}}.
\eeq
The pricing functional of a barrier option is determined by regular solutions of classical Dirichlet problems, opportunely stopped at the barrier hitting times.
The pricing functional for the claim with payoff \eqref{eq:barrier} is given, at time $t\in[0,T]$, by
$$F(t,\w)=f(t\wedge \t_U(\w),\w(t\wedge \t_U(\w))),$$
where $\t_U(\w):=\inf\{t\geq0: \w(t)\in[U,+\infty)\}$ and $f$ is the $\C^{1,2}([0,T)\times(0,U))\cap\C([0,T]\times(0,U))$ solution of the following Dirichlet problem:
\begin{equation}  \label{eq:barrierPDE}
  \left\{\begin{array}{ll}
    \frac12\s^2(t)x^2\partial_{xx}^2f(t,x)+\partial_{t}f(t,x)=0,& (t,x)\in[0,T)\times (0,U),\\
    f(t,U)=0,& t\in[0,T],\\
    f(T,x)=H(x),& x\in(0,U).
  \end{array}\right.
\end{equation}
The delta-hedging strategy is then given by
$$\phi(t,\w)=\partial_{x} f(t,\w(t))\ind_{[0,\t_U(\w))}(t).$$
Analogously to the application in \Sec{asian}, we can compute the hedging error of the delta hedge for the barrier option.
However, unlike for Asian options, the delta hedge for barrier options fails to have the robustness property, because the price collapses at $t=\t_U$, disrupting the positivity of the Gamma.
On the other end, the Gamma of barrier options can be quite large in magnitude, so it is crucial to have a good estimate of volatility, in order to keep the hedging error as small as possible.
\begin{remark}
Let $H$ be the payoff functional of the up-and-out barrier call option with strike price $K$ and barrier $U$ in \eq{barrier}.
Then the Black-Scholes delta hedge for $H$ is not robust to volatility mis-specifications. Any mismatch between the model volatility $\s$ and the realized volatility $\s^{mkt}$ is amplified by the Gamma of the option as the barrier is approached and the resulting error can have an arbitrary sign due to the non-constant sign of the option Gamma near the barrier.
\end{remark}

The assumptions of \prop{convex} are not satisfied, indeed: for any $(t,\w)\in[0,T]\times\C([0,T],\R_+)$,
\begin{align*}
  v^H(e;t,\w)={}&(\w(T)+\w(T)e-K)^+\ind_{(0,U)}\lf\sup_{s\in[0,T]}\lf\w(s)(1+e\ind_{[t,T]}(s))\rg\rg\\
={}&\w(T)\lf e-\frac{K-\w(T)}{\w(T)}\rg^+\ind_{(0,U)}(\g(e))\\
={}&\w(T)\lf e-\frac{K-\w(T)}{\w(T)}\rg^+\ind_{\{\g^{-1}((0,U))\}}(e)
\end{align*}
where $\g:\R\rightarrow\R_+$,
\begin{align*}
  \g(e):={}&\sup_{s\in[0,T]}\lf\w(s)(1+e\ind_{[t,T]}(s))\rg\\
={}&\max\left\{\over\w(t),(1+e)\sup_{s\in[t,T]}\w(s))\right\}\\
={}&\sup_{s\in[t,T]}\w(s)\lf e-\frac{\over\w(t)-\sup_{s\in[t,T]}\w(s)}{\sup_{s\in[t,T]}\w(s)}\rg^++\over\w(t).
\end{align*}
$\g^{-1}(A)$ denote the counter-image of $A\subset\R_+$ via $\g$, and $\over\w(t):=\sup_{s\in[0,t]}\w(s)$.
Since $\g$ is a positive non-decreasing continuous function, we have
$$\g^{-1}((0,U))=\begin{cases}\emptyset,&\text{if }U\leq\over\w(t)\\
\lf-\infty,\frac{U-\sup_{s\in[t,T]}\w(s)}{\sup_{s\in[t,T]}\w(s)}\rg,&\text{otherwise.}\end{cases}$$
Thus, there exist an interval $\mathcal I\subset\R$, $0\in\mathcal I$, such that $v^H(\cdot;t,\w):\mathcal I\rightarrow\R$ is convex if and only if $U>\sup_{s\in[t,T]}\w(s)$.
However, \prop{convex} requires the map $v^H(\cdot;t,\w)$ to be convex for all $\w\in\supp(S,\PP)$ in order to imply vertical convex of the value functional.

Thus, we observe that unlike the case of Asian options, delta-hedging strategies do not provide a robust approach to the hedging of barrier options.

\chapter{Adjoint expansions in local L\'evy models}

This chapter is based on a joint work with Stefano Pagliarani and Andrea Pascucci, published in 2013 \cite{ppr}.

Analytical approximations and their applications to finance have been studied by several authors
in the last decades because of their great importance in the calibration and risk management
processes. The large body of the existing literature (see, for instance, \cite{Hagan99},
\cite{Howison2005}, \cite{WiddicksDuckAndricopoulosNewton2005},
\cite{GatheralHsuLaurenceOuyangWang2010}, \cite{BenhamouGobetMiri2010b},
\cite{CorielliFoschiPascucci2010}, \cite{ChengCostanzinoLiechtyMazzucatoNistor2011}) is mainly
devoted to purely diffusive (local and stochastic volatility) models or, as in
\cite{BenhamouGobetMiri2009} and \cite{XuZheng2010}, to local volatility (LV) models with Poisson
jumps, which can be approximated by Gaussian kernels.

The classical result by Hagan \cite{Hagan99} is a particular case of our expansion, in the sense
that for a standard LV model with time-homogeneous coefficients our formulae reduce to Hagan's
ones (see \Sec{secsimpl}). While Hagan's results are heuristic, here we also provide explicit error estimates for time-dependent coefficients as well.

The results of \Sec{Merton} on the approximation of the transition density for
jump-diffusions are essentially analogous to the results in \cite{BenhamouGobetMiri2009}: however
in \cite{BenhamouGobetMiri2009} ad-hoc Malliavin techniques for LV models with Merton jumps are
used and only a first order expansion is derived. Here we use different techniques (PDE and
Fourier methods) which allows to handle the much more general class of local L\'evy processes: this
is a very significant difference from previous research. Moreover we derive higher order
approximations, up to the $4^{\text{th}}$ order.

Our approach is also more general than the so-called ``pa\-ra\-me\-trix'' methods recently
proposed in \cite{CorielliFoschiPascucci2010} and \cite{ChengCostanzinoLiechtyMazzucatoNistor2011}
as an approximation method in finance. The parametrix method is based on repeated application of
Duhamel's principle which leads to a recursive integral representation of the fundamental
solution: the main problem with the parametrix approach is that, even in the simplest case of a LV
model, it is hard to compute explicitly the parametrix approximations of order greater than one.
As a matter of fact, \cite{CorielliFoschiPascucci2010} and
\cite{ChengCostanzinoLiechtyMazzucatoNistor2011} only contain first order formulae. The adjoint
expansion method contains the parametrix approximation {\it as a particular case}, that is at
order zero and in the purely diffusive case. However the general construction of the adjoint
expansion is substantially different and allows us to find explicit higher-order formulae for the
general class of local L\'evy processes.

\section{General framework}
\label{sec:sec1}

In a local L\'evy model, we assume that the log-price process $X$ of the underlying asset of interest solves the SDE
\begin{equation}\label{X}
    \ud X(t)=\m(t,X(t-))\ud t+\s(t,X(t)) \ud W(t)+ \ud J(t).
\end{equation}
In \eqref{X}, 
$W$ is a standard real Brownian motion on a filtered probability space
$(\O,\F,(\F_t)_{0\leq t\leq T},\mathbb{P})$ with the usual assumptions on the filtration and $J$
is a pure-jump L\'evy process, independent of $W$, with L\'evy triplet
$(\m_{1},0,\n)$. 
In order to guarantee the martingale property for the
discounted asset price $\tilde{S}(t):=S_{0}e^{X(t)-rt}$, we set
\begin{equation}\label{30}
 \m(t,x)=\rle-\m_{1}-\frac{\s^{2}(t,x)}{2},
\end{equation}
where
\begin{equation}\label{31}
 \rle=r-\int_{\R}\left(e^{y}-1-y\mathds{1}_{\{|y|<1\}}\right)\n(dy).
\end{equation}
We denote by
  $$X^{t,x}:T\mapsto X^{t,x}(T)$$
the solution of \eqref{X} starting from $x$ at time $t$ and by
 $$\p_{X^{t,x}(T)}(\x)=E\left[e^{i\x X^{t,x}(T)}\right],\qquad \x\in\R,$$
the characteristic function of $X^{t,x}(T)$.
Provided that $X^{t,x}(T)$ has density $\G(t,x;T,\cdot)$, then its characteristic function is
equal to
 $$\p_{X^{t,x}(T)}(\x)=
 \int_{\R} e^{i\x y}\G(t,x;T,y)dy.$$
Notice that $\G(t,x;T,y)$ is the fundamental solution of the Kolmogorov operator
\begin{equation}\label{L}
\begin{split}
  Lu(t,x)&= \frac{\s^{2}(t,x)}{2}\left({\partial}_{xx}-{\partial}_{x}\right)u(t,x)+\rle{\partial}_{x}u(t,x)+{\partial}_{t}u(t,x)\\
  &\quad+\int_{\R}\left(u(t,x+y)-u(t,x)-{\partial}_{x}u(t,x)y\mathds{1}_{\{|y|<1\}}\right)\n(dy).
\end{split}
\end{equation}

\begin{example}\label{ex3}
Let $J$ be a compound Poisson process with Gaussian jumps, that is
  $$J(t)=\sum_{n=1}^{N(t)} Z_n $$
where $N(t)$ is a Poisson process with intensity $\l$ and $Z_n$ are i.i.d. random variables
independent of $N(t)$ with Normal distribution $\mathcal{N}_{m,\d^{2}}$. In this case, $\n=\l
\mathcal{N}_{m,\d^{2}}$ and
  $$\m_{1}=\int_{|y|<1}y\n(dy).$$
Therefore the drift condition \eqref{30} reduces to
\begin{align}\label{30b}
 \m(t,x)=r_{0}-\frac{\s^{2}(t,x)}{2},
\end{align}
where
\begin{equation}\label{30c}
 r_{0}=r-\int_{\R}\left(e^{y}-1\right)\n(dy)=r-\lambda\left(e^{m+\frac{\delta^2}2}-1\right).
\end{equation}
Moreover, the characteristic operator can be written in the equivalent form
\begin{equation}\label{LPoi}
\begin{split}
  L u(t,x)&=\frac{\s^{2}(t,x)}{2}\left({\partial}_{xx}-{\partial}_{x}\right)u(t,x)+r_{0}{\partial}_{x}u(t,x)+{\partial}_{t}u(t,x)\\
  &\quad+\int_{\R}\left(u(t,x+y)-u(t,x)\right)\n(dy).
\end{split}
\end{equation}
\end{example}
\begin{example}\label{ex4}
Let $J$ be a Variance-Gamma process (cf. \cite{MadanSeneta1990})
obtained by subordinating a Brownian motion with drift $\th$ and standard deviation $\r$, by
a Gamma process with variance $\kappa$ and unitary mean. In this case the L\'evy measure is
given by
\begin{equation}\label{70}
  \n(dx)=\frac{e^{-\l_{1}x}}{\kappa x}\caratt_{\{x>0\}}dx+\frac{e^{\l_{2}x}}{\kappa|x|}\caratt_{\{x<0\}}dx
\end{equation}
where
  $$\l_{1}=\left(\sqrt{\frac{\th^{2}\kappa^{2}}{4}+\frac{\r^{2}\kappa}{2}}+\frac{\th\kappa}{2}\right)^{-1},
  \qquad \l_{2}=\left(\sqrt{\frac{\th^{2}\kappa^{2}}{4}+\frac{\r^{2}\kappa}{2}}-\frac{\th\kappa}{2}\right)^{-1}.$$
The risk-neutral drift in \eqref{X} is equal to
  $$\m(t,x)=r_{0}-\frac{\s^{2}(t,x)}{2}$$
where
\begin{equation}\label{71}
  r_{0}=r+\frac{1}{\kappa}\log\left(1-\l_{1}^{-1}\right)\left(1+\l_{2}^{-1}\right)
  =r+\frac{1}{\kappa}\log\left(1-\kappa\left(\th+\frac{\r^{2}}{2}\right)\right),
\end{equation}
and the expression of the characteristic operator $L$ is the same as in \eqref{LPoi} with $\n$ and
$r_{0}$ as in \eqref{70} and \eqref{71} respectively.
\end{example}

Our goal is to give an accurate analytic approximation of the characteristic function and, when
possible, of the transition density of $X$. The general idea is to consider an approximation of
the volatility coefficient $\s$.
More precisely, to shorten notations we set
\begin{equation}\label{a}
    a(t,x)=\s^{2}(t,x)
\end{equation}
and we assume that $a$ is regular enough: more precisely, for a fixed $N\in\NN$, we make the
following

\smallskip\noindent{\bf Assumption $\text{A}_{N}$.} {\it The function $a=a(t,x)$ is continuously
differentiable with respect to $x$ up to order $N$. Moreover, the function $a$ and its derivatives
in $x$ are bounded and Lipschitz continuous in $x$, uniformly with respect to $t$.}
\smallskip

Next, we fix a basepoint $\bar{x}\in\R$ and consider the $N^{\text{th}}$-order Taylor polynomial of $a(t,x)$ about
$\bar{x}$:
  $$  \a_0(t)+2\sum_{n=1}^{N}\a_n(t)(x-\bar{x})^n,$$
where $\a_0(t)=a(t,\bar{x})$ and
 \begin{equation}\label{43bis}
  \a_n(t)=\frac{1}{2}\frac{\partial_x^na(t,\bar{x})}{n!}, \qquad  n\le N.
 \end{equation}
Then we introduce the $n^{\text{th}}$-order approximation of $L$:
\begin{equation}\label{43}
  L_{n}:=L_{0}+\sum_{k=1}^{n}\a_k(t)(x-\bar{x})^k\lf\partial_{xx}-\partial_x\rg, \qquad  n\le N,
\end{equation}
where
\begin{equation}\label{42}
\begin{split}
  L_{0} u(t,x)&=\frac{\a_0(t)}{2} \lf\partial_{xx}u(t,x)-\partial_xu(t,x)\rg + \rle\partial_{x}u(t,x)+{\partial}_{t}u(t,x)\\
  &\quad+\int_{\R}\left(u(t,x+y)-u(t,x)-\partial_{x}u(t,x)y\caratt_{\{|y|<1\}}\right)\n(dy).
\end{split}
\end{equation}
Following the perturbation method proposed in \cite{PagliaraniPascucci2011}, and also recently
used in \cite{FoschiPagliaraniPascucci2011} for the approximation of Asian options, the
$n^{\text{th}}$-order approximation of the fundamental solution $\G$ of $L$ is defined by
\begin{equation}\label{34}
  \Gamma^{n}(t,x;T,y):=\sum_{k=0}^n G^k(t,x;T,y), \qquad t<T,\ x,y\in\R.
\end{equation}
The leading term $G^0$ of the expansion in \eqref{34} is the fundamental solution of $L_{0}$
and, for any $(T,y)\in\R_{+}\times\R$ and $k\le N$, the functions $G^{k}(\cdot,\cdot;T,y)$ are
defined recursively in terms of the solutions of the following sequence of Cauchy problems on the
strip $]0,T[\times \R$:
\begin{equation}\label{2.2}
  \begin{cases}
     L_{0} G^k(t,x;T,y)\hspace{-9pt} &=- \sum\limits_{h=1}^k\left(L_{h}-L_{h-1}\right) G^{k-h}(t,x;T,y)\\
                       \hspace{-9pt} &=- \sum\limits_{h=1}^k\a_h(t)(x-\bar{x})^h \lf\partial_{xx}-\partial_x\rg
                        G^{k-h}(t,x;T,y),\\
     \hspace{9pt}G^k(T,x;T,y) \hspace{-9pt}&= 0
     .
  \end{cases}
\end{equation}
In the sequel, when we want to specify explicitly the dependence of the approximation $\G^{n}$ on the basepoint $\xbar$, we shall use the notation
\begin{equation}\label{and123}
    \Gamma^{\xbar,n}(t,x;T,y)\equiv \Gamma^{n}(t,x;T,y).
\end{equation}

In \Sec{Merton} we show that, in the case of a LV model with Gaussian jumps, it is
possible to find {\it the explicit solutions} to the problems \eqref{2.2} by an iterative
argument. When general L\'evy jumps are considered, it is still possible to compute the explicit
solution of problems \eqref{2.2} {\it in the Fourier space}. Indeed, in \Sec{LV-J}, we get
an expansion of the characteristic function
$\p_{X^{t,x}(T)}$ having as leading term the characteristic function of the process whose Kolmogorov operator is $L_{0}$ in \eqref{42}.

We explicitly notice that, if the function $\s$ only depends on time, then {\it the approximation
in \eqref{34} is exact at order zero.}

We now provide global error estimates for the approximation
in the purely diffusive case. The proof is postponed to the Appendix (\Sec{app}).
\begin{theorem}\label{t11}
Assume the parabolicity condition
\begin{equation}\label{80}
  m\le \frac{a(t,x)}{2}\le M,\qquad (t,x)\in[0,T]\times\R,
\end{equation}
where $m,M$ are  positive constants and let $\bar{x}=x$ or $\xbar=y$ in \eqref{and123}.
Under Assumption A$_{N+1}$, for any $\e>0$ we have
\begin{equation}\label{81}
  \left|\Gamma(t,x;T,y)-\Gamma^{\xbar,N}(t,x;T,y)\right|\le
  g_{N}(T-t)\bar{\Gamma}^{M+\e}(t,x;T,y),
\end{equation}
for $x,y\in\R$ and $t\in [0,T[$, where
$\bar{\Gamma}^{M}$ is the Gaussian fundamental solution of the heat operator
  $$M{\partial}_{xx}+{\partial}_{t},$$
and $g_{N}(s)=\text{O}\left(s^{\frac{N+1}{2}}\right)$ as $s\to 0^{+}$.
\end{theorem}

Theorem \ref{t11} improves some known results in the literature. In particular in
\cite{BenhamouGobetMiri2010b} asymptotic estimates for option prices in terms of
$(T-t)^{\frac{N+1}{2}}$ are proved under a stronger assumption on the regularity of the
coefficients, equivalent to Assumption A$_{3N+2}$. Here we provide error estimates for the
transition density: error bounds for option prices can be easily derived from \eqref{81}.
Moreover, for small $N$ it is not difficult to find the explicit expression of $g_{N}$.

Estimate \eqref{81} also justifies a time-splitting procedure which nicely adapts to our approximation operators, as shown in detail in Remark 2.7 in
\cite{PagliaraniPascucci2011}.

\section{LV models with Gaussian jumps}
\label{sec:Merton}

In this section we consider the SDE \eqref{X} with $J$ as in Example \ref{ex3}, namely $J$ is a
compound Poisson process with Gaussian jumps. Clearly, in the particular case of a constant
diffusion coefficient $\s(t,x)\equiv \s$, we have the classical Merton jump-diffusion model
\cite{Merton1976}:
  $$X^{\text{Merton}}(t)=\left(r_0-\frac{\s^{2}}{2}\right) t + \s W(t) + J(t),$$
with $r_{0}$ as in \eqref{30c}. We recall that the analytical approximation of this kind of models
has been recently studied by Benhamou, Gobet and Miri in \cite{BenhamouGobetMiri2009} by Malliavin
calculus techniques.

The expression of the pricing operator $L$ was given in \eqref{LPoi} and in this case the leading term of the approximation (cf. \eqref{42}) is equal to
\begin{equation}\label{L0}
\begin{split}
  L_{0} v(t,x)=&\,\frac{\a_0(t)}{2} \lf\partial_{xx}v(t,x)-\partial_xv(t,x)\rg + r_{0} \partial_xv(t,x)\\
  &+ \pa_tv(t,x) + \int_{\R}\left(v(t,x+y)-v(t,x)\right)\nu(d y).
\end{split}
\end{equation}
The fundamental solution of $L_{0}$ is the transition density of a Merton process, that is
\begin{equation}\label{Gamma0}
   G^{0}(t,x;T,y)=e^{-\l(T-t)} \sum_{n=0}^{+\infty} \frac{(\l(T-t))^n}{n!} \Gamma_n(t,x;T,y),
\end{equation}
where
\begin{equation}\label{36}
\begin{split}
   \Gamma_n(t,x;T,y)&=\frac{1}{\sqrt{2\pi \left(A(t,T)+n\d^2\right)}}\,e^{-\frac{\lf x-y+(T-t){r_{0}}-\frac12A(t,T)+nm\rg^2}{2\left(A(t,T)+n\d^2\right)}}, \\
    A(t,T)&=\int_t^T\a_0(s)d s.
\end{split}
\end{equation}
In order to determine the explicit solution to problems \eqref{2.2} for $k\ge 1$, we use some
elementary properties of the functions $\lf\Gamma_n\rg_{n\geq0}$. The following lemma can be
proved as Lemma 2.2 in \cite{PagliaraniPascucci2011}.
\begin{lemma}\label{l1} For any $x,y,\xbar\in\R$, $t<s<T$ and $n,k\in\NN_{0}$, we have
\begin{align}\label{repr}
  \Gamma_{n+k}(t,x;T,y)=&\int_{\R}\Gamma_n(t,x;s,\eta) \Gamma_k(s,\eta;T,y)d \eta,\\
  \pa^{k}_y\Gamma_n(t,x;T,y) =&\, (-1)^{k}\pa^{k}_x\Gamma_n(t,x;T,y),\label{d}\\
  (y-\bar{x})^{k} \Gamma_n(t,x;T,y) =&\, V_{t,T,x,n}^{k}\Gamma_n(t,x;T,y),\label{V}
\end{align}
where $V_{t,T,x,n}$ is the operator defined by
\begin{equation}\label{38}
\begin{split}
  V_{t,T,x,n}f(x) =& \left(x-\bar{x}+(T-t){r_{0}}-\frac12 A(t,T) +nm\right)f(x)\\
  & +\lf A(t,T) +n\d^2\rg\pa_x f(x).
\end{split}
\end{equation}
\end{lemma}
Our first results are the following first and second order expansions of the transition density $\G$.
\begin{theorem}[1st order expansion]\label{t1}
The solution $G^1$ of the Cauchy problem \eqref{2.2} with $k=1$ is given by
\begin{align}\label{11}
  G^1(t,x;T,y) =& \sum_{n,k=0}^{+\infty} J^1_{n,k}(t,T,x) \Gamma_{n+k}(t,x;T,y).
\end{align}
where $J^1_{n,k}(t,T,x)$ is the differential operator defined by
 \begin{equation}\label{13}
    J^1_{n,k}(t,T,x) = e^{-\l(T-t)} \frac{\l^{n+k}}{n!k!} \int_t^T \a_1(s) (s-t)^n(T-s)^k V_{t,s,x,n} d s\,
    ({\partial}_{xx}-{\partial}_{x}).
  \end{equation}
\end{theorem}

\noindent{\it Proof.} By the standard representation formula for solutions to the non-homogeneous
parabolic Cauchy problem  \eqref{2.2} with null final condition, we have
\begin{align*}
  G^1(t,x;T,y) &=\int_t^T\int_{\R}G^0(t,x;s,\eta) \a_1(s) (\eta-\bar{x})\cdot\\
 &\quad\cdot (\pa_{\eta\eta}-\pa_{\eta}) G^0(s,\eta;T,y)d \eta d
  s=
\intertext{(by \eqref{V})}
 &= \sum_{n=0}^{+\infty} \frac{\l^{n}}{n!}\int_t^T \a_1(s) e^{-\l(s-t)}  (s-t)^n\cdot\\
 &\quad\cdot V_{t,s,x,n} \int_{\R}\Gamma_n(t,x;s,\eta) (\pa_{\eta\eta}-\pa_{\eta}) G^0(s,\eta;T,y)d \eta d s=
\intertext{(by parts)}
 &= e^{-\l(T-t)}\sum_{n,k=0}^{+\infty} \frac{\l^{n+k}}{n!k!} \int_t^T\a_1(s)  (T-s)^k (s-t)^n \cdot\\
 &\quad\cdot V_{t,s,x,n} \int_{\R}(\pa_{\eta\eta}+\pa_{\eta}) \Gamma_n(t,x;s,\eta) \Gamma_k(s,\eta;T,y)d \eta ds=
\intertext{(by \eqref{d} and \eqref{repr})}
 &=  e^{-\l(T-t)} \sum_{n,k=0}^{\infty} \frac{\l^{n+k}}{n!k!} \int_t^T \a_1(s) (T-s)^k (s-t)^n  V_{t,s,x,n}d s\cdot\\
 &\quad\cdot (\pa_{xx}-\pa_x) \Gamma_{n+k}(t,x;T,y)
\end{align*}
and this proves \eqref{11}-\eqref{13}. \qquad\endproof

\begin{remark}\label{r4}
A straightforward but tedious computation shows that the operator $J^1_{n,k}(t,T,x)$ can be
rewritten in the more convenient form
\begin{equation}\label{J1}
  J^1_{n,k}(t,T,x) = \sum_{i=1}^{3}\sum_{j=0}^1 f^1_{n,k,i,j}(t,T)(x-\bar{x})^j \pa_x^i,
\end{equation}
for some deterministic functions $f^1_{n,k,i,j}$.
\end{remark}

\begin{theorem}[2nd order expansion]\label{t2}
The solution $G^2$ of the Cauchy problem \eqref{2.2} with $k=2$ is given by
\begin{align}\nonumber
  G^2(t,x;T,y) =& \sum_{n,h,k=0}^{+\infty} J^{2,1}_{n,h,k}(t,T,x) \Gamma_{n+h+k}(t,x;T,y) \\ \label{12}
 & + \sum_{n,k=0}^{\infty } J^{2,2}_{n,k}(t,T,x) \Gamma_{n+k}(t,x;T,y),
\end{align}
where
\begin{align*}
  J^{2,1}_{n,h,k}(t,T,x) =&\, \frac{\l^{n}}{n!} \int_t^T \a_1(s) e^{-\l(s-t)} (s-t)^n V_{t,s,x,n} ({\partial}_{xx}-{\partial}_{x}) \tilde{J}^1_{n,h,k}(t,s,T,x) d s  \\
  J^{2,2}_{n,k}(t,T,x) =&\, e^{-\l(T-t)} \frac{\l^{n+k}}{n!k!} \int_t^T \a_2(s) (s-t)^n(T-s)^k V_{t,s,x,n}^2 d s\, ({\partial}_{xx}-{\partial}_{x})
\end{align*}
and $\tilde{J}^1_{n,h,k}$ is the ``adjoint'' operator of $J^1_{h,k}$, defined by
\begin{equation}\label{Jtilde}
    \tilde{J}^1_{n,h,k}(t,s,T,x) = \sum_{i=1}^3\sum_{j=0}^1 f^1_{h,k,i,j}(s,T)V_{t,s,x,n}^j {\partial}_{x}^i
\end{equation}
with $f^1_{h,k,i,j}$ as in \eqref{J1}. Also in this case we have the alternative representation
\begin{align}
    J^{2,1}_{n,h,k}(t,T,x) =& \sum_{i=1}^{6}\sum_{j=0}^2f^{2,1}_{n,h,k,i,j}(t,T)(x-\bar{x})^j \pa_x^i \label{J21} \\
    J^{2,2}_{n,k}(t,T,x) =& \sum_{i=1}^{6}\sum_{j=0}^2f^{2,2}_{n,k,i,j}(t,T)(x-\bar{x})^j \pa_x^i,\label{J22}
\end{align}
with $f^{2,1}_{n,h,k,i,j}$ and $f^{2,2}_{n,k,i,j}$ deterministic functions.
\end{theorem}

\noindent{\it Proof.} We show a preliminary result: from formulae \eqref{J1} and \eqref{Jtilde}
for $J^1$ and $\tilde{J}^1$ respectively, it follows that
\begin{align}\nonumber
   & \int_{\R}\Gamma_n(t,x;s,\eta) J^1_{h,k}(s,T,\eta) \Gamma_{h+k}(s,\eta;T,y) d \eta =
\intertext{(by \eqref{d} and \eqref{V})}\nonumber
   & = \int_{\R}\tilde{J}^1_{n,h,k}(s,T,x) \Gamma_n(t,x;s,\eta) \Gamma_{h+k}(s,\eta;T,y) d \eta  \\ \nonumber
   & = \tilde{J}^1_{n,h,k}(s,T,x) \int_{\R}\Gamma_n(t,x;s,\eta) \Gamma_{h+k}(s,\eta;T,y) d \eta =
\intertext{(by \eqref{repr})} \label{15}
   & = \tilde{J}^1_{n,h,k}(s,T,x) \Gamma_{n+h+k}(x,t;T,y).
\end{align}
Now we have
  $$G^2(t,x;T,y) = I_1 + I_2, $$
where, proceeding as before,
{\allowdisplaybreaks
\begin{align*}
  I_1 &= \int_t^T\int_{\R}G^0(t,x;s,\eta) \a_1(s) (\eta-\bar{x}) (\pa_{\eta\eta}-\pa_{\eta}) G^1(s,\eta;T,y)d \eta d s \\
 &= \sum_{n,h,k=0}^{+\infty} \frac{\l^{n}}{n!} \int_t^T\a_1(s)e^{-\l(s-t)} (s-t)^n  \cdot  \\
 &\quad  \cdot V_{t,s,x,n}\int_{\R}\Gamma_n(t,x;s,\eta)(\pa_{\eta\eta}-\pa_{\eta}) J^1_{h,k}(s,T,\eta) \Gamma_{h+k}(s,\eta;T,y) d \eta d s \\
 &= \sum_{n,h,k=0}^{+\infty} \frac{\l^{n}}{n!} \int_t^T\a_1(s)e^{-\l(s-t)} (s-t)^n \cdot  \\
 &\quad  \cdot V_{t,s,x,n} (\pa_{xx}-\pa_x) \int_{\R}\Gamma_n(t,x;s,\eta) J^1_{h,k}(s,T,\eta) \Gamma_{h+k}(s,\eta;T,y) d \eta d s=
\intertext{(by \eqref{15})}
 &= \sum_{n,h,k=0}^{+\infty} \frac{\l^{n}}{n!} \int_t^T \a_1(s) e^{-\l(s-t)} (s-t)^n V_{t,s,x,n} (\pa_{xx}-\pa_x)
 \tilde{J}^1_{n,h,k}(s,T,x)d s\cdot\\
 &\quad\cdot \Gamma_{n+h+k}(x,t;T,y) \\
 &= \sum_{n,h,k=0}^{+\infty} J^{2,1}_{n,h,k}(t,T,x) \Gamma_{n+h+k}(t,x;T,y)
\end{align*}}
and
{\allowdisplaybreaks
\begin{align*}
  I_2 &= \int_t^T\int_{\R}G^0(t,x;s,\eta) \a_2(s) (\eta-\bar{x})^2 (\pa_{\eta\eta}-\pa_{\eta}) G^0(s,\eta;T,y)d \eta d s \\
 &=  e^{-\l(T-t)}\sum_{n,k=0}^{+\infty} \frac{\l^{n+k}}{n!k!} \int_t^T \a_2(s) (T-s)^k (s-t)^n \cdot\\
 &\quad\cdot V_{t,s,x,n}^2 \int_{\R}\Gamma_n(t,x;s,\eta) (\pa_{\eta\eta}-\pa_{\eta}) \Gamma_k(s,\eta;T,y)d \eta d s \\
 &= e^{-\l(T-t)} \sum_{n,k=0}^{+\infty} \frac{\l^{n+k}}{n!k!} \int_t^T \a_2(s) (T-s)^k (s-t)^n \cdot\\
 &\quad\cdot V_{t,s,x,n}^2 (\pa_{xx}-\pa_x)  \int_{\R}\Gamma_n(t,x;s,\eta) \Gamma_k(s,\eta;T,y)d \eta d s \\
 &= e^{-\l(T-t)} \sum_{n,k=0}^{+\infty} \frac{\l^{n+k}}{n!k!} \int_t^T \a_2(s) (T-s)^k (s-t)^n \cdot\\
 &\quad\cdot V_{t,s,x,n}^2d s\, (\pa_{xx}-\pa_x) \Gamma_{n+k}(t,x;T,y) \\
 &= \sum_{n,k=0}^{+\infty} J^{2,2}_{n,k}(t,T,x) \Gamma_{n+k}(t,x;T,y).
\end{align*}}
This concludes the proof.\qquad\endproof

\begin{remark}
Since the derivatives of a Gaussian density can be expressed in terms of Hermite polynomials, the
computation of the terms of the expansion \eqref{34} is very fast. Indeed, we have
\begin{equation*}
    \frac{\pa_x^i\Gamma_n(t,x;T,y)}{\Gamma_n(t,x;T,y)} = \frac{(-1)^{i}h_{i,n}(t,T,x-y)}{\left(2 \left(A(t,T) +n\d^2\right)
    \right)^{\frac{i}{2}}}
\end{equation*}
where
  $$h_{i,n}(t,T,z)=\mathbf{H}_{i}\lf \frac{z+(T-t){\m_{0}}-\frac12 A(t,T) +nm}{\sqrt{2 \left(A(t,T) +n\d^2\right)}}\rg$$
and $\mathbf{H}_{i}=\mathbf{H}_{i}(x)$ denotes the Hermite polynomial of degree $i$. Thus we can
rewrite the terms $\lf G^k\rg_{k=1,2}$ in \eqref{11} and \eqref{12} as follows:
{\allowdisplaybreaks
\begin{equation}\label{35}
\begin{split}
  G^1(t,x;T,y) =& \sum_{n,k=0}^{\infty}\mathbf{G}_{n,k}^1(t,x;T,y) \Gamma_{n+k}(t,x;T,y)  \\
  G^2(t,x;T,y) =& \sum_{n,h,k=0}^{\infty} \mathbf{G}_{n,h,k}^{2,1}(t,x;T,y) \Gamma_{n+h+k}(t,x;T,y)\\
  & +\sum_{n,k=0}^{\infty} \mathbf{G}_{n,k}^{2,2}(t,x;T,y) \Gamma_{n+k}(t,x;T,y),
\end{split}
\end{equation}
}
where {\allowdisplaybreaks
\begin{align*}
\mathbf{G}_{n,k}^1(t,x;T,y) =&\sum_{i=1}^{3}(-1)^{i}\sum_{j=0}^1f^1_{n,k,i,j}(t,T)(x-\bar{x})^j
\frac{h_{i,n+k}(t,T,x-y)}{\left(2\left( A(t,T) +(n+k)\d^2\right)\right)^{\frac{i}{2}}} \\
\mathbf{G}_{n,h,k}^{2,1}(t,x;T,y) =&
\sum_{i=1}^{6}(-1)^{i}\sum_{j=0}^1f^{2,1}_{n,h,k,i,j}(t,T)(x-\bar{x})^j
\frac{h_{i,n+h+k}(t,T,x-y)}{\left(2 \left(A(t,T) +(n+h+k)\d^2\right)\right)^{\frac{i}{2}}}  \\
\mathbf{G}_{n,k}^{2,2}(t,x;T,y) =&
\sum_{i=1}^{6}(-1)^{i}\sum_{j=0}^1f^{2,2}_{n,k,i,j}(t,T)(x-\bar{x})^j
\frac{h_{i,n+k}(t,T,x-y)}{\left(2 \left(A(t,T) +(n+k)\d^2\right)\right)^{\frac{i}{2}}}.
\end{align*}}
In the practical implementation, we truncate the series in \eqref{Gamma0} and \eqref{35} to a
finite number of terms, say $M\in\mathds{N}\cup\{0\}$. Therefore we put
\begin{equation*}
\begin{split}
  G^0_{M}(t,x;T,y) &= e^{-\l(T-t)} \sum_{n=0}^{M} \frac{(\l(T-t))^n}{n!}
  \Gamma_n(t,x;T,y),\\
  G^1_{M}(t,x;T,y) &= \sum_{n,k=0}^{M}\mathbf{G}_{n,k}^1(t,x;T,y) \Gamma_{n+k}(t,x;T,y), \\
  G^2_{M}(t,x;T,y) &= \sum_{n,h,k=0}^{M} \mathbf{G}_{n,h,k}^{2,1}(t,x;T,y) \Gamma_{n+h+k}(t,x;T,y)\\
  &\quad +\sum_{n,k=0}^{M} \mathbf{G}_{n,k}^{2,2}(t,x;T,y) \Gamma_{n+k}(t,x;T,y),
\end{split}
\end{equation*}
and we approximate the density $\G$ by
\begin{equation}\label{33}
  \Gamma^{2}_{M}(t,x;T,y):=G^0_{M}(t,x;T,y)+G^1_{M}(t,x;T,y)+G^2_{M}(t,x;T,y).
\end{equation}
\end{remark}

Next we denote by $C(t,S(t))$ the price at time $t<T$ of a European option with payoff function
$\p$ and maturity $T$; for instance,
  $\p(y)=\left(y-K\right)^{+}$
in the case of a Call option with strike $K$. From the expansion of the density in \eqref{33}, we
get the following second order approximation formula.
\begin{remark}
We have
  $$ C(t,S(t)) \approx e^{-r(T-t)} u_{M}(t,\log S(t)) $$
where
\begin{align}\nonumber
 u_{M}(t,x)
 \nonumber
 &=\int_{\R^+} \frac1S \Gamma_{ M }^2(t,x;T,\log S) \p(S) d S\\
 \nonumber
 &= e^{-\l(T-t)} \sum_{n=0}^{ M } \frac{(\l(T-t))^n}{n!} \mathrm{CBS}_n(t,x) \\ \nonumber
 &\quad +\sum_{n,k=0}^{ M } \left(J^1_{n,k}(t,T,x)+J^{2,2}_{n,k}(t,T,x)\right) \mathrm{CBS}_{n+k}(t,x) \\ \label{uM}
 &\quad + \sum_{n,h,k=0}^{ M } J^{2,1}_{n,h,k}(t,T,x) \mathrm{CBS}_{n+h+k}(t,x)
\end{align}
and $\mathrm{CBS}_n(t,x)$ is the BS price\footnote{Here the BS price is expressed as a function of the time $t$ and of the log-asset $x$.} under the Gaussian law
$\Gamma_n(t,x;T,\cdot)$ in \eqref{36}, namely
  $$\mathrm{CBS}_n(t,x) = \int_{\R^+} \frac1S \Gamma_n(t,x;T,\log S) \p(S) d S.$$
\end{remark}

\subsection{Simplified Fourier approach for LV models}
\label{sec:secsimpl}
Equation \eqref{X} with $J=0$ reduces to the standard SDE of a LV model. In this case we can
simplify the proof of Theorems \ref{t1}-\ref{t2} by using Fourier analysis methods. Let us first
notice that $L_{0}$ in \eqref{L0} becomes
\begin{equation}\label{40}
 L_{0}=\frac{\a_0(t)}{2} \lf\partial_{xx}-\partial_x\rg + r \partial_x + \pa_t,
\end{equation}
and its fundamental solution is the Gaussian density
  $$G^0(t,x;T,y) =\frac{1}{\sqrt{2\pi A(t,T)}}\,e^{-\frac{\lf x-y+(T-t)r-\frac12A(t,T)\rg^2}{2A(t,T)}},$$
with $A$ as in \eqref{36}.
\begin{corollary}[1st order expansion]\label{cor1}
In case of $\l=0$, the solution $G^1$ in \eqref{11} is given by
\begin{equation}\label{G1LV}
  G^1(t,x;T,y) = J^1(t,T,x) G^0(t,x;T,y)
  \end{equation}
where $J^1(t,T,x)$ is the differential operator
\begin{equation}\label{J1LV}
 J^1(t,T,x) = \int_t^T \a_1(s) V_{t,s,x} d s\, ({\partial}_{xx}-{\partial}_{x}),
\end{equation}
with $V_{t,s,x}\equiv V_{t,s,x,0}$ as in \eqref{38}, that is
  $$V_{t,T,x}f(x) =\left(x-\bar{x}+(T-t){r}-\frac12 A(t,T)\right)f(x)+ A(t,T)\pa_x f(x).$$
\end{corollary}

\proof Although the result follows directly from Theorem \ref{t1}, here we propose an alternative
proof
of formula \eqref{J1LV}.
The idea is to determine the
solution of the Cauchy problem \eqref{2.2} in the Fourier space, where all the computation can be
carried out more easily; then, using the fact that the leading term $G^{0}$ of the expansion is a
Gaussian kernel, we are able to compute explicitly the inverse Fourier transform to get back to the
analytic approximation of the transition density.

Since we aim at showing the main ideas of an alternative approach, for simplicity we only consider the
case of time-independent
coefficients, precisely we set $\a_{0}=2$ and $r=0$.
In this case we have
 $$L_{0}=\partial_{xx}-\partial_x  + \pa_t$$
and the related Gaussian fundamental solution is equal to
  $$G^{0}(t,x;T,y)=\frac{1}{\sqrt{4\pi (T-t)}}\,e^{-\frac{\lf x-y-(T-t)\rg^2}{4(T-t)}}.$$
Now we apply the Fourier transform (in the variable $x$) to the Cauchy problem \eqref{2.2} with
$k=1$ and we get
\begin{equation}\label{Cpb}
  \begin{cases}
    {\partial}_t\hat{G}^1(t,\x;T,y) &\hspace{-8pt}= \left(\x^2-i\x\right)\hat{G}^1(t,\x;T,y)\\
     &+\a_1(i{\partial}_{\x}+\bar{x}) \lf-\x^2+i\x\rg \hat{G}^0(t,\x;T,y),\\
    \hat{G}^1(T,\x;T,y) &\hspace{-13pt}= 0,  \qquad \x\in\R.
  \end{cases}
\end{equation}
Notice that
\begin{equation}\label{41}
  \hat{G}^0(t,\x;T,y)=e^{-\x^{2}(T-t)+i\x(y+(T-t))}.
\end{equation}
Therefore the solution to the ordinary differential equation \eqref{Cpb} is
\begin{align*}
 \hat{G}^1(t,\x;T,y)&=-\a_{1}\int_{t}^{T}e^{(s-t)(-\x^{2}+i\x)}(i{\partial}_{\x}+\bar{x})
 \left((-\x^{2}+i\x) \hat{G}^0(s,\x;T,y)\right)ds=
\intertext{(using the identity
$f(\x)(i{\partial}_{\x}+\bar{x})(g(\x))=(i{\partial}_{\x}+\bar{x})(f(\x)g(\x))-ig(\x){\partial}_{\x}f(\x)$)}
 &=-\a_{1}\int_{t}^{T}(i{\partial}_{\x}+\bar{x})\left((-\x^{2}+i\x)e^{(s-t)(-\x^{2}+i\x)}
 \hat{G}^0(s,\x;T,y)\right)ds\\
 &\quad+i\a_{1}\int_{t}^{T}(-\x^{2}+i\x)\hat{G}^0(s,\x;T,y){\partial}_{\x}e^{(s-t)(-\x^{2}+i\x)}ds=
\intertext{(by \eqref{41})}
 &=-\a_{1}\int_{t}^{T}(i{\partial}_{\x}+\bar{x})\left((-\x^{2}+i\x)e^{i\x(y+(T-t))-\x^{2}(T-t)}\right)ds\\
 &\quad+i\a_{1}\int_{t}^{T}(-\x^{2}+i\x)(s-t)(-2\x+i)e^{i\x(y+(T-t))-\x^{2}(T-t)}ds=
\intertext{(again by \eqref{41})}
 &=-\a_{1}(T-t)(i{\partial}_{\x}+\bar{x})\left((-\x^{2}+i\x)\hat{G}^0(t,\x;T,y)\right)\\
 &\quad+i\a_{1}\frac{(T-t)^2}{2}(-\x^{2}+i\x)(-2\x+i)\hat{G}^0(t,\x;T,y).
\end{align*}
Thus, inverting the Fourier transform, we get
 \begin{align*}
 G^1(t,x;T,y) &=\a_{1}(T-t)(x-\bar{x})({\partial}_x^{2}-{\partial}_x)G^0(t,x;T,y) + \\
 &\quad -\a_{1}\frac{(T-t)^2}{2}(-2{\partial}_x^3+3{\partial}_x^{2}-{\partial}_x)G^0(t,x;T,y) \\
 &=\a_{1}\left((T-t)^2{\partial}_x^3 + \lf(x-\bar{x})(T-t)-\frac32(T-t)^2\rg{\partial}_x^2 + \right.\\
 &\quad\left. +\lf-(x-\bar{x})(T-t)+\frac{(T-t)^2}{2}\rg{\partial}_x\right)G^0(t,x;T,y),
\end{align*}
where the operator acting on $G^0(t,x;T,y)$ is exactly the same as in \eqref{J1LV}.
\qquad\endproof

\begin{remark}
As in Remark \ref{r4}, operator $J^1(t,T,x)$ can also be rewritten in the form
\begin{equation}\label{37}
  J^1(t,T,x) = \sum_{i=1}^{3}\sum_{j=0}^1 f^1_{i,j}(t,T)(x-\bar{x})^j \pa_x^i,
\end{equation}
where $f^1_{i,j}$ are deterministic functions whose explicit expression can be easily derived.
\end{remark}
The previous argument can be used to prove the following second order expansion.
\begin{corollary}[2nd order expansion]\label{cor2}
In case of $\l=0$, the solution $G^2$ in \eqref{12} is given by
  $$ G^2(t,x;T,y) = J^2(t,T,x)G^0(t,x;T,y) $$
where
\begin{equation}\label{J2LV}
\begin{split}
  J^2(t,T,x) &= \int_t^T \a_1(s) V_{t,s,x} ({\partial}_{xx}-{\partial}_{x}) \tilde{J}^1(t,s,T,x) d s  \\
             & + \int_t^T \a_2(s) V_{t,s,x}^2 d s\, ({\partial}_{xx}-{\partial}_{x})
\end{split}
\end{equation}
and $\tilde{J}^1$ is the ``adjoint'' operator of $J^1$, defined by
  $$ \tilde{J}^1(t,s,T,x) = \sum_{i=1}^3\sum_{j=0}^1 f^1_{i,j}(s,T)V_{t,s,x}^j {\partial}_{x}^i $$
with $f^1_{i,j}$ as in \eqref{37}.
\end{corollary}

\begin{remark}
In a standard LV model, the leading operator of the approximation, i.e. $L_{0}$ in \eqref{40}, has
a Gaussian density $G^{0}$ and this allowed us to use the inverse Fourier transform in order to get the approximated density.
This approach does not work in the general case of models with jumps because typically the explicit expression of the fundamental solution of an integro-differential
equation is not available. On the other hand, for several L\'evy processes used in finance, the characteristic function is known explicitly even if the density is not. This suggests
that the argument used in this section may be adapted to obtain an approximation of the
characteristic function of the process instead of its density. This is what we are going to
investigate in \Sec{LV-J}.
\end{remark}

\section{Local L\'evy models}
\label{sec:LV-J}

In this section, we provide an expansion of the characteristic function for the local L\'evy
model \eqref{X}. We denote by
  $$\hat{\Gamma}(t,x;T,\x)=\F\left(\Gamma(t,x;T,\cdot)\right)(\x)$$
the Fourier transform, with respect to the second spatial variable, of the transition density
$\Gamma(t,x;T,\cdot)$; clearly, $\hat{\Gamma}(t,x;T,\x)$ is the characteristic function of
$X^{t,x}(T)$. Then, by applying the Fourier transform to the expansion \eqref{34}, we find
\begin{equation}\label{34b}
  \p_{X^{t,x}(T)}(\x)\, \approx \,
  \sum_{k=0}^n \hat{G}^k(t,x;T,\x).
\end{equation}
Now we recall that $G^{k}(t,x;T,y)$ is defined, as a function of the variables $(t,x)$, in terms
of the sequence of Cauchy problems \eqref{2.2}. Since the Fourier transform in \eqref{34b} is
performed with respect to the variable $y$, in order to take advantage of such a transformation it
seems natural to characterize $G^{k}(t,x;T,y)$ as a solution of the {\it adjoint operator} in the
dual variables $(T,y)$.

To be more specific, we recall the definition of adjoint operator. Let $L$ be the operator in
\eqref{L}; then its adjoint operator $\tilde{L}$ satisfies (actually, it is defined by) the
identity
  $$\int_{\R^{2}}u(t,x)Lv(t,x)dxdt=\int_{\R^{2}}v(t,x)\tilde{L}u(t,x)dxdt$$
for all $u,v\in C_{0}^{\infty}$. More explicitly, by recalling notation \eqref{a}, we have
\begin{align*}
  \tilde{L}^{(T,y)}u(T,y)&=\frac{a(T,y)}{2}{\partial}_{yy}u(T,y)+b(T,y){\partial}_y u(T,y)\\
  &\quad-{\partial}_Tu(T,y)+c(T,y)u(T,y)\\
  &\quad +\int_{\R}\left(u(T,y+z)-u(T,y)-z{\partial}_y u(T,y)\caratt_{\{|z|<1\}}\right)\bar{\n}(dz),
\end{align*}
where
  $$b(T,y)={\partial}_ya(T,y)-\left(\bar{r}-\frac{a(T,y)}{2}\right),\qquad c(T,y)=\frac12({\partial}_{yy}+{\partial}_y) a(T,y),$$
and $\bar{\n}$ is the L\'evy measure with reverted jumps, i.e. $\bar{\n}(dx)=\n(-dx)$. Here the
superscript in $\tilde{L}^{(T,y)}$ is indicative of the fact that the operator $\tilde{L}$ is
acting in the variables $(T,y)$.

By a classical result (cf., for instance, \cite{GarroniMenaldi1992}) the fundamental solution
$\G(t,x;T,y)$ of $L$ is also a solution of $\tilde{L}$ in the dual variables, that is
\begin{equation}\label{Lad}
 \tilde{L}^{(T,y)}\G(t,x;T,y)=0,\qquad t<T,\ x,y\in\R.
\end{equation}
Going back to approximation \eqref{34b}, the idea is to consider the series of the dual Cauchy
problems of \eqref{2.2} in order to solve them by Fourier-transforming in the variable $y$ and
finally get an approximation of $\p_{X^{t,x}(T)}$.

For sake of simplicity, from now on we only consider the case of time-independent coefficients:
the general case can be treated in a completely analogous way. First of all, 
we consider the integro-differential operator $L_0$ in \eqref{42}, which in this case becomes
\begin{equation}\label{L0F}
\begin{split}
  L_0^{(t,x)}u(t,x)&=\frac{\a_0}{2}({\partial}_{xx}-{\partial}_x)u(t,x)+\bar{r}{\partial}_xu(t,x)+{\partial}_tu(t,x)\\
  &\quad+\int_{\R}\left(u(t,x+y)-u(t,x)-y{\partial}_x u(t,x)\caratt_{\{|y|<1\}}\right)\n(dy),
\end{split}
\end{equation}
and its adjoint operator 
\begin{equation}\label{Ltilde0}
\begin{split}
  \tilde{L}_0^{(T,y)}u(T,y)&=\frac{\a_0}{2}({\partial}_{yy}+{\partial}_y)u(T,y)-\bar{r}{\partial}_y u(T,y)-{\partial}_T u(T,y)\\
  &\quad +\int_{\R}\left(u(T,y+z)-u(T,y)-z{\partial}_y u(T,y)\caratt_{\{|z|<1\}}\right)\bar{\n}(dz).
\end{split}
\end{equation}
By \eqref{Lad}, for any $(t,x)\in\R^{2}$, the fundamental solution $G^{0}(t,x;T,y)$ of $L_0$ solves the
dual Cauchy problem
\begin{equation}\label{50}
  \begin{cases}
     \tilde{L}_0^{(T,y)}G^{0}(t,x;T,y) = 0,\qquad &T>t,\ y\in\R,\\
     G^{0}(t,x;t,\cdot) = \d_x.
  \end{cases}
\end{equation}
It is remarkable that a similar result holds for the higher order terms of the approximation
\eqref{34b}. Indeed, let us denote by $L_{n}$ the $n^{\text{th}}$ order approximation of $L$ in
\eqref{43}:
\begin{equation}\label{43b}
 L_{n}=L_{0}+\sum_{k=1}^{n}\a_k(x-\bar{x})^k\lf\partial_{xx}-\partial_x\rg
\end{equation}
Then we have the following result.
\begin{theorem}
For any $k\ge 1$ and $(t,x)\in\R^{2}$, the function $G^k(t,x;\cdot,\cdot)$ in \eqref{2.2} is the
solution of the following dual Cauchy problem on $]t,+\infty[\times \R$
\begin{equation}\label{51}
  \begin{cases}
     \tilde{L}^{(T,y)}_{0} G^k(t,x;T,y)=- \sum\limits_{h=1}^k\left(\tilde{L}^{(T,y)}_{h}-\tilde{L}^{(T,y)}_{h-1}\right)
     G^{k-h}(t,x;T,y),\\
     G^k(t,x;t,y)= 0,  \qquad y\in\R,
  \end{cases}
\end{equation}
where
\begin{align*}
  \tilde{L}^{(T,y)}_{h}-\tilde{L}^{(T,y)}_{h-1}&=\a_h(y-\bar{x})^{h-2}
  \Big((y-\bar{x})^{2}\partial_{yy}+(y-\bar{x})\left(2h+(y-\bar{x})\right)\partial_y\\
  &\quad +h\left(h-1+y-\bar{x}\right)\Big).
\end{align*}
\end{theorem}
\proof
By the standard representation formula for the solutions of the {\it backward} parabolic Cauchy problem \eqref{2.2}, for $k\geq 1$ we have
\begin{equation}\label{999}
    G^k(t,x;T,y)=\sum_{h=1}^{k}\int_{t}^{T}\int_{\R}G^{0}(t,x;s,\y)M^{(s,\y)}_{h}G^{k-h}(s,\y;T,y)d\y ds,
\end{equation}
where to shorten notation we have set
\begin{equation*}
M^{(t,x)}_{h}=L^{(t,x)}_{h}-L^{(t,x)}_{h-1}.
\end{equation*}
By \eqref{50} and since
\begin{equation*}
 \tilde{M}^{(T,y)}_{h}=\tilde{L}^{(T,y)}_{h}-\tilde{L}^{(T,y)}_{h-1}.
\end{equation*}
the assertion is equivalent to
\begin{equation}\label{998}
 G^k(t,x;T,y)=\sum_{h=1}^{k}\int_{t}^{T}\int_{\R}G^{0}(s,\y;T,y)\tilde{M}^{(s,\y)}_{h}G^{k-h}(t,x;s,\y)d\y ds,
\end{equation}
where here we have used the representation formula for the solutions of the {\it forward} Cauchy problem \eqref{51} with $k\ge 1$.

We proceed by induction and first prove \eqref{998} for $k=1$. By \eqref{999} we have
\begin{align*}
    G^1(t,x;T,y)&=\int_{t}^{T}\int_{\R}G^{0}(t,x;s,\y)M^{(s,\y)}_{1}G^{0}(s,\y;T,y)d\y ds\\
    &=\int_{t}^{T}\int_{\R}G^{0}(s,\y;T,y)\tilde{M}^{(s,\y)}_{1}G^{0}(t,x;s,\y)d\y ds,
\end{align*}
and this proves \eqref{998} for $k=1$.

Next we assume that \eqref{998} holds for a generic $k> 1$ and we prove the thesis for $k+1$.
Again, by \eqref{999} we have
{\allowdisplaybreaks
\begin{align*}
    G^{k+1}(t,x;T,y)&=\sum_{j=1}^{k+1}\int_{t}^{T}\int_{\R}G^{0}(t,x;s,\y)M^{(s,\y)}_{j}G^{k+1-j}(s,\y;T,y)d\y ds\\
    &=\int_{t}^{T}\int_{\R}G^{0}(t,x;s,\y)M^{(s,\y)}_{k+1}G^{0}(s,\y;T,y)d\y ds\\
     &\quad+\sum_{j=1}^{k} \int_{t}^{T}\int_{\R}G^{0}(t,x;s,\y)M^{(s,\y)}_{j}G^{k+1-j}(s,\y;T,y)d\y ds=
\end{align*}}
(by the inductive hypothesis) 
{\allowdisplaybreaks
\begin{align*}
     &=\int_{t}^{T}\int_{\R}G^{0}(t,x;s,\y)M^{(s,\y)}_{k+1}G^{0}(s,\y;T,y)d\y ds\\
     &\quad+\sum_{j=1}^{k}\int_{t}^{T}\int_{\R}G^{0}(t,x;s,\y)M^{(s,\y)}_{j}\cdot\\
     &\quad\cdot\sum_{h=1}^{k+1-j}\int_{s}^{T}\int_{\R}G^{0}(\t,\z;T,y)\tilde{M}^{(\t ,\z)}_{h} G^{k+1-j-h}(s,\y;\t,\z)d\z d\t d\y ds\\
     &=\int_{t}^{T}\int_{\R}G^{0}(t,x;s,\y)M^{(s,\y)}_{k+1}G^{0}(s,\y;T,y)ds d\y\\
     &\quad+\sum_{h=1}^{k}\sum_{j=1}^{k+1-h}\int_{t}^{T}\int_{t}^{\t}\int_{\R^2}G^{0}(t,x;s,\y)G^{0}(\t,\z;T,y)\cdot\\
     &\quad\cdot M^{(s,\y)}_{j}\tilde{M}^{(\t ,\z)}_{h} G^{k+1-j-h}(s,\y;\t,\z)d\y d\z ds d\t \\
     &=\int_{t}^{T}\int_{\R}G^{0}(s,\y;T,y)\tilde{M}^{(s,\y)}_{k+1}G^{0}(t,x;s,\y)ds d\y\\
     &\quad+\sum_{h=1}^{k}\int_{t}^{T}\int_{\R}G^{0}(\t,\z;T,y)\tilde{M}^{(\t ,\z)}_{h} \cdot\\
     &\quad\cdot\left(\sum_{j=1}^{k+1-h}\int_{t}^{\t}\int_{\R}G^{0}(t,x;s,\y)M^{(s,\y)}_{j} G^{k+1-h-j}(s,\y;\t,\z)d\y ds\right) d\z d\t=
     \intertext{(again by \eqref{999})}
     &=\int_{t}^{T}\int_{\R}G^{0}(t,\y;T,y)\tilde{M}^{(s,\y)}_{k+1}G^{0}(t,x;s,\y)ds d\y\\
     &\quad+\sum_{h=1}^{k}\int_{t}^{T}\int_{\R}G^{0}(\t,\z;T,y)\tilde{M}^{(\t ,\z)}_{h} G^{k+1-h}(t,x;\t,\z) d\z d\t \\
     &=\sum_{h=1}^{k+1}\int_{t}^{T}\int_{\R}G^{0}(\t,\z;T,y)\tilde{M}^{(\t ,\z)}_{h} G^{k+1-h}(t,x;\t,\z) d\z d\t. \qquad
\end{align*}}
\endproof

Next we solve problems \eqref{50}-\eqref{51} by applying the Fourier transform in the variable $y$
and using the identity
\begin{equation}\label{52}
  \F_{y}\left(\tilde{L}_0^{(T,y)}u(T,y)\right)(\x)=\psi(\x)\hat{u}(T,\x)-\partial_{T}\hat{u}(T,\x),
\end{equation}
where
\begin{equation}\label{53}
    \psi(\x)=-\frac{\a_0}{2}(\x^2+i\x)+i\bar{r}\x+\int_{\R}\left(e^{iz\x}-1-iz\x\caratt_{\{|z|<1\}}\right)\n(dz).
\end{equation}
We remark explicitly that $\psi$ is the characteristic exponent of the L\'evy process
\begin{equation}\label{60}
 d X^{0}(t)=\left(\bar{r}-\frac{\a_0}{2}\right)d t+ \sqrt{\a_0} d W(t) +d J(t),
\end{equation}
whose Kolmogorov operator is $L^{0}$ in \eqref{L0F}.
Then:
\begin{enumerate}[(i)]
  \item from \eqref{50} we obtain the ordinary differential equation
\begin{equation}\label{CpbF0}
  \begin{cases}
     \partial_T\hat{G}^{0}(t,x;T,\xi)=\psi(\x)\hat{G}^{0}(t,x;T,\xi),\qquad T>t,\\
     \hat{G}^{0}(t,x;t,\xi) = e^{i\x x}.
  \end{cases}
\end{equation}
with solution
\begin{equation}\label{53b}
    \hat{G}^{0}(t,x;T,\xi)=e^{i\x x+(T-t)\psi(\x)}
\end{equation}
which is the $0^{\text{th}}$ order approximation of the characteristic function
$\p_{X^{t,x}(T)}$.

  \item from \eqref{51} with $k=1$, we have
\begin{equation*}
  \begin{cases}
     {\partial}_T\hat{G}^1(t,x;T,\x) \hspace{-9pt} &= \psi(\x)\hat{G}^1(t,x;T,\x)\\
     \hspace{-9pt} &\quad+ \a_1\lf (i{\partial}_{\x}+\bar x)(\x^2+i\x)-2i\x+1 \rg \hat{G}^0(t,x;T,\xi) \\
     \hspace{15pt}\hat{G}^1(t,x;t,\x)\hspace{-9pt} & = 0,
  \end{cases}
\end{equation*}
with solution
    $$\hat{G}^1(t,x;T,\x) = \int_t^Te^{\psi(\x)(T-s)}\a_1 \lf (i{\partial}_{\x}+\bar x)(\x^2+i\x)-2i\x+1 \rg \hat{G}^0(t,x;s,\xi)
    ds=$$
(by \eqref{53b})
\begin{align}\nonumber
    &= -e^{ix\x+\psi(\x)(T-t)}\a_{1} \int_t^T  (\xi^{2}+i\x ) \left(x-\bar{x}-i (s-t)
    \psi'(\xi)\right)ds\\ \label{54}
    &=-\hat{G}^0(t,x;T,\xi)\a_1 (T-t)(\xi^{2}+i\x) \left(x- \bar{x}-\frac{i}{2}(T-t)
    \psi'(\xi)\right),
\end{align}
 which is the first order term in the expansion \eqref{34b}.

 \item regarding \eqref{51} with $k=2$, a straightforward computation based on analogous arguments
 shows that the second order term in the expansion \eqref{34b} is given by
\begin{equation}\label{55}
    \hat{G}^2(t,x;T,\xi)=\hat{G}^0(t,x;T,\xi)\sum_{j=0}^{2}g_{j}(T-t,\x)(x-\bar{x})^{j}
\end{equation}
where
{\allowdisplaybreaks
\begin{align*}
    g_{0}(s,\x)&=\frac{1}{2}s^{2} \a_2 \xi  (i+\xi ) \psi''(\xi)\\
     &\quad -\frac{1}{6}s^{3} \xi  (i+\xi ) \psi''(\xi)\left(\a_1^2 (i+2 \xi )-2 \a_2 \psi''(\xi)+\a_1^2 \xi  (i+\xi )\right)\\
     &\quad -\frac{1}{8}s^{4} \a_1^2 \xi^2 (i+\xi )^2 \psi''(\xi)^2,\\
    g_{1}(s,\x)&= \frac{1}{2}s^{2} \xi  (i+\xi ) \left(\a_1^2 (1-2 i \xi )+2 i \a_2 \psi''(\xi)\right)\\
    &\quad -\frac{1}{2}s^{3} i \a_1^2 \xi ^2 (i+\xi )^2 \psi''(\xi),\\
    g_{2}(s,\x)&=-\a_2 s\xi  (i+\xi )+ \frac{1}{2}s^{2} \a_1^2 \xi ^2 (i+\xi )^2.
\end{align*}}
\end{enumerate}
Plugging \eqref{53b}-\eqref{54}-\eqref{55} into \eqref{34b}, we finally get the second order
approximation of the characteristic function of $X$. In Subsection \ref{HOA}, we also provide the
expression of $\hat{G}^k(t,x;T,\xi)$ for $k=3,4$, appearing in the $4^{\text{th}}$ order
approximation.
\begin{remark}
The basepoint $\bar{x}$ is a parameter which can be
freely chosen in order to sharpen the accuracy of the approximation. In general, the simplest
choice $\bar{x}=x$ seems to be sufficient to get very accurate results.
\end{remark}
\begin{remark}
To overcome the use of the adjoint operators, it would be interesting to investigate an
alternative approach to the approximation of the characteristic function based of the following
remarkable symmetry relation valid for time-homogeneous diffusions
\begin{equation}\label{56}
 m(x)\G(0,x;t,y)=m(y)\G(0,y;t,x)
\end{equation}
where $m$ is the so-called density of the speed measure
 $$m(x)=\frac{2}{\s^{2}(x)}\exp\left(\int_{1}^{x}\left(\frac{2r}{\s^{2}(z)}-1\right)dz\right).$$
Relation \eqref{56} is stated in \cite{ItoMcKean1974} and a complete proof can be found in
\cite{EkstromTysk2011}.
\end{remark}

For completeness, we close this section by stating an integral pricing formula for European
options proved by Lewis \cite{Lewis2001}; the formula is  given in terms of the characteristic
function of the underlying log-price process. Formula below (and other Fourier-inversion methods
such as the standard, fractional FFT algorithm or the recent COS method \cite{Oosterlee2008}) can
be combined with the expansion \eqref{34b} to price and hedge efficiently hybrid LV models with
L\'evy jumps.

We consider a risky asset $S(t)=e^{X(t)}$ where $X$ is the process whose risk-neutral dynamics under a
martingale measure $Q$ is given by \eqref{X}. We denote by
  $H(t,S(t))$
the price at time $t<T$, of a European option with underlying asset $S$, maturity $T$ and payoff
$f=f(x)$ (given as a function of the log-price): to fix ideas, for a Call option with strike $K$
we have
 $$f^{\text{Call}}(x)=\left(e^{x}-K\right)^{+}.$$
The following theorem is a classical result which can be found in several textbooks (see, for
instance, \cite{Pascucci2011book}).
\begin{theorem}\label{t10}
 Let
   $$f_{\g}(x)=e^{-\g x}f(x)$$
 and assume that there exists $\g\in\R$ such that
\begin{enumerate}
  \item[{\it i)}] $f_{\g},\hat{f}_{\g}\in L^{1}(\R)$;
  \item[{\it ii)}] $E^{Q}\left[S(T)^{\g}\right]$ is finite.
\end{enumerate}
Then, the following pricing formula holds:
  $$H(t,S(t))=\frac{e^{-r(T-t)}}{\pi}\int_{0}^{\infty}\hat{f}(\x+i\g)\p_{X^{t,\log S(t)}(T)}(-(\x+i\g))d\x.$$
\end{theorem}
For example, $f^{\text{Call}}$ verifies the assumptions of Theorem \ref{t10} for any $\g>1$ and we
have
  $$\hat{f}^{\text{Call}}(\x+i\g)=\frac{K^{1-\g}e^{i\x \log K}}{\left(i\x-\g\right)\left(i\x-\g+1\right)}.$$
Other examples of typical payoff functions and the related Greeks can be found in
\cite{Pascucci2011book}.

\subsection{High order approximations}\label{HOA} The analysis of \Sec{LV-J} can be carried out to get approximations
of arbitrarily high order. Below we give the more accurate (but more complicated) formulae up to
the $4^{\text{th}}$ order that we used in the numerical section. In particular we give the
expression of $\hat{G}^k(t,x;T,\xi)$ in \eqref{34b} for $k=3,4$. For simplicity, we only consider
the case of time-homogeneous coefficients and $\xbar=x$.

We have
    $$\hat{G}^3(t,x;T,\xi)=\hat{G}^0(t,x;T,\xi)\sum_{j=3}^{7}g_{j}(\x)(T-t)^{j}$$
where {\allowdisplaybreaks
\begin{align*}
 g_3(\xi)&= \frac{1}{2} \a_3 (1-i \xi ) \xi  \psi^{(3)}(\x),\\
 g_4(\xi)&=  \frac{1}{6} i \xi  (i+\xi ) \bigg(2 \psi'(\x) \left(\a_1 \a_2-3 \a_3
   \psi''(\x)\right)\\
   &\quad+\a_1 \a_2 \left(3(i+2 \xi ) \psi''(\x)+2 \xi  (i+\xi ) \psi^{(3)}(\x)\right)\bigg),\\
 g_5(\xi)&=  \frac{1}{24} (1-i \xi ) \xi  \Big(-8 \a_1 \a_2 (i+2 \xi ) \psi'(\x)^2+6 \a_3
   \psi'(\x)^3\\
   &\quad+\a_1 \psi'(\x) \left(\a_1^2 (-1+6 \xi  (i+\xi ))-16 \a_2 \xi  (i+\xi )
   \psi''(\x)\right)\\
   &\quad+\a_1^3 \xi  (i+\xi ) \left(3( i+2 \xi ) \psi''(\x)+\xi  (i+\xi )
    \psi^{(3)}(\x)\right)\Big),\\
 g_6(\xi)&=  -\frac{1}{12} i \a_1 \xi ^2 (i+\xi )^2 \psi'(\x) \Big(\a_1^2 (i+2 \xi ) \psi'(\x)\\
 &\quad-2 \a_2 \psi'(\x)^2+\a_1^2 \xi  (i+\xi ) \psi''(\x)\Big),\\
 g_7(\xi)&=  -\frac{1}{48} i \left(\a_1 \xi (i+\xi ) \psi'(\x)\right)^3.
\end{align*}}
Moreover, we have
    $$\hat{G}^4(t,x;T,\xi)=\hat{G}^0(t,x;T,\xi)\sum_{j=3}^{9}g_{j}(\x)(T-t)^{j}$$
where {\allowdisplaybreaks
\begin{align*}
 g_3(\xi)&= -\frac{1}{2} \a_{4} \xi  (i+\xi ) \psi^{(4)}(\x),\\
 g_4(\xi)&= \frac{1}{6} \xi  (i+\xi ) \Big(2 \psi''(\x) \left(\a_2^2+3 \a_1 \a_3-3 \a_4
   \psi''(\x)\right)\\
   &\quad +2 \left(\left(\a_2^2+2 \a_1 \a_3\right) (i+2 \xi )-4 \a_4
   \psi'(\x)\right) \psi^{(3)}(\x)\\
   &\quad +\left(\a_2^2+2 \a_1 \a_3\right) \xi  (i+\xi )
   \psi^{(4)}(\x)\Big),\\
 g_5(\xi)&=  -\frac{1}{24} \xi  (i+\xi ) \Big(\a_1^2 \a_2 (-7+44 \xi  (i+\xi
   )) \psi''(\x)\\
   &\quad -\left(7 \a_2^2+15 \a_1 \a_3\right) \xi  (i+\xi ) \psi''(\x)^2\\
   &\quad -2
   \psi'(\x)^2 \left(2 \a_2^2+9 \a_1 \a_3-18 \a_4 \psi''(\x)\right)\\
   &\quad +\psi'(\x) \Big((i+2
   \xi ) \left(8 \a_1^2 \a_2-\left(14 \a_2^2+33 \a_1 \a_3\right) \psi''(\x)\right)\\
   &\quad-\left(10
   \a_2^2+21 \a_1 \a_3\right) \xi  (i+\xi ) \psi^{(3)}(\x)\Big)\\
   &\quad +3 \a_1^2 \a_2 \xi  (i+\xi )
   \left(4(i+2 \xi ) \psi^{(3)}(\x)+\xi  (i+\xi ) \psi^{(4)}(\x)\right)\Big),\\
 g_6(\xi)&=  \frac{1}{120} \xi  (i+\xi ) \Big(2 \left(8 \a_2^2+21 \a_1 \a_3\right) (i+2 \xi )
   \psi'(\x)^3-24 \a_4 \psi'(\x)^4\\
   &\quad +2 \psi'(\x)^2 \left(\a_1^2 \a_2 (11-70 \xi  (i+\xi
   ))+\left(26 \a_2^2+57 \a_1 \a_3\right) \xi  (i+\xi ) \psi''(\x)\right)\\
   &\quad +\a_1^2 \psi'(\x)
   \Big((i+2 \xi ) \left(\a_1^2 (-1+12 \xi  (i+\xi ))-112 \a_2 \xi  (i+\xi ) \psi''(\x)\right)\\
   &\quad -38 \a_2 \xi ^2 (i+\xi )^2 \psi^{(3)}(\x)\Big)+\a_1^2 \xi  (i+\xi ) \Big(\a_1^2 (-7+36 \xi  (i+\xi
   )) \psi''(\x)\\
   &\quad -26 \a_2 \xi  (i+\xi ) \psi''(\x)^2+\a_1^2 \xi  (i+\xi ) \left(6 (i+2 \xi )
   \psi^{(3)}(\x)+\xi  (i+\xi ) \psi^{4}(\x)\right)\Big)\Big),\\
 g_7(\xi)&= \frac{1}{144} \xi ^2 (i+\xi )^2 \Big(-32 \a_1^2 \a_2 (i+2 \xi ) \psi'(\x)^3+2
   \left(4 \a_2^2+9 \a_1 \a_3\right) \psi'(\x)^4\\
   &\quad +2 \a_1^4 \xi ^2 (i+\xi )^2 \psi''(\x)^2\\
   &\quad+\a_1^2 \psi'(\x)^2 \left(\a_1^2 (-5+26 \xi  (i+\xi ))-47 \a_2 \xi  (i+\xi )
   \psi''(\x)\right)\\
   &\quad +\a_1^4 \xi  (i+\xi ) \psi'(\x) \left(13 (i+2 \xi ) \psi''(\x)+3 \xi
   (i+\xi ) \psi^{(3)}(\x)\right)\Big),\\
 g_8(\xi)&= \frac{1}{48} \a_1^2 \xi ^3 (i+\xi )^3 \psi'(\x)^2 \Big(\a_1^2 (i+2 \xi )
   \psi'(\x)\\
   &\quad -2 \a_2 \psi'(\x)^2+\a_1^2 \xi  (i+\xi ) \psi''(\x)\Big),\\
 g_9(\xi)&=  \frac{1}{384} \a_1^4 \xi ^4 (i+\xi )^4 \psi'(\x)^4.
   \end{align*}

\section{Numerical tests}
\label{sec:numeric}

In this section our approximation formulae \eqref{34b} are tested and compared with a standard
Monte Carlo method. We consider up to the $4^{\text{th}}$ order expansion (i.e. $n=4$ in
\eqref{34b}) even if in most cases the $2^{\text{nd}}$ order seems to be sufficient to get very
accurate results. We analyze the case of a constant elasticity of variance (CEV) volatility
function with L\'evy jumps of Gaussian or Variance-Gamma type. Thus, we consider the log-price
dynamics \eqref{X} with
 $$\sigma(t,x)=\sigma_{0} e^{(\b-1)x},\quad\b\in[0,1],\ \s_{0}>0,$$
and $J$ as in Examples \ref{ex3} and \ref{ex4} respectively. In our experiments we assume the
following values for the parameters:
\begin{enumerate}[(i)]
  \item $S_0=1$ (initial stock price);
  \item $r=5\%$ (risk-free rate)
  \item $\s_{0}=20\%$ (CEV volatility parameter);
   \item $\b=\frac{1}{2}$ (CEV exponent).
\end{enumerate}
In order to present realistic tests, we allow the range of strikes to vary over the maturities;
specifically, we consider extreme values of the strikes where Call prices are of the order of
$10^{-3}S_{0}$, that is we consider deep-out-of-the-money options which are very close to be
worthless. To compute the reference values, we use an Euler-\MC method with $10$ millions
simulations and $250$ time-steps per year.

\subsection{Tests under CEV-Merton dynamics}
In the CEV-Merton model of Example \ref{ex3}, we consider the following set of parameters:
\begin{enumerate}[(i)]
  \item $\l=30\%$ (jump intensity);
  \item $m=-10\%$ (average jump size);
  \item $\d=40\%$ (jump volatility).
\end{enumerate}
In Table \ref{tab:MertonCEV}, we give detailed numerical results, in terms of prices and implied
volatilities, about the accuracy of our fourth order formula (PPR-$4^{\text{th}}$) compared with
the bounds of the Monte Carlo $95\%$-confidence interval.
\begin{table}[htb]
  \centering

{\footnotesize
\begin{tabular}{c|c|c l@{ -- }l |c l@{ -- }l}
  \hline\hline
  & & \multicolumn{3}{c}{Call prices} & \multicolumn{3}{c}{Implied volatility (\%)}
  \\[1ex]
  $T$&$K$ & PPR-$4^{\text{th}}$ & \multicolumn{2}{c}{MC-$95\%$ c.i.}
  & PPR-$4^{\text{th}}$ & \multicolumn{2}{c}{MC-$95\%$ c.i.}
  \\[1ex]
  \hline \hline
 & 0.5 & 0.50669 & 0.50648 & 0.50666 & 57.81 & 54.03 & 57.31
 \\ &0.75 & 0.26324 & 0.26304 & 0.26321 & 37.91 & 37.48 & 37.84
 \\ 0.25 & 1 & 0.05515 & 0.05501 & 0.05514 & 24.58 & 24.50 & 24.57
 \\ & 1.25 & 0.00645 & 0.00637 & 0.00645 & 30.48 & 30.39 & 30.49
 \\ & 1.5 & 0.00305 & 0.00300 & 0.00306 & 42.05 & 41.93 & 42.07
 \\ \hline & 0.5 & 0.52720 & 0.52700 & 0.52736 & 38.82 & 38.35 & 39.20
 \\ & 1 & 0.13114 & 0.13097 & 0.13125 & 27.06 & 27.01 & 27.08
 \\ 1 & 1.5 & 0.01840 & 0.01836 & 0.01852 & 29.04 & 29.03 & 29.10
 \\ & 2 & 0.00566 & 0.00566 & 0.00575 & 34.45 & 34.45 & 34.55
 \\ & 2.5 & 0.00209 & 0.00208 & 0.00214 & 37.65 & 37.62 & 37.77
 \\ \hline & 0.5 & 0.72942 & 0.72920 & 0.73045 & 32.88 & 32.81 & 33.21
 \\ & 1 & 0.52316 & 0.52293 & 0.52411 & 29.67 & 29.64 & 29.80
 \\ 10 & 5 & 0.05625 & 0.05604 & 0.05664 & 26.12 & 26.09 & 26.17
 \\ & 7.5 & 0.02267 & 0.02246 & 0.02290 & 26.34 & 26.30 & 26.39
 \\ & 10 & 0.01241 & 0.01091 & 0.01126 & 27.05 & 26.54 & 26.66
 \\ \hline \hline
\end{tabular}
  \caption{Call prices and implied volatilities in the CEV-Merton model for the fourth order formula (PPR-$4^{\text{th}}$) and
  the Monte Carlo (MC-$95\%$) with 10 millions simulations using Euler scheme with 250 time
steps per year, expressed as a function of strikes at the expiry T = 3M, 1Y, 10Y. Parameters:
$S_0=1$ (initial stock price), $r=5\%$ (risk-free rate), $\s_{0}=20\%$ (CEV volatility parameter),
$\b=\frac{1}{2}$ (CEV exponent), $\l=30\%$ (jump intensity), $m=-10\%$ (average jump size),
$\d=40\%$ (jump volatility). }
  \label{tab:MertonCEV}
}
\end{table}

Figures \ref{fig1}, \ref{fig2} and \ref{fig3} show the performance of the $1^{\text{st}}$,
$2^{\text{nd}}$ and $3^{\text{rd}}$ approximations against the Monte Carlo $95\%$ and $99\%$
confidence intervals, marked in dark and light gray respectively. In particular, Figure \ref{fig1}
shows the cross-sections of absolute (left) and relative (right) errors 
for the price of a Call with short-term maturity $T=0.25$ and strike $K$ ranging from $0.5$ to
$1.5$. The relative error is defined as
  $$\frac{\text{Call}^{\text{approx}}-\text{Call}^{\text{MC}}}{\text{Call}^{\text{MC}}}$$
where $\text{Call}^{\text{approx}}$ and $\text{Call}^{\text{MC}}$ are the approximated and Monte
Carlo prices respectively. In Figure \ref{fig2} we repeat the test for the medium-term maturity
$T=1$ and the strike $K$ ranging from $0.5$ to $2.5$. Finally in Figure \ref{fig3} we consider the
long-term maturity $T=10$ and the strike $K$ ranging from $0.5$ to $4$.

Other experiments that are not reported here, show that the $2^{\text{nd}}$ order expansion
\eqref{33}, which is valid only in the case of Gaussian jumps, gives the same results as formula
\eqref{34b} with $n=2$, at least if the truncation index $M$ is suitable large, namely $M\ge 8$
under standard parameter regimes. For this reason we have only used formula \eqref{34b} for our
tests.

\subsection{Tests under CEV-Variance-Gamma dynamics}
In this subsection we repeat the previous tests in the case of the CEV-Variance-Gamma model.
Specifically, we consider the following set of parameters:
\begin{enumerate}[(i)]
  \item $\kappa=15\%$ (variance of the Gamma subordinator);
  \item $\th=-10\%$ (drift of the Brownian motion);
  \item $\s=20\%$ (volatility of the Brownian motion).
\end{enumerate}
Analogously to Table \ref{tab:MertonCEV}, in Table \ref{tab:VGCEV} we compare our Call price
formulas with a high-precision Monte Carlo approximation (with $10^{7}$ simulations and $250$
time-steps per year) for several strikes and maturities. For both the price and the implied
volatility, we report our $4^{\text{th}}$ order approximation (PPR $4^{\text{th}}$) and the
boundaries of the Monte Carlo $95\%$-confidence interval.
\begin{table}[htb]
  \centering

{\footnotesize
\begin{tabular}{c|c|c l@{ -- }l |c l@{ -- }l}
  \hline\hline
  & & \multicolumn{3}{c}{Call prices} & \multicolumn{3}{c}{Implied volatility (\%)}
  \\[1ex]
  $T$&$K$ & PPR $4^{\text{th}}$ & \multicolumn{2}{c}{MC 95\% c.i.}
  & PPR $4^{\text{th}}$ & \multicolumn{2}{c}{MC 95\% c.i.}
  \\[1ex]
  \hline \hline
 & 0.8 & 0.23708 & 0.23704 & 0.23722 & 55.61 & 55.57 & 55.72
 \\ & 0.9 & 0.15489 & 0.15482 & 0.15497 & 47.09 & 47.05 & 47.14
 \\ 0.25 & 1 & 0.08413 & 0.08403 & 0.08415 & 39.29 & 39.24 & 39.30
 \\ & 1.1 & 0.03436 & 0.03426 & 0.03433 & 33.27 & 33.22 & 33.26
 \\ & 1.2 & 0.00968 & 0.00961 & 0.00965 & 29.28 & 29.21 & 29.25
 \\ \hline & 0.5 & 0.54643 & 0.54630 & 0.54679 & 61.02 & 60.91 & 61.30
 \\ & 0.75 & 0.35456 & 0.35438 & 0.35479 & 52.35 & 52.28 & 52.44
 \\ 1 & 1 & 0.20071 & 0.20049 & 0.20082 & 45.42 & 45.36 & 45.45
 \\ & 1.5 & 0.03394 & 0.03374 & 0.03387 & 35.16 & 35.09 & 35.14
 \\ & 2 & 0.00188 & 0.00185 & 0.00188 & 29.08 & 29.01 & 29.07
 \\ \hline & 0.5 & 0.80150 & 0.80279 & 0.80502 & 52.60 & 52.95 & 53.53
 \\ &  1 & 0.66691 & 0.66775 & 0.66990 & 49.09 & 49.21 & 49.52
 \\ 10 & 5 & 0.22948 & 0.22836 & 0.22986 & 42.02 & 41.93 & 42.05
 \\ & 7.5 & 0.13680 & 0.13497 & 0.13618 & 40.34 & 40.17 & 40.29
 \\ & 10 & 0.08664 & 0.08418 & 0.08518 & 39.21 & 38.93 & 39.05
\end{tabular}
  \caption{Call prices and implied volatilities in the CEV-Variance-Gamma model for the fourth order formula (PPR-$4^{\text{th}}$) and
  the Monte Carlo (MC-$95\%$) with 10 millions simulations using Euler scheme with 250 time
steps per year, expressed as a function of strikes at the expiry T = 3M, 1Y, 10Y. Parameters:
$S_0=1$ (initial stock price), $r=5\%$ (risk-free rate), $\s_{0}=20\%$ (CEV volatility parameter),
$\b=\frac{1}{2}$ (CEV exponent), $\kappa=15\%$ (variance of the Gamma subordinator), $\th=-10\%$
(drift of the Brownian motion), $\s=20\%$ (volatility of the Brownian motion).}
  \label{tab:VGCEV}
}
\end{table}

Figures \ref{fig4}, \ref{fig5} and \ref{fig6} show the cross-sections of absolute (left) and
relative (right) errors of the $2^{\text{nd}}$, $3^{\text{rd}}$ and $4^{\text{th}}$ approximations
against the Monte Carlo $95\%$ and $99\%$ confidence intervals, marked in dark and light gray
respectively. Notice that, for longer maturities and deep out-of-the-money options, the lower
order approximations give good results in terms of absolute errors  but only the $4^{\text{th}}$
order approximation lies inside the confidence regions. For a more detailed comparison, in Figures
\ref{fig5} and \ref{fig6} we plot the $2^{\text{nd}}$ (dotted line), $3^{\text{rd}}$ (dashed
line), $4^{\text{th}}$ (solid line) order approximations. Similar results are obtained for a wide
range of parameter values.

\section{Appendix: proof of Theorem \ref{t11}}
\label{sec:app}
In this appendix we prove Theorem \ref{t11} under Assumption A$_{N+1}$ where $N\in\NN$ is fixed.
For simplicity we only consider the case of $r=0$ and time-homogeneous coefficients. Recalling
notation \eqref{43bis}, we put
\begin{equation}\label{42b}
\begin{split}
  L_{0}=\frac{\a_0}{2} \lf\partial_{xx}-\partial_x\rg +{\partial}_{t}
\end{split}
\end{equation}
and
\begin{equation}\label{43ba}
  L_{n}=L_{0}+\sum_{k=1}^{n}\a_k(x-\xbar)^k\lf\partial_{xx}-\partial_x\rg, \qquad  n\le N.
\end{equation}

Our idea is to modify and adapt the standard characterization of the fundamental solution given by
the parametrix method originally introduced  by Levi \cite{Levi1907}. The parametrix method is a constructive technique
that allows to prove the existence of the fundamental solution $\G$ of a parabolic operator with
variable coefficients of the form
  $$Lu(t,x)= \frac{a(x)}{2}\left({\partial}_{xx}-{\partial}_{x}\right)u(t,x)+{\partial}_{t}u(t,x).$$
In the standard parametrix method, for any fixed $\x\in\R$, the fundamental solution $\G_{\x}$ of
the frozen operator
  $$L_{\x}u(t,x)= \frac{a(\x)}{2}\left({\partial}_{xx}-{\partial}_{x}\right)u(t,x)+{\partial}_{t}u(t,x)$$
is called a {\it parametrix} for $L$. A fundamental solution $\G(t,x;T,y)$ for $L$ can be
constructed starting from $\G_{y}(t,x;T,y)$ by means of an iterative argument and by suitably
controlling the errors of the approximation.

Our main idea is to {\it use the $N^{\text{th}}$-order approximation $\G^{N}(t,x;T,y)$ in
\eqref{34}-\eqref{2.2} (related to $L_{n}$ in \eqref{42b}-\eqref{43ba}) as a parametrix.} In order
to prove the error bound \eqref{81}, we carefully generalize some Gaussian estimates: in
particular, for $N=0$ we are back into the classical framework, but in general we need accurate
estimates of the solutions of the nested Cauchy problems \eqref{2.2}.

By analogy with the classical approach (see, for instance, \cite{Friedman} or the recent and more
general presentation in \cite{DiFrancescoPascucci2}), we have that $\G$ takes the form
  $$\G(t,x;T,y)=\G^{N}(t,x;T,y)+\int_{t}^{T}\int_{\R}\G^{0}(t,x;s,\x)\Phi^{N}(s,\x;T,y)d\x ds$$
where $\Phi^{N}$ is the function in \eqref{101} below, which is determined by imposing the
condition $L\G=0$. More precisely, we have
  $$0=L\G(z;\z)=L\G^{N}(z;\z)+\int_{t}^{T}\int_{\R}L\G^{0}(z;w)\Phi^{N}(w;\z)dw-\Phi^{N}(z;\z),$$
where, to shorten notations, we have set $z=(t,x)$, $w=(s,\x)$ and $\z=(T,y)$. Equivalently, we
have
 $$\Phi^{N}(z;\z)=L\G^{N}(z;\z)+\int_{t}^{T}\int_{\R}L\G^{0}(z;w)\Phi^{N}(w;\z)dw$$
and therefore by iteration
\begin{equation}\label{101}
 \Phi^{N}(z;\z)=\sum_{n=0}^{\infty}Z_{n}(z;\z)
\end{equation}
where
\begin{align*}
  Z^{N}_{0}(z;\z) & =L\G^{N}(z;\z), \\
  Z^{N}_{n+1}(z;\z)& =\int_{t}^{T}\int_{\R}L\G^{0}(z;w)Z_{n}(w;\z)dw.
\end{align*}
The thesis is a consequence of the following lemmas.
\begin{lemma}\label{lemapp1}
For any $n\le N$ the solution of \eqref{2.2}, with $L_{n}$ as in \eqref{42b}-\eqref{43ba}, takes the form
\begin{equation}\label{Gnlem}
G^n(t,x;T,y)=\sum_{ i\le n,\, j\le n(n+3),\, k\le \frac{n(n+5)}{2}\atop i+j-k\ge n}
c^n_{i,j,k}(x-\xbar)^i(\sqrt{T-t})^{j}\partial_x^k G^0(t,x;T,y),
\end{equation}
where $c^n_{i,j,k}$ are polynomial functions of $\a_0,\a_1,\dots,\a_n$.
\end{lemma}
\begin{proof}
We proceed by induction on $n$. For $n=0$ the thesis is trivial. Next by \eqref{2.2} we have
$G^{n+1}(t,x;T,y)=I_{n,2}-I_{n,1}$ where
 $$I_{n,l}=\sum\limits_{h=1}^{n+1} \a_h\int_t^T\int_{\R}G^0(t,x;s,\eta)
 (\eta-\bar{x})^h \pa_{\eta}^{l} G^{n+1-h}(s,\eta;T,y)d\eta ds,\quad l=1,2.$$
We only analyze the case $l=2$ since the other one is analogous. By the inductive hypothesis
\eqref{Gnlem}, we have that $I_{n,2}$ is a linear combination of terms of the form
\begin{equation}\label{Ia}
 \begin{split}
 \int_t^T\int_{\R}G^0(t,x;s,\eta) (\sqrt{T-s})^{j}(\eta-\bar{x})^{h+i-p}{\partial}_{\y}^{k+2-p}G^0(s,\eta;T,y)d \eta ds
\end{split}
\end{equation}
for $p=0,1,2$ and $h=1,\dots,n+1$; moreover we have
\begin{align}\label{I3a1}
  &i+j-k\ge n+1-h,\\ \label{I3a2}
  &i\le n+1-h, \\ \label{I3a3}
  &j\le (n+1-h)(n+4-h)\le n(n+3),\\ \label{I3a4}
  &k\le \frac{(n+1-h)(n+6-h)}{2}\le \frac{n(n+5)}{2}.
\end{align}
Again we focus only on $p=0$, the other cases being analogous: then by properties \eqref{V},
\eqref{d} and \eqref{repr}, we have that the integral in \eqref{Ia} is equal to
\begin{equation}\label{I3a5}
 \int_t^T (\sqrt{T-s})^{j}V_{t,s,x}^{h+i}ds\, \pa_{x}^{k+2} G^0(t,x;T,y)
\end{equation}
where $V_{t,T,x}\equiv V_{t,T,x,0}$ is the operator in \eqref{38}. Now we remark that
$V^{n}_{t,s,x}$ is a finite sum of the form
\begin{equation}\label{I2a}
  V^{n}_{t,s,x}=\sum_{0\le j_{1},\frac{j_{2}}{2},j_{3}\le n \atop j_{1}+j_{2}-j_{3}\ge n}
  b^{n}_{j_{1},j_{2},j_{3}}(x-\xbar)^{j_{1}}(\sqrt{s-t})^{j_{2}}{\partial}_{x}^{j_{3}}
\end{equation}
for some constants $b^{n}_{j_{1},j_{2},j_{3}}$. Thus the integral in \eqref{I3a5} is a linear
combination of terms of the form
  $$(x-\xbar)^{j_{1}}(\sqrt{T-s})^{j+2+j_{2}} \pa_{x}^{k+2+j_{3}} G^0(t,x;T,y)$$
where
\begin{align}\label{and10a}
  &0\le j_{1},\,\frac{j_{2}}{2},\,j_{3}\le h+i,\\  \label{and10b}
  &j_{1}+j_{2}-j_{3}\ge h+i.
\end{align}
Eventually we have
\begin{align*}
    &j_{1}+j+j_{2}+2-(k+2+j_{3})\ge
\intertext{(by \eqref{and10b})}
    &\ge i+j-k+h \ge
\intertext{(by \eqref{I3a1})}
    &\ge n+1.
\end{align*}
On the other hand, by \eqref{and10a} and \eqref{I3a2} we have
\begin{align*}
    j_{1}\le h+i\le n+1.
\end{align*}
Moreover, by \eqref{and10a}, \eqref{I3a2} and \eqref{I3a3} we have
\begin{align*}
    j+2+j_{2}\le j+2+2(n+1)\le n(n+3)+2+2(n+1)=(n+1)(n+4).
\end{align*}
Finally, by \eqref{and10a}, \eqref{I3a2} and \eqref{I3a4} we have
\begin{align*}
    k+2+j_{3}&\le k+2+h+i\le k+n+3\\
    &\le\frac{n(n+5)}{2}+n+3=\frac{(n+1)(n+6)}{2}.
\end{align*}
This concludes the proof.
\qquad\end{proof}

Now we set $\xbar=y$ and prove the thesis only in this case: to treat the case $\xbar=x$, it
suffices to proceed in a similar way by using the backward parametrix method introduced in
\cite{CorielliFoschiPascucci2010}. 
\begin{lemma}\label{lemapp2}
For any $\epsilon,\t >0$ there exists a positive constant $C$, only dependent on
$\e,\t,m,M,N$ and $\max\limits_{k\le N}\|\a_{k}\|_{\infty}$, such that
\begin{equation}\label{and14}
\left|\partial_{xx}G^n(t,x;T,y)\right| \leq C(T-t)^{\frac{n-2}{2}}\bar{\G}^{M+\epsilon}(t,x;T,y),
\end{equation}
for any $n\le N$, $x,y\in\R$ and $t,T\in\R$ with $0<T-t\le \t$.
\end{lemma}
\begin{proof}
By Lemma \ref{lemapp1} with $\xbar=y$, we have
\begin{align*}
 \left|\partial_{xx}G^n(t,x;T,y)\right| &\leq
 \sum_{ i\le n,\, j\le n(n+3),\, k\le \frac{n(n+5)}{2}\atop i+j-k\ge n}
  \left|c^n_{i,j,k}\right|\left(\sqrt{T-t}\right)^{j}\cdot\\
  &\quad\cdot\left|{\partial}_{xx}\left((x-y)^i\partial_x^k G^0(t,x;T,y)\right)\right|.
\end{align*}
Then the thesis follows from the boundedness of the coefficients $\a_{k}$, $k\le N$, (cf.
Assumption A$_{N}$) and the following standard Gaussian estimates (see, for instance, Lemma A.1
and A.2 in \cite{CorielliFoschiPascucci2010}):
\begin{equation}\label{and13}
\begin{split}
  &\partial_x^k G^0(t,x;T,y) \le c\,\left(\sqrt{T-t}\right)^{-k}\bar{\G}^{M+\epsilon}(t,x;T,y),\\
  &\left(\frac{x-y}{\sqrt{T-t}}\right)^{k} G^0(t,x;T,y) \le c\,\bar{\G}^{M+\epsilon}(t,x;T,y),
\end{split}
\end{equation}
where $c$ is a positive constant which depends on $k,m,M,\e$ and $\t$.
\qquad\end{proof}

\begin{lemma}\label{lemapp3}
For any $\epsilon,\t >0$ there exists a positive constant $C$, only dependent on
$\e,\t,m,M,N$ and $\max\limits_{k\le N+1}\|\a_{k}\|_{\infty}$, such that
\begin{equation}
\left|Z_n^N(t,x;T,y) \right| \leq \kappa_{n}(T-t)^{\frac{N+n-1}{2}}\bar{\G}^{M+\epsilon}(t,x;T,y),
\end{equation}
for any $n\in\NN$, $x,y\in\R$ and $t,T\in\R$ with $0<T-t\le \t$, where
  $$\kappa_{n}=C^{n}\frac{\G_{E}\left(\frac{1+N}{2}\right)}{\G_{E}\left(\frac{n+1+N}{2}\right)}$$
and $\G_{E}$ denotes the Euler Gamma function.
\end{lemma}
\begin{proof}
On the basis of definitions \eqref{34} and \eqref{2.2}, by induction we can prove the following formula:
\begin{equation}\label{and11}
 Z_0^N(z;\z) =L\G^N(z;\z)=\sum_{n=0}^N (L-L_{n})G^{N-n}(z;\z).
\end{equation}
Indeed, for $N=0$ we have
 $$L\G^0(z;\z)=(L-L_0)G^{0}(z;\z),$$
because $L_0G^{0}(z;\z)=0$ by definition. Then, assuming that \eqref{and11} holds for $N\in\NN$, for $N+1$ we have
\begin{align*}
 L\G^{N+1}(z;\z)&= L\G^N(z;\z)+L G^{N+1}(z;\z)=
 \intertext{(by inductive hypothesis and \eqref{2.2})}
 &= \sum_{n=0}^N (L-L_n)G^{N-n}(z;\z) + (L-L_0)G^{N+1}(z;\z)\\
 &\quad- \sum_{n=1}^{N+1} (L_{n}-L_{n-1})G^{N+1-n}(z;\z)\\
 &=\sum_{n=1}^{N+1} (L-L_{n-1})G^{N-(n-1)}(z;\z)+ (L-L_0)G^{N+1}(z;\z)\\
 &\quad- \sum_{n=1}^{N+1} (L_{n}-L_{n-1})G^{N+1-n}(z;\z)\\
 &=(L-L_0)G^{N+1}+\sum_{n=1}^{N+1} (L-L_{n})G^{N+1-n}(z;\z)
\end{align*}
from which \eqref{and11} follows.

Then, by \eqref{and11} and Assumption A$_{N+1}$ we have
\begin{equation}\label{A2}
\left|Z_0^N(z;\z) \right|\leq\sum_{n=0}^N
 \|\a_{n+1}\|_{\infty}|x-y|^{n+1}\left|(\partial_{xx}-\pa_x)G^{N-n}(z;\z)\right|
\end{equation}
and for $n=0$ the thesis follows from estimates \eqref{and14} and \eqref{and13}.
In the case $n\ge 1$, proceeding by induction, the thesis follows from the previous estimates by using the arguments in Lemma 4.3 in \cite{DiFrancescoPascucci2}: therefore the proof is omitted.
\qquad\end{proof}

\begin{figure}[p]
  \centering

  \begin{tabular}{cc}
    \includegraphics[width=.48\linewidth]{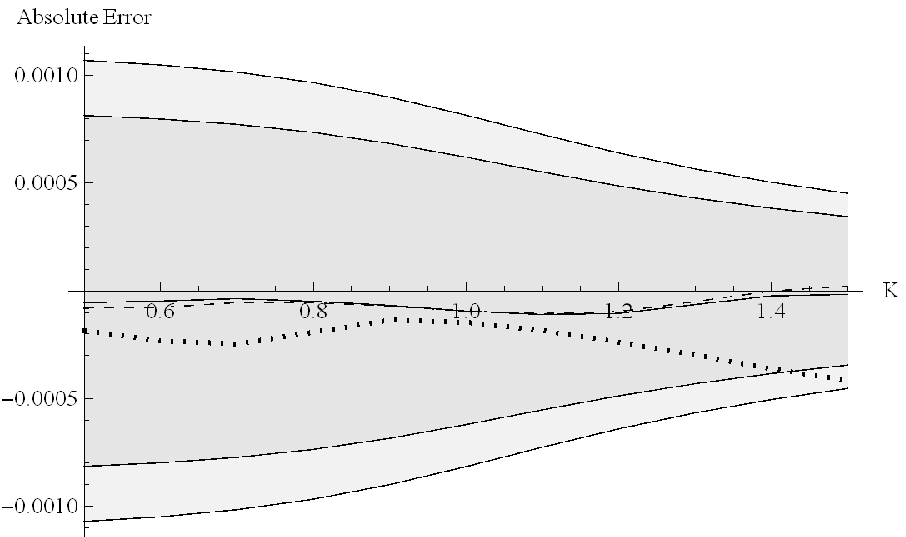}
    &
    \includegraphics[width=.48\linewidth]{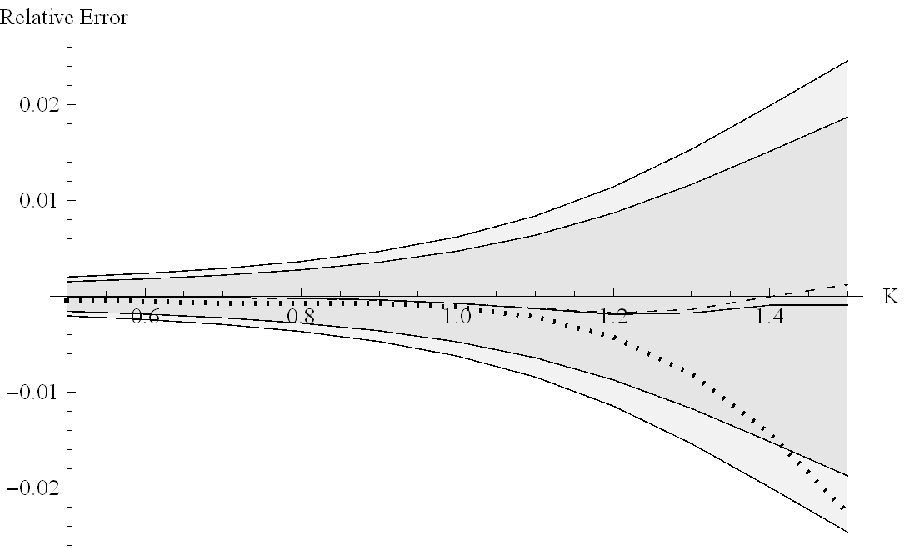}
  \end{tabular}

  \caption{Absolute (left) and relative (right) errors of the $1^{\text{st}}$ (dotted line),
  $2^{\text{nd}}$ (dashed line), $3^{\text{rd}}$ (solid line) order approximations of a Call price
  in the {\bf CEV-Merton} model with maturity $\mathbf{T=0.25}$ and strike $\mathbf{K\in[0.5,1.5]}$.
  The shaded bands show the $95\%$ (dark gray) and $99\%$ (light gray) Monte Carlo confidence regions}
  \label{fig1}
\end{figure}

\begin{figure}[p]
  \centering

  \begin{tabular}{cc}
    \includegraphics[width=.48\linewidth]{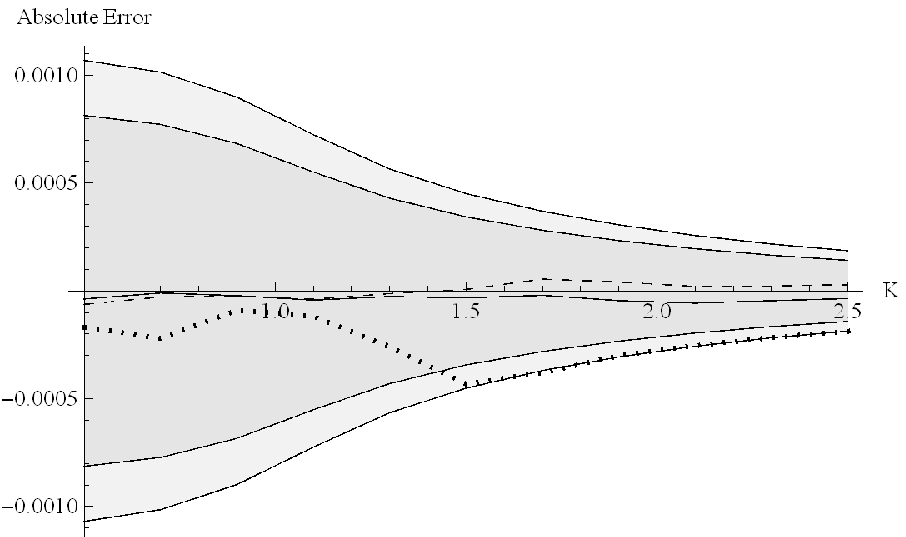}
    &
    \includegraphics[width=.48\linewidth]{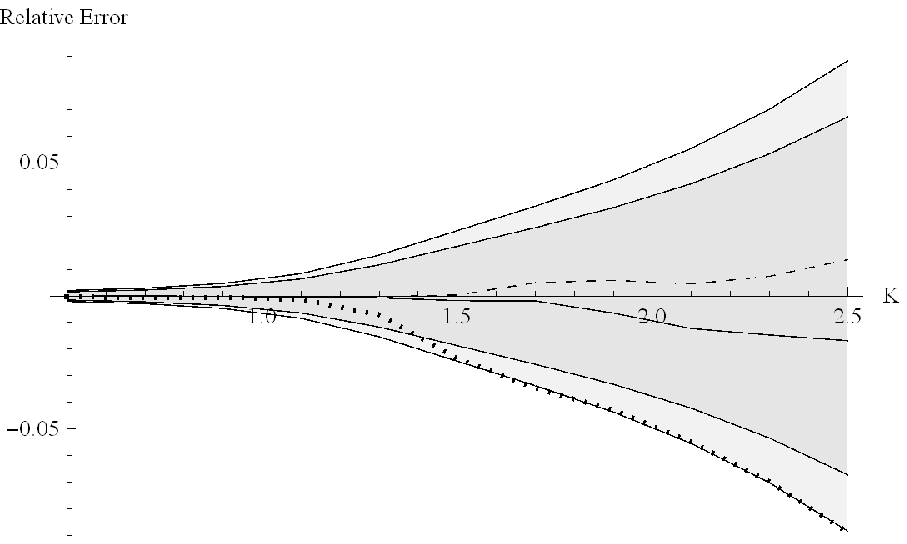}
  \end{tabular}

  \caption{Absolute (left) and relative (right) errors of the $1^{\text{st}}$ (dotted line),
  $2^{\text{nd}}$ (dashed line), $3^{\text{rd}}$ (solid line) order approximations of a Call price
  in the {\bf CEV-Merton} model with maturity $\mathbf{T=1}$ and strike $\mathbf{K\in[0.5,2.5]}$
  }
  \label{fig2}
\end{figure}

\begin{figure}[p]
  \centering

  \begin{tabular}{cc}
    \includegraphics[width=.48\linewidth]{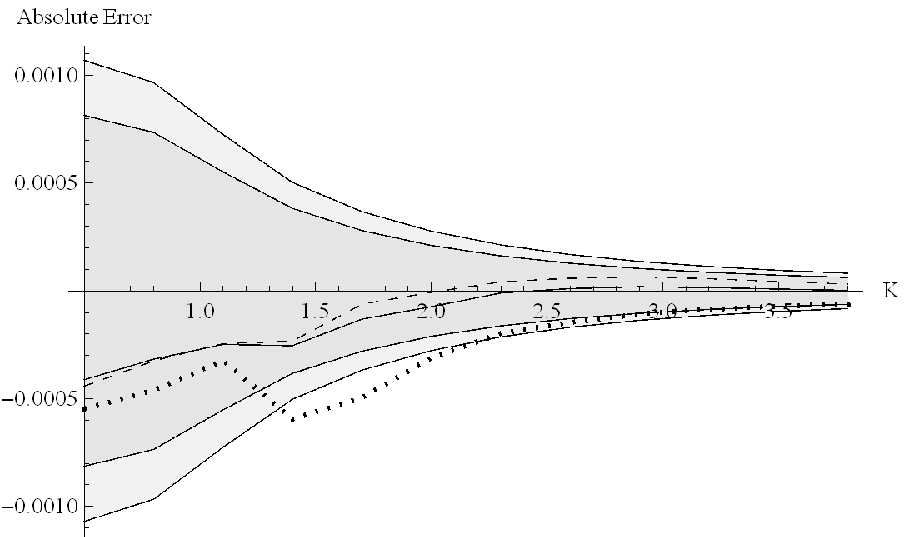}
    &
    \includegraphics[width=.48\linewidth]{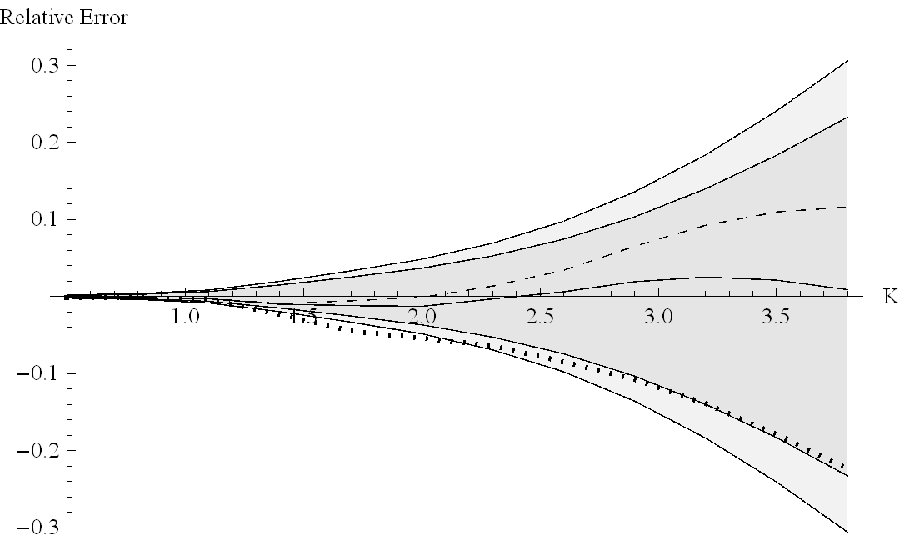}
  \end{tabular}

  \caption{Absolute (left) and relative (right) errors of the $1^{\text{st}}$ (dotted line),
  $2^{\text{nd}}$ (dashed line), $3^{\text{rd}}$ (solid line) order approximations of a Call price
  in the {\bf CEV-Merton} model with maturity $\mathbf{T=10}$ and strike $\mathbf{K\in[0.5,4]}$
  }
  \label{fig3}
\end{figure}

\begin{figure}[p]
  \centering

  \begin{tabular}{cc}
    \includegraphics[width=.48\linewidth]{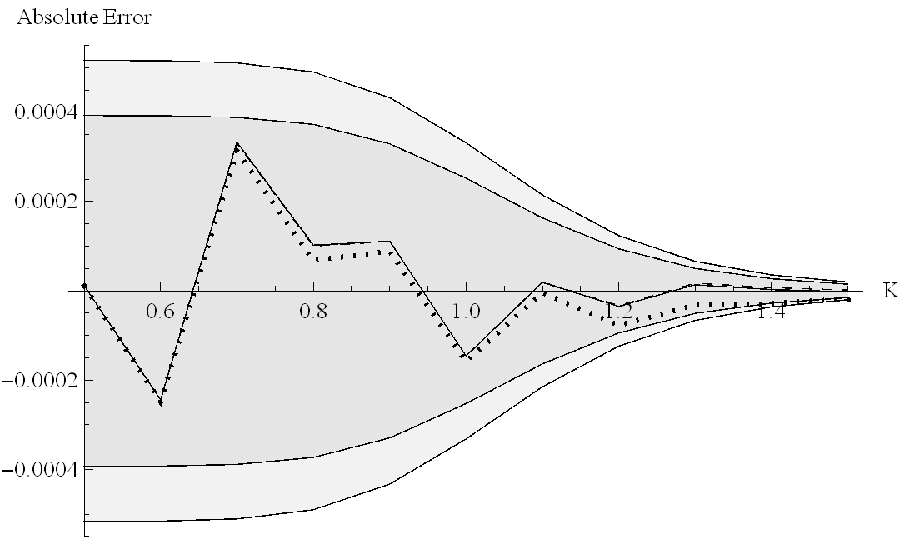}
    &
    \includegraphics[width=.48\linewidth]{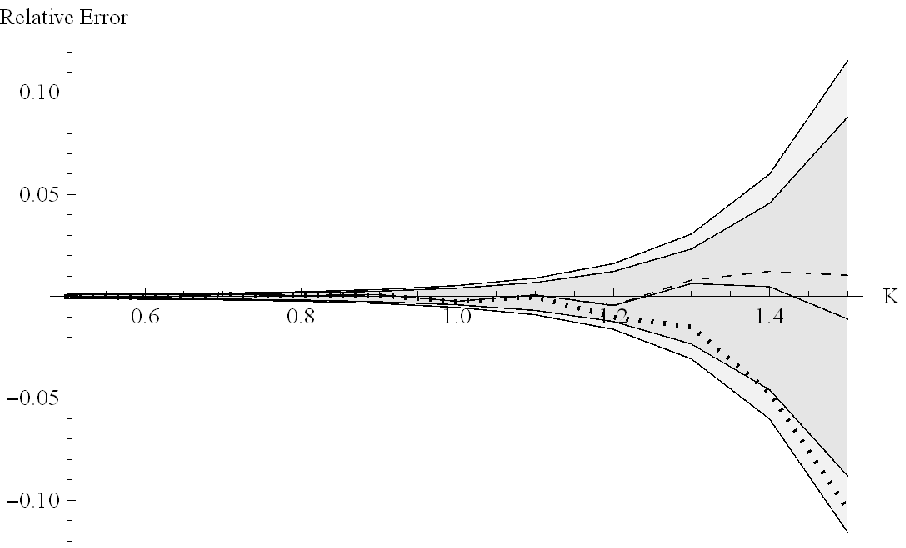}
  \end{tabular}

  \caption{Absolute (left) and relative (right) errors of the $1^{\text{st}}$ (dotted line),
  $2^{\text{nd}}$ (dashed line), $3^{\text{rd}}$ (solid line) order approximations of a Call price
  in the {\bf CEV-Variance-Gamma} model with maturity $\mathbf{T=0.25}$ and strike $\mathbf{K\in[0.5,1.5]}$.
  The shaded bands show the $95\%$ (dark gray) and $99\%$ (light gray) Monte Carlo confidence regions}
  \label{fig4}
\end{figure}

\begin{figure}[p]
  \centering

  \begin{tabular}{cc}
    \includegraphics[width=.48\linewidth]{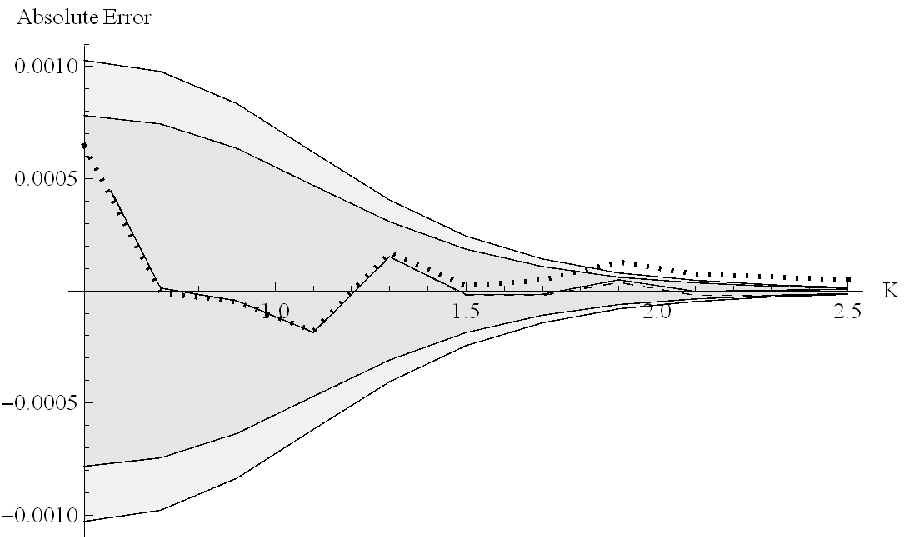}
    &
    \includegraphics[width=.48\linewidth]{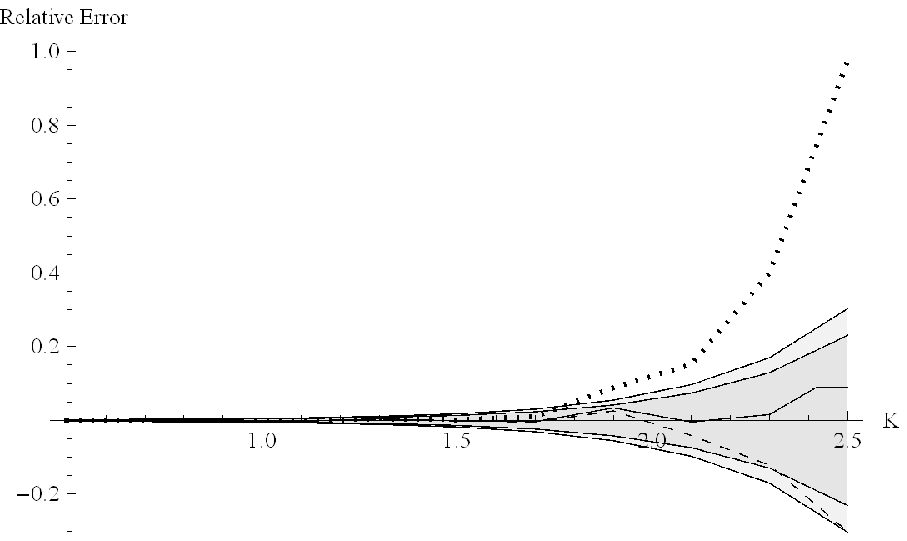}
  \end{tabular}

  \caption{Absolute (left) and relative (right) errors of the $2^{\text{nd}}$ (dotted line),
  $3^{\text{rd}}$ (dashed line), $4^{\text{th}}$ (solid line) order approximations of a Call price
  in the {\bf CEV-Variance-Gamma} model with maturity $\mathbf{T=1}$ and strike $\mathbf{K\in[0.5,2.5]}$
  }
  \label{fig5}
\end{figure}

\begin{figure}[p]
  \centering

  \begin{tabular}{cc}
    \includegraphics[width=.48\linewidth]{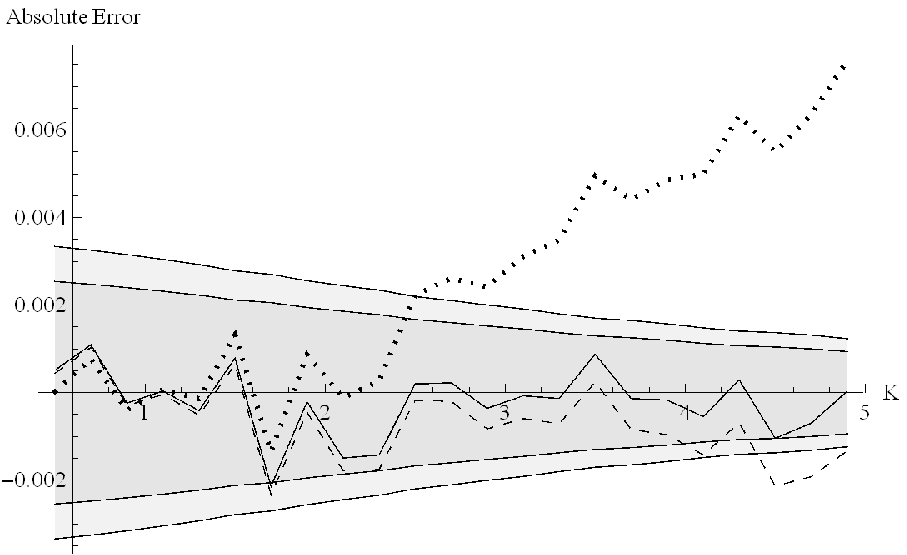}
    &
    \includegraphics[width=.48\linewidth]{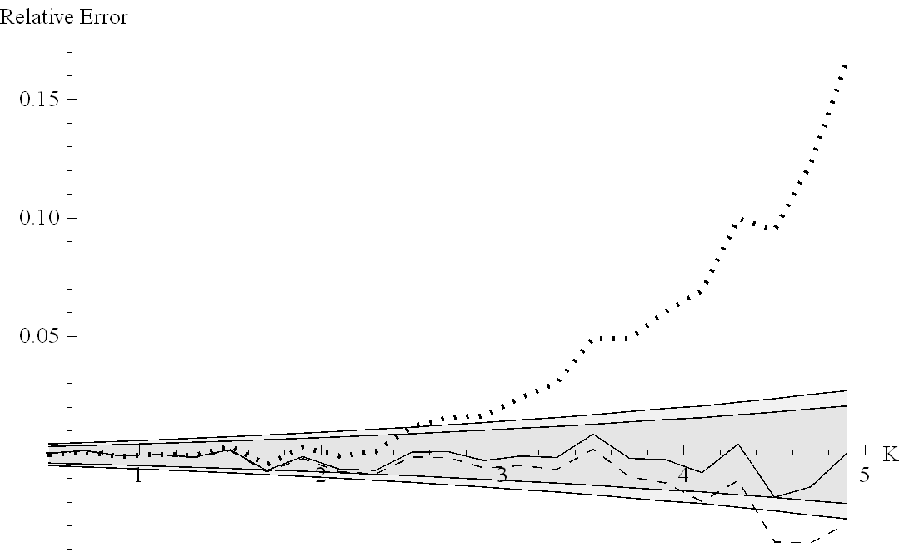}
  \end{tabular}

  \caption{Absolute (left) and relative (right) errors of the $2^{\text{nd}}$ (dotted line),
  $3^{\text{rd}}$ (dashed line), $4^{\text{th}}$ (solid line) order approximations of a Call price
  in the {\bf CEV-Variance-Gamma} model with maturity $\mathbf{T=10}$ and strike $\mathbf{K\in[0.5,5]}$}
  \label{fig6}
\end{figure}

\backmatter

\addcontentsline{toc}{chapter}{Bibliography}
\markboth{\MakeUppercase{Bibliography}}{}
\bibliographystyle{plainnat}
\bibliography{thesis}

\def\cprime{$'$} \def\cprime{$'$} \def\cprime{$'$}
  \def\lfhook#1{\setbox0=\hbox{#1}{\ooalign{\hidewidth
  \lower1.5ex\hbox{'}\hidewidth\crcr\unhbox0}}} \def\cprime{$'$}
  \def\cprime{$'$} \def\cprime{$'$} \def\cprime{$'$} \def\cprime{$'$}
  \def\polhk#1{\setbox0=\hbox{#1}{\ooalign{\hidewidth
  \lower1.5ex\hbox{`}\hidewidth\crcr\unhbox0}}}
\begin{thebibliography}{115}
\providecommand{\natexlab}[1]{#1}
\providecommand{\url}[1]{\texttt{#1}}
\expandafter\ifx\csname urlstyle\endcsname\relax
  \providecommand{\doi}[1]{doi: #1}\else
  \providecommand{\doi}{doi: \begingroup \urlstyle{rm}\Url}\fi

\bibitem[Acciaio et~al.(2013{\natexlab{a}})Acciaio, Beiglb{\"o}ck, Penkner,
  Schachermayer, and Temme]{traj-doob}
B.~Acciaio, M.~Beiglb{\"o}ck, F.~Penkner, W.~Schachermayer, and J.~Temme.
\newblock A trajectorial interpretation of {D}oob's martingale inequalities.
\newblock \emph{Ann. Appl. Probab.}, 23\penalty0 (4):\penalty0 1494--1505,
  2013{\natexlab{a}}.
\newblock ISSN 1050-5164.

\bibitem[Acciaio et~al.(2013{\natexlab{b}})Acciaio, Beiglb{\"o}ck, Penkner, and
  Schachermayer]{abps}
Beatrice Acciaio, Mathias Beiglb{\"o}ck, Friedrich Penkner, and Walter
  Schachermayer.
\newblock A model-free version of the fundamental theorem of asset pricing and
  the super-replication theorem.
\newblock \emph{Mathematical Finance}, 2013{\natexlab{b}}.

\bibitem[Andersen and Andreasen(2000)]{AndersenAndreasen2000}
L.~Andersen and J.~Andreasen.
\newblock Jump-diffusion processes: Volatility smile fitting and numerical
  methods for option pricing.
\newblock \emph{Review of Derivatives Research}, 4:\penalty0 231--262, 2000.

\bibitem[Avellaneda et~al.(1985)Avellaneda, Levy, and Par\'as]{avlevyparas}
M.~Avellaneda, A.~Levy, and A.~Par\'as.
\newblock Pricing and hedging derivative securities in markets with uncertain
  volatilities.
\newblock \emph{Applied Mathematical Finance}, 2\penalty0 (2):\penalty0 73--88,
  1985.

\bibitem[Bachelier(1900)]{bachelier}
L.~Bachelier.
\newblock Th\'eorie de la sp\'eculation.
\newblock \emph{Ann. Sci. \'Ecole Norm. Sup. (3)}, 17:\penalty0 21--86, 1900.
\newblock ISSN 0012-9593.
\newblock URL \url{http://www.numdam.org/item?id=ASENS_1900_3_17__21_0}.

\bibitem[Benhamou et~al.(2009)Benhamou, Gobet, and Miri]{BenhamouGobetMiri2009}
E.~Benhamou, E.~Gobet, and M.~Miri.
\newblock Smart expansion and fast calibration for jump diffusions.
\newblock \emph{Finance Stoch.}, 13\penalty0 (4):\penalty0 563--589, 2009.
\newblock ISSN 0949-2984.
\newblock \doi{10.1007/s00780-009-0102-3}.
\newblock URL \url{http://dx.doi.org/10.1007/s00780-009-0102-3}.

\bibitem[Benhamou et~al.(2010)Benhamou, Gobet, and
  Miri]{BenhamouGobetMiri2010b}
E.~Benhamou, E.~Gobet, and M.~Miri.
\newblock Expansion formulas for {E}uropean options in a local volatility
  model.
\newblock \emph{Int. J. Theor. Appl. Finance}, 13\penalty0 (4):\penalty0
  603--634, 2010.
\newblock ISSN 0219-0249.
\newblock \doi{10.1142/S0219024910005887}.
\newblock URL \url{http://dx.doi.org/10.1142/S0219024910005887}.

\bibitem[Bergman et~al.(1996)Bergman, Grundy, and Wiener]{bergman}
Yaacov~Z. Bergman, Bruce~D. Grundy, and Zvi Wiener.
\newblock General properties of option prices.
\newblock \emph{The Journal of Finance}, \penalty0 (5), 1996.

\bibitem[Biagini et~al.(2014)Biagini, Bouchard, Kardaras, and Nutz]{sara-bkn}
Sara Biagini, Bruno Bouchard, Constantinos Kardaras, and Marcel Nutz.
\newblock {Robust Fundamental Theorem for Continuous Processes}.
\newblock \emph{ArXiv e-prints}, October 2014.
\newblock URL \url{http://arxiv.org/abs/1410.4962v1}.

\bibitem[Bichteler(1981)]{bichteler}
Klaus Bichteler.
\newblock Stochastic integration and {$L^{p}$}-theory of semimartingales.
\newblock \emph{Ann. Probab.}, 9\penalty0 (1):\penalty0 49--89, 1981.
\newblock ISSN 0091-1798.
\newblock URL
  \url{http://links.jstor.org/sici?sici=0091-1798(198102)9:1<49:SIAOS>2.0.CO;2-\&origin=MSN}.

\bibitem[Bick and Willinger(1994)]{bickwill}
Avi Bick and Walter Willinger.
\newblock Dynamic spanning without probabilities.
\newblock \emph{Stochastic Process. Appl.}, 50\penalty0 (2):\penalty0 349--374,
  1994.
\newblock ISSN 0304-4149.
\newblock \doi{10.1016/0304-4149(94)90128-7}.
\newblock URL \url{http://dx.doi.org/10.1016/0304-4149(94)90128-7}.

\bibitem[Black and Scholes(2012)]{bs}
Fischer Black and Myron Scholes.
\newblock The pricing of options and corporate liabilities [reprint of {J}.
  {P}olit. {E}con. {\bf 81} (1973), no. 3, 637--654].
\newblock In \emph{Financial risk measurement and management}, volume 267 of
  \emph{Internat. Lib. Crit. Writ. Econ.}, pages 100--117. Edward Elgar,
  Cheltenham, 2012.

\bibitem[Bouchard and Nutz(2014)]{bouchard-nutz}
Bruno Bouchard and Marcel Nutz.
\newblock {Arbitrage and Duality in Nondominated Discrete-Time Models}.
\newblock \emph{ArXiv e-prints}, February 2014.
\newblock URL \url{http://arxiv.org/abs/1305.6008v2}.

\bibitem[Burzoni et~al.(2015)Burzoni, Frittelli, and Maggis]{bfm}
Matteo Burzoni, Marco Frittelli, and Marco Maggis.
\newblock {Universal Arbitrage Aggregator in Discrete Time Markets under
  Uncertainty}.
\newblock \emph{ArXiv e-prints}, February 2015.
\newblock URL \url{http://arxiv.org/abs/1407.0948v2}.

\bibitem[Carr and Madan(1999)]{CarrMadan1999}
P.~Carr and D.~Madan.
\newblock Option valuation using the fast {F}ourier transform.
\newblock \emph{J. Comput. Finance}, 2(4):\penalty0 61--73, 1999.

\bibitem[Carr et~al.(2004)Carr, Geman, Madan, and Yor]{CGMY2004}
Peter Carr, H{\'e}lyette Geman, Dilip~B. Madan, and Marc Yor.
\newblock From local volatility to local {L}\'evy models.
\newblock \emph{Quant. Finance}, 4\penalty0 (5):\penalty0 581--588, 2004.
\newblock ISSN 1469-7688.
\newblock \doi{10.1080/14697680400024921}.
\newblock URL \url{http://dx.doi.org/10.1080/14697680400024921}.

\bibitem[Cheng et~al.(2011)Cheng, Costanzino, Liechty, Mazzucato, and
  Nistor]{ChengCostanzinoLiechtyMazzucatoNistor2011}
W.~Cheng, N.~Costanzino, J.~Liechty, A.L. Mazzucato, and V.~Nistor.
\newblock Closed-form asymptotics and numerical approximations of 1{D}
  parabolic equations with applications to option pricing.
\newblock \emph{to appear in SIAM J. Fin. Math.}, 2011.

\bibitem[Cont(2006)]{cont2006}
Rama Cont.
\newblock Model uncertainty and its impact on the pricing of derivative
  instruments.
\newblock \emph{Mathematical Finance}, 16\penalty0 (3):\penalty0 519--547,
  2006.
\newblock ISSN 1467-9965.
\newblock \doi{10.1111/j.1467-9965.2006.00281.x}.
\newblock URL \url{http://dx.doi.org/10.1111/j.1467-9965.2006.00281.x}.

\bibitem[Cont(2012)]{cont-notes}
Rama Cont.
\newblock \emph{Functional {I}to calculus and functional {K}olmogorov
  equations}.
\newblock Lecture notes of the {B}arcelona {S}ummer School on {S}tochastic
  Analysis, Centre de Recerca Matematica, July 2012. 2012.

\bibitem[Cont and Fourni\'e(2010)]{ContFournie09a}
Rama Cont and David-Antoine Fourni\'e.
\newblock A functional extension of the {I}to formula.
\newblock \emph{Comptes Rendus Math\'ematique {A}cad. {S}ci. {P}aris {S}er. I},
  348:\penalty0 57--61, 2010.
\newblock URL \url{http://dx.doi.org/10.1016/j.crma.2009.11.013}.

\bibitem[Cont and Fourni{\'e}(2010)]{contf2010}
Rama Cont and David-Antoine Fourni{\'e}.
\newblock Change of variable formulas for non-anticipative functionals on path
  space.
\newblock \emph{J. Funct. Anal.}, 259\penalty0 (4):\penalty0 1043--1072, 2010.
\newblock ISSN 0022-1236.
\newblock \doi{10.1016/j.jfa.2010.04.017}.
\newblock URL \url{http://dx.doi.org/10.1016/j.jfa.2010.04.017}.

\bibitem[Cont and Fourni{\'e}(2013)]{contf2013}
Rama Cont and David-Antoine Fourni{\'e}.
\newblock Functional {I}t\^o calculus and stochastic integral representation of
  martingales.
\newblock \emph{Ann. Probab.}, 41\penalty0 (1):\penalty0 109--133, 2013.
\newblock ISSN 0091-1798.
\newblock \doi{10.1214/11-AOP721}.
\newblock URL \url{http://dx.doi.org/10.1214/11-AOP721}.

\bibitem[Cont and Yi(2014)]{contlu}
Rama Cont and Lu~Yi.
\newblock Weak approximations for martingale representations.
\newblock \emph{Working Paper}, 2014.

\bibitem[Cont et~al.(2011)Cont, Lantos, and Pironneau]{ContLantosPironneau2011}
Rama Cont, Nicolas Lantos, and Olivier Pironneau.
\newblock A reduced basis for option pricing.
\newblock \emph{SIAM J. Financial Math.}, 2:\penalty0 287--316, 2011.
\newblock ISSN 1945-497X.
\newblock \doi{10.1137/10079851X}.
\newblock URL \url{http://dx.doi.org/10.1137/10079851X}.

\bibitem[Corielli et~al.(2010)Corielli, Foschi, and
  Pascucci]{CorielliFoschiPascucci2010}
Francesco Corielli, Paolo Foschi, and Andrea Pascucci.
\newblock Parametrix approximation of diffusion transition densities.
\newblock \emph{SIAM J. Financial Math.}, 1:\penalty0 833--867, 2010.
\newblock ISSN 1945-497X.

\bibitem[Cosso(2013)]{cosso}
Andrea Cosso.
\newblock Viscosity solutions of path-dependent pdes and non-markovian
  forward-backward stochastic equations.
\newblock \emph{ArXiv e-prints}, June 2013.
\newblock URL \url{http://arxiv.org/abs/1202.2502v4}.

\bibitem[Davis(2001)]{davis}
Mark Davis.
\newblock Mathematics of financial markets.
\newblock In \emph{Mathematics unlimited---2001 and beyond}, pages 361--380.
  Springer, Berlin, 2001.

\bibitem[de~Finetti(1931)]{definetti31}
Bruno de~Finetti.
\newblock Sul significato soggettivo della probabilit{\`a}.
\newblock \emph{{\guillemotleft}Fundamenta Mathematicae}, 17:\penalty0
  298--329, 1931.

\bibitem[de~Finetti(1937)]{definetti37}
Bruno de~Finetti.
\newblock La pr{\'e}vision : ses lois logiques, ses sources subjectives.
\newblock \emph{Annales de l'institut Henri Poincar{\'e}}, 7\penalty0
  (1):\penalty0 1--68, 1937.
\newblock URL \url{http://eudml.org/doc/79004}.

\bibitem[Delbaen and Schachermayer(1998)]{ds98}
F.~Delbaen and W.~Schachermayer.
\newblock The fundamental theorem of asset pricing for unbounded stochastic
  processes.
\newblock \emph{Math. Ann.}, 312\penalty0 (2):\penalty0 215--250, 1998.
\newblock ISSN 0025-5831.
\newblock \doi{10.1007/s002080050220}.
\newblock URL \url{http://dx.doi.org/10.1007/s002080050220}.

\bibitem[Delbaen and Schachermayer(1994)]{ds94}
Freddy Delbaen and Walter Schachermayer.
\newblock A general version of the fundamental theorem of asset pricing.
\newblock \emph{Math. Ann.}, 300\penalty0 (3):\penalty0 463--520, 1994.
\newblock ISSN 0025-5831.
\newblock \doi{10.1007/BF01450498}.
\newblock URL \url{http://dx.doi.org/10.1007/BF01450498}.

\bibitem[Delbaen and Schachermayer(2006)]{ds-book}
Freddy Delbaen and Walter Schachermayer.
\newblock \emph{The mathematics of arbitrage}.
\newblock Springer Finance. Springer-Verlag, Berlin, 2006.
\newblock ISBN 978-3-540-21992-7; 3-540-21992-7.

\bibitem[Di~Francesco and Pascucci(2004)]{pascucci-difra}
Marco Di~Francesco and Andrea Pascucci.
\newblock On the complete model with stochastic volatility by hobson and
  rogers.
\newblock \emph{Proceedings: Mathematical, Physical and Engineering Sciences},
  460\penalty0 (2051):\penalty0 pp. 3327--3338, 2004.
\newblock ISSN 13645021.
\newblock URL \url{http://www.jstor.org/stable/4143207}.

\bibitem[Di~Francesco and Pascucci(2005)]{DiFrancescoPascucci2}
Marco Di~Francesco and Andrea Pascucci.
\newblock On a class of degenerate parabolic equations of {K}olmogorov type.
\newblock \emph{AMRX Appl. Math. Res. Express}, 3:\penalty0 77--116, 2005.
\newblock ISSN 1687-1200.

\bibitem[Dudley and Norvai{\v{s}}a(2011)]{dudley-norvaisa}
R.~M. Dudley and R.~Norvai{\v{s}}a.
\newblock \emph{Concrete functional calculus}.
\newblock Springer Monographs in Mathematics. Springer, New York, 2011.
\newblock ISBN 978-1-4419-6949-1.
\newblock \doi{10.1007/978-1-4419-6950-7}.
\newblock URL \url{http://dx.doi.org/10.1007/978-1-4419-6950-7}.

\bibitem[Dudley and Norvai{\v{s}}a(1999)]{dudley-norvaisa99}
Richard~M. Dudley and Rimas Norvai{\v{s}}a.
\newblock \emph{Differentiability of six operators on nonsmooth functions and
  {$p$}-variation}, volume 1703 of \emph{Lecture Notes in Mathematics}.
\newblock Springer-Verlag, Berlin, 1999.
\newblock ISBN 3-540-65975-7.
\newblock With the collaboration of Jinghua Qian.

\bibitem[Dupire()]{dupire}
Bruno Dupire.
\newblock Functional it\^o calculus.
\newblock July 17, 2009, SSRN Preprint.

\bibitem[Ekren et~al.(2014{\natexlab{a}})Ekren, Keller, Touzi, and Zhang]{ektz}
Ibrahim Ekren, Christian Keller, Nizar Touzi, and Jianfeng Zhang.
\newblock On viscosity solutions of path dependent pdes.
\newblock \emph{Ann. Probab.}, 42\penalty0 (1):\penalty0 204--236, 01
  2014{\natexlab{a}}.
\newblock \doi{10.1214/12-AOP788}.
\newblock URL \url{http://dx.doi.org/10.1214/12-AOP788}.

\bibitem[Ekren et~al.(2014{\natexlab{b}})Ekren, Keller, Touzi, and
  Zhang]{ektz1}
Ibrahim Ekren, Christian Keller, Nizar Touzi, and Jianfeng Zhang.
\newblock Viscosity solutions of fully nonlinear parabolic path dependent pdes:
  Part i.
\newblock \emph{ArXiv e-prints}, September 2014{\natexlab{b}}.
\newblock URL \url{http://arxiv.org/abs/1210.0006v3}.

\bibitem[Ekren et~al.(2014{\natexlab{c}})Ekren, Keller, Touzi, and
  Zhang]{ektz2}
Ibrahim Ekren, Christian Keller, Nizar Touzi, and Jianfeng Zhang.
\newblock Viscosity solutions of fully nonlinear parabolic path dependent pdes:
  Part ii.
\newblock \emph{ArXiv e-prints}, September 2014{\natexlab{c}}.
\newblock URL \url{http://arxiv.org/abs/1210.0007v3}.

\bibitem[Ekstr\"om and Tysk(2011)]{EkstromTysk2011}
Erik Ekstr\"om and Johan Tysk.
\newblock Boundary behaviour of densities for non-negative diffusions.
\newblock \emph{preprint}, 2011.

\bibitem[El~Karoui et~al.(1997)El~Karoui, Peng, and
  Quenez]{elkaroui-peng-quenez}
Nicole El~Karoui, Shige Peng, and Marie-Claire Quenez.
\newblock Backward stochastic differential equations in finance.
\newblock \emph{Math. Finance}, 7\penalty0 (1):\penalty0 1--71, 1997.
\newblock ISSN 0960-1627.
\newblock \doi{10.1111/1467-9965.00022}.
\newblock URL \url{http://dx.doi.org/10.1111/1467-9965.00022}.

\bibitem[El~Karoui et~al.(1998)El~Karoui, Jeanblanc-Picqu{\'e}, and
  Shreve]{elkaroui}
Nicole El~Karoui, Monique Jeanblanc-Picqu{\'e}, and Steven~E. Shreve.
\newblock Robustness of the {B}lack and {S}choles formula.
\newblock \emph{Math. Finance}, 8\penalty0 (2):\penalty0 93--126, 1998.
\newblock ISSN 0960-1627.
\newblock \doi{10.1111/1467-9965.00047}.
\newblock URL \url{http://dx.doi.org/10.1111/1467-9965.00047}.

\bibitem[Fang and Oosterlee(2008/09)]{Oosterlee2008}
F.~Fang and C.~W. Oosterlee.
\newblock A novel pricing method for {E}uropean options based on
  {F}ourier-cosine series expansions.
\newblock \emph{SIAM J. Sci. Comput.}, 31\penalty0 (2):\penalty0 826--848,
  2008/09.
\newblock ISSN 1064-8275.
\newblock \doi{10.1137/080718061}.
\newblock URL \url{http://dx.doi.org/10.1137/080718061}.

\bibitem[Federico and Tankov(2015)]{salvatore-tankov}
Salvatore Federico and Peter Tankov.
\newblock Finite-dimensional representations for controlled diffusions with
  delay.
\newblock \emph{Applied Mathematics \& Optimization}, 71\penalty0 (1):\penalty0
  165--194, 2015.
\newblock ISSN 0095-4616.
\newblock \doi{10.1007/s00245-014-9256-2}.
\newblock URL \url{http://dx.doi.org/10.1007/s00245-014-9256-2}.

\bibitem[Flandoli and Zanco(2013)]{flandoli-zanco}
Franco Flandoli and Giovanni Zanco.
\newblock {An infinite-dimensional approach to path-dependent Kolmogorov's
  equations}.
\newblock \emph{ArXiv e-prints}, December 2013.
\newblock URL \url{http://arxiv.org/abs/1312.6165v1}.

\bibitem[F{\"o}llmer(1981)]{follmer}
Hans F{\"o}llmer.
\newblock Calcul d'{I}t\^o sans probabilit\'es.
\newblock In \emph{Seminar on {P}robability, {XV} ({U}niv. {S}trasbourg,
  {S}trasbourg, 1979/1980) ({F}rench)}, volume 850 of \emph{Lecture Notes in
  Math.}, pages 143--150. Springer, Berlin, 1981.

\bibitem[F{\"o}llmer and Schied(2013)]{follmer-schied}
Hans F{\"o}llmer and Alexander" Schied.
\newblock Probabilistic aspects of finance.
\newblock \emph{Bernoulli}, 19\penalty0 (4):\penalty0 1306--1326, Sep 2013.
\newblock \doi{10.3150/12-BEJSP05}.
\newblock URL \url{http://dx.doi.org/10.3150/12-BEJSP05}.

\bibitem[Foschi et~al.(2011)Foschi, Pagliarani, and
  Pascucci]{FoschiPagliaraniPascucci2011}
P.~Foschi, S.~Pagliarani, and A.~Pascucci.
\newblock Black-{S}choles formulae for {A}sian options in local volatility
  models.
\newblock \emph{SSRN eLibrary}, 2011.

\bibitem[Foschi and Pascucci(2008)]{pascucci-foschi}
Paolo Foschi and Andrea Pascucci.
\newblock Path dependent volatility.
\newblock \emph{Decisions in Economics and Finance}, 31\penalty0 (1):\penalty0
  13--32, 2008.
\newblock ISSN 1593-8883.
\newblock \doi{10.1007/s10203-007-0076-6}.
\newblock URL \url{http://dx.doi.org/10.1007/s10203-007-0076-6}.

\bibitem[Fourni\'e(2010)]{fournie}
David-Antoine Fourni\'e.
\newblock \emph{Functional {I}to calculus and applications}.
\newblock ProQuest LLC, Ann Arbor, MI, 2010.
\newblock ISBN 978-1124-18578-1.
\newblock URL
  \url{http://gateway.proquest.com/openurl?url_ver=Z39.88-2004&rft_val_fmt=info:ofi/fmt:kev:mtx:dissertation&res_dat=xri:pqdiss&rft_dat=xri:pqdiss:3420880}.
\newblock Thesis (Ph.D.)--Columbia University.

\bibitem[Friedman(1964)]{Friedman}
Avner Friedman.
\newblock \emph{Partial differential equations of parabolic type}.
\newblock Prentice-Hall Inc., Englewood Cliffs, N.J., 1964.

\bibitem[Garroni and Menaldi(1992)]{GarroniMenaldi1992}
M.~G. Garroni and J.-L. Menaldi.
\newblock \emph{Green functions for second order parabolic integro-differential
  problems}, volume 275 of \emph{Pitman Research Notes in Mathematics Series}.
\newblock Longman Scientific \& Technical, Harlow, 1992.
\newblock ISBN 0-582-02156-1.

\bibitem[Gatheral et~al.(2010)Gatheral, Hsu, Laurence, Ouyang, and
  Wang]{GatheralHsuLaurenceOuyangWang2010}
Jim Gatheral, Elton~P. Hsu, Peter Laurence, Cheng Ouyang, and Tai-Ho Wang.
\newblock Asymptotics of implied volatility in local volatility models.
\newblock \emph{to appear in Math. Finance}, 2010.

\bibitem[Hagan and Woodward(1999)]{Hagan99}
P.S. Hagan and D.E. Woodward.
\newblock Equivalent {B}lack volatilities.
\newblock \emph{Appl. Math. Finance}, 6\penalty0 (6):\penalty0 147--159, 1999.

\bibitem[Harrison and Kreps(1979)]{harr-kreps}
J.~Michael Harrison and David~M. Kreps.
\newblock Martingales and arbitrage in multiperiod securities markets.
\newblock \emph{J. Econom. Theory}, 20\penalty0 (3):\penalty0 381--408, 1979.
\newblock ISSN 0022-0531.
\newblock \doi{10.1016/0022-0531(79)90043-7}.
\newblock URL \url{http://dx.doi.org/10.1016/0022-0531(79)90043-7}.

\bibitem[Harrison and Pliska(1981)]{harr-pliska}
J.~Michael Harrison and Stanley~R. Pliska.
\newblock Martingales and stochastic integrals in the theory of continuous
  trading.
\newblock \emph{Stochastic Process. Appl.}, 11\penalty0 (3):\penalty0 215--260,
  1981.
\newblock ISSN 0304-4149.
\newblock \doi{10.1016/0304-4149(81)90026-0}.
\newblock URL \url{http://dx.doi.org/10.1016/0304-4149(81)90026-0}.

\bibitem[Heston(1993)]{Heston1993}
S.~L. Heston.
\newblock A closed-form solution for options with stochastic volatility with
  applications to bond and currency options.
\newblock \emph{Rev. Finan. Stud.}, 6:\penalty0 327--343, 1993.

\bibitem[Hobson(1998)]{hobson}
David~G. Hobson.
\newblock Volatility misspecification, option pricing and superreplication via
  coupling.
\newblock \emph{Ann. Appl. Probab.}, 8\penalty0 (1):\penalty0 193--205, 1998.
\newblock ISSN 1050-5164.
\newblock \doi{10.1214/aoap/1027961040}.
\newblock URL \url{http://dx.doi.org/10.1214/aoap/1027961040}.

\bibitem[Hobson and Rogers(1998)]{hobson-rogers}
David~G. Hobson and L.~C.~G. Rogers.
\newblock Complete models with stochastic volatility.
\newblock \emph{Mathematical Finance}, 8\penalty0 (1):\penalty0 27--48, 1998.
\newblock ISSN 1467-9965.
\newblock \doi{10.1111/1467-9965.00043}.
\newblock URL \url{http://dx.doi.org/10.1111/1467-9965.00043}.

\bibitem[Howison(2005)]{Howison2005}
Sam Howison.
\newblock Matched asymptotic expansions in financial engineering.
\newblock \emph{J. Engrg. Math.}, 53\penalty0 (3-4):\penalty0 385--406, 2005.
\newblock ISSN 0022-0833.
\newblock \doi{10.1007/s10665-005-7716-z}.
\newblock URL \url{http://dx.doi.org/10.1007/s10665-005-7716-z}.

\bibitem[It{\^o} and McKean(1974)]{ItoMcKean1974}
Kiyosi It{\^o} and Henry~P. McKean, Jr.
\newblock \emph{Diffusion processes and their sample paths}.
\newblock Springer-Verlag, Berlin, 1974.
\newblock Second printing, corrected, Die Grundlehren der mathematischen
  Wissenschaften, Band 125.

\bibitem[Jr.(1987)]{ingersoll}
Jonathan E.~Ingersoll Jr.
\newblock \emph{Theory of Financial Decision Making}.
\newblock Blackwell, Oxford, 1987.
\newblock p.377.

\bibitem[Karandikar(1995)]{karandikar}
Rajeeva~L. Karandikar.
\newblock On pathwise stochastic integration.
\newblock \emph{Stochastic Process. Appl.}, 57\penalty0 (1):\penalty0 11--18,
  1995.
\newblock ISSN 0304-4149.
\newblock \doi{10.1016/0304-4149(95)00002-O}.
\newblock URL \url{http://dx.doi.org/10.1016/0304-4149(95)00002-O}.

\bibitem[Karatzas and Kardaras(2007)]{kk2007}
Ioannis Karatzas and Constantinos Kardaras.
\newblock The num{\'e}raire portfolio in semimartingale financial models.
\newblock \emph{Finance and Stochastics}, 11\penalty0 (4):\penalty0 447--493,
  2007.

\bibitem[Kardaras(2010)]{kardaras}
Constantinos Kardaras.
\newblock Finitely additive probabilities and the fundamental theorem of asset
  pricing.
\newblock In \emph{Contemporary quantitative finance}, pages 19--34. Springer,
  2010.

\bibitem[Knight(1921)]{knight}
Frank Knight.
\newblock \emph{Risk, uncertainty and profit}.
\newblock Houghton Mifflin, Boston, 1921.

\bibitem[Kreps(1981)]{kreps81}
David~M Kreps.
\newblock Arbitrage and equilibrium in economies with infinitely many
  commodities.
\newblock \emph{Journal of Mathematical Economics}, 8\penalty0 (1):\penalty0
  15--35, 1981.

\bibitem[Levi(1907)]{Levi1907}
E.~E. Levi.
\newblock Sulle equazioni lineari totalmente ellittiche alle derivate parziali.
\newblock \emph{Rend. Circ. Mat. Palermo}, 24:\penalty0 275--317, 1907.

\bibitem[Lewis(2001)]{Lewis2001}
A.L. Lewis.
\newblock A simple option formula for general jump-diffusion and other
  exponential {L}évy processes.
\newblock Technical report, Finance Press, August 2001.

\bibitem[Lipton(2002)]{Lipton2002}
A.~Lipton.
\newblock The vol smile problem.
\newblock \emph{Risk Magazine}, 15:\penalty0 61--65, 2002.

\bibitem[Love and Young(1938)]{love-young}
E.~R. Love and L.~C. Young.
\newblock On {F}ractional {I}ntegration by {P}arts.
\newblock \emph{Proc. London Math. Soc.}, S2-44\penalty0 (1):\penalty0 1--35,
  1938.
\newblock ISSN 0024-6115.
\newblock \doi{10.1112/plms/s2-44.1.1}.
\newblock URL \url{http://dx.doi.org/10.1112/plms/s2-44.1.1}.

\bibitem[Lyons(1995)]{lyons}
Terry~J. Lyons.
\newblock Uncertain volatility and the risk-free synthesis of derivatives.
\newblock \emph{Applied Mathematical Finance}, 2\penalty0 (2):\penalty0
  117--133, 1995.

\bibitem[Madan and Seneta(1990)]{MadanSeneta1990}
D.~Madan and E.~Seneta.
\newblock The variance gamma ({VG}) model for share market returns.
\newblock \emph{Journal of Business}, 63:\penalty0 511--524, 1990.

\bibitem[Malliavin(1997)]{malliavin}
Paul Malliavin.
\newblock \emph{Stochastic analysis}.
\newblock Springer, 1997.

\bibitem[Merton(1973)]{merton}
Robert~C. Merton.
\newblock Theory of rational option pricing.
\newblock \emph{Bell J. Econom. and Management Sci.}, 4:\penalty0 141--183,
  1973.
\newblock ISSN 0741-6261.

\bibitem[Merton(1976)]{Merton1976}
Robert~C. Merton.
\newblock {Option pricing when underlying stock returns are discontinuous.}
\newblock \emph{J. Financ. Econ.}, 3\penalty0 (1-2):\penalty0 125--144, 1976.

\bibitem[Milevsky and Posner(1998)]{mil-posner}
Moshe~Arye Milevsky and Steven~E. Posner.
\newblock Asian options, the sum of lognormals, and the reciprocal gamma
  distribution.
\newblock \emph{The Journal of Financial and Quantitative Analysis},
  33\penalty0 (3):\penalty0 409--422, 1998.

\bibitem[Norvai{\v{s}}a(2002)]{norvaisa02}
R.~Norvai{\v{s}}a.
\newblock Chain rules and {$p$}-variation.
\newblock \emph{Studia Math.}, 149\penalty0 (3):\penalty0 197--238, 2002.
\newblock ISSN 0039-3223.
\newblock \doi{10.4064/sm149-3-1}.
\newblock URL \url{http://dx.doi.org/10.4064/sm149-3-1}.

\bibitem[Norvai{\v{s}}a(2001)]{norvaisa}
Rimas Norvai{\v{s}}a.
\newblock {Quadratic variation, p-variation and integration with applications
  to stock price modelling}.
\newblock \emph{ArXiv e-prints}, August 2001.
\newblock URL \url{http://arxiv.org/abs/math/0108090v1}.

\bibitem[Nualart(2009)]{nualart09}
David Nualart.
\newblock \emph{Malliavin calculus and its applications}, volume 110 of
  \emph{CBMS Regional Conference Series in Mathematics}.
\newblock CBMS, Washington, DC, 2009.
\newblock ISBN 978-0-8218-4779-4.

\bibitem[Nutz(2012)]{nutz-int}
Marcel Nutz.
\newblock Pathwise construction of stochastic integrals.
\newblock \emph{Electron. Commun. Probab.}, 17:\penalty0 no. 24, 7, 2012.
\newblock ISSN 1083-589X.
\newblock \doi{10.1214/ECP.v17-2099}.
\newblock URL \url{http://dx.doi.org/10.1214/ECP.v17-2099}.

\bibitem[Pagliarani and Pascucci(2012)]{PagliaraniPascucci2011}
Stefano Pagliarani and Andrea Pascucci.
\newblock Analytical approximation of the transition density in a local
  volatility model.
\newblock \emph{Cent. Eur. J. Math.}, 10(1):\penalty0 250--270, 2012.

\bibitem[Pagliarani et~al.(2013)Pagliarani, Pascucci, and Riga]{ppr}
Stefano Pagliarani, Andrea Pascucci, and Candia Riga.
\newblock Adjoint expansions in local {L}\'evy models.
\newblock \emph{SIAM J. Financial Math.}, 4\penalty0 (1):\penalty0 265--296,
  2013.
\newblock ISSN 1945-497X.
\newblock \doi{10.1137/110858732}.
\newblock URL \url{http://dx.doi.org/10.1137/110858732}.

\bibitem[Pardoux and Peng(1990)]{pardoux-peng90}
{\'Etienne}~Pardoux and Shige Peng.
\newblock Adapted solution of a backward stochastic differential equation.
\newblock \emph{Systems Control Lett.}, 14\penalty0 (1):\penalty0 55--61, 1990.
\newblock ISSN 0167-6911.
\newblock \doi{10.1016/0167-6911(90)90082-6}.
\newblock URL \url{http://dx.doi.org/10.1016/0167-6911(90)90082-6}.

\bibitem[Pardoux and Peng(1992)]{pardoux-peng92}
{\'Etienne}~Pardoux and Shige Peng.
\newblock Backward stochastic differential equations and quasilinear parabolic
  partial differential equations.
\newblock In \emph{Stochastic partial differential equations and their
  applications ({C}harlotte, {NC}, 1991)}, volume 176 of \emph{Lecture Notes in
  Control and Inform. Sci.}, pages 200--217. Springer, Berlin, 1992.
\newblock \doi{10.1007/BFb0007334}.
\newblock URL \url{http://dx.doi.org/10.1007/BFb0007334}.

\bibitem[Pascucci(2011{\natexlab{a}})]{Pascucci2011book}
Andrea Pascucci.
\newblock \emph{P{DE} and martingale methods in option pricing}, volume~2 of
  \emph{Bocconi \& Springer Series}.
\newblock Springer, Milan, 2011{\natexlab{a}}.
\newblock ISBN 978-88-470-1780-1.

\bibitem[Pascucci(2011{\natexlab{b}})]{pascucci}
Andrea Pascucci.
\newblock \emph{P{DE} and martingale methods in option pricing}, volume~2 of
  \emph{Bocconi \& Springer Series}.
\newblock Springer, Milan; Bocconi University Press, Milan, 2011{\natexlab{b}}.
\newblock ISBN 978-88-470-1780-1.
\newblock \doi{10.1007/978-88-470-1781-8}.
\newblock URL \url{http://dx.doi.org/10.1007/978-88-470-1781-8}.

\bibitem[Peng(2012)]{peng}
Shige Peng.
\newblock {Note on Viscosity Solution of Path-Dependent PDE and G-Martingales}.
\newblock \emph{ArXiv e-prints}, February 2012.
\newblock URL \url{http://arxiv.org/abs/1106.1144v2}.

\bibitem[Perkowski(2013)]{perkowski-thesis}
Nicolas Perkowski.
\newblock \emph{Studies of robustness in stochastic analysis and mathematical
  finance}.
\newblock PhD thesis, Humboldt-Universit{\"a}t zu Berlin, 2013.

\bibitem[Perkowski and Pr\"omel(2014)]{perk-promel}
Nicolas Perkowski and David~J. Pr\"omel.
\newblock {Pathwise stochastic integrals for model free finance}.
\newblock \emph{ArXiv}, November 2014.
\newblock URL \url{http://arxiv.org/abs/1311.6187v2}.

\bibitem[Raible(2000)]{Raible2000}
S.~Raible.
\newblock {L}\'evy processes in finance: Theory, numerics, and empirical facts.
\newblock Technical report, PhD thesis, Universit\"at Freiburg, 2000.
\newblock URL \url{http://deposit.ddb.de/cgi-bin/dokserv?idn=961285192}.

\bibitem[{Riedel}(2011)]{riedel}
F.~{Riedel}.
\newblock {Finance Without Probabilistic Prior Assumptions}.
\newblock \emph{ArXiv e-prints}, July 2011.
\newblock URL \url{http://arxiv.org/abs/1107.1078}.

\bibitem[Rogers and Shi(1995)]{rogershi}
L.~C.~G. Rogers and Z.~Shi.
\newblock The value of an {A}sian option.
\newblock \emph{J. Appl. Probab.}, 32\penalty0 (4):\penalty0 1077--1088, 1995.
\newblock ISSN 0021-9002.

\bibitem[Ross(1978)]{ross78}
Stephen~A Ross.
\newblock A simple approach to the valuation of risky streams.
\newblock \emph{Journal of business}, pages 453--475, 1978.

\bibitem[Samuelson(1965)]{sam}
Paul~A. Samuelson.
\newblock Using full duality to show that simultaneously additive direct and
  indirect utilities implies unitary price elasticity of demand.
\newblock \emph{Econometrica}, 33:\penalty0 781--796, 1965.
\newblock ISSN 0012-9682.

\bibitem[Schachermayer(2010{\natexlab{a}})]{schachermayer}
Walter Schachermayer.
\newblock \emph{Fundamental Theorem of Asset Pricing}.
\newblock John Wiley \& Sons, Ltd, 2010{\natexlab{a}}.
\newblock ISBN 9780470061602.
\newblock \doi{10.1002/9780470061602.eqf04002}.
\newblock URL \url{http://dx.doi.org/10.1002/9780470061602.eqf04002}.

\bibitem[Schachermayer(2010{\natexlab{b}})]{schachermayerEMM}
Walter Schachermayer.
\newblock \emph{Equivalent Martingale Measures}.
\newblock John Wiley \& Sons, Ltd, 2010{\natexlab{b}}.
\newblock ISBN 9780470061602.
\newblock \doi{10.1002/9780470061602.eqf04007}.
\newblock URL \url{http://dx.doi.org/10.1002/9780470061602.eqf04007}.

\bibitem[Schied(2014)]{schied-CPPI}
Alexander Schied.
\newblock Model-free {CPPI}.
\newblock \emph{Journal of Economic Dynamics and Control}, 40\penalty0
  (0):\penalty0 84 -- 94, 2014.
\newblock ISSN 0165-1889.
\newblock \doi{http://dx.doi.org/10.1016/j.jedc.2013.12.010}.
\newblock URL
  \url{http://www.sciencedirect.com/science/article/pii/S0165188913002467}.

\bibitem[Schied and Stadje(2007)]{ss}
Alexander Schied and Mitja Stadje.
\newblock Robustness of delta hedging for path-dependent options in local
  volatility models.
\newblock \emph{J. Appl. Probab.}, 44\penalty0 (4):\penalty0 865--879, 2007.
\newblock ISSN 0021-9002.
\newblock \doi{10.1239/jap/1197908810}.
\newblock URL \url{http://dx.doi.org/10.1239/jap/1197908810}.

\bibitem[Sondermann(2006)]{sondermann}
Dieter Sondermann.
\newblock \emph{Introduction to stochastic calculus for finance}, volume 579 of
  \emph{Lecture Notes in Economics and Mathematical Systems}.
\newblock Springer-Verlag, Berlin, 2006.

\bibitem[Stroock and Varadhan(1979)]{str-var}
Daniel~W. Stroock and S.R.~Srinivasa Varadhan.
\newblock \emph{Multidimensional Diffusion Processes}.
\newblock A Series of Comprehensive Studies in Mathematics. Springer-Verlag,
  Berlin Heidelberg New York, 1979.

\bibitem[Tao(2011)]{tao}
Terence Tao.
\newblock \emph{An introduction to measure theory}, volume 126.
\newblock American Mathematical Soc., 2011.

\bibitem[Taqqu and Willinger(1987)]{willtaq87}
Murad~S Taqqu and Walter Willinger.
\newblock The analysis of finite security markets using martingales.
\newblock \emph{Advances in Applied Probability}, pages 1--25, 1987.

\bibitem[Vovk(2008)]{vovk-vol}
Vladimir Vovk.
\newblock Continuous-time trading and the emergence of volatility.
\newblock \emph{Electron. Commun. Probab.}, 13:\penalty0 319--324, 2008.
\newblock ISSN 1083-589X.

\bibitem[Vovk(2011{\natexlab{a}})]{vovk-cadlag}
Vladimir Vovk.
\newblock {Ito calculus without probability in idealized financial markets}.
\newblock \emph{ArXiv e-prints}, August 2011{\natexlab{a}}.
\newblock URL \url{http://arxiv.org/abs/1108.0799v2}.

\bibitem[Vovk(2011{\natexlab{b}})]{vovk-rough}
Vladimir Vovk.
\newblock Rough paths in idealized financial markets.
\newblock \emph{Lith. Math. J.}, 51\penalty0 (2):\penalty0 274--285,
  2011{\natexlab{b}}.
\newblock ISSN 0363-1672.

\bibitem[Vovk(2012)]{vovk-proba}
Vladimir Vovk.
\newblock Continuous-time trading and the emergence of probability.
\newblock \emph{Finance Stoch.}, 16\penalty0 (4):\penalty0 561--609, 2012.
\newblock ISSN 0949-2984.

\bibitem[Widdicks et~al.(2005)Widdicks, Duck, Andricopoulos, and
  Newton]{WiddicksDuckAndricopoulosNewton2005}
Martin Widdicks, Peter~W. Duck, Ari~D. Andricopoulos, and David~P. Newton.
\newblock The {B}lack-{S}choles equation revisited: asymptotic expansions and
  singular perturbations.
\newblock \emph{Math. Finance}, 15\penalty0 (2):\penalty0 373--391, 2005.
\newblock ISSN 0960-1627.
\newblock \doi{10.1111/j.0960-1627.2005.00224.x}.
\newblock URL \url{http://dx.doi.org/10.1111/j.0960-1627.2005.00224.x}.

\bibitem[Willinger and Taqqu(1989)]{willtaq}
Walter Willinger and Murad~S. Taqqu.
\newblock Pathwise stochastic integration and applications to the theory of
  continuous trading.
\newblock \emph{Stochastic Process. Appl.}, 32\penalty0 (2):\penalty0 253--280,
  1989.
\newblock ISSN 0304-4149.
\newblock \doi{10.1016/0304-4149(89)90079-3}.
\newblock URL \url{http://dx.doi.org/10.1016/0304-4149(89)90079-3}.

\bibitem[Wong and Zakai(1965)]{wong-zakai}
Eugene Wong and Moshe Zakai.
\newblock On the convergence of ordinary integrals to stochastic integrals.
\newblock \emph{The Annals of Mathematical Statistics}, pages 1560--1564, 1965.

\bibitem[Wright(1954)]{wright}
EM~Wright.
\newblock An inequality for convex functions.
\newblock \emph{American Mathematical Monthly}, pages 620--622, 1954.

\bibitem[Xu and Zheng(2010)]{XuZheng2010}
Guoping Xu and Harry Zheng.
\newblock Basket options valuation for a local volatility jump-diffusion model
  with the asymptotic expansion method.
\newblock \emph{Insurance Math. Econom.}, 47\penalty0 (3):\penalty0 415--422,
  2010.
\newblock ISSN 0167-6687.
\newblock \doi{10.1016/j.insmatheco.2010.08.008}.
\newblock URL \url{http://dx.doi.org/10.1016/j.insmatheco.2010.08.008}.

\bibitem[Young(1936)]{young36}
L.~C. Young.
\newblock An inequality of the {H}\"older type, connected with {S}tieltjes
  integration.
\newblock \emph{Acta Math.}, 67\penalty0 (1):\penalty0 251--282, 1936.
\newblock ISSN 0001-5962.
\newblock \doi{10.1007/BF02401743}.
\newblock URL \url{http://dx.doi.org/10.1007/BF02401743}.

\bibitem[Young(1938)]{young38}
L.~C. Young.
\newblock General inequalities for {S}tieltjes integrals and the convergence of
  {F}ourier series.
\newblock \emph{Math. Ann.}, 115\penalty0 (1):\penalty0 581--612, 1938.
\newblock ISSN 0025-5831.
\newblock \doi{10.1007/BF01448958}.
\newblock URL \url{http://dx.doi.org/10.1007/BF01448958}.

\end{thebibliography}

\end{document}